\documentclass[reqno]{amsart}
\usepackage{amsmath,amsthm,amssymb,mathrsfs}
\usepackage{latexsym}
\usepackage{eucal}

\def\cal{\mathcal}

\newtheorem{theorem}{Theorem}[section]
\newtheorem{lemma}{Lemma}[section]
\newtheorem{corollary}{Corollary}[section]
\newtheorem{remark}{Remark}[section]

\theoremstyle{definition}
\newtheorem{definition}{Definition}[section]
\newtheorem{notation}{Notation}

\let\Re=\undefined
\DeclareMathOperator{\Re}{Re}
\let\Im=\undefined
\DeclareMathOperator{\Im}{Im}
\DeclareMathOperator{\A}{\mathfrak{A}}
\DeclareMathOperator{\B}{\mathfrak{B}}
\DeclareMathOperator{\f}{\mathfrak{f}}

\let\a=\undefined
\let\b=\undefined
\DeclareMathOperator{\a}{\mathfrak{a}}
\DeclareMathOperator{\b}{\mathfrak{b}}
\DeclareMathOperator{\g}{\mathfrak{g}}

\begin{document}
\begin{titlepage}
\title[Continuous analogs of polynomials \ldots]{Continuous analogs of
polynomials orthogonal on the unit circle. Krein systems}

\author{Sergey A. Denisov}

\maketitle

\end{titlepage}
\,{\quad}\vspace{0.5cm}

\centerline{\it Dedicated to the centenary of Mark
Krein}\vspace{1.5cm}
\begin{center}
{\bf Preface}
\end{center}
In the recent years, the theory of orthogonal polynomials on the
real line (OPRL) and on the unit circle (OPUC) enjoyed the
considerable development. In these lecture notes, we will explain
how to construct the continuous analogs of polynomials orthogonal
on the unit circle. It is possible to built a theory which is as
rigorous and complete as the theory for OPUC. Spectral theories of
one-dimensional Dirac and Schrodinger operators can then be viewed
in the framework of this theory which establishes a solid link
between an approximation theory and quantum mechanics.

 The theory is based on the ideas suggested by M.G. Krein.
 They were developed later by various authors, especially from Krein's school.
 In the meantime, new results were obtained for OPUC and OPRL and
 that was a motivation for us to try to understand their
 continuous analogs. We also try to give systematic exposition of
 the theory but have to refer to the literature once in a while.
 Also, these notes do not cover some aspects of the theory (e.g.
 continuation problems for $G_r$ classes, regularity of
 coefficients, etc.) but we give necessary references. In general, our objective is to give only basics
 of the theory by presenting complete proofs and filling various
 gaps present in the current literature. As a prerequisite for
 reading these notes, we assume that a reader is familiar with
 main facts from the OPUC theory (see, e.g. \cite{Szego, Geronimus, Simon}). The knowledge
 of spectral theory for Schr\"odinger and Dirac operators might also
 be very helpful.

\begin{center}
\bf What is not covered and what is new?
\end{center}

 We didn't include the following subjects
that are related to our topic: solution to the continuous analogs
of Schur and Caratheodory-Toeplitz problems \cite{krein-madamyan,
Krein1}. We also do not discuss matrix-valued version of the
theory. For the recent progress on more specific questions (such
as continuous analog of Szeg\H{o} case, Rakhmanov's Theorem, etc.)
we suggest the reader to consult the journal publications, e.g.
\cite{Tepl, Sylv, Den-Rakhm, Den-ieop, Den-Kup}. Also, we will
deal with rather regular classes of coefficients (not worse than
$L^2_{\rm loc}(\mathbb{R}^+)$) but the general case can also be
treated in the framework of different differential operators (see,
e.g., \cite{melik-adamyan1, melik-adamyan2, Albeverio}).

In these notes, we present quite a few new results. That includes:
approximation of continuous orthogonal system by the sequence of
the discrete ones (Section 8), distribution of zeroes (Section 9),
new criteria for $A(r)\in L^2(\mathbb{R}^+)$ and more on that case
(Section 10), the continuous analog of the Strong Szeg\H{o}
Theorem -- sharp conditions (Section 14). We also gave complete
proofs for results that were present in the literature without any
proofs and gave alternative (hopefully, more transparent) proofs
for several other statements (e.g. the continuous analog of
Baxter's Theorem, scattering theory for Krein systems and Dirac
operators).

\begin{center}
\bf Acknowledgements
\end{center}
These notes are based on the graduate course given at Caltech in
Fall 2001. We are indebted to B. Simon and P. Deift for their
support and encouragement to prepare these lectures. Thanks are
due to A. Teplyaev and L. Sakhnovich for their help and insightful
remarks. This work was supported by NSF grant DMS-0500177, Alfred
P. Sloan Research Fellowship, and Oswald Veblen Fund during the
stay at the Institute for Advanced Study, Princeton, NJ. Finally,
it is my pleasure to dedicate these notes to the centenary of Mark
Grigorievich Krein (April 2007) who was the founder of this
theory.

\begin{center}
\bf Notations \end{center}
\begin{itemize}
\item[$\mathbb{D}$] -- open unit disc in $\mathbb{C}$

\item[$\mathbb{T}$] -- unit circle in $\mathbb{C}$

\item[$H^{p}(\Omega)$] -- Hardy space in the domain $\Omega$,
$0<p\leq \infty$

\item[$N(\Omega)$] -- Nevanlinna class of analytic functions
in $\Omega$

\item[$B(\Omega)$] -- closed unit ball in $H^{\infty }(\Omega)$

\item[$\cal{P}_{\Delta}$] -- the following projection:
\[
\cal{P}_{\Delta} \left[ \int\limits_{-\infty}^\infty \exp(i\lambda
x) f(x)dx\right] =\int\limits_{\Delta} \exp(i\lambda x)f(x)dx
\]
$H^2_{[0,R]}=$Ran$\cal{P}_{[0,R]}$

\item[$\cal{P}_{\pm }$] -- denote $\cal{P}_{\mathbb{R}^{\pm}}$ respectively

\item[$|O|=\sqrt{O^{\ast }O}$] -- absolute value of operator $O$

\item[$\delta(x)$] -- the delta-function at zero

\item[$W(\mathbb{R})$] -- Wiener's Banach algebra of
functions
\[
\hat{f}(\lambda)=\int\limits_{-\infty}^\infty f(x)\exp(i\lambda
x)dx, f(x)\in L^1(\mathbb{R})
\]

\item[$W_+(\mathbb{C}^+)$] -- Banach algebra of functions
\[
\hat{f}(\lambda)=\int\limits_{0}^\infty f(x)\exp(i\lambda x)dx,
f(x)\in L^1(\mathbb{R}^+)
\]

\item[$H^{1/2}(\mathbb{R})$] -- fractional Sobolev space of
functions $f$ whose Fourier transform satisfies
\[
\int\limits_{-\infty}^\infty <t> |\hat{f}(t)|^2dt<\infty,
<t>=\sqrt{t^2+1}
\]

\item[$C_0(\Omega)$] -- continuous on $\overline{\Omega}$
functions vanishing at the boundary of $\Omega$

\item[$\chi_\Delta (x)$] -- the characteristic
function of the set $\Delta$

\item[$\Pi_{\Delta}$] -- the orthogonal projection in
$L^2(\mathbb{R}^+)$ onto $L^2(\Delta), \Delta\subset \mathbb{R}$

\item[$\cal{S}^p$] -- Schatten-Von Neumann class of compact operators

\item[$\ln^+ x$] =\,$\ln x$ if $x\geq 1$ and $=0$ for $0<x<1$
\item[$\ln^- x$] =\,$\ln x$ if $0<x<1$ and $=0$ for $x>1$
\item[$f*g$] -- means the convolution of $f$ and $g$

\end{itemize}

We usually use calligraphic  letters to distinguish between
operators and functions. For example, $\cal{O}$ stands for an
operator and $O(x,y)$ denotes the function of two variables.

\newpage
\tableofcontents

\newpage

\section{Some classes of functions on the real line}\label{S:1}

In this section, we recall some basic facts on positive definite
functions on the real line. Then, we introduce certain class of
functions that we will use later on.

Let $0<r\leq \infty $.

\begin{definition}
The  Lebesgue-measurable function $\phi(x)$ defined on the
interval $(-r,r)$ is called Hermitian if $\phi (-x)=\overline{\phi
(x)}$ for a.e. $x$.
\end{definition}

\begin{definition}
The Lebesgue-measurable function of two variables $K(x,y)$ is
called Hermitian if
\begin{equation}
K(x,y)=%
\overline{K(y,x)}, \label{e1s0}
\end{equation}
for a.e. $0<x,y<r$.
\end{definition}
\begin{definition}
The integral kernel $K(x,y)$ is called positive definite
on $[0,r]$ if for any $N,$ $\{x_{j}\}_{j=1}^{N},(x_{j}\in \lbrack 0,r]),$ $%
\{c_{j}\}_{j=1}^{N},(c_{j}\in \mathbb{C})$, we have inequality
\begin{equation}
\sum\limits_{n,m=1}^{N}c_{n}\overline{c}_{m}K(x_{n},x_{m})\geq 0.
\label{in1}
\end{equation}
\end{definition}

Consider the integral operator ${\cal K}$ in $L^{2}[0,r]$ with
kernel $K(x,y)\in C([0,r]^2)$, i.e.
\[
({\cal K}f)(x)=\int\limits_{0}^{r}K(x,y)f(y)dy.
\]
Clearly, (\ref{e1s0}) means ${\cal K}^{\ast }={\cal K}$. It is an
easy exercise to see that if the continuous kernel is Hermitian,
then (\ref{in1}) is equivalent to ${\cal K}\geq 0$, where
inequality is understood in the operator sense.
\begin{definition}
A function $\phi (x)$ is called positive definite if the
integral kernel $\phi (x-y)$ is positive definite on $\mathbb{R^{+}}.$
\end{definition}
This is equivalent to $\phi(x-y)$ being positive on the whole line
$\mathbb{R}$.

\begin{notation}
The class of continuous positive definite functions on
the whole line is denoted by $P_{\infty }.$
\end{notation}

If $d\mu $ is finite positive measure on $\mathbb{R},$ then
\begin{equation}
\phi (x)=\int\limits_{-\infty }^{\infty }\exp(ixt)d\mu (t)  \label{in2}
\end{equation}
is positive definite. The classical result of Bochner says that
the converse statement is also true. That, in a sense, is the
continuous analog of the solution to trigonometric moment problem.

\begin{theorem}\emph{(Bochner, \cite{Akhiezer1})} A function $\phi$ belongs
to the class $P_{\infty}$ if and only if it admits the
representation (\ref{in2}) with finite positive measure $\mu$. The
measure $\mu$ in this representation is unique.
\end{theorem}

Notice that Bochner's theorem implies that all $P_{\infty }$
functions are necessarily bounded on $\mathbb{R}$.

\begin{notation}
Let $G_{\infty }$ denote the class of continuous Hermitian
functions $g(x)$ defined on the whole line such that $g(0)=0$ and
the integral kernel
\[
K(x,y)=g(x)+g(-y)-g(x-y)
\]%
is positive definite on $\mathbb{R^{+}}$, i.e. on any interval
$[0,r], r>0$.
\end{notation}

Next, we obtain some rather crude estimates on  $g\in G_{\infty
}.$ Later, these bounds will be used to prove the integral
representation for functions of class $G_\infty$.

The following inequality holds
\begin{equation}
|g(2x)|\leq 8|g(x)|  \label{zz}
\end{equation}%
for any $x$. Indeed, since the kernel $K(x,y)$ is positive
definite, estimate (\ref{in1}) is true. Take $N=2,\ x_{1}=x,\
x_{2}=2x$. If $\ c_{1}=\xi \in \mathbb{R},c_{2}=1$, we have
\[
\Re (g(x))\xi ^{2}+\Re (g(2x))\xi +\Re (g(2x))\geq 0
\]
Since $\xi $ is arbitrary real,
\begin{equation}
|\Re (g(2x))|\leq 4\Re g(x)  \label{z1}
\end{equation}
For the same choice of $x_{1(2)}$, we let $c_{1}=i\xi \in i\mathbb{R}, c_{2}=1.$ Then, %
\[
\Re (g(x))\xi ^{2}-[2\Im (g(x))-\Im (g(2x))]\xi +\Re (g(2x))\geq 0
\]%
That yields
\[
\left[ \Im (g(2x))-2\Im (g(x))\right] ^{2}\leq 4\Re g(x)\Re
(g(2x))
\]%
Using (\ref{z1}), we have%
\[
|\Im (g(2x))|\leq 2|\Im g(x)|+4\Re g(x)
\]%
Combining estimates for the real and imaginary parts, we obtain (\ref{zz}).

The following estimate holds true
\begin{equation}
|g(x)|<C(1+|x|^{3})  \label{pl}
\end{equation}
Indeed, if $\max_{x\in [-1,1]}|g(x)|=M$, then $%
\max_{[-2^{n},2^{n}]}|g(x)|\leq 8^{n}M$ by (\ref{zz}). Therefore, $%
|g(x)|\leq 8^{[\log _{2}|x|]+1}M\leq C|x|^{3},$ where $[.]$ means the
integer part of the a number. As we will see later, the estimate (\ref{pl})
is very far from
optimal.

The following integral representation of $G_{\infty }$ functions is an
analog of Bochner's theorem for class $G_{\infty }.$

\begin{theorem}
Function $g(x)\in G_{\infty }$ if and only if
\begin{equation}
g(x)=i\beta x+\int\limits_{-\infty }^{\infty }\left( 1+\frac{i\lambda x}{%
1+\lambda ^{2}}-\exp (i\lambda x)\right) \frac{d\sigma (\lambda )}{\lambda
^{2}} \label{e2s0}
\end{equation}
where $\beta \in \mathbb{R}$ and positive measure $\sigma$
satisfies
the
estimate
\begin{equation}
\int\limits_{-\infty }^{\infty }\frac{d\sigma (\lambda )}{1+\lambda ^{2}}%
<\infty .\label{e2t1}
\end{equation}
Constant $\beta $ and measure $\sigma $ are
uniquely defined. \label{t1s1}
\end{theorem}

\begin{proof} Any functions of the form (\ref{e2s0}) belongs to $G_{\infty
}$. Indeed, notice that the integral in (\ref{e2s0}) converges if
(\ref{e2t1}) holds and defines the continuous function that
vanishes at zero. Then, we have the following representation
\[
g(x)+g(-y)-g(x-y)=\int\limits_{-\infty}^{\infty}
\frac{(1-\exp(i\lambda x))\overline{(1-\exp(i\lambda
y))}}{\lambda^2} d\sigma(\lambda)
\]
which ensures the positivity of the operator with the
corresponding kernel for any $r>0$.

Conversely, due to (\ref{pl}), any $G_{\infty
}$ function allows Laplace transform. Consider
\[
L(z)=z^{2}\int\limits_{0}^{\infty }g(x)\exp (ixz)dx
\]%
This function is analytic in $\mathbb{C}^+$.
 Notice that
\[
\frac{L(z)+\overline{L(z)}}{2\Im z}=|z|^{2}\int\limits_{0}^{\infty
}\int\limits_{0}^{\infty }[g(x-y)-g(x)-g(-y)]\exp (ixz-iy\overline{z}%
)dxdy\leq 0
\]%
Consequently, $-iL(z)$ is Herglotz function and has well-known
integral representation (\cite{AkhGlaz}, chapter 6) which gives
\[
L(z)=i\alpha z-i\beta -i\int\limits_{-\infty }^{\infty }\frac{1-\lambda z}{%
(\lambda +z)(1+\lambda ^{2})}d\sigma (\lambda )
\]%
where%
\[
\int\limits_{-\infty }^{\infty }\frac{d\sigma (\lambda )}{1+\lambda ^{2}}%
<\infty
\]%
and $\alpha \geq 0,\beta \in
\mathbb{R}
.$ Let us take the inverse Laplace transform. Notice that
\begin{equation}
\frac{\lambda ^{2}(1-\lambda z)}{i(\lambda +z)(1+\lambda ^{2})}%
=z^{2}\int\limits_{0}^{\infty }\left[ 1+\frac{i\lambda x}{1+\lambda ^{2}}%
-\exp (i\lambda x)\right] \exp (ixz)dx,z\in
\mathbb{C}^{+},\lambda\in \mathbb{R} \label{e1s1}
\end{equation}
Therefore,
\[
L(z)=i\alpha z-i\beta +z^{2}\int\limits_{0}^{\infty }\exp
(ixz)\int\limits_{-\infty }^{\infty }\left[ 1+\frac{i\lambda x}{1+\lambda
^{2}}-\exp (i\lambda x)\right] \frac{d\sigma (\lambda )}{\lambda^{2}}dx
\]%
Since
\begin{equation}
1=-z^{2}\int\limits_{0}^{\infty }x\exp (ixz)dx,\
z=-iz^{2}\int\limits_{0}^{\infty }\exp (ixz)dx,\ z\in
\mathbb{C}^{+} \label{laplace-1}
\end{equation}
we have the formula%
\[
g(x)=\alpha +i\beta x+\int\limits_{-\infty }^{\infty }\left[ 1+\frac{%
i\lambda x}{1+\lambda ^{2}}-\exp (i\lambda x)\right] \frac{d\sigma (\lambda
)%
}{\lambda ^{2}}
\]%
Due to normalization $g(0)=0,$ $\alpha =0$. Uniqueness follows
from the uniqueness of the Herglotz function representation.
\end{proof}

As a simple corollary one gets the following improvement of
(\ref{pl}): $|g(x)|<C(1+x^2).$ If $\sigma$ in (\ref{e2s0}) is
Heaviside function and $\beta=0$, then $g(x)=x^2/2$. So, the
quadratic growth is possible.

\bigskip If the support of $\sigma$ is a compact, then the second derivative
of $g$ exists and is positive definite. That follows from the
Bochner's theorem. In general, the derivatives of $g(x)$ at zero
can have singularities. In what follows, the reference measure is
$\sigma (\lambda )=\lambda/(2\pi) ,$ and $\beta =0$. It is then
easy to check that $g(x)=|x|/2.$ So, in this case, the second
derivative in the distributional sense is delta function.

The relation between $G_{\infty }$ and $P_{\infty }$ can be
established by

\begin{lemma}
If $f(x)\in P_{\infty }$, then $%
f(0)-f(x)\in G_{\infty }$.
\end{lemma}

\begin{proof} From Bochner's theorem, we have%
\[
f(x)=\int\limits_{-\infty }^{\infty }\exp (ixt)d\mu (t)
\]%
Take $\displaystyle\sigma (\lambda )=\int\limits_{0}^{\lambda
}t^{2}d\mu (t)$, $\displaystyle \beta =-\int\limits_{-\infty
}^{\infty }t(1+t^{2})^{-1}d\mu (t)$ and use Theorem
2.%
\end{proof}

The converse is wrong. Take $f(x)=ix.$ The corresponding kernel
$K(x,y)=0.$ At the same time, $f(x)\notin P_{\infty }$ because
$f(x)$ is not bounded. The class $G_\infty$ is convenient for
description of measures generating one very important class of
canonical differential systems, the Krein systems, which we plan
to study in the next sections.
\bigskip

{\bf Remarks and Historical Notes.}

Classes $P_{r},G_{r}$ can be introduced for finite $r>0$ (see
\cite{Akhiezer1}, p. 190 and references there). Then, measures
$\mu $ and $\sigma $ are not uniquely defined in general. The
class of nonuniqueness for $\sigma $ corresponds to the
continuation problems \cite{Krein1}. The case when quadratic form
(\ref{in1}) has not more than $\varkappa $ negative squares is
more difficult. It was studied in the framework of Pontryagin $\Pi
_{\varkappa }$-spaces \cite{Krein1}.

\newpage

\section{Factorization of integral operators}\label{sect-fact}
To understand better the algebraic aspects of the theory, we will
need some rather simple results on the factorization of integral
operators.  As we know from the linear algebra, given any matrix
$A=\{a_{ij}\}_{i,j=1}^{d}$ with nonzero leading principal minors,
we can always find a lower-triangular matrix $X_{1}$ and an
upper-triangular
matrix $%
X_{2}$, such that $A=X_{1}DX_{2},D$ is a diagonal matrix. The
proof is simple. Since $a_{11}\neq 0$, using the first step of
Gauss algorithm, we can find the lower-triangular matrix $L_{1}$
of elementary transforms, such that $L_{1}A$ has the first column
collinear to $[1,0,\ldots, 0]^{t}$. Then, find an upper-triangular
$U_{1}$ such that the matrix $L_{1}AU_{1}$ has the first raw
collinear to $[1,0,\ldots, 0]$. Notice that the first column stays
the same. The leading principle minors of $L_{1}AU_{1}$ are the
same as those of $A$.
Thus, we can continue this process. In the end, we get a lower-triangular $%
L=L_{d}\ldots L_{1}$ and an upper-triangular $U=U_{1}\ldots U_{d}$
such that
$%
LAU=D$. Denoting $X_{1}=L^{-1}$, $X_{2}=U^{-1},$ we get the
desired statement. Notice that matrices $X_1$ and $X_2$ have $1$
on the diagonal. Taking inverse of $A$, we get the factorization
in the reverse order.

Now, what happens in the continuous case? First, we need to
establish one general result about the resolvent kernels of
integral operators. Fix some $R>0$. For any $0<r\leq R$, consider
integral operators
\[
\cal{K}_rf(x)=\int\limits_{0}^{r}K(x,y)f(y)dy,
\]
with a kernel $K(x,y)$, continuous on  $[0,R]^2$, and acting in
the Hilbert space $L^2[0,r]$. Assume $-1\notin Spec(\cal{K}_r)$
for any $0<r\leq R$. Then, the resolvent
kernel $%
\Gamma_r (s,t)$ exists and

\begin{equation}
\Gamma_r(s,t)+\int\limits_0^r K(s,u) \Gamma_r(u,t)du =K(s,t),
0<s,t<r \label{e9s2}
\end{equation}

\begin{lemma}
Let $K(x,y)\in C([0,R]^2)$ and $-1\notin Spec(\cal{K}_r)$ for any
$0<r\leq R$. Then,
\begin{eqnarray*}
\partial \Gamma _{r}(s,t)/\partial r=-\Gamma _{r}(s,r)\Gamma
_{r}(r,t),\, 0\leq s,t\leq r
\end{eqnarray*}%
Function  $\Gamma _{r}(s,t)$ is jointly continuous in $s,t,r$ and
is continuously differentiable in~$r$. \label{lemma23}
\end{lemma}

\begin{proof}  Notice that Fredholm's formula for resolvent kernel
\begin{equation}
\Gamma_r(s,t)=\frac{\delta_r(s,t)}{\delta_r}
\end{equation}
ensures that $\Gamma_r(s,t)$ is jointly continuous in $s,t,r$ and
is continuously differentiable in $r, r>0$ (Appendix,
Lemma~\ref{fredholm}).

Then, differentiate (\ref{e9s2}) in $r$. We have
\begin{equation}
\frac{\partial \Gamma_r(s,t)}{\partial r} +\int\limits_0^r K(s,u) \frac{%
\partial \Gamma_r(u,t)}{\partial r} du=-K(s,r)\Gamma_r(r,t)
\end{equation}
On the other hand, multiplying both sides of
\begin{equation}
\Gamma_r(s,r)+\int\limits_0^r K(s,u) \Gamma_r(u,r)du =K(s,r)
\end{equation}
by $-\Gamma_r(r,t)$, we obtain an equation
\begin{equation}
-\Gamma_r(s,r) \Gamma_r(r,t)-\int\limits_0^r K(s,u) \Gamma_r(u,r)
\Gamma_r(r,t) du =-K(s,r)\Gamma_r(r,t)
\end{equation}
Now we see that both ${\partial \Gamma_r(s,t)}/{\partial r}$ and $-\Gamma_r(s,r)%
\Gamma_r(r,t)$ satisfy the same integral equation. Therefore, they
are equal. \end{proof}

Although the continuous kernels are very natural, we will also
work with $K(x,y)$ from the following class.
\begin{definition}
Function $K(x,y)$ belongs to the class $\hat{C}([0,R]^2)$, if it
is continuous in each of the triangles: $\Delta_+=\{0\leq x\leq
y\leq R\}$ and $\Delta_-=\{0\leq y\leq x\leq R\}$ but might have
discontinuity on the diagonal $x=y$ if considered as the function
on $[0,R]^2$ (i.e. the limits $K_+(x,x)$ and $K_-(x,x)$ might be
different).
\end{definition}

 For this case, we
have an analogous statement

\begin{lemma}
Let $K(x,y)\in \hat{C}([0,R]^2)$ and $-1\notin Spec(\cal{K}_r)$
for any $0<r\leq R$. Then,
\begin{equation}\label{der-r2}
\partial \Gamma _{r}(s,t)/\partial r=-\Gamma _{r}(s,r)\Gamma
_{r}(r,t),\, s,t\in \Delta_{\pm}
\end{equation}
Function  $\Gamma _{r}(s,t)\in \hat{C}([0,r]^2)$.   It is
continuously differentiable in $r$ for $s,t\in \Delta_{\pm}$
 \label{lemma-discont} and the derivative $\partial \Gamma _{r}(s,t)/\partial r$ is also from
 $\hat{C}([0,r]^2)$.
\end{lemma}
\begin{proof}
The proof is similar to the proof of Lemma \ref{lemma23}. Instead
of the usual Fredholm formula for resolvent kernel we need to use
a modified one given in Lemma~\ref{fredholm-mod}. Indeed, analysis
of $\hat\delta_r$ shows that it is  continuously differentiable.
The Carleman-Hilbert determinant $\hat\delta_r(s,t)$ has
derivative in $r$ for $s,t$ fixed in each of $\Delta_\pm$. This
derivative is also from the class $\hat{C}([0,r]^2)$.
\end{proof}
Notice that $\partial \Gamma _{r}(s,t)/\partial r$ can be regarded
as a continuous function on the whole $[0,r]^2$ because the
right-hand side of (\ref{der-r2}) is continuous on $[0,r]^2$.

The natural analog of the lower-triangular matrix is Volterra
integral operator
\[
\cal{L}f(x)=\int\limits_{0}^{x}L(x,y)f(y)dy
\]
acting on the Hilbert space $L^2[0,R]$. We  also assume that
$L(x,y)\in C(\Delta_-)$. An operator $\cal{U}$ is
upper-triangular, if $\cal{U}^*$ is lower-triangular. The product
and the sum of two lower(upper)-triangular operators are
lower(upper)-triangular as well. Operators $I+\cal{L}$ and
$I+\cal{U}$ are both invertible and $(I+\cal{L})^{-1}-I$ is
lower-triangular,  $(I+\cal{U})^{-1}-I$ is upper-triangular.
Infact, there is the Banach algebra of lower(upper)-triangular
operators \cite{Gohberg, GohKr}. Assuming that we have
factorization
\begin{equation}
 I+\cal{K}_R= (I+\cal{L})(I+\cal{U}) \label{factorization}
\end{equation}
  where
$\cal{L}$ is lower-triangular and $\cal{U}$ is upper-triangular,
we immediately get that $I+\cal{K}_R$ is invertible and
$(I+\cal{K}_R)^{-1}=I-\cal{G}_R
=(I+\cal{V}_{+})(I+\cal{V}_{-})$, where $%
\cal{V}_{\pm }$ is upper(lower)-triangular with kernels
$V_{\pm}(x,y), 0<x,y<R$ and we have the following formula for the
resolvent kernel

\begin{equation}
\Gamma_R (x,y)=\left\{
\begin{array}{cc}
\displaystyle -V_{+}(x,y)-\int\limits_{y}^{R}V_{+}(x,u)V_{-}(u,y)du, & x<y \\
\displaystyle -V_{-}(x,y)-\int\limits_{x}^{R}V_{+}(x,u)V_{-}(u,y)du, & x>y%
\end{array}%
\right. \label{factor}
\end{equation}
It should also be mentioned that $\cal{K}_r, (0<r<R)$ can be
factorized just by using truncations of $\cal{L}$ and $\cal{U}$.
\begin{theorem}
The integral operator $\cal{K}_R$ with kernel $K(x,y)\in
\hat{C}([0,R]^2)$ admits factorization (\ref{factorization}) if
and only if
 $I+\cal{K}_r$ is invertible in $L^2[0,r]$ for any $0<r<R$. In
this case,
\begin{equation}
\begin{array}{ccc}
V_{+}(x,y) &=&-\Gamma _{y}(x,y),\ x<y \\
V_{-}(x,y) &=&-\Gamma _{x}(x,y),\ x>y
\end{array}
\label{kernels}
\end{equation}
where $\Gamma_r(x,y)$ denotes the resolvent kernel of
$I+\cal{K}_r$. \label{volter}
\end{theorem}

\begin{proof}
 Indeed, assume that $I+\cal{K}_r$ is invertible for any
$0<r\leq R$ and $\Gamma _{r}(x,y)$ is the resolvent kernel. Define
$V_{\pm }$ by (\ref{kernels}). Now, let us check (\ref{factor}).
Indeed, from Lemma \ref{lemma-discont},
\[
\Gamma _{y}(x,y)-\int\limits_{y}^{R}\Gamma _{u}(x,u)\Gamma
_{u}(u,y)du=
\]
\[
=\Gamma _{y}(x,y)+\int\limits_{y}^{R}\frac{\partial }{\partial u}%
\Gamma _{u}(x,y)du =\Gamma _{R}(x,y),0<x<y<R
\]%
The case $x>y$ can be checked in the same way. Thus, we have
(\ref{factor}), which means that
$I-\cal{G}_R=(I+\cal{V}_+)(I+\cal{V}_-)$. Now,
(\ref{factorization}) is straightforward. Conversely, assume that
the factorization (\ref{factorization}) exists.  Then
\[
I+\cal{K}_r=\Pi_{[0,r]}
(1+\cal{L})(1+\cal{U})\Pi_{[0,r]}=\Pi_{[0,r]}
(1+\cal{L})\Pi_{[0,r]}\Pi_{[0,r]}(1+\cal{U})\Pi_{[0,r]}
\]
 and it
is clearly invertible. \end{proof}

Another important class of factorizations is the following one.
Instead of integral operator on $L^2[0,R]$ we consider an integral
operator on $L^2[-R,R]$ and define the lower-triangular operator
 as
\[
\hat{\cal{L}}f(x)=\int\limits_{-|x|}^{|x|} \hat L(x,y)f(y)dy
\]
where $\hat{L}(x,y)$ is continuous in $\{ 0\leq x\leq R, |y|\leq
x\}$ and in $\{-R\leq x\leq 0, |y|\leq |x|\}$. Introduce
$\hat\Omega_-=\{|y|\leq |x|\leq R\}$, $\hat\Omega_+=\{|x|\leq
|y|\leq R\}$ and redefine $\hat\Delta_-=\{-R\leq y\leq x\leq R\}$,
$\hat\Delta_+=\{-R\leq x\leq y\leq R\}$.

Similarly, we  say that $\hat{\cal{U}}$ is upper-triangular if
$\hat{\cal{U}}^*$ is lower-triangular. These newly defined
lower-triangular operators possess the same algebraic properties:
sum and product of two lower-triangular operators is
lower-triangular, $(I+\hat{\cal{L}})^{-1}-I$ exists and is
lower-triangular. The same is true about the upper-triangular
operators. In general, the definition of  lower(upper)-triangular
operator depends on the choice of the so-called chain of
orthoprojectors (\cite{GohKr}, Chapter 4).

The natural question is when can we factor the operator
\begin{equation}\label{lower-upper}
I+\hat{\cal{K}}_R=(I+\hat{\cal{L}})(I+\hat{\cal{U}})
\end{equation}
where $\hat{\cal{K}}_r$ is defined as
\[
\hat{\cal{K}}_rf=\int\limits_{-r}^r \hat K(x,y)f(y)dy, 0<r\leq R
\]
We have
\begin{lemma}
Let $\hat K(x,y)\in \hat{C}([-R,R]^2)$ and $-1\notin
Spec(\hat{\cal{K}}_r)$ for any $0<r\leq R$. Then,
\begin{eqnarray*}
\partial \hat\Gamma _{r}(s,t)/\partial r=-\Bigl[\hat\Gamma
_{r}(s,r)\hat\Gamma
_{r}(r,t)+\hat\Gamma_r(s,-r)\hat\Gamma_r(-r,t)\Bigr],\, s,t\in
\hat{\Delta}_{\pm}
\end{eqnarray*}%
The function  $\hat\Gamma _{r}(s,t)\in \hat{C}([-r,r]^2)$.   It is
continuously differentiable in $r$ for $s,t\in \hat\Delta_{\pm}$
and the derivative $\partial \hat\Gamma _{r}(s,t)/\partial r$ is
also from
 $\hat{C}([-r,r]^2)$.
\label{lemma-discont-1}
\end{lemma}
\begin{proof}
We will only check the formula for derivative. The other
properties can be checked just like in Lemma \ref{lemma-discont}.
We have
\begin{equation}\label{expanded-resolv}
\hat\Gamma_r(s,t)+\int\limits_{-r}^r \hat
K(s,u)\hat\Gamma_r(u,t)du=\hat K(s,t)
\end{equation}
Differentiating in $r$, we get $(s\neq t)$
\begin{equation}\label{expanded}
\frac{\partial \hat\Gamma_r(s,t)}{\partial r} +\int\limits_{-r}^r \hat K(s,u) \frac{%
\partial \hat\Gamma_r(u,t)}{\partial r}
du=-(\hat K(s,r)\hat\Gamma_r(r,t)+\hat K(s,-r)\hat\Gamma_r(-r,t))
\end{equation}
On the other hand, take (\ref{expanded-resolv}) with $t=\pm r$,
multiply by $-\hat\Gamma_r(\pm r,t)$ and add:
\begin{eqnarray*}
\left[-(\hat\Gamma_r(s,r)\hat\Gamma_r(r,t)+\hat\Gamma_r(s,-r)\hat\Gamma_r(-r,t)\right]+
\end{eqnarray*}
\[
+\int\limits_{-r}^r \hat K(s,u)\cdot
\left[-(\hat\Gamma_r(u,r)\hat\Gamma_r(r,t)+\hat\Gamma_r(u,-r)\hat\Gamma_r(-r,t)\right]du=
\]
\[
-(\hat K(s,r)\hat\Gamma_r(r,t)+\hat K(s,-r)\hat\Gamma_r(-r,t))
\]
Comparison with (\ref{expanded}) finishes the proof.
\end{proof}
Now, if we have (\ref{lower-upper}), then
$(I+\hat{\cal{K}}_R)^{-1}=I-\hat{\cal{G}}_R
=(I+\hat{\cal{V}}_{+})(I+\hat{\cal{V}}_{-})$, where $%
\hat{\cal{V}}_{\pm }$ is upper(lower)-triangular with kernels
$\hat{V}_{\pm}(x,y), -R\leq x,y\leq R$. Moreover, for the
resolvent kernel

\begin{equation}\label{expanded-res}
\hat\Gamma_R (x,y)=\left\{
\begin{array}{cc}
\displaystyle
-\hat{V}_{+}(x,y)-\int\limits_{|y|<|u|<R}\hat{V}_{+}(x,u)\hat{V}_{-}(u,y)du,
 & (x,y)\in \hat\Omega_+ \\
\displaystyle
-\hat{V}_{-}(x,y)-\int\limits_{|x|<|u|<R}\hat{V}_{+}(x,u)\hat{V}_{-}(u,y)du,
& (x,y)\in \hat\Omega_-
\end{array}%
\right.
\end{equation}
The next Theorem provides the needed factorization
\begin{theorem}
The integral operator $\hat{\cal{K}}_R$ with kernel
$\hat{K}(x,y)\in \hat{C}([-R,R]^2)$ admits factorization
(\ref{lower-upper}) if and only if
 $I+\hat{\cal{K}}_r$ is invertible in $L^2[-r,r]$ for any $0<r\leq R$. In
this case,
\begin{equation}
\begin{array}{ccc}
\hat{V}_{+}(x,y) &=&-\hat\Gamma _{|y|}(x,y),\ (x,y)\in \hat\Omega_+ \\
\hat{V}_{-}(x,y) &=&-\hat\Gamma _{|x|}(x,y),\ (x,y)\in
\hat\Omega_-
\end{array}
\label{kernels1}
\end{equation}
 \label{volter1}
\end{theorem}
\begin{proof}
Calculating the right-hand side in (\ref{expanded-res}) and using
Lemma \ref{lemma-discont-1}  we get (for $(x,y)\in
\hat{\Omega}_-$):
\[
-\hat{V}_-(x,y)-\int\limits_{|x|<|u|<R}
\hat{V}_+(x,u)\hat{V}_-(u,y)du=
\]
\[
=\hat\Gamma_{|x|}(x,y)-\int\limits_{|x|}^R
\hat\Gamma_u(x,u)\hat\Gamma_{u}(u,y)du-\int\limits_{-R}^{-|x|}\hat\Gamma_{|u|}(x,u)
\hat\Gamma_{|u|}(u,y)du
\]
\[=
\hat\Gamma_{|x|}(x,y)-\int\limits_{|x|}^R
\Bigl[\hat\Gamma_u(x,u)\hat\Gamma_{u}(u,y)du+\hat\Gamma_{u}(x,-u)
\hat\Gamma_{u}(-u,y)\Bigr]du
\]
\[
=\hat\Gamma_R(x,y)
\]
The case $(x,y)\in \Omega_+$ can be checked similarly. So,
(\ref{expanded-res}) is true. The other statements of the Theorem
can be verified following the proof of Theorem \ref{volter}.
\end{proof}
Notice that $\hat{V}_-(x,-x)=\hat{V}_+(x,-x), x\neq 0$. The
results of last Lemma and a Theorem can be easily generalized to
the case when the kernel $K(x,y)$ is allowed to have a
discontinuity of the first kind on $\{(x,y): y=- x\}$. We do not
do that since the class $\hat{C}([-R,R]^2)$ is exactly the one we
will need later on.

{\bf Remarks and Historical Notes.}

The proofs of the results in this section are partially taken from
\cite{GohKr}. In \cite{GohKr}, the general case of factorization
along the chain is considered. Recently, the factorization problem
for integral operators with less regular kernels was studied in
\cite{Myk1, Myk2}. Later on, we will need to use the factorization
of Fredholm integral operators along with regularity properties of
the kernels.

\newpage

\section{Continuous analogs of polynomials orthogonal on
the unit circle}\label{three}

In this section we start building the theory of continuous analogs
of polynomials orthogonal on the unit circle (OPUC). For the OPUC
basics, we refer the reader to \cite{Szego, Geronimus, Simon,
Khrushchev1}. Let $H(x)$ be Hermitian function defined on
$\mathbb{R}$ and $H(x)\in L^{1}(0,r)$ for any $r>0$. In the
Hilbert space $L^{2}[0,r]$, consider the following integral
operator
\begin{equation}
{\cal H}_{r}f(x)=\int\limits_{0}^{r}H(x-t)f(t)dt,\, 0<x<r \label{intop}
\end{equation}

This operator is called ``truncated Toeplitz" operator or operator
with the ``displacement  kernel" \cite{Sakh1, Krein1}. It is
obvious that for any $r>0$, this operator is self-adjoint,
compact, and its lower (upper) bound decreases (increases) in $r.$

\begin{definition}
Function $g(x)\in G_{\infty }$ has an accelerant if
there exists Hermitian function $H(x)$ (accelerant) defined on $\mathbb{R}$
such that%
\begin{equation}
g(x)=\frac{|x|}{2}+\int\limits_{0}^{x}(x-s)H(s)ds  \label{e1s2}
\end{equation}
for all $x \in \mathbb{R}$.
\end{definition}

\begin{theorem}
The function $H(x)$ generates $g(x)\in G_{\infty } $ by formula
(\ref{e1s2}) if and only if
\[
I+{\cal H}_{r}\geq 0
\]%
for any $r>0$ and inequality is understood in the operator sense.
\label{t1s2}
\end{theorem}

\begin{proof} It is obvious that $g(x)$ is Hermitian, continuous, and
$g(0)=0$. Consider any $\varphi (x)\in C^{\infty }[0,\infty )$
with compact support
on $[0,r]$. For the kernel $K(x,y)=g(x)+g(-y)-g(x-y)$, we have%
\[
\int\limits_{0}^{\infty }\int\limits_{0}^{\infty }K(x,y)\varphi ^{\prime
}(y)%
\overline{\varphi ^{\prime }(x)}dxdy=I_{1}+I_{2}-I_{3}
\]%
For $I_{1}$ and $I_{2},$

\[
I_{1}=\overline{I_{2}}=-\varphi (0)\int\limits_{0}^{\infty
}g(x)\overline{\varphi ^{\prime }(x)}dx
\]%
\[
I_{3}=\int\limits_{0}^{\infty }\overline{\varphi ^{\prime
}(x)}\int\limits_{0}^{\infty
}g(x-y) \varphi ^{\prime }(y) dydx=-\varphi (0)%
\int\limits_{0}^{\infty }\overline{\varphi ^{\prime
}(x)}g(x)dx+\int\limits_{0}^{\infty
}\overline{\varphi ^{\prime }(x)}\int\limits_{0}^{\infty }g^{\prime }(x-y)
\varphi (y) dydx
\]%
Using (\ref{e1s2}) and integrating by parts, we have
\[
\int\limits_{0}^{\infty }\int\limits_{0}^{\infty }K(x,y)\varphi ^{\prime
}(y)%
\overline{\varphi ^{\prime }(x)}dxdy=\int\limits_{0}^{\infty
}\int\limits_{0}^{\infty }\varphi (y)\overline{\varphi (x)}%
H(x-y)dxdy+\int\limits_{0}^{\infty }|\varphi (x)|^{2}dx=\left( (1+{\cal H}%
_{r})\varphi ,\varphi \right)
\]%
Thus, if $g\in G_{\infty }$, then $I_{1}+I_{2}-I_{3}\geq 0$ and $1+{\cal H}%
_{r}\geq 0$ for any $r>0.$ Conversely, assume that $1+{\cal H}_{r}\geq 0$
for any $r>0$. Take any $\psi (x)\in C^{\infty }[0,\infty )$ with compact
support in $[0,r]$. It can be written as
\[
\psi (x)=\varphi ^{\prime }(x)
\]%
where
\[
\varphi (x)=-\int\limits_{x}^{\infty }\psi (s)ds
\]%
and $\varphi (x)\in C^{\infty }[0,\infty )$, $\varphi (x)$ is supported on
$%
[0,r]$. Therefore,
\[
\int\limits_{0}^{\infty }\int\limits_{0}^{\infty }K(x,y)\psi (y)\overline{%
\psi (x)}dxdy=((1+{\cal H_r})\varphi,\varphi)\geq 0
\]%
and $g(x)\in G_{\infty }$. \end{proof}

As a corollary from Theorem \ref{t1s1} and Theorem \ref{t1s2}, we get the
following
formula for an accelerant%
\begin{equation}
\frac{|x|}{2}+\int\limits_{0}^{x}(x-s)H(s)ds=i\beta x+\int\limits_{-\infty
}^{\infty }\left( 1+\frac{i\lambda x}{1+\lambda ^{2}}-\exp (i\lambda
x)\right) \frac{d\sigma (\lambda )}{\lambda ^{2}}  \label{e2s2}
\end{equation}
where $\beta \in \mathbb{R},$ and
\[
\int\limits_{-\infty }^{\infty }\frac{d\sigma (\lambda )}{1+\lambda ^{2}}%
<\infty
\]%
The straightforward calculation shows that the trivial case
$H(x)=0$ corresponds to $\beta=0$ and
$\sigma_0(\lambda)~=~\lambda/(2\pi)$.

Essentially, (\ref{e2s2}) means that
\[
\int\limits_{-\infty }^{\infty }\exp (i\lambda x)d\sigma (\lambda
)\,``="\,\delta (x)+H(x)
\]%
or
\[
\int\limits_{-\infty }^{\infty }\exp (i\lambda x)d(\sigma (\lambda
)-\sigma_0(\lambda))\,``="\, H(x)
\]%
 In other words, $H(x)$ are ``moments" of $\sigma-\sigma_0$.
Clearly, the integrals in the last formulas do not have to
converge in the usual sense.
\begin{lemma} \label{findingbetta}
Assume that $\sigma$ from (\ref{e2s0}) is known and $g$ has an
accelerant, then the constant $\beta$ is defined uniquely by the
formula
\begin{equation}
\beta=-i\Psi'(0)\label{betaa}
\end{equation}
where
\begin{equation}
\Psi(x)=-\int\limits_{-\infty }^{\infty }\left( 1+\frac{i\lambda
x}{1+\lambda ^{2}}-\exp (i\lambda x)\right) \frac{d[\sigma
(\lambda )-\sigma_0(\lambda)]}{\lambda ^{2}} \in C^1(\mathbb{R})
\end{equation}
\end{lemma}

\begin{proof} The proof follows from the formula (\ref{e2s2}) by
taking the derivative.
\end{proof}

It is up to us to choose the regularity class for  $H(x)$. In
these notes, we will consider two important cases:
\begin{equation}
H(x)\in C[0,\infty), \label{regularity1}
\end{equation}
and
\begin{equation}
H(x)\in L^2_{\rm loc}(\mathbb{R}) \label{regularity-l2loc}
\end{equation}
Other classes of regularity (e.g., $H(x)\in L^p_{\rm
loc}(\mathbb{R}), p\geq 1$) can also be treated. Although, the
case $p=1$ needs special consideration (see discussion in
\cite{Krein1, krein-madamyan}).

We will start our construction with the continuous accelerants.
Then the $L^2_{\rm loc}$ case will be treated by an approximation
argument in the separate section. For (\ref{regularity1}), $H(x)$
might have discontinuity at $0$ but the left and the right limits
must exist and $H(-0)=\overline{H(+0)}$ due to Hermite property.
Notice that the operator $\cal{H}_r$ has a kernel from the class
$\hat{C}([0,r]^2), r>0$.


Assume that we have the strict inequality
\begin{equation}
1+{\cal H}_{r}>0 \label{strict}
\end{equation}
for any $r>0.$ Then, there is the resolvent kernel $\Gamma
_{r}(t,s)$ with nice properties (see Lemma \ref{lemma-discont})
such that

\begin{equation}
\Gamma_r(s,t)=\overline{\Gamma_r(t,s)}
\end{equation}
\begin{equation}
\Gamma_{r}(t,s)+\int\limits_{0}^{r}H(t-u)\Gamma
_{r}(u,s)du=H(t-s), \label{basic1}
\end{equation}
\begin{equation}
\Gamma_r(t,s)+\int\limits_0^r \Gamma_r (t,u) H(u-s) du=H(t-s), \
0\leq s,t\leq  r \label{basic2}
\end{equation}
We emphasize that the last two identities should be understood as
equalities for functions from $\hat{C}([0,r]^2)$ class. In the
meantime, if $H(\pm 0)\in \mathbb{R}$, then $H(x)$ is actually
continuous at $0$ and by Lemma \ref{lemma23} the kernel
$\Gamma_r(x,y)$ is continuous on the diagonal as well.

 Let us introduce the following ``continuous polynomials"
\begin{equation}
P(r,\lambda )=\exp (i\lambda r)-\int\limits_{0}^{r}\Gamma _{r}(r,s)\exp
(i\lambda s)ds  \label{e3s2}
\end{equation}
\begin{equation}
P_{\ast }(r,\lambda )=1-\int\limits_{0}^{r}\Gamma _{r}(s,r)\exp (i\lambda
(r-s))ds )  \label{e31s2}
\end{equation}%
Notice that function $P(r,\lambda )$ is of exponential type exactly $r$, and
$P_{\ast }(r,\lambda )$ is of exponential type not greater than $r$.

Formulas (\ref{e3s2}) and (\ref{e31s2}) can be easily explained.
They are quite natural and have analogs in the OPUC theory.  Let
us consider positive finite measure $\mu (\theta )$ on the unit
circle. We will denote the inner product of two functions $f$ and
$g$ in $L^2(d\mu)$ by $(f,g)_{\mu}$. Let the sequence of moments
be
\[
c_{n}=\int\limits_{-\pi }^{\pi }\exp (-in\theta )d\mu (\theta ),
\,
n\in \mathbb{Z^+}
\]%
Let $\{e_j\}, e_j=[0,\ldots, 0,1,0,\ldots, 0]^t, (j=0,\ldots) $
denotes the standard orthonormal basis. Consider the following
Toeplitz matrix%
\begin{equation}
\cal{T}_{n}=\left[
\begin{array}{cccc}
c_{0} & c_{1} & \ldots & c_{n} \\
\bar{c}_{1} & c_{0} & \ldots & c_{n-1} \\
\ldots & \ldots & \ldots & \ldots \\
\bar{c}_{n} & \bar{c}_{n-1} & \ldots & c_{0}%
\end{array}%
\right]  \label{e4s2}
\end{equation}
If $D_n=\det \cal{T}_n$, then one can easily show that
\begin{equation}
P_{n}(z)=\frac{1}{D_{n-1}}\left\vert
\begin{array}{ccccc}
c_{0} & c_{1} & \ldots & c_{n-1} & 1 \\
\overline{c}_{1} & c_{0} & \ldots & c_{n-1} & z \\
\ldots & \ldots & \ldots & \ldots & \ldots \\
\overline{c}_{n-1} & \overline{c}_{n-2} & \ldots & c_{0} & z^{n-1} \\
\overline{c}_{n} & \overline{c}_{n-1} & \ldots & \overline{c}_{1} & z^{n}%
\end{array}%
\right\vert \label{heine}
\end{equation}
are monic polynomials of degree $n$ (i.e., the coefficient in
front of $z^n$ is $1$) orthogonal with respect to $d\mu$. In
formula (\ref{e3s2}), exponents $\exp (i\lambda r),r\geq 0,\lambda
\in \mathbb{C}$ play the role of $z^{n},n\in \mathbb{Z^{+}},z\in
\mathbb{C}$. The exponential type $r$ is an analog of the integer
index $n$. If one introduces the truncated discrete Toeplitz
operator $\cal{T}_{n}$ given by the matrix (\ref{e4s2}),
then $%
P_{n}(z)$ is the last component of the following vector $D_n
D_{n-1}^{-1} {\cal{T}}_{n}^{-1}[1,z,\ldots ,z^{n}]^{t}$. That
follows from Kramer's rule and (\ref{heine}). Besides Kramer's
rule, there is the following algebraic explanation of
orthogonality relation. Consider the last component of the vector
${\cal{T}}_{n}^{-1}[1,z,\ldots ,z^{n}]^{t}$, i.e.
$({\cal{T}}_{n}^{-1}[1,z,\ldots ,z^{n}]^{t}, e_n)$. Assume that we
have two polynomials of degree not greater than $n$: $A(z)=a_n
z^n+\ldots+a_0$, $B(z)=b_n z^n+\ldots+b_0$. Then
\begin{equation}
(A,B)_{\mu}=(\overline{\cal{T}}_n a,b) \label{p1e1}
\end{equation}
where $\overline{\cal{T}}_n$ is obtained from $\cal{T}_n$ by
conjugating all elements, $a=[a_0,\ldots, a_n]^t$, $b=[b_0,\ldots,
b_n]^t$. Therefore,
\[
(z^j, (\cal{T}_n^{-1}[1,z,\ldots, z^n]^t,e_n))_{\mu}=(z^j,
([1,z,\ldots, z^n]^t,\cal{T}_n^{-1}e_n))_\mu=(\overline{\cal{T}}_n
e_j, \overline{\cal{T}}^{-1}_n e_n)=\delta_{jn}
\]

Now, consider the function $f_{r}(x)=\exp (i\lambda x)$ on an
interval
$%
x\in \lbrack 0,r]$. Then, $P(r,\lambda )$ is the value of the function $(1+%
{\cal H}_{r})^{-1}f_{r}$ at the point $x=r$. Because the spectrum
of the truncated continuous Toeplitz operator $\cal{H}_r$ always
contains $0$, we need to take $(1+{\cal H}_{r})^{-1}$, rather than
${\cal{T}}_n^{-1}$ in the discrete case. This
normalization makes the function $P(r,\lambda )$ ``monic"-- it has $%
\exp (i\lambda r)$ term. In the theory of orthogonal polynomials on the unit
circle, there is a natural procedure that maps any polynomial of degree $n$
to its ``reciprocal". It is defined in the following way

\begin{equation}
\lbrack \ast ](a_{n}z^{n}+\ldots +a_{0})=\bar{a}_{0}z^{n}+\ldots
+\bar{a}_{n}
\end{equation}
or
\begin{equation}
\lbrack \ast ](A_n(z))=z^n \overline{A(\bar{z}^{-1})}
\end{equation}
Notice that the degree of the polynomial might decrease under this
operation. For example, in the space of polynomials of degree not
greater than $n$, $[\ast] (z^n)=1$. The natural analog of $[\ast
]$ operation for the function $f_{r}(\lambda )$ of the exponential
type $r$ is given by the following formula
\begin{equation}
\lbrack \ast ](f_{r}(\lambda ))=\exp (i\lambda
r)\overline{f_{r}(\bar{\lambda})%
}
\end{equation}%
Function $P_{\ast }(r,\lambda )$ is a continuous analog of $[\ast
](P_{n}(z))$, i.e.
\begin{equation}
P_{\ast} (r,\lambda)=\lbrack \ast ] (P(r,\lambda)) \label{mes1}
\end{equation}
It is clear from the formula (\ref{e3s2}) and identity $\Gamma _{r}(s,t)=%
\overline{\Gamma _{r}(t,s)}$. In the discrete case, polynomials $P_{n}(z)$
are {\bf orthogonal} with respect to $d\mu $. In the continuous setting, $%
\{P(r,\lambda )\} $ turns out to be {\bf orthonormal} family in
$L^2(d\sigma)$ with $d\sigma $ from the integral representation
(\ref{e2s2}). Orthogonality is understood in the usual sense--
just like for the Fourier transform. Thus, rather than in the
discrete case, continuous monic polynomials are already
normalized.

Recall the factorization results from the previous section. Notice
that since $I+\cal{H}_r>0$ for any $r>0$, the Theorem \ref{volter}
is applicable. Let us fix any $R>0$ and consider the factorization
(\ref{factorization}): $I+\cal{H}_R=(I+\cal{L})(I+\cal{U})$, where
the lower-diagonal $\cal{L}$ has kernel $L(x,y)$ and the
upper-diagonal $\cal{U}$ has kernel $U(x,y)$, $0<x,y<R$. Since
$\cal{H}$ is Hermitian, $\cal{L}=\cal{U}^*$. If
$I+\cal{L}=(I+\cal{V}_-)^{-1}$ and $I+\cal{U}=(I+\cal{V}_+)^{-1}$,
then

\begin{lemma}
The following formula is true
\begin{equation}
\exp (i\lambda r)=P(r,\lambda
)+\int\limits_{0}^{r}L(r,s)P(s,\lambda )ds, 0<r\leq R \label{eee2}
\end{equation}
\label{transformation}
\end{lemma}

\begin{proof} Take any $R>0$. By (\ref{kernels}), equation
(\ref{e3s2}) can be written as
\[
P(r,\lambda)=(I+\cal{V}_-)\exp(i\lambda r), 0<r<R
\]%
and the both sides are regarded as functions in $L^2[0,R]$. Since
$I+\cal{L}=(I+\cal{V}_-)^{-1}$, we have the statement of the
Lemma\footnote{ Kernel $L(x,y)$ should be regarded as a
transformation kernel from one basis to another \cite{Levitan1}.}.
\end{proof}
In the discrete case, the following formula is true for any $n\in
\mathbb{Z^+}$
\[
z^n=P_n(z)+\sum\limits_{j=0}^{n-1} l_{n,j} P_j(z),\, l_{j,n}\in
\mathbb{C}
\]
The Lemma \ref{transformation} is the continuous analog of that
representation. Now, we are ready to prove the following

\begin{theorem}
The following map is an isometry from $L^{2}(\mathbb{R}^{+})$
into $L^{2}(\mathbb{R},d\sigma )$%
\begin{equation}
\cal{O}:f(r)\longmapsto \left( \cal{O}f\right) (\lambda
)=\int\limits_{0}^{\infty }f(r)P(r,\lambda )dr  \label{zuzu}
\end{equation}
In other words,
\begin{equation}
\int\limits_{-\infty}^\infty
|(\cal{O}f)(\lambda)|^2d\sigma(\lambda)=\int\limits_0^\infty
|f(r)|^2dr\label{plancherel}
\end{equation}
 Integral in (\ref{zuzu}), is understood in the $L^{2}-$ sense.
\label{theorem2s2}
\end{theorem}

\begin{proof} The following is true for any $t,r\in \mathbb{R}$
\[
\frac{|r-t|}{2}+\int\limits_{0}^{r-t}(r-t-s)H(s)ds=i\beta
(r-t)+\int\limits_{-\infty }^{\infty }\left( 1+\frac{i\lambda (r-t)}{%
1+\lambda ^{2}}-\exp (i\lambda (r-t))\right) \frac{d\sigma (\lambda )}{%
\lambda ^{2}}
\]%
Multiply this equality by $f^{\prime \prime }(t)\, (f(t)\in C_{0}^{\infty
}(\mathbb{R}))$ and integrate by parts. We have the following formula%
\begin{equation}
f(r)+\int\limits_{-\infty}^{\infty }H(r-t)f(t)dt=\int\limits_{-\infty
}^{\infty } \left[ \int\limits_{-\infty}^{\infty }f(t)\exp (i\lambda
(r-t))dt%
\right] d\sigma (\lambda )  \label{ee2}
\end{equation}
For $\overline{f(r)}$,
\begin{equation}
\overline{f(r)}+\int\limits_{-\infty}^{\infty }H(r-t)\overline{f(t)}%
dt=\int\limits_{-\infty }^{\infty }\exp (i\lambda r)\overline{%
\int\limits_{-\infty}^{\infty }f(t)\exp (i\lambda t)dt}d\sigma (\lambda )
\label{e6s2}
\end{equation}
Consequently, we have the following analog of (\ref{p1e1}):

\begin{equation}
\int\limits_{-\infty}^{\infty }|f(r)|^{2}dr+\int\limits_{-\infty}^{\infty
}f(r)\int\limits_{-\infty}^{\infty
}\overline{f(t)}H(r-t)dtdr=\int\limits_{-%
\infty }^{\infty }\left\vert \, \int\limits_{-\infty}^{\infty }f(t)\exp
(i\lambda t)dt\right\vert ^{2}d\sigma (\lambda ) \label{e1s2n1}
\end{equation}
 If $f\in L^2[0,r]$, then
(\ref{e1s2n1}) implies
\begin{equation}
((1+\cal{H}_r)f,f)=\int\limits_{-\infty}^{\infty}\left|
\int\limits_0^r \overline{f(t)}\exp(i\lambda t)dt\right|^2
 d\sigma(\lambda)
\label{sect2eq1}
\end{equation}
by simple approximation argument. Now, let us use Lemma
\ref{transformation} for the interval $[0,r]$:
\[
\int\limits_0^r \overline{f(t)}\exp(i\lambda
t)dt=((I+\cal{L})P(t,\lambda),f(t))_{L^2[0,r]}=(P(t,\lambda),(I+\cal{L}^*)f(t))_{L^2[0,r]}
\]
Therefore, (\ref{sect2eq1}) reads
\begin{equation}
\|g\|^2=((I+\cal{L})^{-1}(I+\cal{H}_r)(I+\cal{U})^{-1} g,
g)=\int\limits_{-\infty}^\infty \left|\int\limits_0^\infty
\overline{g(t)}P(t,\lambda)dt\right|^2d\sigma(\lambda)
\label{isom}
\end{equation}
where $g=(I+\cal{U})f$. Since $I+\cal{U}$ is invertible in
$L^2[0,r]$, (\ref{isom}) holds for any $g\in L^2[0,r]$. Since
$\|g\|=\|\bar{g}\|$  and a number $r$ was chosen arbitrarily,
$\cal{O}$ is isometry on $L^2(\mathbb{R}^+)$.
 \end{proof}

As a simple corollary of the Theorem \ref{theorem2s2} and the polarization
identity we get the following
Lemma.

\begin{lemma}
For any $f(r)\in L^2(\mathbb{R}^+)$
\[
f(r)=\int\limits_{-\infty}^\infty \overline{P(r,\lambda)}
(\cal{O}f)(\lambda) d\sigma(\lambda)
\]
where the equality is understood in the $L^2(\mathbb{R}^+)$ sense.
\end{lemma}

\begin{remark}
\rm Notice, that $\cal{O}$ is not necessarily a unitary map. In
the simplest case $H(r)=0$, $P(r,\lambda)=\exp(i\lambda r)$,
$\sigma(\lambda)=\lambda/(2\pi)$ and the range of $\cal{O}$ is
$H^2(\mathbb{R})\subset L^2(\mathbb{R},d\sigma)$.
\end{remark}

Now, let us obtain the differential system for $P(r,\lambda)$ and $%
P_*(r,\lambda)$. In the discrete case, we have

\begin{equation}
\left\{
\begin{array}{cccc}
P_{n+1}(z) & = & zP_n(z)-\bar{a}_n P^*_n(z),& P_0(z)=1 \\
P^*_{n+1}(z) & = & P^*_n(z)-a_n z P_n(z),& P^*_0(z)=1%
\end{array}
\right.  \label{e81s2}
\end{equation}
where $P^*_n(z)=[*]P_n(z)$, $a_n$ are the so-called Verblunsky
coefficients (Geronimus coefficients, Schur coefficients, circle
parameters, or reflection parameters). If one starts with
arbitrary positive finite measure with infinite number of growth
points, then the corresponding $a_n\in \mathbb{D}$. Conversely,
any
sequence $%
a_n\in \mathbb{D}$ yields the unique probability measure with infinite
number of
growth points.

Let us prove one property of the resolvent kernel which would
yield differential equations for $P(r,\lambda)$ and
$P_*(r,\lambda)$. It holds only for the integral operators with
the ``displacement " kernel
\begin{lemma} \label{lemma231}
If  $\Gamma_r(s,t)$ is the resolvent kernel for $\cal{H}_r$, then
\begin{equation}
\Gamma _{r}(s,t)=\Gamma _{r}(r-t,r-s) \label{flip}
\end{equation}
\end{lemma}
\begin{proof} The following relation holds
\[
\cal{H}_r \cal{F}_r=\cal{F}_r \overline{\cal{H}}_r
\]
where the ``flip" operator $\cal{F}_r$ is defined: $\cal{F}_r
f(x)=f(r-x)$ and $\overline{\cal{H}}_r$ is an integral operator with
the
displacement kernel $\overline{H(x)}$.
Then,
\[
(1+\cal{H}_r)^{-1} \cal{F}_r=\cal{F}_r (1+\overline{\cal{H}}_r)^{-1}
\]
Writing this down in terms of the resolvent kernel gives
(\ref{flip}).\end{proof}

As a simple corollary, we also get the following formulas for $P$
and $P_*$ which were originally used by Krein

\begin{equation}
P(r,\lambda ) =\exp (i\lambda r) \left(
1-\int\limits_{0}^{r}\Gamma _{r}(s,0)\exp (-i\lambda s)ds\right)
\label{polyn1}
\end{equation}
\begin{equation}
P_{\ast }(r,\lambda ) =1-\int\limits_{0}^{r}\Gamma _{r}(0,s)\exp (i\lambda
s)ds  \label{e10s2}
\end{equation}
An analog of the relations (\ref{e81s2}) is given by the following
statement.

\begin{theorem}
The following equations hold
\begin{equation}
\left\{
\begin{array}{ll}
P^{\prime}=i\lambda P-\bar A P_*, & P(0,\lambda)=1, \\
P_*^{\prime}=-AP, & P_*(0,\lambda)=1%
\end{array}
\right.  \label{krein2}
\end{equation}
where
\begin{equation}
A(r)=\Gamma_r(0,r)
\end{equation}
\end{theorem}

\begin{proof} Differentiate (\ref{e10s2}) and use
Lemma  \ref{lemma-discont} to get
\[
P_*^{\prime}(r,\lambda)=-\Gamma_r(0,r)\exp(i\lambda r)+\int\limits_0^r
\Gamma_r(0,r) \Gamma_r(r,s) \exp(i\lambda s) ds=-A(r) P(r,\lambda)
\]
where we used the definition of $P(r,\lambda)$. Equation for
$P(r,\lambda)$ can be obtained from the equation for $P_{\ast}
(r,\lambda)$ and (\ref{mes1}). \end{proof}

\begin{definition}
System (\ref{krein2}) is called the Krein system.
\end{definition}

Obviously, coefficient $A(r)$ in the Krein system is an analog of
Verblunsky coefficients.

\begin{lemma} \label{a-h}
Under regularity conditions (\ref{regularity1}), we have $A(r)\in
C[0,\infty)$. We also have $A(0)=H(-0)=\overline{H(+0)}$.
\end{lemma}
\begin{proof}
The continuity of $A(r)$ follows from Lemma \ref{lemma-discont}.
The equality $A(0)=H(-0)$ can be obtained from (\ref{basic2})
where $r\to 0$.
\end{proof}
Notice that $H(x)\in C(\mathbb{R})$ iff $H(0)$-- real iff $A(0)$
is real as well.

Let us consider the model (``free") case. If $%
\sigma(\lambda)=\lambda/(2\pi)$ in (\ref{e2s2}), then $\beta=0$, $%
g(t)=|t|/2$, $H(t)=0$, $\Gamma_r(s,t)=0$, $A(r)=0$, $P(r,\lambda)=\exp(i%
\lambda r)$, $P_*(r,\lambda)=1$.

\begin{lemma}
The following is true \\
$1)$ Christoffel-Darboux formula:
\begin{equation}
P_{\ast }(r,\lambda )\overline{P_{\ast }(r,\mu )}
=P(r,\lambda )\overline{%
P(r,\mu )}-i(\lambda -\bar{\mu})\int\limits_{0}^{r}P(s,\lambda )\overline{%
P(s,\mu )}ds, \quad \lambda,\mu\in \mathbb{C}  \label{e101s2}
\end{equation}\newline
$2)$ For $\lambda\in \mathbb{C}$,
\begin{equation}
P(r,\lambda )=\exp (i\lambda r)\overline{%
P_{\ast }(r,\bar\lambda )} \label{starproperty}
\end{equation}
 $3)$ $P_{\ast }(r,\lambda )$ does
not have any zeroes in $\overline{\mathbb{C}^{+}}$ \newline $4)$
$P(r,\lambda )$  has zeroes in $\mathbb{C}^{+}$  only\newline
\label{lemma24}
\end{lemma}
\begin{proof}  To prove Christoffel-Darboux formula,
multiply the
first equation in (%
\ref{krein2}) by $\overline{P(r,\mu )}$. Multiply both sides of
\[
\overline{P_{\ast }^{\prime }(r,\mu )}=-\overline{A(r)}\,\overline{P(r,\mu
)}
\]%
by $P_{\ast }(r,\lambda )$ and subtract two identities. One has
\begin{equation}
P^{\prime }(r,\lambda )\overline{P(r,\mu )}-\overline{P_{\ast }^{\prime
}(r,\mu )}P_{\ast }(r,\lambda )=i\lambda P(r,\lambda )\overline{P(r,\mu )}
\label{e11s2}
\end{equation}%
Write the same equation with $\lambda $ and $\mu $ interchanged.
Take conjugate and add to (\ref{e11s2}). One gets (\ref{e101s2}).

$2)$ repeats (\ref{mes1}).

Take $\lambda\in \mathbb{C}^+$ and $\mu=\lambda$. Then,
(\ref{e101s2}) guarantees that $|P_*(r,\lambda)|>0$, so
$P_*(r,\lambda)$ has no zeroes in $\mathbb{C}^+$. Assume that
$P_*(r,\lambda)=0$ for some $\lambda\in \mathbb{R}$. Then, by
$2)$, $P(r,\lambda)=0$. However, functions $P(\rho,\lambda),
P_*(\rho,\lambda)$ solve the problem (\ref{krein2}) and then must
vanish for all $\rho\geq 0$. In the meantime
$P(0,\lambda)=P_*(0,\lambda)=1$. This contradiction shows that
$P_*(r,\lambda)$ has no zeroes in $\overline{\mathbb{C}^+}$.

$4)$ follows from $3)$ and (\ref{mes1}).  \end{proof}

Assume that for fixed $r$ we know $P(r,\lambda)$ as the function
in $\lambda$. The natural question is whether we can find
$P(\rho,\lambda)$ and $A(\rho)$ for all $0<\rho<r$? The answer
happens to be positive. Notice that since $P(r,\lambda)$ is given,
we know the values of $\Gamma _{r}(s,0)$ for all $s\in \lbrack
0,r]$. It easily follows from (\ref{polyn1}).

 Consider
\[
g(t,s)=\Gamma _{r}(0,r-t)\Gamma _{r}(r-s,0)-\Gamma _{r}(t,0)\Gamma
_{r}(0,s), 0<s,t<r
\]

\begin{lemma}
The following formula holds true
\begin{equation}
\Gamma _{r}(t,s)=\Gamma _{r}(t-s,0)+\overline{\Gamma _{r}(s-t,0)}%
+\int\limits_{0}^{\min (s,t)}g(t-u,s-u)du,\ 0\leq s\neq t \leq r
\label{equa2}
\end{equation}
where $\Gamma_r(0,s)=\Gamma _{r}(s,0)=0$, $s<0$ for shorthand.
\label{lemmamain}
\end{lemma}

\begin{proof} Assume that $H(x)\in C^{1}[0,r]$ first. Then, $\Gamma_r(s,t)\in C^1(\Delta_{\pm})$. Let us
show that
\begin{equation}\label{wave-eq}
\frac{\partial \Gamma _{r}(s,t)}{\partial t}+\frac{\partial \Gamma
_{r}(s,t)%
}{\partial s}=g(s,t),\quad s,t\in \Delta_{\pm}
\end{equation}
We have
\[
-\Gamma_r(t_1,r)\Gamma_r(r,s_1)=\frac{d}{dr}
\Gamma_r(t_1,s_1)=\frac{d}{dr} \Bigl[\Gamma_r(r-s_1,r-t_1)\Bigr]
\]

\[
=-\Gamma_r(r-s_1,r)\Gamma_r(r,r-t_1)+{\partial_{x_1}}\Gamma_r(r-s_1,r-t_1)
+{\partial_{x_2}} \Gamma_r(r-s_1,r-t_1)
\]
Taking $t=r-t_1, s=r-s_1$, we get (\ref{wave-eq}).

 Solving the linear first
order equation, we obtain (\ref{equa2}). Now, consider $H(x)\in
C[0,r]$. We can approximate $H(x)$ by the sequence of Hermitian
functions
$%
H^{(n)}(x)\in C^{1}[0,r]$ in $C[0,r]$ norm. Since $1+{\cal
H}_r>0$, $1+{\cal
H}^{(n)}_r>0$, the corresponding kernel $\Gamma _{r}^{(n)}(s,t)$
converges to $\Gamma _{r}(s,t)$ uniformly as elements of
$\hat{C}([0,r]^2)$. Then, apply (\ref{equa2}) to $\Gamma
_{r}^{(n)}(s,t)$ and take $n\to \infty$.
\end{proof}
Now, assume we are given $\Gamma_r(s,0)$ for all $0<s<r$. From
(\ref{equa2}), we know the resolvent kernel $\Gamma_r(s,t)$ for
all $0<s,t<r$ and then can find $H(x)$, say, from (\ref{basic2}).
The accelerant $H(x)$ on $[0,r]$ defines $A(x)$ for $x\in [0,r]$
by construction.

 There is yet another way to find $H(x), 0<x<r$ from
 $P(r,\lambda)$. It is given by the following

\begin{lemma}
The following representation is true
\begin{equation}
\frac{1}{|P_{\ast }(r,\lambda )|^{2}}=1+\int\limits_{-\infty }^{\infty }%
\overline{H_{r}(s)}\exp (i\lambda s)ds, \lambda\in \mathbb{R}
\label{est1}
\end{equation}
{\it where} $H_{r}(s)$  is Hermitian function, $H_{r}(s)\in
L^{1}(\mathbb{R})$, and $H_{r}(s)=H(s)$ for
$|s|<r$. \label{lemmaimp}
\end{lemma}

\begin{proof} Formula (\ref{e10s2}), 3) in Lemma \ref{lemma24},
and Levy-Wiener theorem yield existence of function $H_{r}\in
L^1(\mathbb{R})$. Moreover, $H_r$ is continuous on $\mathbb{R}$
except for the points $0,\pm r$, where the left(right) limits
exist. That is clear from the corresponding integral equation. Let
us show that this $H_{r}(s)$ coincides with $H(s)$ for $|s|<r$. We
have
\[
\frac{1}{\overline{P_{\ast }(r,\lambda )}}-1=P_{\ast }(r,\lambda
)\left( 1+\int\limits_{-\infty }^{\infty }\overline{H_{r}(s)}\exp (i\lambda
s)ds\right)-1
\]%
The left-hand side belongs $\overline{H^{2}(\mathbb{R})}$. Therefore,%
\[
\cal{P}_{+}\left[ P_{\ast }(r,\lambda )\left( 1+\int\limits_{-\infty }^{\infty }%
\overline{H_{r}(s)}\exp (i\lambda s)ds\right)-1\right] =0
\]%
which gives%
\[
\Gamma _{r}(s,0)+\int\limits_{0}^{r}H_{r}(s-u)\Gamma _{r}(u,0)du=H_{r}(s),\
0<s<r
\]%
Formula (\ref{est1}) proves that $H_r$ defines an integral
operator $\tilde{\cal{H}}_r$ on $L^2[0,r]$ and
$I+\tilde{\cal{H}}_r>0$. Denote its resolvent kernel  by
$\tilde\Gamma_r(s,t)$. The last equation shows that
\mbox{$\tilde\Gamma_r(s,0)=\Gamma _{r}(s,0),0<s<r$}. Due to Lemma
\ref{lemmamain}, $\tilde\Gamma_r(s,t)=\Gamma_r(s,t)$ for all
$0\leq s,t\leq r$. So, $H_{r}(s)=H(s)$ for $0<s<r$. Since $H_{r}$
is Hermitian, we obtain the statement of the Lemma.
\end{proof}

Thus, if we know $P(r,\lambda)$, we know $P_*(r,\lambda)$ as well
and can find $H(x)$ for $|x|<r$ using the previous Lemma.
\bigskip

{\bf Remarks and historical notes.}

Continuous analogs of polynomials orthogonal on the circle were
introduced by M.G. Krein in the paper \cite{Krein2} but no proofs
were given. We filled this gap. Lemma \ref{lemmamain} is in
\cite{Krein1}, see also \cite{feldman}, p.100.

If one is given the function $H(r)\in C[0,R]$, Hermitian and such
that \mbox{$I+\cal{H}_R>0$}, then the Krein system can be
well-defined on the interval $[0,R]$. In the meantime, the
question of orthogonality with respect to some measure gives rise
to certain continuation problem \cite{Krein1} we do not want to
address here.

\newpage

\section{Krein systems}\label{four}

In the previous section we learned that any accelerant $H(r)$
gives rise to (\ref{krein2}), the system of ODE called the Krein
system. But it makes sense to study this system per se. In this
section, we will show that one can start with the Krein systems
and then define the accelerant $H(r)$ and measure
$\sigma(\lambda)$ uniquely.

Consider the system
\begin{equation}
X'=VX \label{krein3}
\end{equation}
with $X(0,\lambda)=I$,
\begin{equation}
V=\left[
\begin{array}{cc}
i\lambda & -\overline{A(r)} \\
-A(r) & 0%
\end{array}%
\right]\label{potential-V}
\end{equation}
Matrix $V$ has very special algebraic structure and it should
imply very special properties for the fundamental (transfer)
matrix $X(r)$. Assume first that $A(r)\in L^1_{\rm
loc}(\mathbb{R}^+)$.

The first obvious result is
\begin{lemma}
We have \begin{equation} \det X(r)=\exp(i\lambda r)
\label{determinant}
\end{equation}
\end{lemma}
\begin{proof}
Indeed,
\[
\det X(r)=\exp\left[ \int\limits_0^r {\rm Tr}\,
 V(t) dt
\right]=\exp(i\lambda r)
\]
\end{proof}

Consider the signature matrix
\begin{equation}
J=\left[
\begin{array}{cc}
1 & 0\\
0 & -1
\end{array}
\right] \label{signature}
\end{equation}

\begin{definition}
The matrix $M$ is said to be $J$-- contraction if $M^* J M\leq J$.
\end{definition}
\begin{definition}
The matrix $M$ is called $J$-- unitary if $M^* J M=J$
\end{definition}
 Later, we will need the following algebraic
\begin{lemma}
If $M$ is $J$--unitary, then $|\det M|=1$, and $M^{-1}, M^*$ are
$J$--unitary too. If $M$ is $J$-- contraction then $M^*$ is $J$--
contraction also \label{jproperty}
\end{lemma} whose proof is given in the Appendix.

The signature matrix $J$ defines the corresponding indefinite
metric. For general properties of these spaces and operators
acting on them, see \cite{Iohvidov, Rodman}.

The next very important algebraic observation is
\begin{equation}
V^*(r)J+JV(r)=-2\Im\lambda \left[
\begin{array}{cc}
1 & 0\\
0 & 0
\end{array}
\right]  \label{alg1}
\end{equation}
for any $r>0$.
\begin{theorem}
The matrix $X$ is $J$-- contraction for $\lambda\in \mathbb{C}^+$
and is $J$-- unitary for $\lambda\in \mathbb{R}$.
\label{jcontraction}
\end{theorem}
\begin{proof} Consider $Y=X^* J X$. Then, (\ref{alg1})
yields
\[
Y'=-2\Im \lambda\, X^* \left[
\begin{array}{cc}
1 & 0\\
0 & 0
\end{array}
\right] X,\, Y(0,\lambda)=J
\]
Therefore, for any $f\in \mathbb{C}^2$, we have
\begin{eqnarray}
(Yf,f)=(Jf,f), \lambda\in \mathbb{R}\\
(Yf,f)\leq (Jf,f), \lambda\in \mathbb{C}^+\\
(Yf,f)\geq (Jf,f), \lambda\in  \mathbb{C}^-
\end{eqnarray}
which implies the statement of the theorem.\end{proof}

In this section, we {\bf define} the functions
$P(r,\lambda),P_*(r,\lambda)$ as solutions of equation
(\ref{krein3}) corresponding the Cauchy problem
$P(0,\lambda)=P_*(0,\lambda)=1$. Consider also two functions
$\widehat{P}(r,\lambda)$ and $\widehat{P}_*(r,\lambda)$ such that
the vector $\widehat{P}(r,\lambda ),-\widehat{P}%
_{\ast }(r,\lambda )$\ solves (\ref{krein3}) and satisfies initial
 condition $\widehat{P}(0,\lambda )=1,-\widehat{P}_{\ast
}(0,\lambda )=-1.$ The simple calculation shows that

\begin{equation}
X(r)=\frac{1}{2}\left[
\begin{array}{cc}
P+\widehat{P} & P-\widehat{P} \\
P_{\ast }-\widehat{P}_{\ast } & P_{\ast }+\widehat{P}_{\ast }%
\end{array}\right]=
\frac{1}{2} \left[
\begin{array}{cc}
P & \widehat{P} \\
P_{\ast } & -\widehat{P}_{\ast }%
\end{array}\right]
\left[
\begin{array}{cc}
1 & 1 \\
1 & -1%
\end{array}
\right]\label{rotation}
\end{equation}
The sign $``-"$ in the definitions of $\widehat{P}_*(r,\lambda)$
was chosen for the following reason. Notice that
\begin{equation}
JV(r)J=\left[
\begin{array}{cc}
i\lambda & \overline{A(r)} \\
A(r) & 0%
\end{array}
\right]\label{commutation}
\end{equation}
Then,
\begin{lemma}\label{sign-change}
If $X$ solves equation (\ref{krein3}), then $JXJ$ solves the same
equation but with coefficient $A(r)$ having an opposite sign.
\end{lemma}
\begin{proof}
Multiply (\ref{krein3}) from the left by $J$. Then use
(\ref{commutation}) and identity $J^2=I$.
\end{proof}

\begin{corollary}
The vector $\widehat{P}(r,\lambda ),\widehat{P}%
_{\ast }(r,\lambda )$ satisfies the same initial conditions at
zero as $P(r,\lambda), P_*(r,\lambda)$ but solves system
(\ref{krein3}) with $A(r)$ having an opposite sign. This system is
called the dual Krein system. \label{corollary-a}
\end{corollary}

 Instead of dealing
with the transfer matrix $X(r,\lambda)$ which solves
(\ref{krein3}), we will first study functions $P, P_*,
\widehat{P}$, and $\widehat{P_*}$. Below, we list some simple
properties.
\begin{lemma} \label{propertiesP}
\begin{itemize}
\item[1.] For any $\lambda \in \mathbb{C}$
and $r\geq 0$%
\begin{equation}
P(r,\lambda)\widehat{P}_{\ast }(r,\lambda)+P_{\ast
}(r,\lambda)\widehat{P}(r,\lambda)=2\exp (i\lambda r)
\label{e21s3}
\end{equation}
\item[2.] All statements of Lemma \ref{lemma24} are true for both
$P(r,\lambda), P_*(r,\lambda)$ and
$\widehat{P}(r,\lambda),\widehat{P}_*(r,\lambda)$.
\item[3.]
For $\lambda\in \mathbb{C}^+$, we have
\begin{equation}
\Re \left[ P_{\ast }^{-1}(r,\lambda)\widehat{P}_{\ast
}(r,\lambda)\right] \geq \left\vert P_{\ast
}(r,\lambda)\right\vert ^{-2}  \label{e6s3}
\end{equation}
and the last inequality is equality for real $\lambda$.
\end{itemize}
\end{lemma}
\begin{proof}
Part $1)$ follows from (\ref{determinant}) and (\ref{rotation}).

To show $2)$, notice that Christoffel-Darboux formula
(\ref{e101s2}) is the direct consequence of the differential
equations (\ref{krein2}). Then, (\ref{starproperty}) holds because
the functions $\exp(i\lambda
r)\overline{P_*(r,\bar\lambda)},\exp(i\lambda
r)\overline{P(r,\bar\lambda)}$ solve the same Cauchy problem as
$P(r,\lambda), P_*(r,\lambda)$ do. The statements about the zeroes
of $P,P_*$ can be proved in the same way as it was done in Lemma
\ref{lemma24}. The analogous results for
$\widehat{P},\widehat{P}_*$ follow from Corollary
\ref{corollary-a}.

 To prove $3)$, notice that by Lemma \ref{jproperty}, $X^*$ is $J$--contraction for
 $\lambda\in\mathbb{C}^+$ and $J$--unitary for real $\lambda$.
 Writing $XJX^*\leq J$ in terms of
 $P,P_*,\widehat{P},\widehat{P}_*$, we get
 \[
 \frac 12 \left[
\begin{array}{cc}
P\bar{\widehat{P}}+\bar{P}\widehat{P} &  \widehat{P}\bar{P}_*-P\bar{\widehat{P}}_*\\
P_*\bar{\widehat{P}}-\widehat{P}_*\bar{P} &
-(P_*\bar{\widehat{P}}_*+\widehat{P}_*\bar{P}_*)
\end{array}
\right]\leq \left[
\begin{array}{cc}
1 & 0\\
0 & -1
\end{array}
\right]=J
\]
Element $(2,2)$ gives
\begin{equation}
P_*\bar{\widehat{P}}_*+\widehat{P}_*\bar{P}_*\geq 2
\label{bound-below1}
\end{equation}
which implies $3)$. For real $\lambda$, we get equality because
$X^*$ is $J$--unitary.
\end{proof}

\begin{lemma}\label{real-a}
If $A(r)$ is real and $\lambda=0$, the exact solution can be
obtained, i.e.
\[
X(r,0)=\left[
\begin{array}{cc}
\cosh\left(-\int\limits_0^r A(t)dt\right) & \sinh\left(-\int\limits_0^r A(t)dt\right)\\
\sinh\left(-\int\limits_0^r A(t)dt\right) &
\cosh\left(-\int\limits_0^r A(t)dt\right) \end{array}\right]
\]
\end{lemma}
\begin{proof}
The proof is a direct calculation. \end{proof}

\begin{lemma}\label{gronwall-p*}
The following estimate is true if $\lambda\in \mathbb{R}$
\[
\exp\left[-\int\limits_0^r |A(s)|ds\right]\leq
|P_*(r,\lambda)|\leq \exp\left[\int\limits_0^r |A(s)|ds\right]
\]
\end{lemma}
\begin{proof}
The second inequality easily follows from the differential
equations for $P$ and $P_*$. The first one is then immediate from
(\ref{bound-below1}).
\end{proof}

Now, that we studied the general properties of system
(\ref{krein3}), let us show that for any $A(r)\in C[0,\infty)$,
there is the unique accelerant $H(r)\in C[0,\infty)$ that
generates it.

 The following result says that solutions
of Krein system are indeed continuous polynomials.

\begin{lemma}
For any $r>0$,  we have the following
formulas
\begin{equation}
P(r,\lambda ) =\exp (i\lambda r)-\int\limits_{0}^{r}A(r,s)\exp (i\lambda
s)ds \label{kuku1}
\end{equation}
\begin{equation}
P_{\ast }(r,\lambda ) =1-\int\limits_{0}^{r}\overline{A(r,s)}\exp (i\lambda
(r-s))ds  \label{kuku}
\end{equation}
where function $A(r,s)$ is continuous in $s$  and $r:0\leq
s\leq r<\infty$. \label{minlemma}
\end{lemma}

\begin{proof} Consider $Q=\exp (-i\lambda r)P$. We have the
following equations for
$P$~and~$Q$%
\begin{equation}
\left\{
\begin{array}{cccc}
Q^{\prime } & = & -\exp (-i\lambda r)\overline{A}P_{\ast },\,&Q(0,\lambda
)=1
\\
P_{\ast }^{\prime } & = & -\exp (i\lambda r)AQ,\,&P_*(0,\lambda )=1%
\end{array}%
\right.\label{function-q}
\end{equation}%
The corresponding integral equations are
\begin{equation}
Q(r,\lambda )=1-\int\limits_{0}^{r}\exp (-i\lambda s)\overline{A(s)}P_{\ast
}(s,\lambda )ds  \label{e1s3}
\end{equation}%
\begin{equation}
P_{\ast }(r,\lambda )=1-\int\limits_{0}^{r}\exp (i\lambda s)A(s)Q(s,\lambda
)ds \label{e01s3}
\end{equation}%
Let us find the solutions to (\ref{krein2}) in the following form
\begin{equation}
P(r,\lambda )=\exp (i\lambda r)-\int\limits_{0}^{r}A(r,s)\exp (i\lambda
s)ds,P_{\ast }(r,\lambda )=1-\int\limits_{0}^{r}B(r,s)\exp (i\lambda
(r-s))ds
\label{e2s3}
\end{equation}%
where $A$ and $B$ are continuous function. Then, part $2)$ of the
Lemma \ref{lemma24} yields $B(r,s)=\overline{A(r,s)}$. Plug
(\ref{e2s3}) into (\ref{e1s3}) to get the equation for $A(r,s)$
\begin{equation}
A(r,t)=\overline{A(r-t)}-\int\limits_{r-t}^{r}\overline{A(s)}\,\overline{%
A(s,r-t)}ds  \label{erasw}
\end{equation}
Fix any positive $R$. In the triangle $\Delta _{R}=\{0\leq t\leq r\leq R\}$,
consider the operator
\begin{equation}
\lbrack
Of](r,t)=\int\limits_{r-t}^{r}\overline{A(s)}\,\overline{f(s,r-t)}ds
\label{o}
\end{equation}
Let us shows that $O$ is Volterra in $C(\Delta _{R})$. That would allow us
to solve (\ref{erasw}) uniquely.

The following
inequalities hold true

\begin{equation}
\left| \left[ O^{(2n)}f\right] (r,t)\right| \leq \frac{%
||A||_{C[0,R]}^{2n}||f||_{C(\Delta _{R})}(r-t)^{n}t^{n}}{n!\,n!}
\label{e3s3}
\end{equation}
\begin{equation}
\left| \left[ O^{(2n-1)}f\right] (r,t) \right| \leq \frac{%
||A||_{C[0,R]}^{2n-1}||f||_{C(\Delta
_{R})}(r-t)^{n-1}t^{n}}{n!(n-1)!} \label{e4s3}
\end{equation}
Let us prove them by induction. For $n=0,1$, the estimates are
obvious. Assume
that (\ref%
{e3s3}) is true for $n$. Then, we get%
\[
\left\vert \left[ O^{(2n+1)}f\right] (r,t)\right\vert \leq \frac{%
||A||_{C[0,R]}^{2n+1}||f||_{C(\Delta _{R})}}{(n!)^2}(r-t)^{n}\int%
\limits_{r-t}^{r}[s-(r-t)]^{n}ds\leq
\]
\[
\leq \frac{||A||_{C[0,R]}^{2n+1}||f||_{C(%
\Delta _{R})}}{n!(n+1)!}(r-t)^{n}t^{n+1}
\]

Assuming that (\ref{e4s3}) is true for $n$ we obtain%
\[
\left\vert \left[ O^{(2n)}f\right] (r,t)\right\vert \leq \frac{%
||A||_{C[0,R]}^{2n}||f||_{C(\Delta _{R})}}{n!(n-1)!}(r-t)^{n}\int%
\limits_{r-t}^{r}[s-(r-t)]^{n-1}ds\leq
\]
\[
\leq
\frac{||A||_{C[0,R]}^{2n}||f||_{C(%
\Delta _{R})}}{(n!)^2}(r-t)^{n}t^{n}
\]
Therefore, $\|O^{(n)}\|\to 0$  as $n\to\infty$ and $O$ is
Volterra. Notice that the estimates obtained above prove
convergence of the series obtained by the iteration of
(\ref{erasw}). Therefore, the solution $A(r,t)$ is continuous in
$0\leq t\leq r<\infty$. The corresponding $Q$ and $P_{\ast}$ solve
integral equations (\ref{e1s3}) and  (\ref{e01s3}). So, $P$ and
$P_{\ast}$ from (\ref{kuku1}) and (\ref{kuku}) solve the Krein
system.
\end{proof}

\begin{remark}\rm From (\ref{erasw}), we have an identity
\begin{equation}
A(r,0)=\overline{A(r)} \label{recover}
\end{equation}
\end{remark}

\begin{theorem}
For any Krein system (\ref{krein2}) with
\begin{equation}
A(r)\in C[0,\infty),  \label{regularity2}
\end{equation}
there is the unique accelerant $H(x)$ which generates it and
satisfies (\ref{regularity1}). Conversely, any accelerant
satisfying (\ref{regularity1}) gives rise to the Krein system for
which (\ref{regularity2}) is true \label{onetoone}
\end{theorem}

\begin{proof}
The converse statement follows from the construction done in
previous section.

Now, let us find an accelerant that generates the given Krein
system. The clue is given by Lemma \ref{lemmaimp}. Consider the
function $P_{\ast }^{-1}(r,\lambda)\widehat{P}_{\ast
}(r,\lambda)$. We know that $P_{\ast }(r,\lambda)$ does not have
zeroes in
$\overline{\mathbb{C}^{+}}$. By Lemma \ref{minlemma}, both $%
P_{\ast }$ and $\widehat{P}_\ast$ are continuous polynomials, i.e.
\[
P_{\ast }(r,\lambda )=1-\int\limits_{0}^{r}\overline{A(r,r-s)}\exp
(i\lambda s)ds,\
\widehat{P}_{\ast }(r,\lambda )=1-\int\limits_{0}^{r}\overline{\widehat{A}%
(r,r-s)}\exp (i\lambda s)ds \]
 Therefore, Levy-Wiener theorem
yields the
representation%
\begin{equation}
P_{\ast }^{-1}(r,\lambda)\widehat{P}_{\ast
}(r,\lambda)=1+2\int\limits_{0}^{\infty }\overline{H_r(s)}\exp
(is\lambda)ds,  \label{reznik}
\end{equation}
with $H_{r}(s)\in L^{1}(\mathbb{R}^{+})$. $H_{r}(s)$ is continuous
on $[0,r]$ and $[r,\infty]$ but the right and the left limits at
$r$ are not necessarily the same. From differential equations for
$P_*$ and $\widehat{P}_*$ and Lemma \ref{propertiesP} (parts $(2)$
and $(3)$), we have
\begin{equation}\label{integration}
\frac{d}{dr}\left[ P_{\ast }^{-1}(r,\lambda)\widehat{P}_{\ast }(r,\lambda)%
\right] =\frac{2A(r)\exp (ir\lambda)}{P_{\ast }^{2}(r,\lambda)},\,
\lambda\in \mathbb{R}
\end{equation}
Consequently%
\[
P_{\ast }^{-1}(r_{2},\lambda)\widehat{P}_{\ast
}(r_{2},\lambda)-P_{\ast }^{-1}(r_{1},\lambda)\widehat{P}_{\ast
}(r_{1},\lambda)=\int\limits_{r_{1}}^{r_{2}}\frac{2A(s)\exp
(is\lambda)}{%
P_{\ast }^{2}(s,\lambda)}ds,\,0<r_{1}<r_{2}
\]%
and that implies
\begin{equation}
H_{r_{1}}(s)=H_{r_{2}}(s)=H(s) \label{consist}
\end{equation}
for $0<s<r_{1}$. Therefore,  the  function $H(s)$ is well-defined
and continuous on $[0,\infty)$.

 Therefore, if we let $
H_r(-s)=H_r(s)$, $H(-s)=\overline{H(s)}, s>0$, then $H(s)$ is
Hermitian and satisfies (\ref{regularity1}). Now, let us show that
it actually generates the Krein system with given coefficient
$A(r)$.

Indeed, from Lemma \ref{propertiesP}, part 3, we have
\begin{equation}
\frac{1}{|P_{\ast }(r,\lambda )|^{2}}=\Re \left[ P_{\ast
}^{-1}(r,\lambda)\widehat{P}_{\ast
}(r,\lambda)\right]=1+\int\limits_{-\infty}^{\infty
}\overline{H_r(s)}\exp (is\lambda)ds, \label{realpart}
\end{equation}
Notice that the last identity yields
\begin{equation}
\int\limits_{0}^{\infty }|h(x)|^{2}dx+\int\limits_{0}^{\infty
}\int\limits_{0}^{\infty }H_{r}(x-y)h(y)\overline{h(x)}dydx\geq 0
\label{xsa}
\end{equation}
for any $h\in C^\infty _0(0,\infty)$ and the inequality is strict
for any nontrivial $h$. Indeed, one needs to rewrite (\ref{xsa})
in terms of Fourier transform. Then, (\ref{consist}) shows that
$\cal{H}_r+I>0$ for any $r>0$ and $H$ does generate the Krein
system with coefficient $A^{(1)}(r)$ that satisfies
(\ref{regularity2}). Now, let us prove that $A^{(1)}(r)=A(r)$.
Denote the solutions of (\ref{krein2}) with coefficient
$A^{(1)}(r)$ by $P^{(1)}$ and $P^{(1)}_*$. But
$P^{(1)}_*(r,\lambda)=P_*(r,\lambda)$ for any $r>0,\lambda\in
\mathbb{C}$. Indeed, from Lemma \ref{lemmaimp} applied to
$P^{(1)}_*$ and (\ref{realpart}), we get
\[
\cal{P}_{[-r,r]} \left[ \frac{1}{|P_{\ast }(r,\lambda
)|^{2}}-1\right]=\cal{P}_{[-r,r]}\left[\frac{1}{|P^{(1)}_{\ast
}(r,\lambda )|^{2}}-1\right]
\]
Then, by Lemma \ref{uniqueness1} from Appendix, we get
$P^{(1)}_*(r,\lambda)=P_*(r,\lambda)$. Therefore,
$P^{(1)}(r,\lambda)=P(r,\lambda)$ for all $r>0$ and
$A^{(1)}(r)=A(r)$.\end{proof}
\begin{remark}\rm
 Notice that the values of
accelerant on $[0,R]$ depends solely on the values of $A(r)$ on
$[0,R]$ and vice versa.
\end{remark}
The application of Levy-Wiener theorem to $P_{\ast
}^{-1}(r,\lambda)\widehat{P}_{\ast }(r,\lambda)$ shows that
\[
2H_r(+0)=-\widehat{A}(r,r)+A(r,r)
\]
and therefore
\[
H(+0)=\lim_{r\to 0}
H_r(+0)=[-\widehat{A}(0,0)+A(0,0)]/2=\overline{A(0)}
\]
where we used (\ref{recover}). Therefore, if $A(0)\in \mathbb{R}$,
then $H(x)$ is continuous at $0$ and $H(0)\in \mathbb{R}$.

Theorem \ref{onetoone} establishes a one-to-one correspondence
between continuous $A(r)$, defined on $\mathbb{R}^+$, and
continuous accelerants for which (\ref{strict}) is true. But what
happens to a map $\{H(x)\rightarrow A(r)\}$ if (\ref{strict})
fails at a finite point? For OPUC, if the measure has only $k$
growth points, then $D_{k-1}\neq 0$, $D_k=0$. The corresponding
$|a_{j}|<1, j=0,\ldots,k-1, |a_k|=1$. For the Krein system, the
situation is similar. Assume that $1+\cal{H}_r>0$ for all $r<R$
and $\ker(I+\cal{H}_R)\neq 0$. Following argument given above, one
can construct $A(r)\in C[0,R)$. Vice versa, given $A(r)\in
C[0,R)$, we can define $H\in C[0,R)$ such that (\ref{strict})
holds up to $R$. But as long as $\ker(I+\cal{H}_R)\neq 0$, $A(r)$
blows up as $r$ approaches $R$ from the left. More precisely, this
process is governed by a pair of simple (and clearly very crude)
estimates
\begin{equation}
|A(r)|\leq \|\Gamma_r(t,0)\|_{C[0,r]}\leq
\|(I+\cal{H}_r)^{-1}\|_{C[0,r]} \|H\|_{C[0,r]}
\end{equation}
and
\begin{equation}
\|(I+\cal{H}_r)^{-1}\|_{C[0,r]}\leq
1+Cr\|A\|_{C[0,r]}\left(1+r\|A\|_{C[0,r]}\right)\exp [C r
\|A\|_{C[0,r]} ] \label{second1}
\end{equation}
The first estimate easily follows from (\ref{basic1}) with $s=0$.
It shows that $A(r)$ can not blow up unless (\ref{strict}) fails
at a finite point (here we also assume that $H(x)\in C[0,R]$). One
can get (\ref{second1}) from the following inequalities.
\[
\|(I+\cal{H}_r)^{-1}\|_{C[0,r]}= \|I-\Gamma_r\|_{C[0,r]}\leq 1+r
\max_{0\leq s,t \leq r} |\Gamma_r(s,t)|
\]
To estimate the last maximum, we follow the proof of Lemma
\ref{minlemma} (estimates (\ref{e3s3}) and (\ref{e4s3})). This
gives us the following bound
\[
\max_{s\in [0,r]} |\Gamma_r(0,s)|\leq C\|A\|_{C[0,r]}\exp [C r
\|A\|_{C[0,r]} ]
\]
because $A(r,s)$ from Lemma \ref{minlemma} is actually equal to
$\Gamma_r(r,s)=\Gamma_r(r-s,0)$ by (\ref{flip}).

 Using Lemma \ref{lemmamain}, we obtain an estimate
\[
\max_{0\leq s,t \leq r} |\Gamma_r(s,t)|\leq
C\|A\|_{C[0,r]}(1+r\|A\|_{C[0,r]})\exp[Cr\|A\|_{C[0,r]}]
\]
which yields (\ref{second1}). Obviously the left-hand side of
(\ref{second1}) is non-decreasing in $r$.  If it blows up at a
finite time (i.e. (\ref{strict}) fails at a finite time), then $A$
blows up at the same point as well. Comparing the OPUC and Krein
systems, we see that infinity (for Krein systems) plays the role
of $\mathbb{T}=\partial \mathbb{D}$ for OPUC.

 {\bf Remarks and historical notes.}

The problem of constructing accelerant from the Krein system with
locally integrable coefficient $A(r)$ was solved by Rybalko in
\cite{Rybalko}. In \cite{Rybalko}, estimates on $A(r,t)$ from the
Lemma \ref{minlemma} are a bit stronger than what we obtain. Some
generalizations of Krein systems were considered by L. Sakhnovich
in \cite{Sakh1}.

\newpage

\section{Accelerant and $A(r)$ are from $L^2_{\rm loc}(\mathbb{R}^+)$ class}

In this section, we will show that the accelerant from $L^2_{\rm
loc}(\mathbb{R})$ class generates the Krein system with $A(r)\in
L^2_{\rm loc}(\mathbb{R}^+)$ and, conversely, the Krein system
with $A(r)\in L^2_{\rm loc}(\mathbb{R}^+)$ generates the
accelerant $H(x)\in L^2_{\rm loc}(\mathbb{R})$. Moreover, this
highly nonlinear map is homeomorphism in $L^2_{\rm loc}$, i.e., in
$L^2[0,R]$ for any $R>0$. We will prove that all statements from
the previous two sections find their analogs for $L^2_{\rm
loc}$--case.

First, consider an accelerant $H(x)\in L^2[0,r]$ for any $r$.
Then, $I+\cal{H}_r>0$ and  $\Gamma_r(x,y)$ is well-defined as
$L^2([0,r]^2)$-- function. We also have
\[
\Gamma_r(t,s)+\int\limits_0^r H(t-u)\Gamma_r(u,s)du=H(t-s)
\]
Fix any $s\in [0,r]$ in the last equation. Then, $H(t-s)$ is
continuous in $s$ as $L^2[0,r]$-- function in $t$. Therefore,
$\Gamma_r(t,s)$ is continuous in $s$ in $L^2[0,r]$ norm with
respect to the first coordinate. It is also Hermitian function so
the same is true for $s$ and $t$ interchanged. Let
$g_r(s)=\Gamma_r(s,0)$ for $s\in [0,r]$ and $g_r(s)=0$ for $s\in
[r,R]$. Due to the formula $\Gamma_r(s,0)=(I+\cal{H}_r)^{-1}H$, we
have the continuity of $g_r$ in $r$ with respect to $L^2[0,R]$
norm.

The special displacement structure of the kernel of $\cal{H}_r$
allows to generalize Theorems \ref{volter} and \ref{volter1}.

\begin{theorem}
For Hermitian $H(x)\in L^2[-R,R]$, the operator $I+\cal{H}_R$
admits factorization (\ref{factorization}) if and only if
$I+\cal{H}_r>0$ for any $r\in (0,R]$. In this case,
\begin{equation}
\begin{array}{ccc}
V_{+}(x,y) &=&-\Gamma _{y}(x,y),\ x<y \\
V_{-}(x,y) &=&-\Gamma _{x}(x,y),\ x>y
\end{array}
\label{kernels2}
\end{equation}
where $\Gamma_r(x,y)$ denotes the resolvent kernel of
$I+\cal{H}_r$. \label{volter2}
\end{theorem}
\begin{proof}
Now, assume that $I+\cal{H}_r>0$ for any $r\in (0,R]$. Define
$V_{\pm }$ by (\ref{kernels2}).
Notice that operators $\cal{V}_{\pm}$ are well-defined. It follows
from the representation
\[
[\cal{V}_-f](x)=-\int\limits_0^x
\Gamma_x(x,y)f(y)dy=-\int\limits_0^x \Gamma_x(s,0) f(x-s)ds
=-\int\limits_0^R g_x(s)f(x-s)ds
\]
which shows that $\cal{V}_-$ is actually bounded from $L^2[0,R]$
to $L^\infty[0,R]$. Analogous formula is true for $\cal{V}_+$.
These operators also have Hilbert-Schmidt and Volterra properties.
Approximate $H(x)$ by $H^{(n)}(x)\in C[-R,R]$ in $L^2[-R,R]$ norm
and apply Theorem \ref{volter}. For each $n$, formula
(\ref{factor}) is true. Moreover, $V^{(n)}_\pm(x,y) \to
V_\pm(x,y)$, $\Gamma^{(n)}_R(x,y)\to \Gamma_R(x,y)$ as
$n\to\infty$ and convergence is in $L^2([0,R]^2)$. Taking
$n\to\infty$, we get (\ref{factor}) for $G$ and $V_{\pm}$. That
implies the needed factorization. The converse statement is simple
and repeats the argument in Theorem \ref{volter}.
\end{proof}

The analog of the Theorem \ref{volter1} can also be easily proved
in the same way giving

\begin{theorem}
For Hermitian $H(x)\in L^2[-2R,2R]$, the operator
$I+\widehat{\cal{H}}_R$ admits factorization (\ref{lower-upper})
if and only if
 $I+\widehat{\cal{H}}_r$ is invertible in $L^2[-r,r]$ for any $0<r\leq R$. In
this case,
\begin{equation}
\begin{array}{ccc}
\hat{V}_{+}(x,y) &=&-\hat\Gamma _{|y|}(x,y),\ (x,y)\in \hat\Omega_+ \\
\hat{V}_{-}(x,y) &=&-\hat\Gamma _{|x|}(x,y),\ (x,y)\in
\hat\Omega_-
\end{array}
\label{kernels1-l2}
\end{equation}
 \label{volter7}
\end{theorem}

All functions that are used in the definition of continuous
polynomials are now well-defined as elements of $L^2[0,r]$.
Indeed, $\Gamma_r(r,s)$ and $\Gamma_r(0,s)$ are both from
$L^2[0,r]$ and we can consider the corresponding continuous
polynomials $P(r,\lambda),P_*(r,\lambda)$. Moreover, Lemma
\ref{transformation} and Theorem \ref{theorem2s2} are true for
$H\in L^2_{\rm loc}(\mathbb{R})$ as well. Indeed, their proofs
were based on the factorization of Fredholm operators (the analog
of which we just proved).

Now, let us define the coefficient $A(r)$ and show that $A(r)\in
L^2_{\rm loc}(\mathbb{R}^+)$ and that equations (\ref{krein2}) are
true.

Consider an accelerant $H(x)\in L^2[-R,R]$ and approximate it with
$H^{(n)}(x)\in C[-R,R]$ in $L^2[-R,R]$ norm. Then, each
$H^{(n)}(x)$ generates the Krein system on $[0,R]$ with
$A^{(n)}(r)\in C[0,R]$. We have
\[
\overline{A^{(n)}}(r)=H^{(n)}(r)-\int\limits_0^r
H^{(n)}(r-u)\Gamma^{(n)}_r(u,0)du \to H(r)-\int\limits_0^r
H(r-u)\Gamma_r(u,0)du
\]
in $L^2[0,R]$. The conjugate of last function will be denoted by
$A(r)$, i.e.
\[
\overline{A(r)}=H(r)-\int\limits_0^r H(r-u)\Gamma_r(u,0)du
\]
Notice that the second term is continuous function in $r$.
Therefore, all singularities of $A(r)$ and $H(r)$ coincide. In
particular, $A(r)\in L^2[0,R]$. For any $\lambda\in \mathbb{C}$,
the polynomials $P^{(n)}(r,\lambda), P_*^{(n)}(r,\lambda)$
converge to $P(r,\lambda)$ and $P_*(r,\lambda)$ uniformly in $r\in
[0,R]$. Notice that equations (\ref{krein2}) for $P^{(n)}$ and
$P_*^{(n)}$ are equivalent to the system of integral equations

\[
P^{(n)}(r,\lambda)=1+ \int\limits_0^r \left[i\lambda
P^{(n)}(s,\lambda)-\overline{A^{(n)}(s)}P_*^{(n)}(s,\lambda)\right]ds
\]
\[
P_*^{(n)}(r,\lambda)=1-\int\limits_0^r
A^{(n)}(s)P^{(n)}(s,\lambda)ds
\]
Taking the limit $n\to\infty$, we see that $P$ and $P^*$ satisfy
the corresponding equations that are equivalent to (\ref{krein2}).
The proofs of Lemma \ref{lemma24} and Lemma \ref{lemmaimp} work
for $L^2_{\rm loc}$ case. Lemma \ref{lemmamain} can be shown by an
approximation argument.

Now, following the arguments from the Section \ref{four}, we
consider $A(r)\in L^2_{\rm loc}(\mathbb{R}^+)$.  We have
\begin{theorem}
For any Krein system (\ref{krein2}) with
\begin{equation}
A(r)\in L^2_{\rm loc} (\mathbb{R}^+)
\end{equation}
there is the unique accelerant which generates it and satisfies
\begin{equation}
H(x)\in L^2_{\rm loc}(\mathbb{R}) \label{regularity3}
\end{equation}
Conversely, any accelerant satisfying (\ref{regularity3}) gives
rise to Krein system with $A(r)\in L^2_{\rm loc}(\mathbb{R}^+)$.
This map is a homeomorphism.
\end{theorem}
\begin{proof}
As we just showed, the $L^2_{\rm loc}$ accelerant does generates
the Krein system with $L^2_{\rm loc}$ coefficient. Now, let us
start with $A(r)\in L^2_{\rm loc}(\mathbb{R}^+)$. Notice that all
arguments from the proof of Theorem \ref{onetoone} are valid as
long as we have an analog of Lemma~\ref{minlemma}. Therefore, we
just need
\begin{lemma}
If $A(r)\in L^2_{\rm loc}(\mathbb{R}^+)$, and $P(r,\lambda)$,
$P_*(r,\lambda)$ are defined as solutions to the Krein system,
then the representations (\ref{kuku1}), (\ref{kuku}) hold with
$A(r,s)\in L^2[0,r]$.
\end{lemma}
\begin{proof}
We need to consider the operator ${O}$ given by (\ref{o}) and show
that it is Volterra in the space of functions $f(r,y), 0\leq y\leq
r\leq R$ such that
\[
\|f\|_{\infty,2}=\sup_{r\in [0,R]}\left[\int\limits_0^r
|f(r,y)|^2dy\right]^{1/2}<\infty
\]
Let us prove by induction that
\begin{equation}
\left\|\left[O^{(n)}f\right](r,t)\right\|_{t,L^2[0,r]}\leq
\left[\int\limits_0^r |A(s)|ds\right]^n\cdot \|f\|_{\infty,2}/(n!)
\label{ind-1}
\end{equation}
For $n=0$, the statement is elementary. Assume that this estimate
is true for $n$. Then, for $n+1$, we have
\[
\left[O^{(n+1)}f\right](r,t)=\int\limits_{r-t}^r
A(s)\left[O^{(n)}f\right](s,r-t)ds
\]
Application of Minkowski inequality and induction assumption
yields
\[
\left\|\left[O^{(n+1)}f\right](r,t)\right\|_{t,L^2[0,r]}\leq
\|f\|_{\infty,2}(n!)^{-1}\int\limits_0^r
|A(s)|\left[\int\limits_0^s |A(u)|du\right]^n
\]
\[
=\|f\|_{\infty,2}[(n+1)!]^{-1}\left[\int\limits_0^r
|A(s)|ds\right]^{n+1}
\]
Thus, we have (\ref{ind-1}) and
\[
\|A(r,t)\|_{t,L^2[0,r]}\leq \|A\|_{L^2[0,r]} \exp
\left[\|A\|_{L^1[0,r]}\right]
\]
This estimate finishes the proof of the Lemma.
\end{proof}
We are left with proving
\begin{lemma} \label{l2regular}
The following estimates are true for any $R>0$
\begin{equation}\label{l2-l2-1}
\|H(x)\|_{L^2[-R,R]}\leq
C\|A(r)\|_{L^2[0,R]}\exp\left(C\|A\|_{L^1[0,R]}\right)
\end{equation}
\begin{equation}\label{l2-l2-2}
 \|A(r)-\overline{H(r)}\|_{L^\infty [0,R]}\leq
 \|H\|^2_{L^2[0,R]} \|(I+\cal{H}_R)^{-1}\|_{2,2}
\end{equation}
and the map $A(r)\to H(x)$ is homeomorphism in $L^2_{\rm loc}$.
\end{lemma}
\begin{proof}
 From (\ref{reznik}) and (\ref{integration}), we have
\[
\int\limits_0^\infty \overline{H_R(s)}\exp(i\lambda
s)ds=\int\limits_0^R \frac{A(s)\exp(i\lambda
s)}{P_*^2(s,\lambda)}ds=
\]
\[
=\frac{1}{P_*^2(R,\lambda)}\int\limits_0^R A(s)\exp(i\lambda
s)ds-\int\limits_0^R \left[\int\limits_0^s A(u)\exp(i\lambda
u)du\right] \frac{A(s)P(s,\lambda)}{P^3_*(s,\lambda)}ds
\]
Therefore, using $|P(s,\lambda)|=|P_*(s,\lambda)|$ and Lemma
\ref{gronwall-p*}, we get
\[
\|H_R\|_2\leq
\|A\|_{L^2[0,R]}\exp(2\|A\|_{L^1[0,R]})+\int\limits_0^R
\|A\|_{L^2[0,s]} |A(s)|\exp(2\|A\|_{L^1[0,s]})ds
\]
which yields (\ref{l2-l2-1}) since $H(x)=H_R(x)$ for $|x|<R$. This
argument also shows that $H(x)$ depends on $A(r)$ continuously in
the $L^2_{\rm loc}$.

To prove (\ref{l2-l2-2}), we use (\ref{basic1}) to write
\begin{equation}\label{operator-kernel}
\Gamma_r(t,0)+\int\limits_0^r H(t-u)\Gamma_r(u,0)du=H(t)
\end{equation}
Taking $t=r$,
\begin{equation}
\overline{A(r)}+\int\limits_0^r H(r-u)\Gamma_r(u,0)du=H(r)
\label{a-eq}
\end{equation}
 The
formula (\ref{operator-kernel}) implies
$\Gamma_r(t,0)=(I+\cal{H}_r)^{-1}H$ and, therefore,
$\|\Gamma_r(t,0)\|_{L^2[0,r]}\leq \|(I+\cal{H}_r)^{-1}\|_{2,2}
\|H\|_{L^2[0,r]}$. That yields (\ref{l2-l2-2}). Also, the map
$H(x)\longrightarrow A(r)$ is continuous in $L^2_{\rm loc}$.
\end{proof}
\end{proof}
We want to mention here that slight modification of the arguments
allows to prove that the map $H(x)\longrightarrow A(r)$ is
homeomorphism in $L^p_{\rm loc}(\mathbb{R})-L^p_{\rm
loc}(\mathbb{R}^+)$ for any $p>1$. For one direction, one just
have to iterate equation (\ref{operator-kernel}) sufficiently many
times to achieve the necessary gain in regularity and then plug it
in (\ref{a-eq}). The other direction is straightforward.
\newpage

\section{Continuous analogs of Wall polynomials
and Schur function.  Bernstein-Szeg\H{o} approximation}
\bigskip Let us introduce the continuous analogs of the so-called Wall
polynomials. Consider the functions $\A(r,\lambda), \B(r,\lambda),
\A_*(r,\lambda), \B_*(r,\lambda)$ defined by
\begin{equation}
X(r,\lambda)= \left[
\begin{array}{cc}
\A_*(r,\lambda) & \B_*(r,\lambda)\\
\B(r,\lambda) & \A(r,\lambda)
\end{array}
\right] \label{fsr}
\end{equation}
were $X(r,\lambda)$ is the transfer matrix given in
(\ref{krein3}). By analogy with OPUC theory, it makes sense to
call $\A$ and $\B$ the continuous Wall polynomials. They can be
rewritten in the following way

\[
\A(r,\lambda )=\frac{P_{\ast }(r,\lambda )+\widehat{P}_{\ast }(r,\lambda )%
}{2},\A_{\ast }(r,z)=\frac{P(r,\lambda )+\widehat{P}(r,\lambda
)}{2}
\]
\begin{equation}
\B(r,\lambda ) =\frac{P_{\ast }(r,\lambda )-\widehat{P}_{\ast }(r,\lambda )%
}{2},\B_{\ast }(r,\lambda )=\frac{P(r,\lambda
)-\widehat{P}(r,\lambda )}{2} \label{e19s3}
\end{equation}

\begin{lemma} \label{wall}
For continuous Wall polynomials, the following identities are true
\begin{itemize}
\item[1)] For $\lambda\in \mathbb{R}$,
\begin{equation}
|\A|^2-|\B|^2=1, |\A|=|\A_*|, |\B|=|\B_*|,\overline{\A_*}
B_*=\overline{\B} \A \label{simplect}
\end{equation}

\item[2)]
For $\lambda \in\mathbb{C}$,
\begin{equation}
\A(r,\lambda)\A_*(r,\lambda)-\B(r,\lambda)\B_*(r,\lambda)=\exp(i\lambda
r),\label{energy}
\end{equation}

\item[3)]
For $\lambda\in \mathbb{C}^+$,
\begin{equation}
|\A|^2-|\B|^2\geq 1, |\A|^2-|\B_*|^2\geq 1  \label{contr1}
\end{equation}
\end{itemize}
\end{lemma}
\begin{proof}
The proof follows directly from Theorem \ref{jcontraction} or
Lemma \ref{propertiesP}.
\end{proof}

Notice that (\ref{contr1}) implies
$\B(r,\lambda)\A^{-1}(r,\lambda)\in B(\mathbb{C}^+)$.

\begin{theorem}
The ratio $\B(r,\lambda)\A^{-1}(r,\lambda)$ converges to
$\f(\lambda)\in B(\mathbb{C}^+)$. This convergence is uniform over
all compacts in $\mathbb{C}^+$. \label{schur-function}
\end{theorem}

\begin{proof}
 Take $r_{1}<r_{2}$. Consider the Krein system on the
interval $[r_{1},\infty ).$  Denote the transfer matrix from $r_1$
to $r_2$ by $X(r_1,r_2,\lambda)$. Then, in our notations,
$X(0,r,\lambda)=X(r,\lambda)$. If we introduce

\begin{equation}
X(r_1,r,\lambda)= \left[
\begin{array}{cc}
\a_*(r,\lambda) & \b_*(r,\lambda)\\
\b(r,\lambda) & \a(r,\lambda)
\end{array}
\right],r>r_1 \label{fsr1}
\end{equation}
then $\a$ and $\b$ are Wall polynomials for the same Krein system
considered on the interval $[r_1,\infty)$. Obviously, Lemma
\ref{wall} will hold for these functions as well.

The semigroup relation
\[
X(0,r_{2})=X(r_1,r_{2})\cdot X(0,r_{1})
\]%
yields%
\begin{equation}
\begin{array}{ccc}
\A(r_{2}) &=&\b(r_2)\B_{\ast }(r_{1})+\a(r_2)\A(r_{1}) \\
\B(r_{2}) &=&\b(r_2)\A_{\ast }(r_{1})+\a(r_2)\B(r_{1})
\end{array}\label{relations-aA}
\end{equation}
For the function $\B\A^{-1}$,%
\begin{equation}
\frac{\B(r_{2})}{\A(r_{2})}=\frac{\B(r_{1})+(\b\a^{-1})\A_{\ast }(r_{1})}{%
\A(r_{1})+(\b\a^{-1})\B_{\ast }(r_{1})}  \label{schur1}
\end{equation}
and
\[
\frac{\B(r_{2})}{\A(r_{2})}-\frac{\B(r_{1})}{\A(r_{1})}=\frac{(\b\a^{-1})
(\A(r_1)\A_*(r_1)-\B(r_1)\B_*(r_1))
}{\A(r_{1})[\A(r_{1})+(\b\a^{-1})\B_{\ast }(r_{1})]}=
\]
\begin{equation}
=\frac{(\b\a^{-1})\exp (i\lambda
r_{1})}{\A(r_{1})[\A(r_{1})+(\b\a^{-1})\B_{\ast }(r_{1})]}
\label{energy22}
\end{equation}
where we have used (\ref{energy}).

 We have $\b\a^{-1}\in
B(\mathbb{C}^+)$. Therefore, from (\ref{contr1}), we get
\begin{equation}
\left\vert \A(r_{1})[\A(r_{1})+(\b\a^{-1})\B_{\ast
}(r_{1})]\right\vert \geq \left( |\A(r_{1})|-|\B_*(r_{1})|\right)
|\A(r_{1})|\geq
\end{equation}
\begin{equation}
\geq
\frac{%
|\A(r_{1})|^{2}-\left\vert \B_*(r_{1})\right\vert ^{2}}{2}\geq
\frac 12, \lambda\in \mathbb{C}^+ \label{e22s3}
\end{equation}
Then,
\[
\left\vert
\frac{\B(r_{2})}{\A(r_{2})}-\frac{\B(r_{1})}{\A(r_{1})}\right\vert
\leq 2\exp (-r_{1}\Im \lambda )
\]%
That means $\B(r,\lambda )\A^{-1}(r,\lambda )$ converges to a
certain function $\f(\lambda )$ as $\rho\to\infty$. This
convergence is uniform in any compact in $\mathbb{C}^{+}$.
\end{proof}

We will call $\f(\lambda)$ the Schur function corresponding to
$[0,\infty)$. Notice that $\f(\lambda )\in \B(\mathbb{C}^{+})$.
The following formula is the consequence of  (\ref{schur1}) if one
takes $r_2\to\infty$
\begin{equation}
\f(\lambda )=\frac{\B(\rho,\lambda)+\f_{\rho}(\lambda )\A_{\ast
}(\rho,\lambda)}{\A(\rho,\lambda)+\f_{\rho}(\lambda )\B_{\ast
}(\rho,\lambda)} \label{e18s3}
\end{equation}
where $\f_{\rho}(z)$ is Schur's function for the same Krein but on
the
interval $%
[\rho,\infty )$.

Not every function from $B(\mathbb{C}^+)$ is Schur's function of
some Krein system. The characterization of that special subclass
will be given later but now we just want to mention that
\begin{equation}
\f(\lambda)\to 0 \label{feature}\end{equation} if $\Im\lambda\to
+\infty$. It easily follows from the formula (\ref{e18s3}) with
any fixed $\rho$ and relations
\[
\f_{\rho}(\lambda )\in B(\mathbb{C}^{+}),\A(\rho,\lambda
)\rightarrow 1,\B(\rho,\lambda)\rightarrow 0,\A_{\ast
}(\rho,\lambda )\rightarrow 0,\B_{\ast }(\rho,\lambda)\rightarrow
0
\]
as $\Im\lambda\to +\infty$.

 In the previous section, we constructed an accelerant from the
given Krein system. Then, the measure $\sigma$ and the constant
$\beta$ can be found  from the formula (\ref{e2s2}). But there is
more direct way to find these data. The next Theorem develops an
analog of the Weyl-Titchmarsh theory \cite{Levitan} for Krein
systems.

\begin{theorem}
The ratio $\hat{P}_*(r,\lambda)P_*^{-1}(r,\lambda)$ converges to
the function $F(\lambda)$ uniformly in any compact in
$\mathbb{C}^+$ as $r\to\infty$. This function $F(\lambda)$ has the
positive real part in $\mathbb{C}^+$ and allows the following
representation
\begin{equation}
F(\lambda )/2=-i\beta  +i\int\limits_{-\infty }^{\infty }
\frac{1+\lambda t}{(\lambda-t)(1+t^2) } d\sigma (t)
\label{weyl-titchmarsh}
\end{equation}
where $d\sigma$ and $\beta$ coincide with those from the formula
(\ref{e2s2}). Moreover, the sequence of measures
\begin{equation}
d\sigma_r(\lambda)=\frac{d\lambda}{(2\pi) |P_*(r,\lambda)|^2} \to
d\sigma(\lambda)
\end{equation}
in the weak-{\rm($\ast$)} sense (analog of Bernstein-Szeg\H{o}
approximation).\label{Bernstein-Szego}
\end{theorem}
\begin{proof}
From (\ref{e19s3}),
\[
P_{\ast }=\A+\B,\widehat{P}_{\ast }=\A-\B,P=\A_{\ast }+B_{\ast },\widehat{P}%
=\A_{\ast }-\B_{\ast }
\]%
Therefore,
\[
P_{\ast }^{-1}(r,\lambda )\widehat{P}_{\ast }(r,\lambda )=(\A-\B)(\A+\B)^{-1}=%
\frac{1-\A^{-1}\B}{1+\A^{-1}\B}
\]%
That shows convergence of $P_{\ast }^{-1}(r,\lambda
)\widehat{P}_{\ast }(r,\lambda )$ to the function
\begin{equation}
F(\lambda )=(1-\f)(1+\f)^{-1}  \label{e23s3}
\end{equation}
Again, this convergence is uniform for compacts in
$\mathbb{C}^{+}$. Function $F(\lambda )$ has positive real part in
$\mathbb{C}^{+}$ because $\f(\lambda)\in B(\mathbb{C}^+)$.
Therefore, $F(\lambda)/2$ admits the following integral
representation \cite{AkhGlaz}
\begin{equation}
F(\lambda )/2=-i\beta -i\alpha \lambda +i\int\limits_{-\infty
}^{\infty } \frac{1+\lambda t}{(\lambda-t)(1+t^2) } d\sigma (t)
\label{e1s3-1}
\end{equation}
where $\alpha ,\beta \in \mathbb{R},\alpha \geq 0$ and
non-decreasing function $\sigma $ is such that
\[
\int\limits_{-\infty }^{\infty }\frac{d\sigma (t)}{t^{2}+1}<\infty
\]%

Recall the way this formula is obtained. If we map $\lambda\in
\mathbb{C}^+$ onto $z\in \mathbb{D}$ by conformal mapping
$z=(\lambda-i)(\lambda+i)^{-1}$, then the function
$g(z)=2^{-1}F(i(z+1)(1-z)^{-1})$ is the Herglotz function in
$\mathbb{D}$. It has a canonical representation, which can be
written as follows
\[
g(z)=-i\beta +\int\limits_\mathbb{T} \frac{\xi+ z}{\xi- z}
d\tau(\xi)
\]
If $\alpha=d\sigma \{1\}$, the mass at point $1$, then we have
formula (\ref{e1s3-1}) with
$(1+t^2)^{-1}d\sigma(t)=d\tau[(t-i)(t+i)^{-1}]$.

Next, our goal is to show that $\alpha=0$, and $\beta$ and
$d\sigma$ coincide with those from (\ref{e2s2}) in section 3.

From \cite{Atkinson}, p. 630, we have
\[
\alpha =\lim_{\eta \rightarrow +\infty }\frac{iF(i\eta )}{i\eta }
\]%
But (\ref{feature}) implies  $\lim_{\eta \rightarrow +\infty
}F(i\eta )=1$ so $\alpha=0$. In other words, measure $d\tau$ has
no mass at $\xi=1$.

Notice now that each function $P_{\ast }^{-1}(r,\lambda )\widehat{P}%
_{\ast }(r,\lambda )/2$ has positive real part as well and admits
the same representation (\ref{e1s3-1}) with $\alpha _{r}=0$,
$\beta_r\in \mathbb{R}$, and absolutely continuous measure
$d\sigma _{r}$, which is the transplantation of some $d\tau_r$.
From (\ref%
{e6s3}), we have $\sigma _{r}^{\prime }(\lambda )=(2\pi)^{-1}
|P_{\ast }(r,\lambda )|^{-2}$. Measures $d\sigma _{r}$ are analogs
of the so-called Bernstein-Szeg\H{o} approximations for OPUC. They
converge weakly to $d\sigma $. Indeed,  the convergence
of $P_{\ast }^{-1}(r,\lambda )\widehat{P}%
_{\ast }(r,\lambda )/2$ to $F(\lambda)$ in $\mathbb{C}^+$ implies
convergence of the corresponding functions within the unit disc
$\mathbb{D}$.  By Stone-Weierstrass theorem, that yields weak
convergence of measures $d\tau_r$ to $d\tau$, $\beta_r$ to $\beta$
and so the weak convergence of $d\sigma_r$ to $d\sigma$. Since
$d\tau$ has no mass at $\xi=1$, the family of measures
$(1+t^2)^{-1}d\sigma_r(t)$ is tight, i.e. for any $\epsilon>0$,
there is $T(\epsilon), R(\epsilon)>0$ such that
\begin{equation}
\int\limits_{|t|>T(\epsilon)} \frac{d\sigma_r(t)}{1+t^2}<\epsilon
\label{tightness}
\end{equation}
if $r>R(\epsilon)$.

 Let us show now that $\sigma $ coincides with the measure
from (\ref{e2s2}).
Using (\ref{e1s1}) and (\ref{reznik}), we obtain%
\[
\frac 12 +\int\limits_0^\infty \overline{H_r(x)}\exp(i\lambda x)dx
\]
\[
=-\lambda ^{2}\int\limits_{0}^{\infty }\left[ -i\beta
_{r}x+\int\limits_{-\infty
}^{\infty }\left( 1-\frac{it x}{1+t ^{2}}-\exp (-it x)\right) \frac{%
d\sigma _{r}(t )}{t ^{2}}\right] \exp (i\lambda x)dx
\]%

From (\ref{laplace-1}), we get
\[
\frac{|t|}{2}+\int\limits_{0}^{t}(t-s)H_{r}(s)ds=i\beta
_{r}t+\int\limits_{-\infty }^{\infty }\left( 1+\frac{it s}{1+s ^{2}}%
-\exp (it s)\right) \frac{d\sigma _{r}(s )}{s ^{2}}
\]%
Fix $t$ and take $r\to\infty$ in the last equation. The formula
(\ref{consist}) and tightness (\ref{tightness}) yield
\[
\frac{|t|}{2}+\int\limits_{0}^{t}(t-s)H(s)ds=i\beta
t+\int\limits_{-\infty
}^{\infty }\left( 1+\frac{it s}{1+s ^{2}}-\exp (it s)\right) \frac{%
d\sigma (s )}{s ^{2}}
\]
Since the integral representation of $G_\infty$ functions is
unique (Theorem \ref{t1s1}), we get the statement of the Theorem.
\end{proof}
We also have the following important
\begin{corollary}\label{scale-property}
For any $f(x)\in L^2[0,\rho]$, the following identity is true
\begin{equation}
\int\limits_{-\infty}^\infty \left|\int\limits_0^\rho
f(x)\exp(i\lambda x)dx\right|^2d\sigma(\lambda)=
\int\limits_{-\infty}^\infty \left|\int\limits_0^\rho
f(x)\exp(i\lambda x)dx\right|^2 \frac{d\lambda}{2\pi
|P(\rho,\lambda)|^2}\label{form1-scale}
\end{equation}
\end{corollary}
\begin{proof}
We have $|P(r,\lambda)|=|P_*(r,\lambda)|$ for real $\lambda$.
Then, the Plancherel theorem for the Fourier integrals and
(\ref{realpart}) yield that the right hand side of
(\ref{form1-scale}) is equal to
\[
\int\limits_0^\rho |f(x)|^2dx+\int\limits_0^\rho
\int\limits_0^\rho
\overline{H_\rho(x-u)}f(u)\overline{f(x)}dudx
\]
\begin{equation}
=\int\limits_0^\rho |f(x)|^2dx+\int\limits_0^\rho
\int\limits_0^\rho \overline{H(x-u)}f(u)\overline{f(x)}dudx
\label{form2-scale}
\end{equation}
where (\ref{consist}) is used to get the last equality. Then, the
formula (\ref{e1s2n1}) shows that the l.h.s. of
(\ref{form1-scale}) is equal to r.h.s. of (\ref{form2-scale}).
\end{proof}
The formula similar to (\ref{e18s3}) is true for Weyl-Titchmarsh
function as well. From (\ref{e18s3}) and relation between $F$ and
$\f$, we get
\begin{equation}
F(\lambda)=\frac{\widehat{P}_*(\rho,\lambda)-\widehat{P}(\rho,\lambda)+
F_\rho(\lambda)(\widehat{P}_*(\rho,\lambda)+\widehat{P}(\rho,\lambda))}
{{P}_*(\rho,\lambda)+{P}(\rho,\lambda)+
F_\rho(\lambda)({P}_*(\rho,\lambda)-{P}(\rho,\lambda))}
\end{equation}

{\bf Remarks and historical notes.} The Weyl-Titchmarsh theory for
Krein systems was developed to some extent in \cite{Rybalko}.

\newpage
\section{Dual system. Some simple considerations}

The dual Krein system is obtained by changing the sign of the
coefficient $A(r)$ (see Corollary \ref{corollary-a}). Due to
Corollary \ref{corollary-a}, functions $\widehat{P}(r,\lambda),
\widehat{P}_*(r,\lambda)$ are continuous orthogonal polynomials
for the dual system. They are usually called the dual continuous
orthogonal polynomials.

The dual Krein system can be characterized by the dual accelerant.
Let us call it $\widehat{H}$. The relation between accelerant and
dual accelerant is very simple. \begin{lemma} For the dual
accelerant $\widehat{H}$, we have
\begin{equation}
H(x)+\widehat{H}(x)+2\int\limits_{0}^{x}H(x-s)\widehat{H}(s)ds=0,
\, x\in \mathbb{R} \label{dualaccelerant}
\end{equation}
\end{lemma}
\begin{proof}
We have
\[
\frac{P_*(r,\lambda)}{\widehat{P}_*(r,\lambda)} \cdot
\overline{\left(\frac{P_*(r,\lambda)}{\widehat{P}_*(r,\lambda)}\right)}=1
\]
Substitute (\ref{reznik}) to the seconds factor and analogous
formula to the first one gives
\[
\left(1+2\int\limits_0^\infty \overline{H_r(x)}\exp(i\lambda x)dx
\right)\left(1+2\int\limits_0^\infty
\overline{\widehat{H}_r(x)}\exp(i\lambda x)dx \right)=1
\]
which implies
\[
H_r(x)+\widehat{H}_r(x)+2\int\limits_0^x
H_r(x-t)\widehat{H}_r(t)dt=0
\]
Then, use (\ref{consist}) to get (\ref{dualaccelerant}).
\end{proof}
Clearly, the last Theorem allows one to find $\hat{H}$ from $H$ by
solving Volterra equation. The  algebraic explanation to
(\ref{dualaccelerant}) is as follows. Consider two operators
\[
[\cal{A}f](x)=f(x)+2\int\limits_0^x H(x-u)f(u)du,
[\widehat{\cal{A}}f](x)=f(x)+2\int\limits_0^x
\widehat{H}(x-u)f(u)du
\]
acting in $L^2[0,r]$. They both have positive real parts:
\begin{equation}
\Re\cal{A}=I+\cal{H}_r, \Re
\widehat{\cal{A}}=I+\widehat{\cal{H}}_r \label{car-to1}
\end{equation}
and the formula (\ref{dualaccelerant}) is equivalent to
\begin{equation}
\widehat{\cal{A}}=\cal{A}^{-1} \label{car-to2}
\end{equation}
These identities arise naturally from the solution to continuous
Caratheodory-Toeplitz problem \cite{krein-madamyan}.
 The
relation between $\Gamma_r(x,y)$ and the dual resolvent kernel
$\widehat\Gamma_r(x,y)$ is also quite simple and can be obtained
from (\ref{car-to1}) and (\ref{car-to2}).

Consider the dual Weyl-Titchmarsh function
$\widehat{F}(\lambda )=\lim_{r\rightarrow \infty }P_{\ast }(r,\lambda )%
\widehat{P}_{\ast }^{-1}(r,\lambda )$. Then, by Theorem
\ref{Bernstein-Szego},
\[
\widehat{F}=F^{-1}, \widehat{\f}=-\f
\]

Now, let us study how the parameters of the Krein system change
upon some simple transformations of the coefficient $A(r)$.

\begin{lemma}[Shift]
Let  $A(r)$ be coefficient of (\ref{krein2}). Then
$A^{(t)}(r)=A(r)\exp (irt), t\in \mathbb{R}$ corresponds to
\[
\sigma^{(t)}(\lambda )=\sigma (\lambda +t),
H^{(t)}(x)=\exp(-itx)H(x),
\Gamma_r^{(t)}(x,y)=\Gamma_r(x,y)\exp(-it(x-y)) \] \label{lemma36}
\end{lemma}

\begin{proof} Let pairs $\{P,P_{\ast }\}$, $\{P^{(t)},P_{\ast
}^{(t)}\}$ be solutions of Krein system with coefficient $A$ and
$A^{(t)}$, respectively. Introduce $Q=\exp (-i\lambda r)P$,
$Q^{(t)}=\exp (-i\lambda r)P^{(t)}$. We have (see formula
(\ref{function-q})):
\begin{equation}
\left\{
\begin{array}{cccc}
Q^{\prime }=&-\bar{A}\exp (-i\lambda r)P_{\ast },& Q(0,\lambda)=1 \\
P_{\ast }^{\prime }=&-A\exp (i\lambda r)Q,& P_{\ast }(0,\lambda)=1
\end{array}
\right.
\end{equation}%
and
\begin{equation}
\left\{
\begin{array}{cccc}
Q^{(t)\prime }=&-\bar{A}\exp (-i(\lambda +t)r)P_{\ast }^{(t)},&
Q^{(t)}(0,\lambda)=1 \\
P_{\ast }^{(t)\prime }=&-A\exp (i(\lambda +t)r)Q^{(t)},& P_{\ast
}^{(t)}(0,\lambda)=1
\end{array}
\right.
\end{equation}
Therefore, \ $P_{\ast }^{(t)}(r,\lambda )=P_{\ast }(r,\lambda +t)$. By Theorem \ref{Bernstein-Szego}, $%
(2\pi )^{-1}|P_{\ast }^{(t)}(r,\lambda )|^{-2}d\lambda
\rightharpoonup d\sigma^{(t)}$, we get the shift in the measure.
For dual system, we have the same results. Then, by
(\ref{reznik}),
\begin{equation}
\frac{\widehat{P}_{\ast }^{(t)}(r,\lambda)}{P_{\ast
}^{(t)}(r,\lambda)}=1+2\int\limits_{0}^{\infty
}\overline{H_r^{(t)}(s)}\exp
(is\lambda)ds=1+2\int\limits_{0}^{\infty }\overline{H_r(s)}\exp
(is(\lambda+t))ds,  \label{reznik1}
\end{equation}
and we have the needed formula for $H^{(t)}$. The way resolvent
kernel changes is easy to obtain from (\ref{basic1}) or
(\ref{basic2}).
\end{proof}
It is an easy exercise to show directly that if  $H(x)$ is an
accelerant, then $H(x)\exp(-itx)$ is an accelerant as well.
Coefficient $\beta^{(t)}$ from formula (\ref{e2s2}) changes in a
more intricate way. It can be recovered by noticing that
$F^{(t)}(\lambda)=F(\lambda+t)$ and by integral representations
for both Weyl-Titchmarsh functions. We then get

\[
\beta^{(t)}=\beta+\int\limits_{-\infty}^\infty \left(
\frac{s-t}{1+(s-t)^2}-\frac{s}{1+s^2}\right) d\sigma(s)
\]

 Lemma \ref{lemma36} has its analog in the OPUC theory, rather
than the following Lemma

\begin{lemma}[Dilation]
For any $\gamma>0$, coefficient $\gamma A(\gamma r)$ corresponds
to
\[
\sigma_{(\gamma)}(\lambda)=\gamma\sigma (\gamma ^{-1}\lambda
),\gamma
>0, H_{(\gamma)}(x)=\gamma H(\gamma x),
\Gamma_{(\gamma),r/\gamma}(x,y)=\gamma\Gamma_r(\gamma x,\gamma y)
\]
\end{lemma}

\begin{proof} Under the change of variables $\rho =\gamma r$,
system (\ref{krein2}) changes as follows
\begin{equation}
\left\{
\begin{array}{ll}
\displaystyle \frac{dP(\gamma r,\lambda )}{dr}=i\gamma \lambda
P(\gamma r,\lambda )-\gamma
\bar{A}(\gamma r)P_{\ast }(\gamma r,\lambda ), & P(0,\lambda )=1, \\
\displaystyle \frac{dP_{\ast }(\gamma r,\lambda )}{dr}=-\gamma
A(\gamma r)P(\gamma r,\lambda
), & P_{\ast }(0,\lambda )=1%
\end{array}%
\right.
\end{equation}%
which proves the Lemma following the same lines as in the proof of
Lemma \ref{lemma36}.
\end{proof}
Again, if $H(x)$ is an accelerant, then one directly checks that
$\gamma H(\gamma x)$ is also an accelerant. Since
$F_{(\gamma)}(\lambda)=F(\lambda/\gamma)$, we have
\[
\beta_{(\gamma)}=\beta+\int\limits_{-\infty}^\infty \left(
\frac{s\gamma^2}{1+s^2\gamma^2}-\frac{s}{1+s^2}\right) d\sigma(s)
\]
\begin{lemma}[Conjugation]
The coefficient $\overline{A(r)}$  %
corresponds to the measure $\overline{\sigma }:\overline{\sigma
}(I)=\sigma (-I)$  for any Borel set $I$, accelerant
$\overline{H(x)}$, resolvent kernel $\overline{\Gamma_r(x,y)}$,
and coefficient $-\beta$.\label{conjugation}
\end{lemma}

\begin{proof} Take conjugation of
(\ref{krein2}). The pair
$\{%
\overline{P(r,-\bar\lambda )},\overline{P_{\ast }(r,-\bar\lambda
)}\}$ solves Krein system with parameter $\lambda $ and
coefficient $\overline{A(r)}.$ New Weyl-Titchmarsh function is
equal to $\overline{F(-\bar\lambda)}$. The integral representation
for $F(\lambda)$ yields the value of $\beta$.
\end{proof}

Consider the case when $A(r)$ is real. Then, $\overline{\sigma
}=\sigma $ and $H(x)$ is also real. It is continuous on the whole
line provided that $A(r)\in C[0,\infty)$.
 In (\ref{e1s3-1}), constant $\beta=0$. Later on, we will consider
 this case in greater details.

The next calculation will be important to understand the
scattering problem for Krein system and Dirac operators.
Consider Krein system on the interval $[0,R]$ with coefficient $%
A^{(R)}(r)=A(R-r)$ for $r\in \lbrack 0,R]$. For $r>R$, we let
$A^{(R)}(r)=0$.

\begin{lemma}[Mirror symmetry]  The Schur function of Krein system with the
coefficient $A^{(R)}(r)$  is equal to
\[
f^{(R)}(\lambda )=\frac{\B(R,-\lambda )}{\A_{* }(R,-\lambda )}
\] \label{mirror}
\end{lemma}
\begin{proof} Consider the matrix $X$ that solves Krein system
\[
X^{\prime }=VX,X(0,\lambda )=I
\]%
and $V$ is given by (\ref{potential-V}). At the same time, matrix
$Y(r,\lambda )=JX(R-r,-\lambda )X^{-1}(R,-\lambda)J$ solves the
following system%
\[
Y^{\prime }=V(R-r)Y,Y(0)=I
\]%
Therefore,
new Wall polynomials are%
\begin{eqnarray*}
\A^{(R)}(R,\lambda ) &=&\exp (i\lambda R)\A_{\ast }(R,-\lambda ) \\
\B^{(R)}(R,\lambda ) &=&\exp (i\lambda R)\B(R,-\lambda )
\end{eqnarray*}%
That finishes the proof. \end{proof}

{\bf Remarks and historical notes.} The dual systems were studied
before, e.g. \cite{Krein1}. They also appear in the solution to
various continuation problems and in continuous
Caratheodory-Toeplitz problem.

\newpage

\section{Szeg\H{o} distance and Krein systems}\label{sect-szego}

As any orthogonal system, $\{P(r,\lambda)\}$ has the reproducing
kernel. Consider the scale $S_{\rho }= \cal{P}_{[0,\rho]}
L^2(\mathbb{R})$ of Paley-Wiener spaces.
 Recall that for any $\rho >0$, this space consists of functions $%
\widehat{f}(\lambda )$ that can be represented as%
\[
\widehat{f}(\lambda )=\int\limits_{0}^{\rho }f(x)\exp (i\lambda
x)dx
\]%
with $f(x)\in L^{2}[0,\rho ]$.

\begin{lemma}
The following function $K_\rho(\lambda',\lambda)$:
\begin{equation}
K_{\rho }(\lambda' ,\lambda )=\int\limits_{0}^{\rho
}\overline{P(x,\lambda'
)}%
P(x,\lambda
)dx=i\frac{P_*(\rho,\lambda)\overline{P_*(\rho,\lambda')}-
P(\rho,\lambda)\overline{P(\rho,\lambda')}}{\lambda-\bar\lambda'}
\label{reproduce0}
\end{equation}
is the reproducing kernel in $S_\rho$ space, i.e.
\begin{equation}
 \widehat{f}(\lambda' )=\langle \widehat{f}(\lambda ),K_{\rho
}(\lambda' ,\lambda )\rangle ,\ \lambda' \in \mathbb{C}
\label{reproducing}
\end{equation}
where the inner product is defined as follows
\begin{equation}
\langle \widehat{f}_{1},\widehat{f}_{2}\rangle
=\int\limits_{-\infty }^{\infty
}\widehat{f_{1}}%
(\lambda )\overline{\widehat{f_{2}}(\lambda )}d\sigma (\lambda
)=\int\limits_{-\infty }^{\infty }\widehat{f_{1}}(\lambda
)\overline{\widehat{f_{2}}(\lambda
)}\frac{d\lambda }{%
2\pi |P(\rho ,\lambda )|^{2}} \label{subordinacy-s}
\end{equation}
\end{lemma}
\begin{proof}
The second equality in (\ref{reproduce0}) is formula
(\ref{e101s2}). For any fixed $\lambda'$, $P(x,\lambda')\in
L^2[0,\rho]$, so the kernel $K_\rho(\lambda',\lambda)$ itself is
an element of $S_\rho$. Formula (\ref{reproducing}) follows from
Plancherel-type identity (\ref{plancherel}). Indeed, by
(\ref{eee2}), we have
\[
\widehat{f}(\lambda')=\int\limits_0^\rho f_1(x) P(x,\lambda)dx
\]
with some $f_1(x)\in L^2[0,\rho]$. Now, using (\ref{plancherel}),
we get
\[
\langle \widehat{f}(\lambda),
K_\rho(\lambda',\lambda)\rangle=\int\limits_0^\rho
f_1(x)P(x,\lambda')dx=\widehat{f}(\lambda')
\]
The second equality in (\ref{subordinacy-s}) is the contents of
Corollary \ref{scale-property}.
\end{proof}

The reproducing kernel property (\ref{reproducing}) yields
\[
K_\rho(\lambda_1,\lambda_2)=\langle
K_\rho(\lambda_1,\lambda),K_\rho(\lambda_2,\lambda)\rangle
\]
Together with Cauchy inequality, that implies
\begin{equation}
\left\vert \widehat{f}(\lambda ^{\prime })\right\vert ^{2}\leq
||\widehat{f}%
||_{L^{2}(\mathbb{R},d\sigma )}^{2}K_{\rho }(\lambda ^{\prime
},\lambda ^{\prime }) \label{e25s3}
\end{equation}%
and the equality holds if and only if $\widehat{f}(\lambda
)=\gamma K_\rho(\lambda ^{\prime },\lambda ),$ $|\gamma |=1$.

\begin{lemma}
The following identity is true
\begin{equation}
\frac{1}{K_{\rho }(\lambda ^{\prime },\lambda ^{\prime })}=\min_{\hat{f}%
\in S_{\rho },\hat{f}(\lambda ^{\prime })=1}\int\limits_{-\infty
}^{\infty }|\widehat{f}(\lambda )|^{2}d\sigma (\lambda
)=m^2_\rho(\lambda') \label{minimum}
\end{equation}
for any $\lambda ^{\prime }\in \mathbb{C}$. The minimizer is
unique and is given by
$\hat{f}(\lambda)=K_\rho^{-1}(\lambda',\lambda')K_\rho(\lambda',\lambda)$.
\end{lemma}

\begin{proof} Divide\ (\ref{e25s3}) by $K_{\rho }(\lambda
^{\prime },\lambda ^{\prime })\left\vert \widehat{f}(\lambda
^{\prime })\right\vert ^{2}$. \end{proof}

Now, the natural question to ask is what happens if
$\rho\to\infty$? Since $S_{\rho_1}\subset S_{\rho_2}$ for
$\rho_1<\rho_2$, the minimum $m_\rho(\lambda')$ decreases. Now,
can we characterize the case when it decreases to zero? To do
that, we need the following classical result (see, e.g.
\cite{DymMcKean}, page 84).  We give its proof in Appendix
(Theorem \ref{appendix-szego}).

Assume that $d\sigma$ is a positive measure on $\mathbb{R}$ such
that
\[
\int\limits_{-\infty}^\infty
\frac{d\sigma(\lambda)}{1+\lambda^2}<\infty
\]
Consider the linear manifold $X$ of functions $\hat{f}(\lambda)$,
having the following representation
\[
\hat{f}(\lambda)=\int\limits_{r_1}^{r_2} \exp(i\lambda x)
f(x)dx,\quad 0\leq r_1<r_2
\]
where $f(x)\in C^1[r_1,r_2]$ and is zero outside
$[r_1,r_2]\subseteq [0,\infty)$. Notice that each
$\hat{f}(\lambda)\in L^2(d\sigma)$. Denote the closure of $X$ in
$L^2(d\sigma)$ by $\bar{X}$.

\begin{theorem}
The linear manifold $X$ is not dense in $L^2(d\sigma)$ if and only
if
\begin{equation}
\int\limits_{-\infty}^\infty \frac{\ln \sigma'
(\lambda)}{1+\lambda^2}d\lambda >-\infty
\end{equation}
Moreover, the following formula is always true
\begin{equation}
{\rm Dist} \left( \frac{1}{\lambda-\lambda_0}, \bar{X}
\right)_{L^2(d\sigma)} = \frac{1}{\sqrt{2\Im\lambda_0}} \exp
\left[ \frac{\Im \lambda_0}{2\pi}\int\limits_{-\infty}^{\infty}
\frac{\ln(2\pi
\sigma'(\lambda))}{|\lambda-\lambda_0|^2}d\lambda\right],\,
\lambda_0\in \mathbb{C}^+ \label{sdcalc}
\end{equation}
 \label{nnn2}
\end{theorem}

Next, we will apply this Theorem to the Krein systems. Let
$d\sigma$ be the measure generated by some Krein system (the
measure from (\ref{e2s2})). Recall the definition of $S_r$ and
notice that for each finite $r$, $S_r\subset L^2(d\sigma)$ (see
(\ref{sect2eq1})). Denote the closure of $\cup_{r>0} S_r$ in
$L^2(d\sigma)$ by $\bar{S}$.
\begin{lemma}
If $d\sigma$ is generated by some Krein system, then
$\bar{X}=\bar{S}$. \label{coincidence1}
\end{lemma}
\begin{proof}
It is clear that $\bar{X}\subseteq \bar{S}$. On the other hand,
any function $\hat{f}\in S_r$ can be approximated in
$L^2(d\sigma)$ by a sequence of functions from $X$. That easily
follows from (\ref{sect2eq1}).
\end{proof}

\begin{lemma}
The following formula
\begin{equation}
{\rm Dist} \left( \frac{2\Im \lambda_0}{\lambda-\lambda_0},
\bar{S} \right)_{L^2(d\sigma)} = \inf\limits_{r>0}
m_r(\lambda_0)=m_\infty(\lambda_0) \label{distance}
\end{equation}
is true for any $\lambda_0\in \mathbb{C}^+$. \label{sdl1}
\end{lemma}
\begin{proof} Denote the l.h.s. by $I_1$ and the r.h.s. by $I_2$.
We have
\[
\left|\frac{2\Im\lambda_0}{\lambda-\lambda_0}-\hat
f(\lambda)\right|=\left|
\frac{2i\Im\lambda_0}{\lambda-\overline{\lambda}_0}-\frac{\lambda-\lambda_0}
{\lambda-\overline{\lambda}_0}\cdot i\hat f(\lambda)\right|
\]
and so
\begin{equation}
I_1 ={\rm Dist} \left(
\frac{2i\Im\lambda_0}{\lambda-\overline{\lambda}_0},
\frac{\lambda-\lambda_0} {\lambda-\overline{\lambda}_0}
\bar{S}\right)_{L^2(d\sigma)} \label{sdf2}
\end{equation}
Since
\begin{equation}
2\Im\lambda_0\int\limits_0^\infty \exp(i\lambda
x)\exp(-i\bar\lambda_0
x)dx=\frac{2i\Im\lambda_0}{\lambda-\bar\lambda_0},\,
\frac{\lambda-\lambda_0}{\lambda-\bar\lambda_0}=1-
\frac{2i\Im\lambda_0}{\lambda-\bar\lambda_0} \label{calc-13}
\end{equation}
any function
\[
\frac{2i\Im\lambda_0}{\lambda-\overline{\lambda}_0}-\frac{\lambda-\lambda_0}
{\lambda-\overline{\lambda}_0}\hat f(\lambda), \hat f(\lambda)\in
\bar{S}
\]
can be approximated in $L^2(d\sigma)$ by the sequence of functions
$\hat f_{r_n}(\lambda)\in S_{r_n},$ \mbox{$\hat
f_{r_n}(\lambda_0)=1$.} Therefore, $I_1\geq I_2$. Assume that
$\hat f(\lambda)$-- arbitrary function from some $S_r$ and $\hat
f(\lambda_0)=1$. Then
\begin{equation}
\hat f(\lambda)=\frac{2i
\Im\lambda_0}{\lambda-\overline{\lambda}_0}-\frac{\lambda-\lambda_0}{\lambda-\overline{\lambda}_0}
\hat g(\lambda) \label{sdf1}
\end{equation}
with $\hat g(\lambda)\in \bar{S}$. Indeed, the function
\[
\hat f(\lambda)-\frac{2i\Im
\lambda_0}{\lambda-\overline{\lambda}_0}
\]
belongs to $H^2(\mathbb{C}^+)$ and has zero at
$\lambda=\lambda_0$. So, by Paley-Wiener Theorem, (\ref{sdf1})
holds with $\hat g\in H^2(\mathbb{C}^+)$. Let us show now that
$\hat g\in \bar{S}$. The first formula in (\ref{calc-13}) suggests
\[
\hat
g(\lambda)=\frac{\lambda-\overline{\lambda}_0}{\lambda-\lambda_0}\left[
-\hat f(\lambda)-\frac{2\Im
\lambda_0}{1-\exp(in(\lambda_0-\overline{\lambda}_0))}\int\limits_0^n
\exp(ix(\lambda-\overline{\lambda}_0))dx\right]+r_n(\lambda)
\]
The Paley-Wiener Theorem yields that the first term belongs to
$S_\rho$ where $\rho=\max (r,n)$. One can easily see that
$\|r_n\|_{2,\sigma}\to 0$ as $n\to\infty$. Thus, the formula
(\ref{sdf1}) holds. Due to (\ref{sdf2}), $I_1\leq I_2$ and
therefore $I_1=I_2$. \end{proof} The next result describes the
continuous analog of the Szeg\H{o} case in OPUC theory. If any of
the conditions bellow is satisfied, we will say that $d\sigma\in$
(Szeg\H{o}).

\begin{theorem}{\rm (The Szeg\H{o} case)}
The following statements are equivalent
\begin{itemize}
\item[(a)] { The operator $\cal{O}$ from the Theorem
\ref{theorem2s2} is not unitary.}

\item[(b)] { Inequality
\begin{equation}
\int\limits_{-\infty}^\infty \frac{\ln
\sigma'(\lambda)}{1+\lambda^2} d\lambda>-\infty
\end{equation}
holds. }

\item[(c)]
{
\[ \sup\limits_{r>0, \hat f(\lambda)\in S_r, \|\hat f(\lambda)\|_{2,\sigma}=1}
|\hat f(\lambda_0)|=m_\infty^{-1} (\lambda_0)<\infty
\]
for at least one (and then for all) $\lambda_0\in \mathbb{C}^+$. }

\item[(d)]
{ $P(r,\lambda_0)\in L^2(\mathbb{R}^+)$ for at least one
(and then for all) $\lambda_0\in \mathbb{C}^+$.}

\item[(e)]
{ $\liminf_{r\to\infty} |P_*(r,\lambda_0)|<\infty$ for
at least one (and then for all) $\lambda_0 \in \mathbb{C}^+ $.}
\end{itemize}
\label{sdmt}
\end{theorem}

\begin{proof} (a) and (b) are equivalent. Indeed, by Lemma
\ref{transformation}, the range of $\cal{O}$ coincides with
$\bar{S}$. From Lemma \ref{coincidence1}, we have
$\bar{X}=\bar{S}$ and then we only need to use Theorem~\ref{nnn2}.

The formula (\ref{e101s2}) with $\lambda=\mu=\lambda_0$ shows that
if (d) holds at some $\lambda_0$ then (e) is true as well at the
same point. The converse is also true.

The identity (\ref{minimum}) and the formula for reproducing
kernel yield
\[
m^{-2}_\infty(\lambda_0)=\sup\limits_{r>0, f(\lambda)\in S_r,
\|f(\lambda)\|_{2,\sigma}=1} |f(\lambda_0)|^2=K_\infty
(\lambda_0,\lambda_0)=\int\limits_0^\infty |P(x,\lambda_0)|^2 dx
\]
and that proves equivalence of (c) and (d) for fixed $\lambda_0$.

Let us show that (c) is satisfied with some $\lambda_0$ if and
only if (a) holds. Assume that the operator $\cal{O}$ is unitary.
That means its range is the whole $L^2(d\sigma)$ and
$(\lambda-\lambda_0)^{-1}\in \bar{S}$, for any $\lambda_0\in
\mathbb{C}^+$. Due to Lemma \ref{sdl1}, (c) fails. Conversely,
assume  (c) fails for some $\lambda_0\in \mathbb{C}^+$. Then, by
Lemma \ref{sdl1}, $(\lambda-\lambda_0)^{-1}\in \bar{S}$. The
Theorem~\ref{nnn2} now implies
\[
\int\limits_{-\infty}^{\infty} \frac{\ln
\sigma'(\lambda)}{|\lambda-\lambda_0|^2}d\lambda=-\infty
\]
but that means (b) fails and therefore (a) fails too. Notice that
both (a) and (b) do not depend on parameter $\lambda_0$.
Therefore, if any of (c), (d), or (e) holds for some $\lambda_0$
then it holds for all $\lambda\in \mathbb{C}^+$.
\end{proof}

It is not in general true that $|P_*(r,\lambda)|,\lambda\in
\mathbb{C}^+ $ is even bounded as $r\to\infty$ under the
conditions of the Theorem \ref{sdmt}. That is due to continuous
nature of the problem. Moreover, it is possible that
$|P_*(r,\lambda)|$ has a limit, but $\lim_{r\to\infty}
P_*(r,\lambda)$ does not exist inspite of the fact that the
corresponding $A(r)\to 0$ at infinity  and $A(r)\in
L^p(\mathbb{R}^+)$ for any $p>2$. This phenomena was observed for
the first time by Teplyaev (see \cite{Tepl, Teplyaev}). That can
be explained as follows:  in the discrete case, the orthonormal
polynomials are usually normalized such that they have the
positive leading coefficient. For Krein systems, normalization is
quite different: $P_*(r,\lambda)$ are normalized to be equal to
$1$ at infinity, the point on the boundary of $\mathbb{C}^+$.
Therefore, the argument of $P_*(r,\lambda)$ is not stabilized and
that leads to the ambiguity in the definition of
$\lim_{r\to\infty} P_*(r,\lambda)$.

Consider some $\lambda_0 \in \mathbb{C}^+$. If conditions in
Theorem \ref{sdmt} are satisfied, then there is a sequence $r_n\in
[n,n+1]\to \infty$ such that $P(r_n,\lambda_0)\to 0$. Take the
outer function
\begin{equation}
\Pi(\lambda)= \frac{1}{\sqrt{2\pi}} \exp \left[ \frac{1}{2\pi i}
\int\limits_{-\infty}^{\infty} \frac{(1+s\lambda) \ln
\sigma'(s)}{(\lambda-s)(1+s^2)} ds\right] \label{function-pi}
\end{equation}
that satisfies $|\Pi(\lambda)|=[2\pi\sigma'(\lambda)]^{-1/2}$ for
a.e. $\lambda\in \mathbb{R}$. Notice that $\left[(\lambda+i)
\Pi(\lambda)\right]^{-1}\in H^2(\mathbb{C}^+)$ and is outer.

\begin{lemma}
If $d\sigma\in${\rm (Szeg\H{o})} and $r_n\to\infty$ is such that
$P(r_n,\lambda_0)\to 0$ for some $\lambda_0\in \mathbb{C}^+$, then
the following convergence $|P_*(r_n,\lambda)|\to |\Pi(\lambda)|$
takes place uniformly in $\mathbb{C}^+$. \label{converge}
\end{lemma}
\begin{proof} From the Theorem \ref{Bernstein-Szego}, we know that
the  sequence $h_n(\lambda)=\left[(\lambda+i)
P_*(r_n,\lambda)\right]^{-1}$ is bounded in $H^2(\mathbb{C}^+)$,
i.e. $\|h_n(\lambda)\|_{L^2(\mathbb{R})}$ is bounded. Assume that
$h(\lambda)$ is any $L^2(\mathbb{R})$-weak limit point of this
sequence. Then, $h(\lambda)\in H^2(\mathbb{C}^+)$ and the
convergence is uniform in $\mathbb{C}^+$ over the same subsequence
$n_k$. From (\ref{e6s3}),

\begin{equation}
\Re \left[ P_{\ast }^{-1}(r_{n_k},\lambda)\widehat{P}_{\ast
}(r_{n_k},\lambda)\right] \geq \left\vert P_{\ast
}(r_{n_k},\lambda)\right\vert ^{-2}
\end{equation}
Taking $k\to\infty$, we get
\[
|(\lambda+i)|^2 |h(\lambda)|^2\leq \Re F(\lambda)
\]
for $\lambda\in \mathbb{C}^+$, where $F(\lambda)$ is the
Weyl-Titchmarsh function. From (\ref{weyl-titchmarsh}), we get
\[
(\lambda^2+1)|h(\lambda)|^2\leq 2\pi\sigma'(\lambda)
\]
for a.e. $\lambda\in \mathbb{R}$. Therefore, from the
multiplicative representation of $(\lambda+i)h(\lambda)$
\[
|(\lambda+i)h(\lambda)|\leq  |\Pi^{-1}(\lambda)|, \lambda\in
\mathbb{C}^+
\]
At the same time, at $\lambda_0$, we have
$|(\lambda_0+i)h(\lambda_0)|=|\Pi^{-1}(\lambda_0)|$. It follows
from (\ref{reproduce0}), (\ref{minimum}), (\ref{sdcalc}),
(\ref{distance}), (\ref{function-pi}). Therefore,
$(\lambda+i)h(\lambda)$ is an outer function different from
$\Pi^{-1}(\lambda)$ only by a unimodular constant factor. Thus,
for any subsequence $n_k$, $|P_*(r_{n_k},\lambda)|\to
|\Pi(\lambda)|$. That means we actually have convergence over the
whole  $r_n$.\end{proof}

It is known that all outer functions $\Pi_\gamma(\lambda)$,
satisfying $|\Pi_\gamma(\lambda)|^{-2}=2\pi\sigma'(\lambda)$ a.e
on $\mathbb{R}$, have the following representation
\begin{equation}
\Pi_\gamma(\lambda)= \frac{1}{\sqrt{2\pi}} \exp \left[i\gamma+
\frac{1}{2\pi i} \int\limits_{-\infty}^{\infty} \frac{(1+s\lambda)
\ln \sigma'(s)}{(\lambda-s)(1+s^2)} ds\right], \gamma\in [0,2\pi)
\label{function-pi1}
\end{equation}
i.e. they can be parameterized by the angle $\gamma$.
 The function
$\Pi(\lambda)=\Pi_0(\lambda)$ satisfies the following
normalization condition: $\Pi(i)>0$. There are some quite
interesting examples \cite{Teplyaev} when $P_*(r_n,\lambda)\to
\Pi_\gamma(\lambda)$  and the constant $\gamma$ depends on the
choice of subsequence $r_n$. In the meantime, the following is
true \cite{Sakhnovich7}
\begin{lemma}\label{real-valued-case}
Assume that $A(r)$ is real-valued and $r_n$ is such that
$P(r_n,\lambda_0)\to 0$ for at least some $\lambda_0\in
\mathbb{C}^+$. Then, $P_*(r_n,\lambda)\to\Pi(\lambda)$ uniformly
in $\lambda\in \mathbb{C}^+$.
\end{lemma}
\begin{proof}
Following the proof of the previous Lemma, we get convergence of
$P_*(r_{n_k},\lambda)$ to some $\Pi_\gamma(\lambda)$ uniformly for
$\lambda\in \mathbb{C}^+$. Taking $\lambda=i$, we get $\gamma=0$.
Indeed, $P_*(0,i)=1$, $P_*(r,i)$ is real and has no zeroes for
$r>0$. Therefore, it must be positive for all $r>0$. Thus
$\Pi_\gamma(i)>0$ and $\gamma=0$.
\end{proof}

In the OPUC theory, we can not directly characterize the set of
moments such that the corresponding measure belongs to Szeg\H{o}
class. The same is true for the Krein systems: we are not aware of
the characterization of the Szeg\H{o} case in terms of accelerant.
In the meantime,
\begin{lemma}
If $d\sigma\in$ {\rm (Szeg\H{o})}, then for the dual system we
also have $d\hat\sigma\in $ {\rm (Szeg\H{o})}.
\end{lemma}
\begin{proof}
We know that for any Krein system,
\[
\frac{\widehat{P}_*(r,\lambda)}{P_*(r,\lambda)}\to F(\lambda)
\]
as $r\to\infty$ and $F(\lambda)$ has positive real part in
$\mathbb{C}^+$. Therefore, if condition (e) of the Theorem
\ref{sdmt} is satisfied for the original Krein system, it must be
satisfied for the dual one as well.
\end{proof}

Next, let us show that we have weighted $L^2$--convergence for
$P_*(r,\lambda)$ for $\lambda\in \mathbb{R}$. First, we need the
following auxiliary result
\begin{lemma}\label{auxil-1}
Assume that the Szeg\H{o} case holds. Let $r_n$ be a sequence such
that $P(r_n,i)\to 0$. Then
\[
\lim_{r_n\to\infty} \int\limits_{-\infty}^\infty
\frac{|P(r_n,\lambda)|^2}{\lambda^2+1} d\sigma(\lambda)=\frac 12
\]
\label{kernel-diag}
\end{lemma}

\begin{proof} Indeed, we have
\[
K_r(i,\lambda)=i\frac{P_*(r,\lambda)\overline{P_*(r,i)}-P(r,\lambda)\overline{P(r,i)}}{\lambda+i}
\]
or
\begin{equation}\label{auxil-decay}
i\frac{P_*(r,\lambda)}{\lambda+i}\cdot \overline{P_*(r,i)}
=i\frac{P(r,\lambda)}{\lambda+i}\cdot
\overline{P(r,i)}+K_r(i,\lambda)
\end{equation}
From Lemma \ref{converge}, we know that $|P_*(r_n,i)|\to
|\Pi(i)|$. Then, since $P(r_n,i)\to 0$,
\[
\lim_{r_n\to\infty} \left\|
\frac{P_*(r_n,\lambda)}{\lambda+i}\right\|^2_{2,d\sigma}=|\Pi(i)|^{-2}
\lim\limits_{r_n\to\infty} |K_{r_n} (i,i)|=\frac 12
\]
\end{proof}

The following result establishes an $L^2(d\sigma, \mathbb{R})$
asymptotics of $P_*(r,\lambda)$. It will be used later to prove
existence of wave operators for Dirac equation.

\begin{lemma} Assume that $d\sigma\in${\rm (Szeg\H{o})} and
$r_n\to\infty$ is such that $P(r_n,i)\to 0$ and
$P_*(r_n,\lambda)\to \Pi_\gamma(\lambda)$ for $\lambda\in
\mathbb{C}^+$, ($\gamma\in [0,2\pi)$). Then,

\begin{equation}
\int\limits_{-\infty }^{\infty }\frac{1}{\lambda ^{2}+1}\left|
\frac{P_{\ast }(r_n,\lambda )}{\Pi_\gamma (\lambda)}-1\right|
^{2}d\lambda \rightarrow 0 \label{asy}
\end{equation}
as $r_n\rightarrow \infty $. \label{asymp-l2}
\end{lemma}

\begin{proof} The left-hand side of (\ref{asy}) is equal to
\begin{equation}
\int\limits_{-\infty }^{\infty }\frac{1}{\lambda ^{2}+1}\left|
\frac{P_{\ast
}(r_n,\lambda )}{\Pi_\gamma (\lambda)}\right| ^{2}\mathit{d\lambda +}%
\int\limits_{-\infty }^{\infty }\frac{d\lambda }{\lambda
^{2}+1}-2\Re \int\limits_{-\infty }^{\infty }\frac{P_{\ast
}(r_n,\lambda )}{(\lambda ^{2}+1)\Pi_\gamma (\lambda)}d\lambda
\label{one}
\end{equation}
By the Cauchy formula,
\begin{equation}
\int\limits_{-\infty }^{\infty }\frac{P_{\ast }(r_n,\lambda
)}{(\lambda ^{2}+1)\Pi_\gamma (\lambda)}d\lambda =\frac{\pi
P_{\ast }(r_n,i)}{\Pi_\gamma (i)}\to \pi \label{two}
\end{equation}
as $r_n\to\infty$. The first term of (\ref{one}) can be written as
\begin{equation*}
2\pi \int\limits_{-\infty }^{\infty }\frac{1}{\lambda
^{2}+1}\left| P_{\ast }(r_n,\lambda )\right| ^{2}d\sigma (\lambda
)-2\pi \int\limits_{-\infty }^{\infty }\frac{1}{\lambda
^{2}+1}\left| P_{\ast }(r_n,\lambda )\right|
^{2}d\sigma_{s}\mathit{(\lambda )},
\end{equation*}
where $d\sigma_s(\lambda)$ is the singular component of
$d\sigma(\lambda)$. From the Lemma \ref{auxil-1}, we infer
\begin{equation}
\int\limits_{-\infty }^{\infty }\frac{1}{\lambda ^{2}+1}\left|
P_{\ast }(r,\lambda )\right| ^{2} d\sigma (\lambda )\rightarrow
\frac12 \label{star}
\end{equation}
as $r\rightarrow \infty$. Bearing in mind (\ref{one}),
(\ref{two}), and (\ref{star}), we have (\ref{asy}). \end{proof}

\textbf{Remark.} We also proved
\begin{equation}
\int\limits_{-\infty }^{\infty }\frac{1}{\lambda ^{2}+1}\left|
P_{\ast }(r_n,\lambda )\right| ^{2}\mathit{d\sigma }_{s} (\lambda
)\rightarrow 0 \label{singularcontribution}
\end{equation}
as $r_n\rightarrow \infty.$

We want to finish this section with the following observation that
relates the regularity of $|\Pi(\lambda)|$ at $+i\infty$ to some
approximation problem. Consider the function
$f_{\lambda_0}(\lambda)=(2\Im\lambda_0)^{1/2}(\lambda-\lambda_0)^{-1},\lambda\in
\mathbb{R}$. Since
\[
(2\Im\lambda_0)^{1/2}(\lambda-\lambda_0)^{-1}=
i(2\Im\lambda_0)^{1/2}\int\limits_{-\infty}^0 e^{-i\lambda_0
x}e^{i\lambda x}dx
\]
the function $f_{\lambda_0}(\lambda)$ has frequency concentrating
near zero as $\Im\lambda_0\to+\infty$ and the constant $L^2$ norm.
How regular the distance $ {\rm
Dist}(f_{\lambda_0}(\lambda),\bar{S}) $ behaves as
$\Im\lambda_0\to +\infty$ depends on the regularity of
$\Pi(\lambda)$ at infinity. Indeed, from Lemma \ref{sdl1}
\[
{\rm Dist}(f_{\lambda_0}(\lambda),\bar{S})=|\Pi(\lambda_0)|^{-1}
\]
That infinitesimal phenomena is not present in the discrete case.

 {\bf Remarks and historical notes.}
The approximation results of this section can be interpreted in
the framework of prediction theory for stationary Gaussian
processes with continuous time \cite{Rozanov}. The original paper
by Krein contained some inaccuracies in the formulation of the
Theorem \ref{sdmt} and the same mistake was made in some later
papers. The correct statement was given later by Teplyaev
\cite{Tepl,Teplyaev}. Some sufficient conditions for the Szeg\H{o}
case were given in the series of papers \cite{Sakhnovich7,
Den-ieop}. There is no known criteria in terms of $A(r)$ for the
Szeg\H{o} case to hold. In the meantime, if one assumes some
regularity of $A(r)$, say, $A(r)\in L^\infty(\mathbb{R}^+)$ then
$d\sigma\in$ (Szeg\H{o}) if and only if $A(r)\in
H^{-1}(\mathbb{R}^+)$ (see \cite{Den-ieop}). It is probably
impossible to give reasonable characterization of the Szeg\H{o}
case in terms of $A(r)$ without any apriori assumptions. For
example, one can construct a sequence of compactly supported
$A^{(n)}(r)$ with growing $H^{-1}(\mathbb{R}^+)$ norms but such
that the corresponding sequence $P^{(n)}_*(\infty,i)$ is bounded.
This can be achieved by a simple modification of Teplyaev's
example \cite{Teplyaev}.

\newpage

\section{Schur's algorithm and approximation of continuous orthogonal system by
 discrete ones}\label{sect-appro}

It is well-known \cite{Simon}, that any function $f(z)\in
B(\mathbb{D})$ can be expanded into the continued fraction (the
so-called Schur's algorithm). This expansion can be obtained by
the iteration of
\begin{equation}
 f_n(z)=\frac{zf_{n+1}(z)+a_n}{1+\bar{a}_n zf_{n+1}(z)},
 f_0(z)=f(z), a_n=f_n(0) \label{opuc-schur}
\end{equation}
By doing so, we obtain the one-to-one correspondence between
$B(\mathbb{D})$ and all sequences $\{a_n\}$ such that $|a_n|\leq
1$. The Geronimus theorem asserts that these so-called Schur
parameters are actually equal to Verblunsky parameters for the
measure $d\tau$ in the representation
\begin{equation}
\frac{1+zf(z)}{1-zf(z)}=\int_{\mathbb{T}}
\frac{\xi+z}{\xi-z}d\tau(\xi) \end{equation}

In this section we will study Schur's function associated to Krein
system. The Schur function $\f(\lambda)$ associated to the Krein
system was introduced in the Theorem~\ref{schur-function}.
Functions $\f_{\rho}(\lambda)$ in the representation (\ref{e18s3})
are Schur's functions that correspond to the same Krein systems
but on the interval $[\rho,\infty)$. This is the same as if we
would take $A_\rho(r)=A(r+\rho)$ with $r\in \mathbb{R}^+$.

\begin{lemma}
For any fixed $\lambda\in \mathbb{C}^+$, the Schur functions
$\f_r(\lambda)$ are continuously differentiable in $r$ and satisfy
the following equation
\begin{equation}
\frac{d\f_r(\lambda)}{dr}=-i\lambda
\f_r(\lambda)+A(r)-\overline{A(r)}\f^2_r(\lambda) \label{schur}
\end{equation}
\end{lemma}
\begin{proof} The smoothness of $\f_r(\lambda)$ in $r$ follows
immediately from (\ref{e18s3}). Taking derivative of
($\ref{e18s3}$) in $r$ at $r=0$, we get
\begin{equation}\label{schur-evolution}
\frac{d\f_r(\lambda)}{dr}\Big|_{r=0}=-i\lambda
\f_0(\lambda)+A(0)-\overline{A(0)}\f^2_0(\lambda)
\end{equation}
Now, (\ref{schur}) follows from the definition of $\f_r(\lambda)$.
\end{proof}

It is very important to keep in mind that the initial condition
for (\ref{schur}) is $\f_0(\lambda)=\f(\lambda)$ and it is {\bf
not} independent of the coefficient $A(r)$. Now, let us compare
the continuous and discrete Schur algorithms. Consider the
following M\"obius transform
\[
\tau_\gamma(z)=\frac{z+\gamma}{1+\bar\gamma z}
\]
The inverse to $\tau_\gamma(z)$ is equal to $\tau_{-\gamma}(z)$.

For any $\gamma\in \mathbb{D}$, $\tau_\gamma$ is a conformal map
of $\mathbb{D}$ onto $\mathbb{D}$ that takes $0$ to $\gamma$.
Another important property of $\tau_\gamma$ is the preservation of
the pseudohyperbolic distance on $\mathbb{D}$:
\begin{equation}
\rho(z_1,z_2)=\left| \frac{z_1-z_2}{1-\bar{z}_1 z_2}\right|,
\rho(\tau_\gamma(z_1), \tau_\gamma(z_2))=\rho(z_1, z_2)
\label{pseudohyp}
\end{equation}

Given any function $f(z)$ from $B(\mathbb{D})$ the Schur algorithm
can also be defined as follows. The $n$--th Schur's iterate
$f_n(z)$
 is defined by the relation
\begin{equation}
f(z)=S_{z,a_0, \ldots, a_{n-1}}(f_n)= \tau_{a_0} \circ z\tau_{a_1}
\circ \ldots z\tau_{a_{n-1}} \circ (zf_{n}) \label{sprocedure}
\end{equation}
and
\begin{equation}
f_n(z)=S^{-1}_{z,a_0, \ldots, a_{n-1}}(f)= z^{-1}\tau_{-a_{n-1}}
\circ z^{-1}\tau_{-a_{n-2}} \circ \ldots z^{-1}\tau_{-a_{0}} \circ
f \label{sprocedure1}
\end{equation}
Notice that for $z\in \mathbb{T}$ the map $S_{z,a_0, \ldots,
a_{n-1}}$ is the composition of rotations and M\"obius transforms.

 The differential equation
(\ref{schur}) is a continuous analog of (\ref{opuc-schur}). It is
Riccati equation and the Cauchy problem for solving it from the
right to the left happens to be well-posed for suitable initial
data:
\begin{lemma} \label{combination}
For any $A(r)\in L^1[0,R]$ and any $\f(0)=\f_0\in
\overline{\mathbb{D}}, \lambda\in \overline{\mathbb{C}}^+$, there
is the unique solution $\f_r(\lambda)$ to Cauchy problem for
equation (\ref{schur}) with initial condition $\f(R)=\f_0$.
\end{lemma}
\begin{proof}
Consider $Y(r,\lambda)=(y_1(r), y_2(r))^t$, solution to the
following Cauchy problem
\[
\left\{
\begin{array}{cc}
y_1'=-i\lambda y_1+{A(r)}y_2,& y_1(R,\lambda)=\f_0,\\
 y_2'=\overline{A(r)}y_1, &y_2(R,\lambda)=1
 \end{array}
 \right.
 \]
and solve it for $r\in [0,R]$. Simple calculations show that
 \begin{equation}
 Y(r,\lambda)=X_{B(r)}(0,R-r,\lambda)Y(R)=X^t(r,R,\lambda)Y(R)
 \label{relation-backward}
 \end{equation}
where $Y(R)=(f_0, 1)^t$ and $X_B$ is the transfer matrix for the
Krein system with coefficient $B(r)=\overline{A(R-r)}$. We have
$\lambda\in \overline{\mathbb{C}^+}$ so $X_B$ is $J$--contraction
by the Theorem~\ref{jcontraction}. Thus, we have
$|y_1(r,\lambda)|\leq |y_2(r,\lambda)|$. In particular,
$y_2(r,\lambda)\neq 0$ since otherwise $Y\equiv 0$ on $[0,R]$.
Consider $f(r,\lambda)=y_1(r,\lambda)y_2^{-1}(r,\lambda)$. The
straightforward calculation shows that $f(r,\lambda)$ is solution
to our Cauchy problem and uniqueness follows from the general
theory of ODE.
\end{proof}
In analogy with discrete case, we denote the solution of this
Cauchy problem at zero by $S_{\lambda,A,R}(\f_0)$ and now we have
$\f(\lambda)=S_{\lambda,A,r}(\f_r(\lambda))$, the direct analog of
(\ref{sprocedure}). The formula (\ref{relation-backward}) shows
that $S_{\lambda, A, r}$ allows the following representation
\begin{equation} \label{mobius-easy}
S_{\lambda, A,
r}(z)=\frac{\A_*(r,\lambda)z+\B(r,\lambda)}{\B_*(r,\lambda)z+\A(r,\lambda)}
\end{equation}
Notice that the inverse to $S$ in discrete case is not contraction
anymore and we have the same problem in the continuous setting.

As we mentioned earlier, the class of Schur functions
$\f(\lambda)$ in the continuous case can not be all
$B(\mathbb{C}^+)$. For example, (\ref{feature}) must hold. Next,
we will describe the subclass of $B(\mathbb{C}^+)$ in which $A(r)$
have the meaning of intrinsic parameters of the function
$f(\lambda)$ just like $\{a_n\}$ are intrinsic parameters of
$f(z)\in B(\mathbb{D})$.

Let a function $C(x)\in L^2_{\rm loc}(\mathbb{R}^+)$ be given. For
any $R>0$, consider the operator $\cal{C}_R$ acting in $L^2[0,R]$
by the following formula
\begin{equation}
\cal{C}_R f(x)=\int\limits_0^x C(x-u)f(u)du \label{volterra-s}
\end{equation}

 We start with the definition.
\begin{definition} \label{classS}
The function $s(\lambda)\in S(\mathbb{C}^+)$ if the following is
true:
\begin{itemize}
\item[(1)] There is a function $C(x)\in L^2_{\rm loc}(\mathbb{R}^+)$ such that for
any $R>0$ there is a function $\Phi_R(\lambda)\in
H^\infty(\mathbb{C}^+)$:
\begin{equation}
s(\lambda)=\int\limits_0^R C(x)\exp(i\lambda x)dx+\exp(i\lambda
R)\Phi_R(\lambda) \label{schur-alg}
\end{equation}

\item[(2)] For any $R>0$,
\begin{equation}
\|\cal{C}_R\|_{L^2[0,R]}<1 \label{contraction}
\end{equation}

\end{itemize}
\end{definition}
It is an easy exercise to see that $C(x)$ and $\Phi_R(\lambda)$
are both uniquely defined for any $s(\lambda)\in S(\mathbb{C}^+)$.
Conversely, if the function $C(x)$ is given, then there is at most
one $s(\lambda)\in S(\mathbb{C}^+)$ having $C(x)$ as a function in
the formula (\ref{schur-alg}). Indeed, assume that there are two
$s^{(1)}(\lambda), s^{(2)}(\lambda)$ having the same $C(x)$ in
(\ref{schur-alg}). Then,
\mbox{$(s^{(1)}(\lambda)-s^{(2)}(\lambda))/(\lambda+i)\in
H^2(\mathbb{C}^+)$}. At the same time,
\[
\frac{s^{(1)}(\lambda)-s^{(2)}(\lambda)}{\lambda+i}=\exp(i\lambda
R)\frac{\Phi^{(1)}_R(\lambda)-\Phi^{(2)}_R(\lambda)}{\lambda+i}
\]
for any $R>0$. Since
\mbox{$(\Phi^{(1)}_R(\lambda)-\Phi^{(2)}_R(\lambda))/(\lambda+i)\in
H^2(\mathbb{C}^+)$} as well, we have

\[
\cal{P}_{[0,R]} \left[
\frac{s^{(1)}(\lambda)-s^{(2)}(\lambda)}{\lambda+i}\right] =0
\]
for any $R>0$. So, $s^{(1)}(\lambda)=s^{(2)}(\lambda)$.

Consider $s(\lambda)\in S(\mathbb{C}^+)$. We have $s(\lambda)\in
H^\infty(\mathbb{C}^+)$. In the space $H^2(\mathbb{C}^+)$, denote
the operator of multiplication by this function by $\cal{S}$.
Also, consider the operator $\cal{C}$ acting on $L^2[0,\infty)$
and given by the formula
\begin{equation}
\cal{C} f(x)=\int\limits_0^x C(x-u)f(u)du \label{volterra-line}
\end{equation}
Since $\Pi_{[0,R]} \cal{C} \Pi_{[0,R]}=\cal{C}_R$ and
$\|\cal{C}_R\|<1$, operator $\cal{C}$ is well-defined on
$L^2[0,\infty)$ and is contraction, i.e. $\|\cal{C}\|_{2,2}\leq
1$.

\begin{lemma}
One has the following inclusion: $S(\mathbb{C}^+)\subset
B(\mathbb{C}^+)$. The operators $\cal{C}$ and $\cal{S}$ are
unitary equivalent.
\end{lemma}

\begin{proof}
Indeed, $s(\lambda)\in H^\infty (\mathbb{C}^+)$ and it is known
that
\[
\|s\|_{H^\infty(\mathbb{C}^+)}=\sup_{\|g\|_{H^2(\mathbb{C}^+)}=1}
\|sg\|_{H^2(\mathbb{C}^+)}=\|\cal{S}\|
\]
 The condition (\ref{contraction}) is
equivalent to the estimate $\|\cal{P}_{[0,R]} \cal{S}
\cal{P}_{[0,R]}\|<1$. Then, $\|\cal{S}\|\leq  1$, or
$s(\lambda)\in B(\mathbb{C^+})$. The unitary equivalence of
$\cal{S}$ and $\cal{C}$ follows from the unitary equivalence of
operators $\Pi_{[0,R]} \cal{C}\Pi_{[0,R]}$ and
$\cal{P}_{[0,R]}\cal{S}\cal{P}_{[0,R]}$ via the Fourier transform.
\end{proof}

Recall the definition of the accelerant: given  Hermitian $H(x)\in
L^2_{\rm loc}(\mathbb{R})$, we say that it is an accelerant if the
operator
\[
I+\cal{H}_R>0
\]
for any $R>0$ and $\cal{H}_R$ is given by (\ref{intop}). Given any
$C(x)\in L^2_{\rm loc}(\mathbb{R}^+)$, consider the function
$H(x)\in L^2_{\rm loc}(\mathbb{R}^+)$ which is the solution to
\begin{equation}
H(x)+C(x)+\int\limits_0^x C(x-u)H(u)du=0 \label{transfer}
\end{equation}
The direct iteration of the equation proves existence and
uniqueness of this $H(x)$. Let $H(-x)=\overline{H(x)}$ for $x>0$.
\begin{lemma}\label{c-h}
The function $H(x)$ is an accelerant if and only if  $C(x)$ is
such that (\ref{contraction}) holds for any $R>0$.
\end{lemma}
\begin{proof}

For any $R>0$, consider the Caley transform of $\cal{C}_R$:
\begin{equation}
I+\cal{U}_R=(I-\cal{C}_R)(I+\cal{C}_R)^{-1}
\end{equation}
Since $\cal{C}_R$ is a Volterra operator, the Caley transform does
exist. Moreover, $\cal{U}_R$ is a Volterra operator with the
kernel given exactly by ${H}$:
\begin{equation}
\cal{U}_Rf(x)=2\int\limits_0^x {H(x-s)} f(s)ds \label{u-def}
\end{equation}
This is an easy corollary from (\ref{transfer}).
 Clearly, $\Re
(I+\cal{U}_R)={I+\cal{H}}_R$. Now, the equivalence of
$I+{\cal{H}}_R>0$ and (\ref{contraction}) is a simple algebraic
fact.
\end{proof}

 Now, we can easily characterize the class of
all $C(x)$ that  generate $s(\lambda)\in S(\mathbb{C}^+)$.

\begin{theorem}
The Schur functions of Krein systems with $A(r)\in L^2_{\rm
loc}(\mathbb{R}^+)$ are in one-to-one correspondence with
functions $s(\lambda)\in S(\mathbb{C}^+)$. For each $s(\lambda)\in
S(\mathbb{C^+})$, the coefficient $A(r)$ of the associated Krein
system plays the role of the Schur parameter.\label{s0krein}
\end{theorem}
 \begin{proof}
Assume $s(\lambda)\in S(\mathbb{C}^+)$ and $C(x)$ is the
corresponding function. Denote by $H(x)$ the accelerant
corresponding to $\bar{C}(x)$, i.e.

\begin{equation}
H(x)+\bar{C}(x)+\int\limits_0^x \bar{C}(x-u)H(u)du=0
\label{transfer1}
\end{equation}
Then, by Lemma \ref{c-h}, $H(x)$ is an accelerant that generates
the Krein system with coefficient $A(x)\in L^2_{\rm
loc}(\mathbb{R}^+)$. Consider the corresponding Schur function
$\f(\lambda)$. Let us show that $\f(\lambda)=s(\lambda)$. For each
$R>0$, we have (\ref{e18s3})
\begin{equation}
\f(\lambda )=\frac{\B(R,\lambda)+\f_R(\lambda)\A_{\ast
}(R,\lambda)}{\A(R,\lambda)+\f_R(\lambda)\B_{\ast }(R,\lambda)}
\label{param-sh}
\end{equation}
for any $R$. Clearly, by (\ref{energy22}) with $r_1=\infty, r_2=R$
\begin{equation}
\f(\lambda)-\frac{\B(R,\lambda)}{\A(R,\lambda)}=
\frac{\exp(i\lambda
R)\f_R(\lambda)}{\A(R,\lambda)(\A(R,\lambda)+\f_R(\lambda)
\B_*(R,\lambda))} \label{sasp}
\end{equation}
By (\ref{e22s3}), the right-hand side of (\ref{sasp}) is equal to
$\exp(i\lambda R)\Phi_R(\lambda)$ with $\Phi_R(\lambda)\in
H^\infty (\mathbb{C}^+)$. Consider the function
$\B(R,\lambda)\A^{-1}(R,\lambda)$. Due to Levy-Wiener Theorem,
\[
\frac{\B(R,\lambda)}{\A(R,\lambda)}=\int\limits_0^\infty
C_R(x)\exp(i\lambda x)dx= \int\limits_0^R C_R(x)\exp(i\lambda
x)dx+
\]
\[
+ \exp(i\lambda R)\int\limits_0^\infty C_R(x+R)\exp(i\lambda x)dx
 ,\, C_R(x)\in L^1(\mathbb{R}^+)\cap L^2(\mathbb{R}^+)
\]
and the last term can be written as $\exp(i\lambda
R)\Phi_R(\lambda)$ with $\Phi_R(\lambda)\in H^\infty
(\mathbb{C}^+)$. We can also write
\[
\frac{\B(R,\lambda)}{\A(R,\lambda)}=
\left(1-\frac{\hat{P}_*(R,\lambda)}{P_*(R,\lambda)}\right)
\left(1+\frac{\hat{P}_*(R,\lambda)}{P_*(R,\lambda)}\right)^{-1}
\]
 From (\ref{reznik}),
(\ref{consist}) and (\ref{transfer1}), we infer $C_R(x)=C(x)$ for
$x\in [0,R]$. Since $R$ is arbitrary positive, $\f(\lambda)\in
S(\mathbb{C}^+)$ and $\f(\lambda)=s(\lambda)$.

Now, assume that the Krein system with the coefficient $A(r)$ is
given. Repeating the arguments above, one has $\f(\lambda)\in
S(\mathbb{C}^+)$.
 \end{proof}
\begin{remark}\rm It follows from (\ref{transfer}) that $C(x)$ on the interval
$[0,R]$ depends only on the values of $A(r)$ on the same interval.
This is because an accelerant has analogous property.
\end{remark}

Notice that $s(\lambda)\in S(\mathbb{C^+})$ implies certain
regularity at infinity, a boundary point of $\mathbb{C}^+$. For
instance, $s(iy)\to 0$ as $y\to+\infty$. So, there are plenty of
functions in $B(\mathbb{C}^+)$ that do not belong to
$S(\mathbb{C}^+)$.

The same arguments immediately yield
\begin{remark}
$A(r)\in C[0,\infty)$ iff $H(x)\in C[0,\infty)$ iff $C(x)\in
C[0,\infty)$ in the corresponding representation for
$\f(\lambda)$.
\end{remark}

Consider this case for the rest of the section. Let $C(x,r)$ be
function associated to $f_r(\lambda)$ by formula
(\ref{schur-alg}).
\begin{lemma}{\rm (Continuous analog of Geronimus theorem).}
The following relation holds true:
 $C(0,r)=-{A(r)}$.\label{linksg}
 \end{lemma}
\begin{proof} It is enough to prove the statement for $r=0$. From
(\ref{transfer1}), we have $C(0)=-\overline{H(+0)}$. Then, by
Lemma \ref{a-h},  $C(0)=-{A(0)}$.
\end{proof}

Notice that the value of $C(r)$ at zero gives the main term of
asymptotics of $\f(\lambda)$ as $\lambda=iy, y\to +\infty$. For
instance, if $A(r)\in C^1[0, \infty)$, then $H(x),C(x,r)\in
C^1[0,\infty)$ as well (see (\ref{erasw}) and (\ref{transfer1}))
and Lemma \ref{linksg} yields
$\f(r,iy)=-{A(r)}/y+\bar{o}(y^{-1})$. In general, for $A(r)\in
C[0,\infty)$, we have asymptotics in the mean (see Lemma
\ref{regular-inf} in Appendix):
\[
\lim_{y\to\infty}\frac{1}{y} \int\limits_0^y s\f(r,is)ds =-{A(r)}
\]
Anyway, the number $C(0,r)=-{A(r)}$ is an intrinsic parameter of
the function $\f(r,\lambda)$. Therefore, Lemma \ref{linksg} can be
regarded as the continuous analog of the celebrated Geronimus
theorem which says that the Schur parameters of the function from
$B(\mathbb{D})$ coincide with the Verblunsky parameters of the
associated sequence of orthogonal polynomials.

\begin{lemma}
Assume $A(r)\in C^1[0,\infty)$, then $C(x,r)$ is continuously
differentiable in $x$ and $r$ and satisfies the following
nonlinear integro-differential equation

\begin{equation}
\frac{\partial C(x,r)}{\partial r}= \frac{\partial
C(x,r)}{\partial x}-\overline{A(r)}\int\limits_0^x
C(x-u,r)C(u,r)du \label{c-equation}
\end{equation}
\end{lemma}
\begin{proof}
Let us prove smoothness in $\Omega_T=\{r\geq 0, x\in [0,T]\}$ for
any $T>0$. Apply (\ref{erasw}) to the Krein system on
$[r,\infty)$. We have the corresponding $r$--dependent function
$A^{(r)}(T,x)\in C^1(\Omega_T)$. Notice that
$A^{(r)}(T,x)=\Gamma_T^{(r)}(T-x,0)$ and
\[
\Gamma_T^{(r)}(x,0)+\int\limits_0^T
H^{(r)}(x-u)\Gamma_T^{(r)}(u,0)du=H^{(r)}(x)
\]

To prove (\ref{c-equation}) for, say, $0<r<R$ and $0<x<T$, we can
consider new $A_1$ equal to $A$ on $[0, R+T]$, smooth on
$\mathbb{R}^+$ with compact support. Then, new $C_1(x,r)=C(x,r)$
for $0<r<R, 0<x<T$. Moreover,  $C_1(x,r)\in L^1(\mathbb{R}^+)$ in
$x$ and the formula (\ref{schur-alg}) holds with $R=\infty$.
 Then, if one  substitutes (\ref{schur-alg}) to
(\ref{schur}), equation (\ref{c-equation}) pops up.
\end{proof}

Plug $A(r)=-C(0,r)$ into this equation and solve the first order
PDE with boundary condition $C(x,0)$ regarded as known. Then
(\ref{c-equation}) becomes a nonlinear integral equation which one
tries to solve by iterations. Since $A(r)=-C(0,r)$, that gives us
a solution to inverse problem since $C(x,0)$ can be read off the
spectral data, say $d\sigma$.

Next, let us focus on the differential equations (\ref{schur}) and
(\ref{krein2}). Looking at the formula (\ref{sprocedure}), one
might guess that the map $S_{\lambda, A, r}$ should also be
represented as a combination of M\"obius transforms and certain
multiplications. This is indeed the case.  The following result
gives approximation of continuous orthogonal polynomials by the
sequence of properly scaled polynomials orthogonal on the unit
circle. These discrete polynomials are given in terms of  Schur
parameters that depend upon the step of discretization $h$.

\begin{theorem}
Let $A(r)\in C[0,\infty)$. Fix any $r>0$ and consider the sequence
of Verblunsky coefficients
\begin{equation}\label{discrete-A}
a^{(h)}_0=h{A(t_1)}, a^{(h)}_1=h{A(t_2)}, \ldots,
a^{(h)}_{n-1}=h{A(t_n)}
\end{equation}
 and
\[
a^{(h)}_{j}=0, j\geq n, (h=r/n, t_j=jh)
\]
where $h$ is chosen so small that all of these coefficients are
less than one in absolute value. Consider the discrete transfer
matrix generated by these coefficients
\[
M(0,k,z)=W(a_k) Z\ldots W(a_0) Z
\]
with
\[
Z(z)=\left[
\begin{array}{cc}
z & 0\\
 0 &1
\end{array}
\right], W(a_j)=\left[
\begin{array}{cc}
1 & -\bar{a}_j\\
-a_j &1
\end{array}
\right]
\]
Then, we have
\begin{equation}
X(r,\lambda)=\lim_{h\to 0} M(0,n,\exp(i\lambda h))
\end{equation}
and the convergence is uniform over $\lambda$ in compacts in
$\mathbb{C}$.
 \label{approx-1}
\end{theorem}
\begin{proof}
Consider small $h>0$. Then, from the definition and properties of
the multiplicative integral \cite{D-K}, we have
\begin{eqnarray}
X(r,\lambda)=\int\limits_0^{\stackrel{\curvearrowleft}{r}} \exp
[V(t)dt]= \lim_{h\to 0}\Big[ (1+h V(t_n))\ldots (1+h V(t_1))
\Big]=&\\\lim\limits_{h\to 0}\Big[  W(a^{(h)}_{n-1}) Z(w)
W(a^{(h)}_{n-2}) Z(w)\ldots  W(a^{(h)}_0) Z(w)\Big]&
\label{multiplicative}
\end{eqnarray}
where $w=\exp(i\lambda h)$. Now, the statement of the Theorem is
an elementary corollary from (\ref{multiplicative}).
\end{proof}
The next Corollary follows directly from the Theorem.
\begin{corollary}\label{approx-poly}
For continuous orthogonal polynomials,
\[
\left[
\begin{array}{cc}
P(r,\lambda) & \widehat{P}(r,\lambda)\\
P_*(r,\lambda) & \widehat{P}_*(r,\lambda)
\end{array}
\right]=\lim_{h\to 0}\left[
\begin{array}{cc}
P_n(w) & \widehat{P}_n(w)\\
P_n^*(w) & \widehat{P}_n(w)
\end{array}
\right]
\]
where $w=\exp(i\lambda h)$, polynomials $P_k, P_k^*$ are monic
orthogonal polynomials generated by the prescribed Verblunsky
parameters and $\widehat{P}_k, \widehat{P}^*_k$ are dual to them.
The convergence is uniform in $\lambda$ from any compact in
$\mathbb{C}$.
\end{corollary}

\begin{corollary} \label{moebius-appr-cor}
Let $A(r)\in C[0,\infty)$. For the map $S_{\lambda,A,r}(z)$, we
have
\begin{equation}\label{moebius-approximation}
S_{\lambda,-A,r}(z)=\lim_{h\to 0} S_{w,a^{(h)}_0, \ldots,
a^{(h)}_{n-1}}(z),
\end{equation}
$w=\exp(i\lambda h)$ and the convergence is again uniform over the
compacts in $\mathbb{C}$.
\end{corollary}
\begin{proof}
(\ref{moebius-approximation}) follows from the formula
(\ref{mobius-easy}). Indeed, in discrete setting,  there is a
formula analogous to (\ref{mobius-easy}) (see \cite{Khrushchev1},
formula (4.19))
\[
S_{z,a_0, \ldots,
a_{n}}(f)=\frac{A_n(z)+zB_n^*(z)f}{B_n(z)+zA_n^*(z)f}
\]
We also have
\begin{equation} \label{wall-approxim}
\left[
\begin{array}{cr}
\A_*(r,\lambda) & \B_*(r,\lambda)\\
\B(r,\lambda) & \A(r,\lambda)
\end{array}
\right]= \lim_{h\to 0}\left[
\begin{array}{cr}
wB_n^*(w) & -A_n^*(w)\\
-wA_n(w) & B_n(w)
\end{array}
\right] \end{equation} and $A_n, B_n$ are the standard Wall
polynomials\footnote{We want to emphasize some abuse in notations
in the definition of continuous Wall polynomials. If one wants to
be consistent with discrete case, then the choice must be made
according to (\ref{wall-approxim}) so that for the transfer matrix
$X$ in Krein system:
\begin{equation}\label{footnote1}
X=\left[ \begin{array}{cc} B_* & -A_*\\
-A & B\end{array}\right]
\end{equation}
In the meantime, we want to keep our notations to  be consistent
later on with terminology accepted in the scattering theory.
 }.
  On the other hand, by Lemma~\ref{sign-change}, $JXJ$ is
 the transfer matrix for Krein systems with coefficient $-A$.
 Comparing the corresponding formulas to (\ref{mobius-easy}), we get the statement of the Corollary.
\end{proof}
\begin{remark} \rm One might wonder why the formula
(\ref{moebius-approximation}) contains the sign minus in front of
$A$? The answer to this question is contained in the definition of
continuous Schur function and map $S_{\lambda, A,r}$. Indeed, we
defined $\f$ as
\[
\f(\lambda)=\lim_{r\to\infty}
\frac{X_{21}(r,\lambda)}{X_{22}(r,\lambda)}
\]
In the discrete case, the Schur function is defined as
\[
f(z)=\lim_{n\to\infty} \frac{A_n(z)}{B_n(z)}
\]
The Schur function  with opposite sign corresponds to the dual
system with coefficient of opposite sign. So, having
(\ref{footnote1}) in mind (see the footnote below), we see why the
opposite sign was picked up.
\end{remark}

If we view an operation $S_{z,a_0, \ldots, a_{n-1}} (w)$
introduced in (\ref{sprocedure}) as a map of $w\in \mathbb{D}$ to
$\mathbb{D}$ with parameters $z,a_0,\ldots, a_{n-1} \in
\mathbb{D}$, then
\begin{itemize}
\item For $z\in \mathbb{T}$, it preserves the pseudohyperbolic
distance. This is simply because both multiplication by $z$ and
the M\"obius transform preserve this distance.

\item For $z\in \mathbb{D}$, it acts as a contraction, i.e.
\[
\rho (S_{z,a_0,\ldots, a_{n-1}} (w_1), S_{z,a_0,\ldots, a_{n-1}}
(w_2))\leq \rho (w_1,w_2)
\]
The contractive property follows solely from the contractive
property of multiplication by $z$.
\end{itemize}
Analogous properties for the map $S_{\lambda, A, r}$ is given in
the following
\begin{lemma}
For any $\lambda\in \mathbb{R}$, the map $S_{\lambda, A, r}$
preserves the pseudohyperbolic metric and for $\lambda\in
\mathbb{C}^+$ it is contraction.
\end{lemma}
\begin{proof}
For continuous $A$, the proof follows immediately from the
properties of discrete map $S$ and Corollary
\ref{moebius-appr-cor}. Approximating $A\in L^1[0,r]$ by
continuous functions, we get the statement of the Lemma in general
case.
\end{proof}

Now, we can really regard the map $S_{\lambda, A, r}$ as a
combination of M\"obius transforms and rotations.  In the
particular case $\lambda=0$, we have the following representation
\[
S_{0, A, r} (z)=\Phi_0^r (\overline{A}, A, z)
\]
where we use notation $\Phi_0^r (F, G, z)$ for the {\bf
continuous} continued fraction invented by Puig Adam \cite{Puig}
and later developed by Wall \cite{cWall}. In this case, equation
(\ref{schur}) takes the form of the Riccati-Stiltjes equation
\begin{equation}
\frac{df_r(0)}{dr}= A(r)-\overline{A(r)}f^2_r(0)
\end{equation}
We do not get deeper into this subject and refer the interested
reader to the original papers.

One should notice that there are many ways to approximate Krein
system by the sequence of OPUC's. For example, one can take the
following system of Verblunsky coefficients:
\begin{equation}
hA(t_1),0,hA(t_2),0,\ldots,hA(t_n), 0, 0, \ldots \label{alternative}
\end{equation}
Then, the only difference will be a different scaling, e.g.
\[
X(r,\lambda)=\lim_{h\to 0} M(0,2n,\exp(i\lambda h/2))
\]

There are at least two other ways to approximate the Krein system
with the sequence of OPUC. They are discussed below and we will
make use of them later on.

 Previously, we started with finite differences approximation to
a system of ODE. That produced the approximation of the related
analytic functions. Now we start with an accelerant and
approximate it first.

\begin{theorem}
Assume that we are given an accelerant $H(x)\in C[0,\infty)$.  Fix
any $R>0$ and
 let $h=R/n$. Consider the Toeplitz matrices

\[
\cal{T}_j=\left[
\begin{array}{cccc}
1& hH(-h)&  \ldots & hH(-jh) \\
\ldots& \ldots& \ldots& \ldots\\
hH(jh)& hH((j-1)h)& \ldots& 1
\end{array}
\right],\quad j=1,\ldots, n
\]
For $h$ small enough, $\cal{T}_n>0$ and it generates the Schur coefficients $\{a^{(h)}_j\}_{j=0}^{j=n}$ such that

\begin{equation}
\lim_{h\to 0} \sup_{\delta<jh<R} |A(jh)-h^{-1}{a^{(h)}_j}|\to 0
\label{rer}
\end{equation}
where $\delta$ is a small fixed number. \label{approx-2}
\end{theorem}

\begin{proof} Indeed, we know that $I+\cal{H}_R>0$,
$H(-x)=\overline{H(x)}$, and $H(x)\in C[0,\infty)$. Thus
$\cal{T}_n>0$ for $n$ large enough and it generates the Schur
parameters $\{a^{(h)}_j\}_{j=0}^{j=n}$
\begin{equation}
-\overline{a}^{(h)}_{j-1}=\frac{ \det \cal{T}_j}{\det
\cal{T}_{j-1}} \cal{T}^{-1}_j (j, 0) \label{goddv}
\end{equation}
Consider the resolvent equation

\[
\Gamma_R (x,0)+\int\limits_0^R H(x-u) \Gamma_R (u,0)du=H(x),
\]
Its $h$--step discretization leads to the system of linear
algebraic equations with the matrix $\cal{T}_n$. If one takes any
$\delta<r<R$, then the matrix $\cal{T}_{[rh^{-1}]+1}$ is the
discretization of the operator $I+\cal{H}_r$ but with the step of
discretization (relative to the length of the interval $[0,r]$)
slightly bigger than that for $[0,R]$. Nevertheless, it tends to
zero as $h\to 0$ and this is why we need to keep $r>\delta>0$.
 It allows us to use the following argument. We have

\[
\Gamma_r (x,0)+\int\limits_0^r H(x-u) \Gamma_r (u,0)du=H(x),
\]

The discretization with the step $h$ gives

\[
\left[
\begin{array}{cccc}
1& hH(-h)&  \ldots & hH(-jh) \\
\ldots& \ldots& \ldots& \ldots\\
hH(jh)& hH((j-1)h)& \ldots& 1
\end{array}
\right]
\left[
\begin{array}{c}
\gamma^{(r,h)}(0) \\
\ldots \\
\gamma^{(r,h)}(j)
\end{array}
\right]
=
\left[
\begin{array}{c}
H(+0) \\
\ldots \\
H(jh)
\end{array}
\right],
\]
and $j=[rh^{-1}]+1$. Application of the standard arguments that
use Hadamard's Lemma on the determinants yields
\begin{equation}\label{approx-kernel}
\lim_{h\to 0} \sup_{k=0,\ldots, j} \left|\frac{\hat{\delta}_r
(kh,0)}{\hat{\delta}_r}-\gamma^{(r,h)}(k)\right|= 0
\end{equation}
where $\hat{\delta}_r(x,y)$ and $\hat\delta_r$ are introduced in
Lemma \ref{fredholm-mod}.  From this Lemma, we also know that
\[
\Gamma_r(x,0)=\frac{\hat\delta_r(x,0)}{\hat\delta_r}
\]
The convergence in (\ref{approx-kernel}) is uniform in $r$ as long
as $\delta <r<R$. Since $A(r)=\Gamma_r (0,r)$, we have
\[
\sup_{j: \delta<jh<R} |\overline{A(jh)}-\gamma^{(r,h)}(j)|\to 0
\]
as $h\to 0$. At the same time,  Kramer's rule gives us the
following
\begin{eqnarray*}
\gamma^{(r,h)}(j)= \frac{1}{\det \cal{T}_j} \det\left[
\begin{array}{ccccc}
1& hH(-h)&  \ldots & hH(-(j-1)h)& H(+0) \\
\ldots& \ldots& \ldots& \ldots &\ldots \\
hH(jh)& hH((j-1)h)& \ldots& hH(h) & H(jh)
\end{array}
\right]\\
=
\frac{1}{\det \cal{T}_j}
\det\left[
\begin{array}{ccccc}
1& hH(-h)&  \ldots & hH(-(j-1)h)& H(+0)-h^{-1} \\
\ldots& \ldots& \ldots& \ldots &\ldots \\
hH(jh)& hH((j-1)h)& \ldots& hH(h) & 0
\end{array}
\right]\\
=-(H(+0)-h^{-1})\frac{\det \cal{T}_{j-1}}{\det \cal{T}_{j}}\,
\overline{a}^{(h)}_{j-1}
\end{eqnarray*}
where the last formula follows from (\ref{goddv}). We have $\det
\cal{T}_{j}\to \hat\delta_r$ uniformly in $r$ as $h\to 0$ and
$\hat\delta_r$ is defined in Lemma \ref{fredholm-mod}. Therefore,
(\ref{rer}) follows.
\end{proof}

The next Theorem is technical but we will need it later in the
proof of the Strong Szeg\H{o} Theorem. It will give an
approximation of Krein system through its measure although we need
to take $d\sigma$ of a very special kind. Assume that $d\sigma$ is
purely absolutely continuous with density
\[
2\pi\sigma'(\lambda)=\exp\left( \int_\mathbb{R} l(x) e^{i\lambda
x}dx \right)
\]
with $l(x)$-- Hermitian, continuous on $\mathbb{R}$ with compact
support within $[-R,R]$. This measure $d\sigma$ will generate the
Krein system with continuous $H(x)$ and $A(r)$.  For $H(x)$, we
have an expansion (see (\ref{e2s2}))
\begin{equation}\label{expansion}
\overline{H(x)}=l(x)+\frac{\left[l*l\right](x)}{2!}+\ldots+\frac{\left[l*\ldots*l\right](x)}{k!}+\ldots
\end{equation}
\begin{theorem}
For large $n$, consider $h=Rn^{-1}$, $x_j=jh, j=-n,\ldots, n$ and
the a.c. measure $d\mu_n$ on $\mathbb{T}$ with the density given
by the formula
\[
\mu_n'(\theta)=\exp\left[ \sum\limits_{j=-n}^n hl(x_j)
z^{j}\right], z=e^{i\theta}
\]
Let $a^{(h)}_j$ be the associated Verblunsky parameters. Then
\begin{equation} \lim_{h\to 0} \sup_{\delta<jh<R}
|A(jh)-h^{-1}{a^{(h)}_j}|\to 0 \label{rer1}
\end{equation}
where $\delta$ is any small fixed number. \label{approx-3}
\end{theorem}
\begin{proof}
The $0$--th moment of the measure is equal to
\begin{equation}
c_n^{(0)}=1+hl(0)+\ldots+\frac{h^k}{k!}\sum_{j_1+\ldots+j_k=0}l(x_{j_1})\cdot
\ldots \cdot l(x_{j_k})+\ldots \label{series-approx}
\end{equation}
which can be written as
\[
1+hH(0)+\bar{o}(h)
\]
as it follows from (\ref{expansion}), approximation of the
integral by the Riemann sum, and simple estimates on the tail of
the series (\ref{series-approx}). The same is true about the
higher moments, i.e.
\[
c_n^{(k)}=hH(-kh)+\bar{o}(h), |k|<n
\]
Moreover, $|h^{-1}\bar{o}(h)|\to 0$ as $h\to 0$ uniformly in
$|k|<n$. Application of the same arguments that proved Theorem
\ref{approx-2} completes the proof.
\end{proof}

 {\bf Remarks
and historical notes.} \\The continuous analogs of Schur and
Caratheodory-Toeplitz problems were considered in
\cite{krein-madamyan}. The corresponding classes of analytic
contractions were introduced in the same paper. Equation
(\ref{schur-evolution}) is rather standard in the theory of
inverse problems. For Schr\"odinger operators, the equation
analogous to (\ref{c-equation}) was obtained and studied in
\cite{Simon-4}, \cite{Simon-5}. The explicit approximation of
Krein system by sequence of scaled OPUC's is new to the best of
our knowledge. For discretization of continuous Toeplitz
operators, see \cite{Ellis} Chapter 8. It is a very good exercise
to take $A(r)={\rm const}$ on $[0,R]$ and explicitly compute
polynomials that correspond to discretization with step $h=R/n$.

\newpage

\section{Zeroes of $P(r,\lambda)$}\label{sect-zeros}

In this section, we study zeroes of the function $P(r,\lambda)$.
It is convenient for us to write $P(r,\lambda)=\exp(i\lambda
r)Q(r,\lambda)$ and consider $Q(r,\lambda)$ instead. From
(\ref{polyn1}), we know that
\[
Q(r,\lambda)= 1-\int\limits_{0}^{r}\Gamma _{r}(s,0)\exp (-i\lambda
s)ds
\]
Assume that $A(r)$ is not identically zero. By the Hadamard
Theorem,  function $Q(r,\lambda)$ has the following factorization
\begin{eqnarray*}
Q(r,\lambda)=C\lambda^m\exp(\alpha \lambda)\prod_{n=1}^\infty
\left( 1-\lambda/\lambda_n\right) \exp(\lambda/\lambda_n),\,
|\lambda_1|\leq |\lambda_2|\leq \ldots,
\end{eqnarray*}
Since $P(r,\lambda)$ has zeroes in $\mathbb{C}^+$ only, $m=0$.
Also,
\[
C=Q(r,0),\, \alpha=\frac{\partial Q(r,0)/\partial
\lambda}{Q(r,0)}
\]
Clearly, $|\lambda_1(r)|\to +\infty$ as $r\to 0$. For each $r>0$,
zeroes $\lambda_n(r)$ accumulate at infinity in a very regular
way. For instance, if $A(x)$ is smooth on $x\in [0,r]$ and
$A(r)\neq 0$, then $\lambda_n$ has the following trivial
asymptotics at infinity (see Lemma \ref{zeroes} in Appendix)
\[
\begin{array}{cc}
\lambda_n=\lambda_n^0+\bar{o}(1), n\to\infty, \\
\lambda_n^0= x_n+iy_n, x_n^2+y_n^2= |A(r)|^2\exp(2ry_n),
\\
x_n=r^{-1}\left[\pi/2+\pi n-Arg (A(r))\right], n\in \mathbb{Z}
\end{array}
\]
i.e.  the zeroes are accumulating evenly near the graph of the
logarithm. Moreover, as $r\to\infty$, the graph of logarithm is
getting closer to the real axis and the spacing between
consecutive zeroes decreases \footnote{More on the asymptotics of
$\{\lambda_n\}$ can be found in recent preprint \cite{Hryniv}.}.

In the meantime, an interesting question is distribution of zeroes
for finite $r$ and $\lambda$ inside the compacts in
$\mathbb{C}^+$. The Fej\'er Theorem for polynomials $\varphi_n(z)$
orthogonal on $\mathbb{T}$ with respect to measure $d\mu$ says
that all zeroes of each $\varphi_n(z)$ are inside the convex hull
of ${\rm supp}(d\mu)$. Let us prove similar statements for
$P(r,\lambda)$. Assume that ${\rm supp}(d\sigma)$ has a gap, say
$(a,b)$. We want to show that $\lambda_n$ stay away from $(a,b)$.
Let $M_r$ be introduced by the following formula
\[
M_r=\sup_{\lambda\in \mathbb{R}} |P(r,\lambda)|
\]
Lemma \ref{gronwall-p*} and $|P(r,\lambda)|=|P_*(r,\lambda)|$ for
$\lambda\in \mathbb{R}$  yield
\[
M_r\leq \exp \left[\int\limits_0^r |A(s)|ds\right]
\]

\begin{theorem}{\rm (Continuous analog of Fejer Theorem).}
Let $(a,b)\cap {\rm supp}(d\sigma)=\emptyset$. Then,
$P(r,\lambda)$ has no zeroes in  $\Omega_r$ given by
\[
\Omega_r:  \lambda=x+iy, \, \int\limits_{-\infty}^\infty
\frac{2yd\sigma(\lambda)}{(\lambda-x)^2+y^2}< M_r^{-2}
\]
 \label{fejer-1}
\end{theorem}
\begin{proof} Assume that $\lambda_0=x+iy$ is a zero of
$P(r,\lambda)$. Consider
\mbox{$f(\lambda)=P(r,\lambda)/(\lambda-\lambda_0)\in S_r$}. By
(\ref{subordinacy-s}),
\[
\int\limits_{-\infty}^\infty |f(\lambda)|^2\frac{d\lambda}{2\pi
|P(r,\lambda)|^2}=\int\limits_{-\infty}^\infty
|f(\lambda)|^2d\sigma(\lambda)
\]
That can be rewritten as
\[
\frac{1}{2y}=\int\limits_{-\infty}^\infty
\frac{|P(r,\lambda)|^2}{(\lambda-x)^2+y^2}\,d\sigma(\lambda)
\]
 Thus, a simple estimate
follows
\[
1\leq M_r^2 \int\limits_{-\infty}^\infty
\frac{2yd\sigma(\lambda)}{(\lambda-x)^2+y^2}
\]
\end{proof}
Clearly, $\Omega_r$ contains a domain in $\mathbb{C}^+$ contiguous
to $(a,b)$.
\begin{remark}\rm One can modify this proof in the
following way. For simplicity, assume $(a,b)=(-1,1)$. Introduce
\[
N_r=\int\limits_{-\infty}^\infty
\frac{|P(r,\lambda)|^2}{\lambda^2+1} d\sigma(\lambda)
\]
Then $P(r,\lambda)$ has no zeroes in the following set
\[
\Omega'_r: \lambda=x+iy, y<\left[2N_r \sup_{\lambda\in {\rm
supp}(\sigma)} \frac{\lambda^2+1}{(\lambda-x)^2+y^2}\right]^{-1}
\]
For a large class of coefficients $A(r)$, $N_r$ is bounded in $r$.
This is because {$2N_r= {\rm Tr} \Im G_i(r,r)$} where $G$ is the
resolvent kernel for corresponding Dirac operator which will be
introduced later. For example, $A\in L^\infty(\mathbb{R}^+)$ is
sufficient for $N_r$ to be bounded in $r>0$.
\end{remark}

 The next Theorem yields yet another
result on the distribution of $\lambda_n$.

\begin{theorem}
If $z_1$ is a zero of $P(r,\lambda)$, then there is no any other
zero of $P(r,\lambda)$ in $\Omega_1$
\[
\Omega_1: \lambda\in \mathbb{C}^+, |\lambda-\bar{z}_1|< {\rm Dist}
(\lambda, {\rm supp}(d\sigma))
\]
\label{fejer-2}
\end{theorem}
\begin{proof} Assume $z_1$ is a zero of $P(r,\lambda)$. By the
variational principle, function
$f_0(\lambda)=K_r(z_1,\lambda)/K_r(z_1,z_1)$ minimizes
$\|f(\lambda)\|_{2,\sigma}$ in the set of all \mbox{$f(\lambda)\in
S_r, |f(z_1)|=1$}. Since $P(r,z_1)=0$,
\[
f_0(\lambda)=\frac{(z_1-\bar{z}_1)P_*(r,\lambda)}{(\lambda-\bar{z}_1)P_*(r,z_1)}
\]
In the meantime, if $P(r,\lambda)$ has a zero $z_2\in \Omega_1$,
then $P_*(r,\bar{z}_2)=0$ and the function
\[
f_1(\lambda)=\frac{z_1-\bar{z}_2}{\lambda-\bar{z}_2} f_0(\lambda)
\]
belongs to $S_r$, $f_1(z_1)=1$, but
$\|f_1\|_{2,\sigma}<\|f_0\|_{2,\sigma}$, a contradiction.
\end{proof}
Notice that this Theorem makes no assumptions on coefficient
$A(r)$. It also implies that the isosceles triangle with base
$(a,b)$ and angles $\pi/6,\pi/6, 2\pi/3$ can contain only finite
number of zeroes.

The next result is the continuous analog of the Widom's theorem on
the zeroes of OPUC. It says that the zeroes of $P(r,\lambda)$ can
not accumulate in the compact of $\mathbb{C^+}$ provided that the
support of $d\sigma$ is not the whole $\mathbb{R}$. The proof is a
rather simple modification of  proof for the discrete case.

For any compact $K\subset \mathbb{C}^+$, define $N_K(r)$ as the
number of zeroes of $P(r,\lambda)$ in $K$. Fix any $R>0$. We have
elementary estimates
\[
\max_{\lambda\in K, 0\leq r\leq R} |P(r,\lambda)|<C(A, R, K),
\min_{0\leq r\leq R} |P(r,iy)|>1/2
\]
if $y$ is large enough.  Therefore, by Jensen's formula
(\cite{Rudin}, Theorem 15.18), we know that $N_K(r)$ is bounded
for $r\in [0,R]$.

\begin{theorem} {\rm(Continuous analog of Widom's theorem).}
Assume that the measure $d \sigma$ of the Krein system is such
that ${\rm supp}(d\sigma)\neq \mathbb{R}$.  Then, we have
\begin{equation}
\sup_{r>0} N_K(r) <\infty \label{accum}
\end{equation}
\end{theorem}

\begin{proof} Fix $K$ and Krein system with the measure
$d\sigma$. Cover $K$ by disjoint cubes $C_j$ with side
$\varepsilon$. We choose $\varepsilon$ small enough to satisfy the
following conditions. For any cube $C_j$, consider $\xi\in C_j$
and a map $\phi_\xi(\lambda)=(\lambda-\xi)^{-1}$. The reflected
cube $\bar{C}_{j}=\{\bar{z}, z\in C_j\}$ will be mapped to a set
$D_{j,\xi}$ and the support of the measure ${\rm supp}(d\sigma)$
to a set $F_{j,\xi}$, a proper subset of some circle. For each
$j$, consider
\begin{equation}
D_j=\cup_{\xi\in C_j} D_{j, \xi}, F_j=\cup_{\xi\in C_j} F_{j, \xi}
\label{sets}
\end{equation}
We now require $\varepsilon$ to be so small that for each $j$ we
have: $D_j$ and $F_j$ are disjoint, $\mathbb{C}\setminus  F_j$ is
connected. We can always satisfy these conditions because the
function $\phi_\xi(\lambda)$ is jointly continuous and ${\rm
supp}(d\sigma)$ has a gap in it.

Fix this $\varepsilon$. Assume (\ref{accum}) is wrong. Clearly,
among all cubes $C_j$ there will be al least one, call it
$C_{j^\prime}$, such that for any $k$ we can find $r$  so that
$P(r,\lambda)$ has $n$ zeroes in $\overline{C_{j^{\prime}}}$ and
$n>k$. Denote these zeroes by $\lambda_j, j=1, \ldots, n$. Fix
this cube. Let $D_{j'}$ and $F_{j'}$ be the corresponding sets
defined by (\ref{sets}). Let $m<n$ be some fixed number to be
specified later.

By the variational principle, the function
\[
f_0(\lambda)=\frac{K_r(\lambda_1,\lambda)}{K_r(\lambda_1,\lambda_1)}=
\frac{\lambda_1-\bar{\lambda}_1}{\lambda-\bar{\lambda}_1}
\frac{P_*(r,\lambda)}{P_*(r,\lambda_1)}
\]
minimizes $\|f\|_{2,\sigma}$ in the set $f\in S_r,
|f(\lambda_1)|=1$. We can write
\[
f_0(\lambda)=g(\lambda) \frac{(\lambda-\bar{\lambda}_2)\ldots
(\lambda-\bar{\lambda}_{m+1})} {(\lambda_1-\bar{\lambda}_2)\ldots
(\lambda_1-\bar{\lambda}_{m+1})}
\]
Notice that $g(\lambda_1)=1$. We will find a polynomial
$Q(\lambda)$ satisfying the following properties: $\deg Q\leq m$,
$Q(\lambda_1)=1$, and
\begin{equation}
|Q(\lambda)|\leq 2^{-1} \left|
\frac{(\lambda-\bar{\lambda}_2)\ldots
(\lambda-\bar{\lambda}_{m+1})} {(\lambda_1-\bar{\lambda}_2)\ldots
(\lambda_1-\bar{\lambda}_{m+1})} \right| \label{property-imp}
\end{equation}
for any $\lambda\in {\rm supp}(d\sigma)$. That would give us a
contradiction since for the function
$f(\lambda)=g(\lambda)Q(\lambda)$ we have: $f(\lambda_1)=1$,
\[
\int\limits_{\mathbb{R}} |f(\lambda)|^2 d\sigma \leq 2^{-2}
\int\limits_{\mathbb{R}} |f_0(\lambda)|^2 d\sigma
\]
and
\[
f(\lambda)=f_0(\lambda)Q(\lambda)\left[
\frac{(\lambda-\bar{\lambda}_2)\ldots
(\lambda-\bar{\lambda}_{m+1})} {(\lambda_1-\bar{\lambda}_2)\ldots
(\lambda_1-\bar{\lambda}_{m+1})}\right]^{-1}\in S_r
\]
by Paley-Wiener Theorem. Thus we have a contradiction with the
variational principle.

To find $Q(\lambda)$, we first take a map
$z=(\lambda-\lambda_1)^{-1}$. It sends $\lambda_1$ to infinity,
the support of $d\sigma$ will be mapped to
$F_{j^\prime,\lambda_1}$, and the set $\bar{C}_{j^\prime}$
(reflection of $C_{j^\prime}$ with respect to $\mathbb{R}$) will
go to a compact $D_{j^\prime,\lambda_1}$.

We now use Widom's Lemma (see Appendix, Lemma \ref{widom}) for two
compacts $D_{j'}$ and $F_{j'}$. They are disjoint and
$\mathbb{C}\setminus F_{j'}$ is connected so the Lemma is
applicable and the number $m$ can be chosen so that for any points
$z_j\in D, j=1,\ldots, m$, we can find a monic polynomial
$\tilde{Q}(z)$ of degree $m$ such that:

\begin{equation}
\left| \frac{\tilde{Q}(z)}{(z-z_1)\ldots(z-z_m)} \right| \leq
\frac12 \label{widom1}
\end{equation}
for all $z\in F$. In particular, for the points
$z_j=(\bar{\lambda}_{j+1}-\lambda_1)^{-1}\in
D_{j^\prime,\lambda_1}\subseteq D, j=1,\ldots, m$ there is
$\tilde{Q}(z)$ so that we have (\ref{widom1}) for any $z\in
F_{j^\prime,\lambda_1}\subseteq F$ . Translating it back to the
$\lambda$ variable, we have (\ref{property-imp}) where
$Q(\lambda)=(\lambda-\lambda_1)^{m}\tilde{Q}((\lambda-\lambda_1)^{-1})$.
Clearly, $\deg Q\leq m, Q(\lambda_1)=1$.
\end{proof}

{\bf Remarks and historical notes.} The asymptotics of zeroes for
the exponential functions of the special type (e.g.,
$P(r,\lambda)$) is a classical question. The problem here, of
course, is how this asymptotics depends on the regularity of the
function in representation (function $\Gamma_r(0,t)$ in our case).
We do not consider this problem here, interested reader can check
\cite{Hryniv} for related results.

The results from this section are new. We addressed only some of
the basic questions about the distribution of zeroes. Clearly,
there are many questions left open.

\newpage

\section{The case $A(r)\in L^2(\mathbb{R}^+)$}\label{sect-l2}

Now, let us study an important class of Krein systems: $A(r)\in
L^2(\mathbb{R}^+)$.

\begin{theorem}
If $A(r)\in L^2[0,\infty)$, then $d\sigma\in$ {\rm (Szeg\H{o})}.
Moreover, \[P_*(r,\lambda)\to \Pi_{\alpha}(\lambda),
\widehat{P}_*(r,\lambda)\to \widehat{\Pi}_{\beta}(\lambda),
\A(r,\lambda)\to \A(\lambda), \B(r,\lambda)\to \B(\lambda) \] as
$r\to\infty$ uniformly in $\Im \lambda>\varepsilon$,
$\varepsilon>0$. Function $\B(\lambda)\in N(\mathbb{C}^+)$,
$\f(\lambda)=\B(\lambda)\A^{-1}(\lambda)$, $\A^{-1}(\lambda)$ is
an outer function from $B(\mathbb{C}^+)$,
\begin{equation}
\int\limits_{-\infty}^\infty \ln|\A(\lambda)|d\lambda=\pi
\int\limits_0^\infty |A(r)|^2 dr \label{tracing}
\end{equation}
\label{main-lsz}
\end{theorem}
\begin{proof}

From (\ref{e01s3}), we have
\begin{equation}
P_*(r,\lambda)=1-\int\limits_{0}^{r}\exp (i\lambda
s)A(s)ds+\int\limits_{0}^{r}\exp (i\lambda s)A(s)\int\limits_0^s
\exp(-i\lambda t) \overline{A(t)}P_*(t,\lambda)dtds
\label{int-eq-l2}
\end{equation}
or
\begin{equation}
\begin{array}{ll}
\displaystyle P_*(r,\lambda)=1-\int\limits_{0}^{r}\exp (i\lambda
s)A(s)ds
\\
\displaystyle +\int\limits_{0}^{r} P_*(t,\lambda) \left[
\exp(-i\lambda t) \overline{A(t)}
 \int\limits_t^r
\exp(i\lambda s) A(s)ds\right] dt
\end{array}
\label{int-eq-l3}
\end{equation}
Use Cauchy-Schwarz and Young inequalities to get
\[
\left\|
  \exp(-i\lambda t) \overline{A(t)}
 \int\limits_t^r
\exp(i\lambda s) A(s)ds \right\|_{L^1[0,\infty)}\leq
\frac{\|A\|_2^2}{\Im\lambda}, \, \lambda\in \mathbb{C}^+
\]
From Gronwall-Belmann inequality,
\begin{eqnarray*}
|P_*(r,\lambda)|\leq C(\lambda)\exp \left(
[\Im\lambda]^{-1}\|A\|_2^2 \right), \lambda\in \mathbb{C}^+,
\,C(\lambda)\leq 1+\|A\|_2 [\Im \lambda]^{-1/2}
\end{eqnarray*}
Recall the function (\ref{function-pi}) and Lemma \ref{converge}.
From (\ref{int-eq-l3}), we have $P_*(r,\lambda)\to
\Pi_\alpha(\lambda)$ as $r\to\infty$, $\lambda\in \mathbb{C}^+$.
Similarly, $\widehat{P}_*(r,\lambda)\to
\widehat{\Pi}_\beta(\lambda)$. Thus, $\A(r,\lambda)\to
\A(\lambda)$, $\B(r,\lambda)\to \B(\lambda)$ as $r\to\infty$.
Moreover, from Lemma \ref{wall},  $|\A(\lambda)|^2\geq
1+|\B(\lambda)|^2,\, \Im\lambda>0$. Consequently,
$\A^{-1}(\lambda)\in B(\mathbb{C}^+), \A(\lambda)\in
N(\mathbb{C}^+), \B(\lambda)\in N(\mathbb{C}^+)$. We also have
\[
\A(\lambda)=(\Pi_\alpha(\lambda)+\widehat{\Pi}_\beta(\lambda))/2=\frac{\Pi_\alpha(\lambda)}{2}\left[
1+\frac{\widehat{\Pi}_\beta(\lambda)}{\Pi_\alpha(\lambda)}\right]
\]
Clearly, $\Pi_\alpha(\lambda)$ and $\widehat{\Pi}_\beta(\lambda)$
are outer from $N(\mathbb{C}^+)$. Due to Theorem
\ref{Bernstein-Szego},
\[
F(\lambda)=\frac{\widehat{\Pi}_\beta(\lambda)}{\Pi_\alpha(\lambda)}
\]
Since $F(\lambda)$ has positive real part, the function
$1+\widehat{\Pi}_\beta(\lambda) \Pi_\alpha^{-1}(\lambda)$ is outer
from $N(\mathbb{C}^+)$. Consequently, $\A(\lambda)$ is outer from
$N(\mathbb{C}^+)$ and $\A^{-1}(\lambda)$ is outer from
$B(\mathbb{C}^+)$. Thus,
\begin{equation}
\A^{-1}(\lambda)=\exp\left[i\gamma-\frac{1}{\pi
i}\int\limits_{-\infty}^\infty \frac{(1+s\lambda)\ln
|\A(s)|}{(s-\lambda)(1+s^2)}ds\right], \gamma\in [0,2\pi),
\label{represent-1}
\end{equation}
For $\A(r,\lambda)$, we have
\begin{equation}
\A(r,\lambda)=1+\int\limits_{0}^{r} \A(t,\lambda) \left[
\exp(-i\lambda t) \overline{A(t)}
 \int\limits_t^r
\exp(i\lambda s) A(s)ds\right] ds \label{int-eq-a}
\end{equation}
Iterating this identity, estimating the lower order terms, and
taking $r\to\infty$, one has
\begin{eqnarray*}
\A(iy)=1+\int\limits_{0}^\infty \left[ \exp(yt) \overline{A(t)}
 \int\limits_t^\infty
\exp(-ys) A(s)ds\right]+O(y^{-2}), y\to +\infty
\end{eqnarray*}
Clearly,
\[
\int\limits_{0}^{\infty} \left[ \exp(yt) \overline{A(t)}
 \int\limits_t^\infty
\exp(-ys) A(s)ds\right]=\int\limits_{-\infty}^\infty
\frac{|\hat{A}(\omega)|^2}{y-i\omega} d\omega
\]
where $\hat{A}(\omega)$ is the Fourier transform of the function
$A(t)\cdot \chi_{\mathbb{R}^+}(t)$ and
\mbox{$|\hat{A}(\omega)|^2\in L^1(\mathbb{R})$}. Thus,
\begin{equation}
\A(iy)=1+y^{-1}\int\limits_0^\infty |A(s)|^2 ds +\bar{o}(y^{-1}),
y \to +\infty \label{equa1-1}
\end{equation}
From (\ref{represent-1}), we have
\begin{equation}
|\A(iy)|=\exp\left[ \frac{y}{\pi}\int\limits_{-\infty}^\infty
\frac{\ln |\A(\lambda)|}{\lambda^2+y^2}\, d\lambda\right]
\label{equa1-2}
\end{equation}
Since $|\A(\lambda)|\geq 1$ for a.e. $\lambda\in \mathbb{R}$,
relations (\ref{equa1-1}) and (\ref{equa1-2}) imply $\ln
|\A(\lambda)|\in L^1(\mathbb{R})$ and  (\ref{tracing}).
\end{proof}
\begin{corollary}
For a.e. $\lambda\in \mathbb{R}$, we have
\begin{equation} \label{energy-l2}
|\A(\lambda)|^2=1+|\B(\lambda)|^2
\end{equation}
\end{corollary}
\begin{proof}
The equation is equivalent to
\[
\frac{|\Pi_\alpha(\lambda)+\widehat\Pi_\beta(\lambda)|^2}{4}
=\frac{|\Pi_\alpha(\lambda)-\widehat\Pi_\beta(\lambda)|^2}{4}+1
\]
or
\[
\Re F(\lambda)=\frac{1}{|\Pi_\alpha(\lambda)|^2}=2\pi
\sigma'(\lambda)
\]
and the last identity is elementary and follows, e.g., from the
integral representations for both functions $F(\lambda)$ and
$\Pi_\alpha(\lambda)$.
\end{proof}

\begin{corollary}
If $A(r)\in L^2(\mathbb{R}^+)$, then
\begin{equation}
2\pi\int\limits_0^\infty |A(r)|^2 dr=-\int\limits_{-\infty}^\infty
\ln[1-|\f(\lambda)|^2]d\lambda \label{tracing-2}
\end{equation}
\label{corol-3}
\end{corollary}
\begin{proof} From (\ref{energy-l2}), we have
$1-|\f(\lambda)|^2=|\A(\lambda)|^{-2}$ for a.e. $\lambda\in
\mathbb{R}$. Now, (\ref{tracing-2}) follows from (\ref{tracing}).
\end{proof}

Notice that (\ref{tracing-2}) implies
\begin{equation}\label{fh2}\f(\lambda)\in L^2(\mathbb{R})\end{equation}
 So, $\f(\lambda)\in H^2(\mathbb{C}^+)\cap
B(\mathbb{C}^+)$ and (compare with (\ref{schur-alg})):
\begin{equation}
\f(\lambda)=\int\limits_0^\infty C(x)\exp(i\lambda x)dx,\lambda\in
\mathbb{C}^+ \label{hardy-2}
\end{equation}

\begin{corollary}
If $A(r)\in L^2(\mathbb{R}^+)$, then $\ln[2\pi
\sigma'(\lambda)]\in
L^1(\mathbb{R})+L^2(\mathbb{R})$.\label{cor-sum2}
\end{corollary}
\begin{proof} We have for a.e. $\lambda\in \mathbb{R}$
\[
[2\pi
\sigma'(\lambda)]^{-1}=|\Pi_\alpha(\lambda)|^2=|\A(\lambda)+\B(\lambda)|^2=
|\A(\lambda)|^2\cdot |1+\f(\lambda)|^2
\]
and therefore
\[-\ln[2\pi \sigma'(\lambda)]=2\ln
|\A(\lambda)|+2\ln|1+\f(\lambda)|\] By (\ref{tracing}), the first
term is from $L^1(\mathbb{R})$. As about the second term,
\[
\ln_+|1+\f(\lambda)|\leq \ln [1+|\f(\lambda)|]\leq
|\f(\lambda)|\in L^2(\mathbb{R}) \] by (\ref{fh2}). For the
negative part of the logarithm, we use an elementary estimate
$|1+\f|\geq 1-|\f|$ which yields (check the definition of $\ln^-$)
\[
\left|\,\ln^- |1+\f(\lambda)|\right|\leq \ln
(1-|\f(\lambda)|)^{-1}=-\ln(1-|\f(\lambda)|^2)+\ln(1+|\f(\lambda)|)
\] for $|1+\f(\lambda)|<1$. The second term is again in $L^2(\mathbb{R})$.
The first one is in $L^1(\mathbb{R})$ due to Corollary
\ref{corol-3}.\end{proof}

\begin{corollary}
For the constants $\alpha$  and $\gamma$ from the multiplicative
representations of $\Pi_\alpha(\lambda)$ (formula {\rm
(\ref{function-pi1})})  and $\A(\lambda)$ (formula {\rm
(\ref{represent-1})}) we have
\[
\alpha=\frac{1}{2\pi} \int\limits_{-\infty}^\infty \frac{s\ln
[2\pi \sigma'(s)]}{1+s^2}ds,\,
\gamma=\frac{1}{\pi}\int\limits_{-\infty}^\infty \frac{s \ln
|\A(s)|}{1+s^2}ds
\]
\end{corollary}
\begin{proof} From (\ref{represent-1}) and (\ref{equa1-1}), we
have
\[
\gamma=\frac{1}{\pi} \lim_{y\to\infty}
\int\limits_{-\infty}^\infty \frac{s(y^2-1)\ln
|\A(s)|}{(s^2+y^2)(1+s^2)}ds=\frac{1}{\pi}\int\limits_{-\infty}^\infty
\frac{s \ln |\A(s)|}{1+s^2}ds
\]
because $\ln|\A(s)|\in L^1(\mathbb{R})$.

Take $r\to\infty$ in (\ref{int-eq-l3}). One has
\[
\Pi_\alpha(\lambda)=1+O([\Im\lambda]^{1/2}])
\]
as $\Im\lambda\to +\infty$. Then, from (\ref{function-pi1}), we
have
\[
\alpha=\lim_{y\to\infty} \frac{1}{2\pi} \int
\frac{s(1-y^2)+iy(1+s^2)}{(s^2+y^2)(1+s^2)}\ln[2\pi\sigma'(s)]ds
\]
Using Corollary \ref{cor-sum2}, we get the needed formula for
$\alpha$.\end{proof}

Clearly, the Corollary implies the following integral
representations
\begin{equation}\label{a1}
\A(\lambda)=\exp\left[\frac{1}{\pi i}\int\limits_{-\infty}^\infty
\frac{\ln |\A(s)|}{s-\lambda}ds\right],
\end{equation}
\begin{equation}
\Pi_\alpha(\lambda)=  \exp \left[ -\frac{1}{2\pi i}
\int\limits_{-\infty}^{\infty} \frac{ \ln [2\pi
\sigma'(s)]}{s-\lambda} ds\right] \label{pisub}
\end{equation}
As we know from the discussion of the general Szeg\H{o} case, the
function $(\lambda+i)^{-1}\Pi^{-1}_\alpha(\lambda)\in
H^2(\mathbb{C}^+)$. For our situation, much more is true
\begin{theorem}\label{trace-formula1}
If $A(r)\in L^2(\mathbb{R}^+)$, then
\begin{equation}\label{miracle0}
\Pi^{-1}_\alpha(\lambda)=1+\int\limits_0^\infty
\gamma(x)e^{ix\lambda}dx=1+\hat{\gamma}(\lambda),
\hat{\gamma}(\lambda)\in H^2(\mathbb{C}^+)
\end{equation}
where
\begin{equation}\label{miracle}
\sigma(E)+\int\limits_0^\infty |\gamma(x)|^2
dx=\int\limits_0^\infty |A(r)|^2dr
\end{equation}
where $E$-- support of $d\sigma_s$, the singular component of
$d\sigma$.
\end{theorem}

\begin{proof}
From Lemma \ref{asymp-l2}, as $r\to\infty$:
\[
\frac{P_*(r,\lambda)}{\lambda+i}\to \frac{\Pi_\alpha(\lambda)\cdot
\chi_{E^c}(\lambda)}{\lambda+i}
\]
in $L^2(\mathbb{R},d\sigma)$, where  $E^c$-- the complement of
$E$. On the other hand, for any $r>0, \lambda\in \mathbb{R}$, we
have (by (\ref{krein2})):
\begin{equation}\label{integral-of-p}
P_*(r,\lambda)=1-\int\limits_0^r A(t)P(t,\lambda)dt
\end{equation}
That implies
\[
\frac{\Pi_\alpha(\lambda)\cdot
\chi_{E^c}(\lambda)}{\lambda+i}=\frac{1}{\lambda+i}-\frac{\tilde{A}(\lambda)}{\lambda+i}
\]
where the generalized Fourier transform
\[
\tilde{A}(\lambda)=\int\limits_0^\infty A(r)P(r,\lambda)dr\in
L^2(\mathbb{R},d\sigma)
\]
by Theorem \ref{theorem2s2}. Thus,
\begin{equation}
\tilde{A}(\lambda)=1-\Pi_\alpha(\lambda)\cdot \chi_{E^c}(\lambda)
\end{equation}
and
\[
\int\limits_0^\infty |A(r)|^2
dr=\sigma\{E\}+\frac{1}{2\pi}\int\limits_{-\infty}^\infty
\left|1-\frac{1}{\Pi_\alpha(\lambda)}\right|^2d\lambda
\]
Now, we have
\[
\int\limits_{-\infty}^\infty
\left|1-\frac{1}{\Pi_\alpha(\lambda)}\right|^2d\lambda<\infty,
\frac{1}{\lambda+i}\cdot\left(1-\frac{1}{\Pi_\alpha(\lambda)}\right)\in
H^2(\mathbb{C}^+)
\]
Therefore, by elementary Lemma \ref{auxil-2} in Appendix,
$1-\Pi^{-1}_\alpha(\lambda)\in H^2(\mathbb{C}^+)$. The
Paley-Wiener Theorem now implies (\ref{miracle0}) and
(\ref{miracle}).
\end{proof}
An interesting corollary from this result is that the total
variation over the whole line of the singular part of the measure
$d\sigma$ is finite.

After proving the representation for $\Pi_\alpha(\lambda)$, the
following result is quite natural. By Levy-Wiener Theorem,
\[
\frac{1}{P_*(r,\lambda)}=1+\int\limits_0^\infty
\gamma_r(x)e^{i\lambda x} dx
\]
where $\gamma_r(x)\in L^1(\mathbb{R}^+)\cap L^2(\mathbb{R}^+)$.
\begin{lemma}
Assume $A(r)\in L^2(\mathbb{R}^+)$ and $E=\emptyset$. Then,
$\gamma_r(x)\to \gamma(x)$ in $L^2(\mathbb{R}^+)$.
\end{lemma}
\begin{proof}
From (\ref{integral-of-p}) and Corollary \ref{scale-property}, we
get
\begin{equation}
\int\limits_0^r |A(x)|^2dx=\int\limits_{-\infty}^\infty
|1-P_*(r,\lambda)|^2 d\sigma(\lambda)=\frac{1}{2\pi}
\int\limits_{-\infty}^\infty
\left|1-{P_*(r,\lambda)}\right|^2\frac{d\lambda}{|P_*(r,\lambda))|^2}
\end{equation}

\[
=\frac{1}{2\pi} \int\limits_{-\infty}^\infty
\left|1-\frac{1}{P_*(r,\lambda)}\right|^2d\lambda=\int\limits_0^\infty
|\gamma_r(x)|^2dx
\]
Then, since $P_*(r,\lambda)\to \Pi_\alpha(\lambda)$ for
$\lambda\in \mathbb{C}^+$, we also have that $\gamma_r(x)\to
\gamma(x)$ weakly in $L^2(\mathbb{R}^+)$. But since $E=\emptyset$,
we also get $\|\gamma_r(x)\|_2\to\|\gamma(x)\|_2$ from
(\ref{miracle}). Therefore, $\gamma_r(x)\to\gamma(x)$ in
$L^2(\mathbb{R}^+)$.
\end{proof}
If the singular part of the measure $d\sigma$ is not trivial, then
we have only the bound: $\|\gamma_r\|_2\leq \|A\|_2$.

Now, let us characterize the class of Schur coefficients
corresponding to $A(r)\in L^2(\mathbb{R}^+)$. We start with
\begin{theorem}

 For any $A(r)\in L^2[0,R]$ and any
$R>0$, we have

\begin{equation}
\int\limits_{-\infty}^\infty \ln|\A(R,\lambda)|d\lambda=\pi
\int\limits_0^R |A(r)|^2 dr \label{tracing1}
\end{equation}
and
\begin{equation}
 2\pi\int\limits_0^R |A(r)|^2 dr=-\int\limits_{-\infty}^\infty
\ln[1-|\B(R,\lambda)\A^{-1}(R,\lambda)|^2]d\lambda
\label{trunc-trace}
\end{equation}
\label{trunc-theorem}
\end{theorem}
\begin{proof}
The proof repeats those of Theorem \ref{main-lsz} and Corollary
\ref{corol-3}.
\end{proof}

Let us introduce a certain subclass of $B(\mathbb{C}^+)$. Consider
a function $\f(\lambda)\in B(\mathbb{C}^+)$ such that for its
boundary value:
\begin{equation}
\int\limits_{-\infty}^\infty \ln (1-|\f(\lambda)|^2)
d\lambda>-\infty \label{l-2}
\end{equation}
Then, a simple estimate $\ln(1-|\f(\lambda)\|^2)\leq
-|\f(\lambda)|^2$ yields that $\f(\lambda)\in L^2(\mathbb{R})$.
Thus $\f\in H^2(\mathbb{C}^+)$ and
\[
\f(\lambda)=\int\limits_0^\infty C(x)e^{i\lambda x}dx
\]
with $C(x)\in L^2(\mathbb{R}^+)$.
\begin{definition}
We say that $\f(\lambda)\in S_0(\mathbb{C}^+)$ if $\f(\lambda)\in
B(\mathbb{C}^+)$ and (\ref{l-2}) holds.
\end{definition}

The set $S_0(\mathbb{C}^+)$ can be regarded as the  metric space
\cite{Sylv} with the distance given by the formula
\begin{equation}
\rho_s^2(\f,\g)=-\int\limits_{-\infty}^\infty \ln\left[ 1-
\rho^2(\f(\lambda),\g(\lambda))\right]d\lambda
\end{equation}
where pseudohyperbolic distance $\rho(\cdot,\cdot)$ is defined by
{\rm (\ref{pseudohyp})}. It turns out that the resulting metric
space is complete (see Lemma 1.5, Theorem 1.6, Corollary 1.9 in
\cite{Sylv}. The geometry of this space is studied in the same
paper).

Let us write $\rho_s(\f)=\rho_s(0,\f)$ for short-hand.
\begin{lemma}
If in Krein system $\f_r(\lambda)\in S_0(\mathbb{C}^+)$ for some
$r$, then $\f(\lambda)\in S_0(\mathbb{C}^+)$ and
\begin{equation}
\rho_s^2(\f (\lambda))=\rho_s^2(\B (r,\lambda)\A^{-1}
(r,\lambda))+\rho_s^2(\f_r (\lambda))\label{layer-strip1}
\end{equation}
In particular, this is true for any $A(r)\in L^2(\mathbb{R}^+)$.
\label{layer-strip}
\end{lemma}
\begin{proof}
Let us prove this lemma using certain ``orthogonality" argument.
For $\lambda\in \mathbb{R}$, we have
\[
\A_*(r,\lambda)=e^{i\lambda
r}\overline{\A(r,\lambda)},\B_*(r,\lambda)=e^{i\lambda r}
\overline{\B(r,\lambda)}
\]
Therefore, (\ref{e18s3}) yields
\[
|\f (\lambda)|^2=\left| \frac
{\B(r,\lambda)\bar{\A}^{-1}(r,\lambda)+\exp(i\lambda r) \f
_r(\lambda)}{1+ \bar{\B} (r,\lambda)\A^{-1}(r,\lambda)
\exp(i\lambda r) \f_r(\lambda)} \right|^2, \lambda\in \mathbb{R}
\]
The formula
\[
\ln \left(
1-\left|\frac{z+w}{1+\bar{z}w}\right|^2\right)=\ln(1-|z|^2)+\ln(1-|w|^2)-2
\ln |1+\bar{z}w|
\]
and Theorem \ref{trunc-theorem} give
\[
\begin{array}{ll}
\rho_s^2(\f (\lambda))=\rho_s^2(\B (r,\lambda)\A^{-1}
(r,\lambda))+\rho_s^2(\f_r (\lambda))\\
\displaystyle
 +2 \int\limits_{-\infty}^\infty \ln|1+\bar{\B}
(r,\lambda)\A^{-1}(r,\lambda) \exp(i\lambda r)
\f_r(\lambda)|d\lambda\\
=\displaystyle \rho_s^2(\B (r,\lambda)\A^{-1}
(r,\lambda))+\rho_s^2(\f_r (\lambda))+2
\int\limits_{-\infty}^\infty \ln|1+\B_*
(r,\lambda)\A^{-1}(r,\lambda)  \f_r(\lambda)|d\lambda
\end{array}
\]
The last integral is zero because
\[
\B_*(r,\lambda)\A^{-1}(r,\lambda), \A^{-1}(r,\lambda),
\f_r(\lambda)\in B(\mathbb{C}^+),
\]
\[
\B_*(r,\lambda)=\bar{o}(1)\quad {\rm as}\quad \lambda\in
\overline{\mathbb{C}^+}, \,|\lambda|\to \infty,\quad
\B_*(r,iy)=\bar{o}(y^{-1/2}) \quad { \rm as}\quad y\to+\infty
\]
\[
\f_r(\lambda)\in H^2(\mathbb{C}^+), \quad {\rm so} \quad
\f_r(iy)=\bar{o}(y^{-1/2}) \]functions $\B_*(r,\lambda),
\f_r(\lambda)\in L^2(\mathbb{R})$ and the mean-value formula
(Lemma \ref{mean-value}, Appendix) is applicable.
  \end{proof}

The result above is sometimes called ``the layer stripping". That
is because $\rho_s^2(\f)$ is equal to the sum of the terms that
correspond to different intervals of the coordinate $r$. The
formula (\ref{tracing-2}) is called the non-linear Plancherel
Theorem. The both results are well-known in the theory of
orthogonal polynomials.

 It is an important observation that any function from
 $S_0(\mathbb{C}^+)$ gives rise to a certain Krein system.
Indeed, we can show $S_0(\mathbb{C}^+)\subset S(\mathbb{C}^+)$,
where the class $S(\mathbb{C}^+)$ was introduced in the
Definition~\ref{classS}, and the Theorem \ref{s0krein} applies. To
prove this inclusion, notice that $C(x)$ generates operator
$\cal{C}_R$ (see (\ref{volterra-s})) for any $R>0$. By a standard
argument, $\|\cal{C}_R\|_{L^2[0,R]}<1$ for all $0<R<\infty$. Then,
for any $R>0$,
\[
f(\lambda)=\int\limits_0^R C(x)e^{i\lambda x} dx+e^{i\lambda
R}\Phi_R(\lambda), \lambda\in \mathbb{C}^+
\]
with
\[
\Phi_R(\lambda)=\int\limits_0^\infty C(x+R)e^{i\lambda x}dx
\]
Since $\Phi_R(\lambda)\in H^2(\mathbb{C}^+)$ by Paley-Wiener
Theorem and both
\[
f(\lambda),\int\limits_0^R C(x)e^{i\lambda x} dx\in
H^\infty(\mathbb{C}^+)
\]
we get $\Phi_R(\lambda)\in L^\infty(\mathbb{R})$. Therefore,
$\Phi_R(\lambda)\in H^\infty(\mathbb{C}^+)$ and $f(\lambda)\in
S(\mathbb{C}^+)$.

In \cite{Sylv}, the analog of the following Theorem was
established.
\begin{theorem}\label{test-l2}
For Krein system, $A(r)\in L^2(\mathbb{R}^+)$ iff the
corresponding Schur function $f(\lambda)\in S_0(\mathbb{C}^+)$.
\end{theorem}
\begin{proof} If $A(r)\in L^2(\mathbb{R}^+)$, then $\f(\lambda)\in S_0(\mathbb{C}^+)$ follows from the
Corollary~\ref{corol-3}. Assume now that $\f(\lambda)\in
S_0(\mathbb{C}^+)$. The observation made right before the Theorem
says there is the corresponding Krein system with $A(r)\in
L^2_{\rm loc}[0,\infty)$. Let us show that actually $A(r)\in
L^2(\mathbb{R}^+)$. Indeed, fix any $R>0$. Then,
\begin{equation}\label{backward}
\f_R(\lambda)=\frac{\f(\lambda)
\A(R,\lambda)-\B(R,\lambda)}{\A_*(R,\lambda)-\f(\lambda)
\B_*(R,\lambda)}
\end{equation}
and
\[
|\f_R(\lambda)|=\left|
\frac{\f(\lambda)-\B(R,\lambda)\A^{-1}(R,\lambda)}{1-\bar{\f}(\lambda)
\B(R,\lambda)\A^{-1}(R,\lambda)}\right|, \lambda\in \mathbb{R}
\]
Therefore,
\[
\rho_s(\f_R(\lambda))=\rho_s(\f(\lambda),
\B(R,\lambda)\A^{-1}(R,\lambda))<\infty
\]
and $\f_R(\lambda)\in S_0(\mathbb{C}^+)$. Now, the Lemma
\ref{layer-strip} is applicable together with
Theorem~\ref{trunc-theorem}:
\begin{equation}
2\pi \int\limits_0^R |A(r)|^2dr =\rho_s^2(\B (R,\lambda)\A^{-1}
(R,\lambda))\leq \rho_s^2(\f)<C
\end{equation}
uniformly in $R$ which means $A(r)\in L^2(\mathbb{R}^+)$.
 \end{proof}
This result has an interesting corollary.
\begin{corollary}\label{convexity}
\begin{itemize}
\item[(i)]{If $\f(\lambda)$ is the Schur function corresponding to
$A_{\f}(r)\in L^2(\mathbb{R}^+)$ and $g(\lambda)\in
B(\mathbb{C}^+)$ then $g\f $ generates the Krein system with
$A_{g\f }\in L^2(\mathbb{R}^+)$ and $\|A_{g\f }\|_2\leq
\|A_{\f}\|_2$.}

\item[(ii)]{The  set of measures $d\sigma$ that correspond to
$A_{d\sigma}(r)\in L^2(\mathbb{R}^+)$ is convex.}
\end{itemize}
\end{corollary}
\begin{proof}
The first statement is obvious due to Theorem \ref{test-l2}. The
second one follows from the following calculations.
\[
f=\frac{1-F}{1+F}, 1-|f|^2=\frac{4\Re F}{|1+F|^2}
\]
If $F_j$ and $\f_j$ correspond to $d\sigma_j, j=0,1$ and $F_{t}$
and $\f_{t}$-- to \mbox{$td\sigma_1+(1-t)d\sigma_0, t\in [0,1]$},
then
\[
1-|f_{t}|^2= \frac{4(t\Re F_1+(1-t)\Re F_0)}{|1+t F_1+(1-t)
F_0|^2}=(1-|f_1|^2)\frac{t|1+F_1|^2}{|1+t F_1+(1-t) F_0|^2}
\]
\[
+(1-|f_0|^2)\frac{(1-t)|1+F_0|^2}{|1+t F_1+(1-t) F_0|^2} \geq
\min\limits_{j=0,1} (1-|f_j|^2)
\]
due to convexity of $|z|^2, z\in \mathbb{C}$. Therefore
\[
\int\limits_{-\infty}^\infty \ln (1-|f_{t}(\lambda)|^2)d\lambda
\geq \int\limits_{-\infty}^\infty \ln
(1-|f_0(\lambda)|^2)d\lambda+\int\limits_{-\infty}^\infty \ln
(1-|f_1(\lambda)|^2)d\lambda
\]
and $\f_t(\lambda)\in S_0(\mathbb{C}^+)$.
\end{proof}
{\bf Remark.} Clearly, the last estimate is not optimal. As about
the first statement, notice that multiplication $\f_A$ by any
inner function does not change $L^2$ norm of the coefficient
$A(r)$.

The Theorem \ref{trace-formula1} leads to the following natural
question: what can be the singular component of $d\sigma$ if
$A(r)\in L^2(\mathbb{R}^+)$? The answer is given by the following
result which can be regarded as another criteria for $A(r)\in
L^2(\mathbb{R}^+)$. In particular, it says that the singular
component can be any singular measure finite over $\mathbb{R}^+$.

\begin{theorem}\label{l2-test2}
Let $d\sigma$ be a nonnegative measure on $\mathbb{R}$ with
decomposition
 $d\sigma=d\sigma_s+\sigma'(\lambda)d\lambda$, where $\sigma_s(\mathbb{R})<\infty$
  and $\ln [2\pi\sigma'(\lambda)]\in L^1(\mathbb{R})+ L^2(\mathbb{R})$.
Assume also that
\begin{equation}
\exp \left[ \frac{1}{2\pi i} \int\limits_{-\infty}^{\infty} \frac{
\ln [2\pi \sigma'(t)]}{t-\lambda} dt\right]-1=\int\limits_0^\infty
\gamma(x)\exp(i\lambda x)dx=\hat{\gamma}(\lambda)\in
H^2(\mathbb{C}^+)
\end{equation}
Then $d\sigma$ generates the Krein system with $A(r)\in
L^2(\mathbb{R}^+)$ and {\rm (\ref{miracle})} holds true.
\end{theorem}
\begin{proof}
We have
\begin{equation}\label{derivative-l2}
2\pi
\sigma'(\lambda)=1+\hat{\gamma}+\bar{\hat{\gamma}}+|\hat{\gamma}|^2
\end{equation}
Take $H(x)$ Hermitian such that
\begin{equation}
H(x)=(2\pi)^{-1}\left[\overline{\gamma(x)}+\int\limits_0^\infty
\overline{\gamma(x+u)}{\gamma}(u)du\right]+\int_{\mathbb{R}}
\exp(i\lambda x)d\sigma_s(\lambda)\label{sst-refl2}
\end{equation}
for $x>0$ and define $\beta$ by (\ref{betaa}). Then, it is not
difficult to check that the formula (\ref{e2s2}) holds.
 Moreover, $H(x)\in L^2_{\rm loc}(\mathbb{R})$ is an accelerant and
generates some $A(r)\in L^2_{\rm loc}(\mathbb{R})$. That follows
from (\ref{e1s2n1}), (\ref{derivative-l2}), and the uniqueness
theorem for analytic functions.

Let us show that $A(r)\in L^2(\mathbb{R}^+)$. Indeed, take
$d\sigma_n$ such that it has only finite number of jumps and
\[
\int_{\mathbb{R}} \exp(i\lambda x)d\sigma_n(\lambda)\to
\int_{\mathbb{R}} \exp(i\lambda x)d\sigma_s(\lambda)
\]
uniformly in $x\in [0,R]$ with any fixed $R>0$.

 Also, take purely a.c. $d\mu_\epsilon$ so that
\[
2\pi \mu_\epsilon'(\lambda)=|1+\hat{\gamma}_\epsilon(\lambda)|^2,
\hat\gamma_\epsilon(\lambda)=\hat\gamma(\lambda+i\epsilon),
\epsilon>0,
\]
Take $d\sigma_{n,\epsilon}=d\sigma_n+d\mu_\epsilon$. Then, the
corresponding accelerants $H_{n,\epsilon}(x)\to H(x)$ and
$A_{n,\epsilon}(r)\to A(r)$ in $L^2[0,R]$ for any fixed $R>0$ as
long as $n\to\infty, \epsilon\to 0$.

Let us show that Theorem \ref{test-l2} yields $A_{n,\epsilon}\in
L^2(\mathbb{R}^+)$. Indeed, fix $n$ and $\epsilon$. Then, we can
write $d\sigma_{n,\epsilon}$ as a convex combination
\[
d\sigma_{n,\epsilon}=t\left[\frac{d\sigma_n}{t}+\frac{d\lambda}{2\pi}\right]+(1-t)\left[
\frac{\mu'_\epsilon(\lambda)}{1-t}-\frac{t}{(1-t)2\pi}\right]d\lambda
\]
and $t\in (0,1)$. If we can show that both measures in this convex
combination generate square summable $A(r)$, then the second claim
of the Corollary \ref{convexity} finishes the argument. We will
apply Theorem \ref{test-l2} to
\[
\frac{d\sigma_n}{t}+\frac{d\lambda}{2\pi}\]
and
\begin{equation}
\nu(\lambda)d\lambda,
\,\nu(\lambda)=\frac{\mu'_\epsilon(\lambda)}{1-t}-\frac{t}{(1-t)2\pi}
\label{second-measure}
\end{equation}
The first measure gives rise to
\[
F(\lambda)=1-2it^{-1}\int_{\mathbb{R}}\frac{d\sigma_n(t)}{t-\lambda}
\]
This representation follows from the formula (\ref{e1s3-1}) and
$F(iy)\to 1$ as $y\to +\infty$. Since
\begin{equation}
1-|\f|^2=\frac{4\Re F}{|1+F|^2} \label{fF}
 \end{equation}
 and $d\sigma_n$ has
only finite number of jumps
\[
1-|\f|^2=1+O(|\lambda|^{-2}) \] as $|\lambda|\to\infty$. The local
singularities of $\ln (1-|\f|^2)$ are integrable and thus we have
$\ln (1-|\f|^2)\in L^1(\mathbb{R})$ and the Theorem \ref{test-l2}
can be applied.

Let us show that Theorem \ref{test-l2} can also be applied to the
measure (\ref{second-measure}) as long as $t$ is chosen properly.
Indeed, we can always take $t$ so small that the density
$\nu(\lambda)$ of this measure is strictly positive. We have
\[
F(\lambda)=1+i\int_{\mathbb{R}}
\frac{\hat\gamma_\epsilon(s)+\overline{\hat\gamma_\epsilon(s)}+
\hat\gamma_\epsilon(s)\overline{\hat\gamma_\epsilon(s)}}{\pi(1-t)(\lambda-s)}ds
\]
Since $\hat\gamma_\epsilon$ is infinitely smooth, $F$ is
continuous up to the boundary and $\ln(1-|\f|^2)$ is locally
integrable by (\ref{fF}). Therefore, we are left with showing that
$|\f(\lambda)|\to 0$ if $\lambda\in \mathbb{R}, \lambda\to\infty$
and that $|\f(\lambda)|\in L^2(\mathbb{R})$. Since
$\hat\gamma_\epsilon(s)\in L^2(\mathbb{R})$ and it decays at
infinity, the simple properties of Hilbert transform imply
$F(\lambda)-1\in L^2(\mathbb{R})$ and $F(\lambda)-1\to 0$ as
$\lambda\to\infty$. Since
\[
\f=\frac{1-F}{1+F}
\]
we have $\ln(1-|\f|^2)\in L^1(\mathbb{R})$ and Theorem
\ref{test-l2} applies.

The Krein system with coefficient $A_{n,\epsilon}$ has
$\Pi_{\epsilon}(\lambda)$--function with inverse
\[
\Pi_{\epsilon}^{-1}(\lambda)=1+\hat\gamma_{\epsilon}(\lambda)
\]
 Then, notice that (\ref{miracle}) gives
a bound $\|A_{n,\epsilon}\|_2<C$ uniformly in $n$ and $\epsilon$
and $A(r)\in L^2(\mathbb{R}^+)$ because $A_{n,\epsilon}(r)\to
A(r)$ in $L^2[0,R]$ with any $R>0$. We get (\ref{miracle}) and the
function $\gamma(x)$ coincides with the one introduced in
Theorem~\ref{trace-formula1}.
\end{proof}

 Later on, we will need
the following bound which sharpens the Lemma  \ref{asymp-l2}.
\begin{lemma}
If $A(r)\in L^2(\mathbb{R}^+)$, then
\begin{equation}
\int\limits_{-\infty }^{\infty }\frac{1}{\lambda ^{2}+1}\left|
\frac{P_{\ast }(r,\lambda )}{\Pi_\alpha (\lambda)}-1\right|
^{2}d\lambda + \int\limits_{-\infty }^{\infty }\frac{1}{\lambda
^{2}+1}\left| P_{\ast }(r,\lambda )\right| ^{2}\mathit{d\sigma
}_{s} (\lambda )<C\int\limits_r^\infty |A(s)|^2ds
\label{decay-improved}
\end{equation}
\end{lemma}
\begin{proof}
First, we notice that
\[
\Pi_\alpha(\lambda)-P_*(r,\lambda)=-\int\limits_r^\infty
A(s)P(s,\lambda)ds, \lambda\in \mathbb{C}^+
\]
Therefore,
\[
|\Pi_\alpha(\lambda)-P_*(r,\lambda)|^2 \leq
\left(\int\limits_r^\infty |A(s)|^2ds\right)\cdot\left(
\int\limits_r^\infty |P(s,\lambda)|^2ds\right)
\]
\[
\leq \frac{1}{2\Im\lambda}\left(\int\limits_r^\infty
|A(s)|^2ds\right)\cdot\left[
|\Pi_\alpha(\lambda)|^2-|P_*(r,\lambda)|^2\right]
\]
by Cauchy-Schwarz and analog of (\ref{e101s2}). For the last
expression, we use
\[
|\Pi_\alpha|^2-|P_*|^2\leq |\Pi_\alpha-P_*|\cdot
\left(|\Pi_\alpha|+|P_*|\right)
\]
and then
\begin{equation}
|\Pi_\alpha(\lambda)-P_*(r,\lambda)|\leq C(\lambda)
\int\limits_r^\infty |A(s)|^2ds, \Im \lambda>0 \label{addl-decay}
\end{equation}
The last estimate also implies
\[
\int\limits_r^\infty
|P(s,\lambda)|^2ds=(2\Im\lambda)^{-1}\left(|\Pi_\alpha(\lambda)|^2-|P_*(r,\lambda)|^2\right)
\]
\begin{equation}
< C(\lambda)
|\Pi_\alpha(\lambda)-P_*(r,\lambda)|<C(\lambda)\int\limits_r^\infty
|A(s)|^2ds\label{residue}
\end{equation}
  For $P(r,\lambda)$, we have
\[
P^2(r,\lambda)=-2\int\limits_r^\infty
P'(s,\lambda)P(s,\lambda)ds=-2\int\limits_r^\infty \left[i\lambda
P^2(s,\lambda)-\overline{A(s)}P(s,\lambda)P_*(s,\lambda)\right]ds,
\lambda\in \mathbb{C}^+
\]
Therefore, for $\lambda\in \mathbb{C}^+$
\begin{equation}
|P(r,\lambda)|^2\leq C(\lambda)\int\limits_r^\infty
|P(s,\lambda)|^2ds+C(\lambda)\left[\int\limits_r^\infty
|A(s)|^2ds\cdot \int\limits_r^\infty
|P(s,\lambda)|^2ds\right]^{1/2}
\end{equation}
\begin{equation}
 \leq
C(\lambda)\int\limits_r^\infty |A(s)|^2ds \label{est-l2-1}
\end{equation}
where the last inequality follows from (\ref{residue}).

Now, let us improve estimates from Lemma \ref{auxil-1}. From
(\ref{auxil-decay}), we have
\[
|P_*(r,i)|^2\int\limits_{-\infty}^\infty
\frac{|P_*(r,\lambda)|^2}{\lambda^2+1}d\sigma(\lambda) =
\frac{|P_*(r,i)|^2-|P(r,i)|^2}{2}+|P(r,i)|^2\int\limits_{-\infty}^\infty
\frac{|P(r,\lambda)|^2}{\lambda^2+1}d\sigma(\lambda)
\]
\begin{equation}
+2\Re \left[i \int\limits_{-\infty}^\infty
\frac{P(r,\lambda)\overline{P(r,i)}}{\lambda+i}
\overline{K_r(i,\lambda)}d\sigma(\lambda)\right]\label{est-l2-2}
\end{equation}
For real $\lambda$, $|P(r,\lambda)|=|P_*(r,\lambda)|$. Therefore,
\begin{equation}
\int\limits_{-\infty}^\infty
\frac{|P_*(r,\lambda)|^2}{\lambda^2+1}d\sigma(\lambda) = \frac 12
+2\Bigl[|P_*(r,i)|^2-|P(r,i)|^2\Bigr]^{-1} \Re \left[i
\overline{P(r,i)} \int\limits_{-\infty}^\infty
\frac{P(r,\lambda)}{\lambda+i}
\overline{K_r(i,\lambda)}d\sigma(\lambda)\right]\label{reduction-ser}
\end{equation}
Using the representation \[
 K_r(i,\lambda)=K_\infty
(i,\lambda)-\displaystyle \int\limits_r^\infty
P(s,\lambda)\overline{P(s,i)}ds\] and the property of reproducing
kernel, we get
\[
\int\limits_{-\infty}^\infty \frac{P(r,\lambda)}{\lambda+i}
\overline{K_r(i,\lambda)}d\sigma(\lambda)=\frac{P(r,i)}{2i}-\int\limits_{-\infty}^\infty
\frac{P(r,\lambda)}{\lambda+i} \int\limits_r^\infty
\overline{P(s,\lambda)}{P(s,i)}dsd\sigma(\lambda)
\]
The last integral can be bounded by Cauchy-Schwarz and Theorem
\ref{theorem2s2}  as follows
\[
\left|\int\limits_{-\infty}^\infty \frac{P(r,\lambda)}{\lambda+i}
\int\limits_r^\infty
\overline{P(s,\lambda)}{P(s,i)}dsd\sigma(\lambda)\right|
<C(\lambda)\left[\int\limits_r^\infty |P(s,i)|^2ds\right]^{1/2}
<C(\lambda)\left[ \int\limits_r^\infty |A(s)|^2ds \right]^{1/2}
\]
where we used (\ref{residue}) for last inequality. Estimate
(\ref{est-l2-1}) and (\ref{reduction-ser}) yield
\begin{equation}
\int\limits_{-\infty}^\infty
\frac{|P_*(r,\lambda)|^2}{\lambda^2+1}d\sigma(\lambda)=\frac 12
+O\left[\int\limits_r^\infty |A(s)|^2ds\right]\label{diag-kr}
\end{equation}
Now, repeating the proof of Lemma \ref{asymp-l2} with estimates
(\ref{addl-decay}) and (\ref{diag-kr}), we obtain
(\ref{decay-improved}).
\end{proof}
The estimates in the last Lemma are not sharp but good enough for
us.

{\bf Remarks and historical notes.} The results in this section
are partially new. For the Helmholtz equation, analog of Lemma
\ref{layer-strip} was obtained in \cite{Sylv} where the nonlinear
Fourier transform was introduced. In our case, this transform is
given by the map $\cal{F}: A(r)\in
L^2(\mathbb{R}^+)\longrightarrow \f(\lambda)\in
S_0(\mathbb{C}^+)$. In \cite{Sylv}, the space $S_0(\mathbb{C}^+)$
is studied in detail as well as properties of the nonlinear
Fourier transform. For example, its homeomorphic property is
proved by means of weak convergence argument. See also \cite{Tao}.
For the Schr\"odinger operators and Jacobi matrices, the analysis
is more involved \cite{DKS}. The recent paper \cite{ksr}, contains
analysis of $L^2(\mathbb{R})$ potentials for Schr\"odinger
operators. In the OPUC theory, many of these results were
well-known for quite a long time. The paper \cite{Den-Kup} studies
the relation between the decay of the tail
$\|A\|^2_{L^2[r,\infty)}$ and the Hausdorff dimension of the
support of $d\sigma_s$. In conclusion, we want to say that the
case of square summable coefficient is studied pretty well by now.
Perhaps, the only problem left open is the following nonlinear
(non-commutative) analog of the Carlesson Theorem in the Fourier
analysis. Prove (or disprove) that solution of the ODE:
$P'_*(r,\lambda)=A(r)\exp(i\lambda r)\overline{P_*(r,\lambda)},
P_*(0,\lambda)=1$ has a limit at infinity for a.e. $\lambda\in
\mathbb{R}$. We assume here, of course, that $A(r)\in
L^2(\mathbb{R}^+)$. This is a deep and difficult problem whose
analog for OPUC case is also open for quite a long time. We
mention the paper \cite{CK} for some recent closely related
results in this direction.

\newpage

\section{Continuous analog of the Baxter theorem. The case $A(r)\in
L^1(\mathbb{R}^+)$}\label{sect-baxter} In this section, we assume
that $A(r)$ and $H(x)$ are both from regularity class $L^2_{\rm
loc}(\mathbb{R}^+)$. Our goal is to prove the following analog of
Baxter's theorem in the OPUC theory. The proof is an adaptation of
the one for the discrete case (\cite{Simon}, Chapter 5).

\begin{theorem}\label{baxter-theorem}
For any Krein system, $A(r)\in L^1(\mathbb{R}^+)\cap L^2_{\rm
loc}(\mathbb{R}^+)$ if and only if the accelerant $H(r)\in
L^1(\mathbb{R})\cap L^2_{\rm loc}(\mathbb{R}^+)$ and the
Hopf-Wiener operator
\[
(I+\cal{H}_\infty)f=f(x)+\int\limits_0^\infty H(x-y)f(y)dy
\]
is strictly positive on $L^2(\mathbb{R}^+)$. The last condition is
equivalent to
\[
1+\rho(\lambda)>0, \lambda>0
\]
where $\rho(\lambda)\in W(\mathbb{R})$ is the Fourier transform of
$\overline{H(x)}$. Moreover, the measure $d\sigma$ is purely
absolutely continuous, has continuous derivative and
\begin{equation}
 \exp(-2\|A\|_1)\leq 2\pi\sigma'(\lambda)=1+\rho(\lambda)\leq
\exp(2\|A\|_1)
\end{equation}
\end{theorem}

\begin{proof}
 Assume that we are given $A(r)\in L^1(\mathbb{R}^+)$.
We are then in the Szeg\H{o} case. Indeed, an elementary
application of Gronwall-Bellman inequality to (\ref{function-q})
yields the uniform convergence of $P_*(r,\lambda)$ and $\widehat
P_*(r,\lambda)$ to $\Pi_\alpha(\lambda)$ and
$\widehat\Pi_\alpha(\lambda)$ in $\lambda\in
\overline{\mathbb{C}^+}$. Both $\Pi_\alpha(\lambda)$ and
$\widehat\Pi_\alpha(\lambda)$ are continuous in
$\overline{\mathbb{C}^+}$. Moreover,
\[
|P_*(r,\lambda)|\leq \exp(\|A\|_1),|\widehat{P}_*(r,\lambda)|\leq
\exp(\|A\|_1), \lambda\in \overline{\mathbb{C}^+}, r>0
\]
\[
|P_*(r,\lambda)|\geq \exp(-\|A\|_1),|\widehat{P}_*(r,\lambda)|\geq
\exp(-\|A\|_1), \lambda\in \overline{\mathbb{C}^+}, r>0
\]
and then
\[
\exp(-\|A\|_1)\leq |\Pi_\alpha(\lambda)|\leq
\exp(\|A\|_1),\exp(-\|A\|_1)\leq |\widehat\Pi_\alpha(\lambda)|\leq
\exp(\|A\|_1)
\]
The Theorem \ref{Bernstein-Szego} says that measures $(2\pi)^{-1}
|P_*(r,\lambda)|^{-2}d\lambda$ converge to $d\sigma(\lambda)$ in
the weak-($\ast$) sense. Therefore, $d\sigma(\lambda)$ is purely
a.c. and its continuous density allows an estimate
\begin{equation}
 \exp(-2\|A\|_1)\leq 2\pi\sigma'(\lambda)=|\Pi_\alpha(\lambda)|^{-2}\leq
\exp(2\|A\|_1)
\end{equation}
Now, let us show that $H(x)\in L^1(\mathbb{R}^+)$. Indeed, let
\[
y_1(r,\lambda)=P_*(r,\lambda)-1,
y_2(r,\lambda)=P(r,\lambda)-\exp(i\lambda r)
\]
For each $r>0$, $y_{1(2)}\in W_+(\mathbb{C^+})$ and
$\|y_1\|_{W_+}=\|y_2\|_{W_+}$ by (\ref{e3s2}) and (\ref{e31s2}).
From (\ref{krein2}), we get
\begin{equation}\label{banach}
y_1(r)=a(r)-\int\limits_0^r A(s)y_2(s)ds
\end{equation}
where
\[
a(r)=-\int\limits_0^r A(s)\exp(is\lambda)ds\in W_+
\]
and (\ref{banach}) is considered as an integral equation for
functions with values in $W_+$. Taking the norm of the both sides
in (\ref{banach}), we get
\[
\|y_1(r)\|_{W_+}\leq \int\limits_0^r |A(s)|ds+\int\limits_0^r
|A(s)|\cdot \|y_1(s)\|_{W_+}ds
\]
The Gronwall-Bellman inequality yields convergence of $y_1(r)$ to
some $y_1$ in the $W_+$ norm (as $r\to\infty$) and
\[
\|y_1\|_{W_+}\leq \|A\|_1\exp(\|A\|_1)
\]
Therefore, $\Pi_\alpha(\lambda)=1+y_1$. In the same way, we have
$\widehat{\Pi}_{\hat\alpha}(\lambda)=1+\hat{y}_1$ and
$F(\lambda)=\widehat{\Pi}_{\hat\alpha}/\Pi_\alpha=(1+\hat{y}_1)/(1+y_1)=1+h$,
where
\begin{equation}
h=\frac{\hat{y}_1-y_1}{1+y_1}\in W_+ \label{h-f1}
\end{equation}
since the
spectrum of $1+y_1$ does not contain zero. Now, (\ref{reznik})
implies
\begin{equation}
h=2\int\limits_0^\infty \overline{H(x)}\exp(i\lambda x)dx
\label{h-f2}
\end{equation}
and so $H(x)\in L^1(\mathbb{R})$ since $h\in W_+$. We also have
$2\pi\sigma'(\lambda)=1+\rho(\lambda)$.

Now, assume that we are given $H(x)\in L^1(\mathbb{R})$ and
$I+\cal{H}_\infty>0$. Then, the equivalence of
$I+\cal{H}_\infty>0$ and an estimate $1+\rho(\lambda)>0$ follows
from the simple identity
\[
((I+\cal{H}_\infty)f,f)=2\pi \int\limits_{-\infty}^\infty
(1+\rho(-\lambda))|\hat{f}(\lambda)|^2d\lambda
\]
Last identity shows that $H$ generates a Krein system with
$A(r)\in L^2_{\rm loc}(\mathbb{R}^+)$. Together with formula
(\ref{e1s2n1}) and a simple approximation argument (like in the
proof of Lemma \ref{uniqueness} in Appendix), it also imply
$d\sigma=(2\pi)^{-1}(1+\rho(\lambda))d\lambda$.

We need a simple
\begin{lemma} \label{auxil-l1}
If $\|H\|_1<1$, then $A(r)\in L^1(\mathbb{R}^+)$ and $\|A\|_1\leq
\|H\|_1/(1-\|H\|_1)$.
\end{lemma}
\begin{proof}
From (\ref{basic1}), we have
\begin{equation}\label{baxter-aux}
\Gamma_r(t,0)+\int\limits_0^r H(t-u)\Gamma_r(u,0)du=H(t)
\end{equation}
Iterating this identity, we have for
\begin{equation}
\Gamma_r(t,0)=H(t)-\int\limits_0^r
H(t-u_1)H(u_1)du_1+\int\limits_0^r H(t-u_1)\int\limits_0^r
H(u_1-u_2)H(u_2)du_2du_1-\ldots \label{h-iterates}
\end{equation}
where the series converges in $L^1[0,r]$. Then,
$|\Gamma_r(t,0)|\leq g(t)$, where
\[
g=h+h\ast h+h\ast h\ast h+\ldots, h(t)=|H(t)|
\]
Notice that $g(t)$ does not depend on $r$ and $\|g\|_1\leq
\|H\|_1/(1-\|H\|_1)$. So, by taking $t=r$ in (\ref{baxter-aux}),
we get
\begin{equation}
\overline{A(r)}+\int\limits_0^r
H(r-u)\Gamma_r(u,0)du=H(r)\label{lemma-later}
\end{equation}
\begin{equation}
|A(r)|\leq \int\limits_0^r |H(r-u)|\cdot |g(u)|du+|H(r)|
\label{lemma-bis}
\end{equation}
The Young inequality finishes the proof.
\end{proof}
To finish the proof, we will apply this Lemma to the interval
$[R,\infty)$, where $R$ is so large that the accelerant for the
Krein system considered on $[R,\infty)$ has small $L^1$ norm. To
prove the existence of such $R$, we need to use Baxter's Lemma
(see Appendix, Corollary \ref{baxter-cor2}) to the operator
$I+\cal{H}_\infty$ acting in the space $L^1(\mathbb{R}^+)$. Since
$H(x)\in L^1(\mathbb{R})$ and $1+\rho(\lambda)>0$, this Lemma is
applicable and  gives
\begin{equation} \label{baxter-impl}
\|(I+\cal{H}_r)^{-1}\|_{L^1[0,r], L^1[0,r]}\leq C,(r>r_0); \quad
\|\Gamma_r(0,x)\chi_{[0,r]}(x)-\Gamma(x)\|_1\to 0, (r\to\infty)
\end{equation}
where
\[
\Gamma(x)=(I+\cal{H}_\infty)^{-1} H(x)\in L^1(\mathbb{R}^+)
\]
is solution to the Wiener-Hopf equation. Then, we are immediately
in the Szeg\H{o} case since \[ P_*(r,\lambda)=1-\int\limits_0^r
\Gamma_r(0,s) \exp(i\lambda s)ds\to
\Pi_\alpha(\lambda)=1-\int\limits_0^\infty \Gamma(s)\exp(i\lambda
s)ds
\]
and convergence is uniform in $\overline{\mathbb{C}^+}$.

For $F(\lambda)$, we use (\ref{weyl-titchmarsh}) and
\[
F(\lambda)/2=-i\beta+i\int\limits_{-\infty}^\infty
\left[\frac{1}{\lambda-t}+\frac{t}{t^2+1}\right]
\frac{1+\rho(t)}{2\pi}dt
\]
Since $F(i\infty)=1$, we have
\[
F(\lambda)=1+\frac{i}{\pi}\int\limits_{-\infty}^\infty
\frac{1}{\lambda-t} \rho(t)dt
\]
and the integral is understood in v.p. sense. That immediately
implies
\[
F(\lambda)=1+2\int\limits_0^\infty \overline{H(x)}\exp(i\lambda
x)dx
\]

 From
$\widehat{F}=F^{-1}$
 we have $\widehat{H}(x)\in L^1(\mathbb{R})$ for the dual
accelerant
 and
\[
\|\widehat{\Gamma}_r(0,x)\chi_{[0,r]}(x)-\widehat{\Gamma}(x)\|_1\to
0, (r\to\infty);\quad
\widehat{\Gamma}(x)=(I+\widehat{\cal{H}}_\infty)^{-1}
\widehat{H}(x)\in L^1(\mathbb{R}^+)
\]
So, we also have \begin{equation}\label{aux-conv}
\|\B(r,\lambda)-\B(\lambda)\|_{W_+}\to 0,
\|\A(r,\lambda)-\A(\lambda)\|_{W_+}\to 0
\end{equation}
where
\[
\A(\lambda)=1-\int\limits_0^\infty
\frac{\Gamma(x)+\widehat{\Gamma}(x)}{2}\exp(i\lambda x)dx,
\B(\lambda)=\int\limits_0^\infty
\frac{\widehat\Gamma(x)-{\Gamma}(x)}{2}\exp(i\lambda x)dx
\]
From (\ref{sasp}),
\begin{equation}
\exp(i\lambda R)\f_R(\lambda)=\left[
\f(\lambda)\A(R,\lambda)-\B(R,\lambda)\right] \cdot\left[
\A(R,\lambda)+\f_R(\lambda)\B_*(R,\lambda)\right]
\end{equation}
Consider the first factor in the right-hand side. It can be
written as
\[
\A^{-1}(\lambda)\left[
\B(\lambda)\A(R,\lambda)-\B(R,\lambda)\A(\lambda)\right]
\]
Since $\A(\lambda)-1\in W_+$ and $\A(\lambda)$ has no zeroes in
$\mathbb{C}^+$, relations (\ref{aux-conv}) imply this factor goes
to zero in $W_+$ norm as $R\to\infty$. On the other hand, this
factor can be written as
\begin{equation}
\A(R,\lambda)\left[\f(\lambda)-\frac{\B(R,\lambda)}{\A(R,\lambda)}\right]\label{sec-fac}
\end{equation}
and (\ref{sasp}) says that the Fourier coefficient of the second
factor in (\ref{sec-fac}) is equal to $0$ on $[0,R]$. Thus,
\[
\f(\lambda)\A(R,\lambda)-\B(R,\lambda)=\exp(i\lambda
R)Q_R(\lambda)
\]
where $\|Q_R(\lambda)\|_{W_+}\to 0$. Now,
\[
f_R(\lambda)=\frac{Q_R(\lambda)\A(R,\lambda)}{1-Q_R(\lambda)\B_*(R,\lambda)}
\]
and $\|f_R(\lambda)\|_{W_+}\to 0$ since
$\|\A(R,\lambda)\|_{W_+}<C, \|\B_*(R,\lambda)\|_{W_+}<C$ uniformly
in $R$. Since $F_R=(1-\f_R)(1+\f_R)^{-1}$, we also have
$\|1-F_R\|_{W_+}\to 0$ or $\|H_R(x)\|_1\to 0$ where $H_R(x)$--
accelerant corresponding to the interval $[R,\infty)$. The Lemma
\ref{auxil-l1} now yields $A(r)\in L^1[R,\infty)$ as long as
$\|H_R\|<1$.
\end{proof}
Let us obtain the formula for $\alpha$ in the representation for
$\Pi_\alpha(\lambda)$, just like we did for the square summable
$A(r)$. Since $\Pi_\alpha(\lambda)=1+y_1(\lambda)$ and
$y_1(\lambda)\in W_+$, we have a trivial asymptotics:
$\Pi_\alpha(\lambda)=1+\bar{o}(1)$ as $\lambda\to\infty,
\lambda\in \overline{C^+}$. From the multiplicative representation
for $\Pi_\alpha$, we get
\[
\alpha=\lim_{y\to\infty} \frac{1}{2\pi} \int
\frac{s(1-y^2)+iy(1+s^2)}{(s^2+y^2)(1+s^2)}\mu(s)ds,
\mu(s)=\ln(1+\rho(s))\in W(\mathbb{R})
\]
and simple estimates yield
\[
\alpha=\frac{1}{2\pi}\, {\rm v.p.} \int\limits_{-\infty}^\infty
\frac{s\mu(s)}{1+s^2}ds
\]
Thus $\Pi_\alpha(\lambda)$ allows the same representation
(\ref{pisub}). Notice also that in contrast with square summable
case, function $\A(\lambda)$ does not allow asymptotical formula
(\ref{equa1-1}). It should also be mentioned that the Theorem does
not provide a quantitative estimate on $\|A\|_1$ in terms of
$\|H\|_1$ and, say, $\|(1+\rho)^{-1}\|_\infty$. On the other hand,
it is easy to bound $\|(1+\rho)^{-1}\|_\infty$ in terms of
$\|A\|_1$.

Later on, we will need the following result
\begin{lemma}\label{auxil-scat}
Assume that conditions of Baxter's Theorem hold. Then $A(r)\in
C_0(\mathbb{R}^+)$ iff $H(x)\in C_0(\mathbb{R}^+)$ iff $C(x)\in
C_0(\mathbb{R}^+)$.
\end{lemma}
\begin{proof}
Since $H(x), C(x)\in L^1(\mathbb{R}^+)$, we have $H(x)\in
C_0(\mathbb{R}^+)$ iff $C(x)\in C_0(\mathbb{R}^+)$ because of
(\ref{transfer}). Assume $H(x)\in C_0(\mathbb{R}^+)$. Then,
(\ref{lemma-bis})  implies $A(r)\in C_0(\mathbb{R}^+)$. Now, let
$A(r)\in L^1(\mathbb{R}^+)\cap C_0(\mathbb{R}^+)$. Consider
(\ref{erasw}). Since $A(r,t)=\Gamma_r(t,r)=\Gamma_r(0,r-t)$,
\[
\Gamma_r(0,t)=\overline{A(t)}-\int\limits_t^r \overline{A(s)}\cdot
\overline{\Gamma_s(0,t)}ds
\]
Take $R$ large and iterate this identity for $\Delta_R=\{R\leq
t\leq r\}$. Since $A(r)\in L^1(\mathbb{R}^+)\cap
C_0(\mathbb{R}^+)$, we will get convergence. To be more precise,
if
\[
\gamma_R=\sup\limits_{\Delta_R}|\Gamma_r(0,t)|, \alpha_R=\max_{t>
R} |A(t)| \]
then
\[
\gamma_R\leq \alpha_R+\gamma_R\int\limits_R^\infty |A(s)| ds
\]
and $\gamma_R\to 0$ as $R\to\infty$. Since $\Gamma_R(0,t)\to
\Gamma(t)$ in $L^1(\mathbb{R}^+)$,  we have
$\|\Gamma(t)\|_{L^\infty[R,\infty)}\to 0$ as $R\to \infty$. The
same is true about $\widehat{\Gamma}(t)$. Application of formulas
(\ref{h-f1}) and (\ref{h-f2}) finishes the proof.
\end{proof}

{\bf Remarks and historical notes.} For an excellent exposition of
the proof for the Baxter theorem in OPUC case, see \cite{Simon}.
In our case, some modifications were needed. Apparently, the first
proof of the Baxter theorem for continuous case was given in
\cite{melik-adamyan1}. See also \cite{DI, Krein-sf} (one has to
pay attention to some inaccuracies in statements regarding the
regularity of coefficients generated by summable accelerants).
\newpage

\section{Dirac systems}\label{sect-dirac}

In this section, we relate Krein systems to the well-known object
in mathematical physics: one-dimensional Dirac operator. Consider
the Krein system, given by (\ref{krein2}) and assume some
regularity conditions, e.g. $a(r), b(r)\in L^2_{\rm
loc}(\mathbb{R}^+)$. Let $\lambda\in \mathbb{C}$ and

\begin{eqnarray*}
\varphi(r,\lambda)=\frac{\exp(-i\lambda r)}{2}\left[
P(2r,\lambda)+P_*(2r,\lambda)\right],\\
\psi(r,\lambda)=\frac{\exp(-i\lambda r)}{2i}\left[
P(2r,\lambda)-P_*(2r,\lambda)\right]
\end{eqnarray*}%
These functions are of the exponential type $r$ and are  not from
$H^{2}(\mathbb{C}^{+})$ anymore. They should be regarded as
analogs of trigonometric polynomials (or Laurent polynomials). For
the free case, i.e. $A(r)=0$, one has
$\varphi(r,\lambda)=\cos(r\lambda),
\psi(r,\lambda)=\sin(r\lambda)$. If $\lambda\in \mathbb{R}$,
\begin{equation}
\varphi (r,\lambda )=\Re {\cal E}(r,\lambda ),\psi (r,\lambda )=\Im {\cal
E}%
(r,\lambda ), \cal{E}(r,\lambda)=\exp(-i\lambda r)P(2r,\lambda)
\end{equation}%
Define $\cal{E}(r,\lambda)$ for $r<0$ by
$\cal{E}(-r,\lambda)=\overline{\cal{E}(r,\lambda)}$. Let
$a(r)=2\Re A(2r),b(r)=2\Im A(2r)$. Consider the following Dirac
operator

\begin{equation}
\cal{D}\left[
\begin{array}{c}
f_{1} \\
f_{2}%
\end{array}%
\right] =\left[
\begin{array}{cc}
-b(r) & d/dr-a(r) \\
-d/dr-a(r) & b(r)%
\end{array}%
\right] \left[
\begin{array}{c}
f_{1} \\
f_{2}%
\end{array}%
\right]  \label{Dir}
\end{equation}%
where the Hilbert space is $f_{1},f_{2}\in L^{2}(R^{+})\times
L^{2}(R^{+})$ and operator is made self-adjoint by imposing
condition $f_{2}(0)=0$. Another way to write $\cal{D}$ is as
follows
\[
\cal{D}=\mathscr{J}\frac{d}{dr} +Q(r)
\]
where potential $Q(r)$ is
\begin{equation}
Q(r)=\left[
\begin{array}{cc}
-b(r) & -a(r) \\
-a(r) & b(r)%
\end{array}%
\right] \label{dir-pote}
\end{equation}
and
\begin{equation}
\mathscr{J}=\left[
\begin{array}{cc}
0 & 1 \\
-1 & 0%
\end{array}%
\right]
\end{equation}

This form of Dirac operator is called canonical. Any Dirac
operator can be reduced to this form by a suitable change of
variables \cite{Levitan}, p. 48--50. Just like the Krein system,
the Dirac operator in the canonical form has a lot of structure
due to a special choice of potential.

 Rather than, say,
Schr\"odinger operator,  $\cal{D}$ can always be defined as the
closure of the naturally chosen minimal operator (\cite{Levitan},
Theorem 7.1, p.493 or  \cite{Weidmann}, p. 99). We will start with
the following
\begin{lemma}
Functions  $\varphi$  and $\psi$ are
generalized eigenfunctions of $\cal{D}$,  i.e.

\begin{equation}
\cal{D} \left[
\begin{array}{c}
\varphi \\
\psi%
\end{array}
\right] =\lambda \left[
\begin{array}{c}
\varphi \\
\psi%
\end{array}
\right], \varphi(0,\lambda)=1, \psi(0,\lambda)=0  \label{q4}
\end{equation}
More generally,  the fundamental solution $X_d$ for the system
(\ref{Dir}) can be expressed via the fundamental solution for
Krein system in the following way
\[
X_d(r,\lambda)=\exp(-i\lambda r)U_0X(2r,\lambda)U_0^{-1}
\]
and
\[
U_0=\left[
\begin{array}{cc}
1&1\\
-i& i
\end{array}\right]
\]
\end{lemma}
\begin{proof} The proof is a straightforward calculation. \end{proof}

Next, we will show that the system $\{\cal{E}(x,\lambda)\}$ is an
orthogonal system in $L^2(\mathbb{R},d\sigma)$. There are many
ways to see that but we prefer an algebraic one, based on the
proper factorization of certain integral operator. Just like in
Section 3, we start with the following consideration. Let $H(x)$
be an accelerant and $H(x)\in L^2_{\rm loc}[0,\infty)$. For any
$r>0$, consider the integral operator $\hat{\cal{H}}_r$:
\begin{equation}
\hat{\cal{H}}_r f(x)=\int\limits_{-r}^r H(x-u)f(u)du
\end{equation}
given in $L^2[-r,r]$. Since $\hat{\cal{H}}_r$ is the translation
of $\cal{H}_{2r}$ defined on $[0,2r]$, we have
$I+\hat{\cal{H}}_r>0$ for any $r>0$ and the conditions of the
Theorem \ref{volter1} in section 2 are satisfied. Moreover, we can
express one resolvent via the other one, i.e.
\begin{equation}\label{translation-of-res}
\hat{\Gamma}_r(x,y)={\Gamma}_{2r}(r+x, r+y),\quad  |x|, |y|<r
\end{equation}

 Let us consider some
$R>0$ and the factorization
$I+\hat{\cal{H}}_R=(I+\hat{\cal{L}})(I+\hat{\cal{U}})$ where the
lower-diagonal (in the sense of Theorems \ref{volter1} and
\ref{volter7}) operator $ \hat{\cal{L}}$ has kernel $\hat{L}(x,y),
|y|<|x|<R$. Since $\hat{\cal{H}}_R$-- Hermitian, we get
$\hat{\cal{U}}=\hat{\cal{L}}^*$. If
$I+\hat{\cal{L}}=(I+\hat{\cal{V}}_-)^{-1}$ and
$I+\hat{\cal{U}}=(I+\hat{\cal{V}}_+)^{-1}$, then
\begin{lemma}\label{transformation-dirac}
The following is true
\begin{equation}
 \exp(i\lambda
x)=\cal{E}(x,\lambda)+\int\limits_{-|x|}^{|x|}
\widehat{L}(x,u)\cal{E}(u,\lambda)du, x\in [-R,R]
\end{equation}
\end{lemma}
\begin{proof}
 We have the following formula for ${\cal E }
(x,\lambda)$ if $x>0$

\begin{equation}
{\cal E}(x,\lambda )=P(2x,\lambda )\exp (-i\lambda x)=\left( \exp
(2i\lambda x)-\int\limits_{0}^{2x}\Gamma _{2x}(2x,s)\exp
(is\lambda )ds\right) \exp (-i\lambda x)
\end{equation}%
Doing the change of variables $s-x=t$ in the integral, we obtain
\begin{eqnarray}
{\cal E}(x,\lambda ) &=&\exp (i\lambda
x)-\int\limits_{-x}^{x}\Gamma
_{2x}(2x,x+t)\exp (i\lambda t)dt= \\
&=&\exp (i\lambda x)-\int\limits_{-x}^{x}\Gamma _{2x}(x-t,0)\exp
(i\lambda t)dt
\end{eqnarray}%
We also have
\begin{eqnarray}
{\cal E}(-x,\lambda ) =\exp (-i\lambda
x)-\int\limits_{-x}^{x}\Gamma _{2x}(0,x+t)\exp (i\lambda t)dt, x>0
\end{eqnarray}
Therefore, from (\ref{translation-of-res}) and the Theorem
\ref{volter7}, we get
$\cal{E}(x,\lambda)=(I+\hat{\cal{V}}_-)\exp(i\lambda x)$. The
Lemma then follow from $I+\hat{\cal{L}}=(I+\hat{\cal{V}}_-)^{-1}$.
\end{proof}
Clearly, this Lemma is an analog of Lemma \ref{transformation} but
for the different chain. Now, we are ready to relate Krein systems
to Dirac operators. But first we need the following

\begin{definition}
We say that  a non-decreasing function $\sigma _{d}(\lambda
),\lambda \in \mathbb{R}$ is the spectral measure for Dirac
operator (\ref{Dir}), if the following is true (see
\cite{levitan-2}, Chapter 8): for any $f_{1}\in
L^{2}(\mathbb{R}^{+})$ and $f_{2}\in L^{2}(\mathbb{R}^{+}) $, the
operator $\cal{F}$ given by
\begin{equation}
\left[\cal{F}f\right](\lambda)=\int\limits_0^\infty \Bigl( f_1(r)\varphi(r,\lambda)+f_2(r)%
\psi(r,\lambda) \Bigr) dr
\end{equation}
is unitary onto $L^2(\mathbb{R},d\sigma_d)$.
\end{definition}

The next Theorem establishes a further link between the Krein
systems and Dirac operators
\begin{theorem}
The measure $d\sigma_d(\lambda)=2d\sigma(\lambda)$ is the spectral
measure for Dirac operator. Moreover, the mapping
\[
f(x)\in L^{2}(\mathbb{R})\rightarrow \left[{\cal
W}f\right](\lambda )=\int\limits_{-\infty }^{\infty }f(x){\cal
E}(x,\lambda )dx
\]
is unitary onto  $L^{2}(\mathbb{R},d\sigma )$.\label{number2}
\end{theorem}
\begin{proof}
Let us first show that $\cal{W}$ is an isometry map. Indeed, let
$f(x)\in L^2(-R,R)$. From (\ref{e1s2n1}), we get
\begin{equation}\label{identity-dirac}
((I+\hat{\cal{H}}_R) f,f)=\int\limits_{-\infty}^\infty \left|
\int\limits_{-R}^R \overline{f(t)}\exp(i\lambda
t)dt\right|^2d\sigma(\lambda)
\end{equation}
In the meantime, from Lemma \ref{transformation-dirac},
\[
\int\limits_{-R}^R \overline{f(t)}\exp(i\lambda
t)dt=((I+\hat{\cal{L}})\cal{E}(t,\lambda),f(t))_{L^2[-R,R]}=
(\cal{E}(t,\lambda),(I+\hat{\cal{L}}^*)f(t))_{L^2[-R,R]}
\]
Therefore, (\ref{identity-dirac}) gives
\[
\int\limits_{-\infty}^\infty \left|\,\int\limits_{-\infty}^\infty
\overline{g(t)}\cal{E}(t,\lambda)dt\right|^2d\sigma(\lambda)=
((I+\hat{\cal{L}})^{-1}(I+\hat{\cal{H}}_R)(I+\hat{\cal{U}})^{-1}g,g)=\|g\|^2
\]
where $g=(I+\hat{\cal{L}}^*)f=(I+\hat{\cal{U}})f$. Since
$I+\hat{\cal{U}}$ is invertible on $L^2[-R,R]$ and $R$ was chosen
arbitrarily, we learn that $\cal{W}$ is isometry. Now, let us show
that $\cal{W}$ is also unitary. Indeed, for any $R>0$ and $f(x)\in
L^2[-R,R]$, we get (Lemma \ref{transformation-dirac})
\[
\int\limits_{-R}^R f(x)\exp(i\lambda
x)dx=([(I+\hat{\cal{L}})\cal{E}](x,\lambda),\bar{f}(x))=(\cal{E}(x,\lambda),
[(I+\hat{\cal{U}})\bar{f}](x))
\]
Functions of that kind are dense in
$L^2(\mathbb{R},d\sigma(\lambda))$ because the span of
characteristic functions of the intervals $\{[a,b)\}$ are dense
and each of these characteristic functions can be approximated by
the Fourier transform of finitely supported $L^2(\mathbb{R})$
function due to the regularity condition (\ref{e2t1}). Therefore,
the range of $\cal{W}$ is the whole of
$L^2(\mathbb{R},d\sigma(\lambda))$. That means $\cal{W}$ is
unitary.

Now, we can conclude the proof of the Theorem. Take any
$f_{1}(r)\in L^{2}(\mathbb{R}^{+})$ and $f_{2}(r)\in
L^{2}(\mathbb{R}^{+})$. Let $f_{1}(-x)=f_{1}(x),\
f_{2}(-x)=-f_{2}(x),\ x>0$. Consider $f(x)=f_{1}(x)-if_{2}(x)$ on
$\mathbb{R}$. Function $\varphi $ is
even, $%
\psi $ is odd. So,
\begin{eqnarray*}
\left[{\cal W}f\right](\lambda ) &=&\int\limits_{-\infty }^{\infty
}f(x){\cal E}(x,\lambda
)dx=\int\limits_{-\infty }^{\infty }(f_{1}-if_{2})(\varphi +i\psi )dx= \\
&=&2\int\limits_{0}^{\infty }\left( f_{1}\varphi +f_{2}\psi
\right) dx=2\left[\cal{F}f\right](\lambda)
\end{eqnarray*}%
That proves $\cal{F}$ is unitary mapping to
$L^2(\mathbb{R},2d\sigma(\lambda))$. Since the range of $\cal{W}$
is the whole $L^2(\mathbb{R},d\sigma(\lambda))$, we can write
$\cal{W}^{-1}g=f(x)=f_1(x)-if_2(x)$, where
$f_1(x)=[f(x)+f(-x)]/2$, $f_2(x)=-[f(x)-f(-x)]/(2i)$ and $g$ is
arbitrary from $L^2(\mathbb{R},d\sigma)$. Then, $2\cal{F}f=g$ and
$\cal{F}$ is unitary.
\end{proof}
It is easy to show that the spectral measure for Dirac operator is
uniquely defined (see Lemma \ref{uniqueness}) in Appendix.

 As usual, the following representation can be easily
obtained from the Theorem:
\begin{equation}
\int\limits_{-\infty}^\infty \left[\begin{array}{cc}
\varphi(x,\lambda)\varphi(y,\lambda) & \varphi(x,\lambda)\psi(y,\lambda)\\
\psi(x,\lambda)\varphi(y,\lambda) & \psi(x,\lambda)\psi(y,\lambda)
\end{array}\right]d\sigma_d(\lambda)=\left[\begin{array}{cc}
\delta(x-y) & 0\\
0 & \delta(x-y)
\end{array}\right] \label{delta-function}
\end{equation} and this identity should be understood in the
weak-$[L^2(\mathbb{R}^+)]^2$ sense (i.e. it is true after
multiplication by $L^2$ functions and integration in $x$ and $y$).

In case $d\sigma_0 (\lambda )=d\lambda /(2\pi )$ discussed above, one has $%
a(r)=b(r)=0$, ${\cal E}(x,\lambda )=\exp (i\lambda x)$, $d\sigma
_{d}(\lambda )=d\lambda /\pi $. The map $\cal{W}$ is then the
standard Fourier transform.

The Dirac operator plays the role of the so-called CMV matrix for
polynomials orthogonal on the unit circle. Many results about the
Krein systems and functions $P(r,\lambda)$ can be viewed from that
perspective.

It is also quite helpful to introduce the  auxiliary dissipative
operator. For any $R>0$, consider the operator $\cal{D}_R$ on
$L^2[0,R]\times L^2[0,R]$ given by the differential system
(\ref{Dir}) and the boundary conditions:
\begin{equation}
f_2(0)=0, f_1(R)+if_2(R)=0 \label{bc}
\end{equation}
 The domain of definition for that
operator consists in functions from $W^{1,2}[0,R]\times
W^{1,2}[0,R]$ satisfying (\ref{bc}). It is an elementary
calculation to show that $\cal{D}_R$ is dissipative since
\[
\Im (\cal{D}_R f,f)=2|f_2(R)|^2\geq 0
\]
for all $f$ in the domain of $\cal{D}$. This operator has compact
resolvent, an integral operator that can be written explicitly in
terms of the solutions to the corresponding equation.  The formula
is as follows and can be easily checked
\[
(\cal{D}_{R} -\lambda)^{-1} (f_1, f_2)^t= \left[\begin{array}{c}
\varphi(r,\lambda)\\
\psi(r,\lambda)
\end{array}\right]\int\limits_r^R
(-Z_{12}(s,\lambda)f_1(s)+Z_{11}(s,\lambda)f_2(s))ds
\]
\[
+ \left[\begin{array}{c}
y_1(r,\lambda)\\
y_2(r,\lambda)
\end{array}\right]
\int\limits_0^r
(Z_{22}(s,\lambda)f_1(s)-Z_{21}(s,\lambda)f_2(s))ds
\]
where
\begin{equation}
(y_1(r,\lambda),y_2(r,\lambda))^t=X_d(r,\lambda)X_d^{-1}(R,\lambda)(1,i)^t
\end{equation}
and
\[
Z(r,\lambda)=\left[\begin{array}{cc}
\varphi(r,\lambda) & y_1(r,\lambda)\\
\psi(r,\lambda) & y_2(r,\lambda)
\end{array}\right]^{-1}
\]
Since $(y_1,y_2)^t$ is solution to the Cauchy problem, it always
exists. Therefore, the kernel of resolvent has a pole at point
$\lambda$ if and only if
\[
0=\det\left[\begin{array}{cc}
\varphi(R,\lambda) & 1\\
\psi(R,\lambda) & i
\end{array}\right]=i\varphi(R,\lambda)-\psi(R,\lambda)=i\exp(-i\lambda
R)P(2R,\lambda)
\]

Thus, the spectrum of this operator is discrete and  coincides
with the zeroes of $P(2R,\lambda)$.  Since the dissipative
operator has the spectrum in  $\mathbb{C}^+$, this is another,
operator-theoretic explanation to the fact that all zeroes of
$P(r,\lambda)$ are in the upper half-plane. The zeroes of
$P_*(2R,\lambda)$, on the other hand, are naturally characterized
by the spectrum of the operator $\cal{D}_R^*$, adjoint to
$\cal{D}_R$. The symmetry of zeroes for $P(r,\lambda)$ and
$P_*(r,\lambda)$ is now a consequence of a simple fact in operator
theory.

One can say even more, infact

\begin{lemma}
The following representation is true
\begin{equation}
P_*(2R,\lambda)=\exp \left[-\int\limits_0^R [a(s)+ib(s)]
\exp(2i\lambda s) ds \right] \det {}_2  \left(
\frac{\cal{D}_R^*-\lambda}{\cal{D}_{0,R}^*-\lambda}\right)
\label{link2}
\end{equation}
 where $\cal{D}_{0,R}$ denotes
the operator with $a(r)=b(r)=0$. The regularized determinant is
understood  as the regularized determinants of operator $I+
(\cal{D}_{0,R}^*-\lambda)^{-1}Q$ (see \cite{simon-traces}, p.
106).
\end{lemma}

\begin{proof} A simple calculation shows that the spectrum of $\cal{D}_{0,R}^*$ is empty
(infact, is equal to infinity). Therefore,
$(\cal{D}_{0,R}^*-\lambda)^{-1}$ is always well-defined. We have

\[
(\cal{D}^*_{0,R} -\lambda)^{-1} (f_1, f_2)^t= \left[
\begin{array}{c}
\cos(r\lambda)\\
\sin(r\lambda)
\end{array}
\right] \int\limits_r^R [i\exp(i\lambda s)f_1(s)+\exp(i\lambda
s)f_2(s)]ds
\]
\begin{equation}
+ \left[
\begin{array}{c}
\exp(ir\lambda)\\
-i\exp(ir\lambda)
\end{array}
\right] \int\limits_0^r [i\cos (s\lambda)f_1(s)+i\sin
(s\lambda)f_2(s)]ds \label{greenk}
\end{equation}
Simple calculations show that $(\cal{D}_{0,R}^*-\lambda)^{-1}Q\in
\cal{S}^2$ and the regularized determinant exists.

Then, we use the following trick (see \cite{simon-traces}, p. 75).
Introduce the so-called ``coupling constant" $\mu\in \mathbb{C}$
and the potentials $Q_\mu=\mu Q$. Then, consider the corresponding
functions $P_*(2R,\lambda,\mu)$ and
\[
f(\lambda,\mu)= \exp \left[-\mu\int\limits_0^R [a(s)+ib(s)]
\exp(2i\lambda s) ds \right] \det {}_2  (I+\mu
(\cal{D}_{0,R}^*-\lambda)^{-1}Q)
\]
It is easy to see that these functions have the same zeroes. That
follows from the properties of regularized determinants and
relation between spectrum of $\cal{D}_{0,R}^*$ and zeroes of
$P_*(2R,\lambda)$ discussed above. For fixed $\lambda$, the
function $P_*(2R,\lambda,\mu)$ is of exponential type in $\mu$.
Therefore, it can be factored
\begin{equation}\label{factor-p-formula}
P_*(2R,\lambda,\mu)=\exp(c_1\mu+c_0) \prod_{n=1}^\infty
(1-\mu/\mu_n)\exp(\mu/\mu_n)
\end{equation}
where the constants $c_0,c_1$ and zeroes $\mu_n$ all depend on
$\lambda$. Since $P_*(2R,\lambda,0)=1$, we get $c_0=0$. Taking
logarithm of both sides in (\ref{factor-p-formula}), and comparing
the Taylor coefficients in front of $\mu$, we get
\[
c_1=-\int\limits_0^{2R} \exp(i\lambda s) A(s) ds=-\int\limits_0^R
[a(s)+ib(s)] \exp(2i\lambda s) ds
\]
On the other hand, for $\det_2$, we have the following
factorization result (see \cite{simon-traces}, Theorem 9.2, part
(a))
\[
\det {}_2  (I+\mu
(\cal{D}_{0,R}^*-\lambda)^{-1}Q)=\prod_{n=1}^\infty
(1-\mu/\mu_n)\exp(\mu/\mu_n)
\]
Comparing these two expansions, we get the statement of the Lemma
(take $\mu=1$).
\end{proof}

The determinantal representations are usually very useful in
practice. They provide the natural factorization of entire
functions of interest. In case potential $A(r)$ is small at
infinity (say $A(y)\in L^p(\mathbb{R}^+), p<\infty$), one can get
asymptotical expansion of any order by using the further
regularization of $\det_2$ involving $\det_3, \det_4, \ldots$.
Formulas for the kernel $K(x,y)$ show that it has discontinuity on
the diagonal. Therefore, one could have used the Carleman-Hilbert
determinant instead of $\det_2$ regularization.

Since we have determinantal formula for $P_*(2R,\lambda)$ for
finite $R$, we might hope to get analogous result for
$\Pi_\alpha(\lambda)$ in case $A(r)\in L^2(\mathbb{R}^+)$ by just
taking $R\to \infty$.

\begin{theorem}
If $A(r)\in L^2(\mathbb{R}^+)$, then $
(\cal{D}_{0}-\lambda)^{-1}Q\in \cal{S}^2$ and
\begin{equation}
 \Pi_\alpha(\lambda)= \exp \left[-\int\limits_0^\infty [a(s)+ib(s)]
\exp(2i\lambda s) ds \right] \det {}_2  \left(
\frac{\cal{D}-\lambda}{\cal{D}_0-\lambda}\right), \lambda\in
\mathbb{C}^+ \label{link3}
\end{equation}
Here $\cal{D}_0$ denotes the free Dirac operator, i.e. $\cal{D}$
with $a(r)=b(r)=0$.
\end{theorem}
\begin{proof}
The integral operator $ (\cal{D}_{0}-\lambda)^{-1}Q$ has the
following kernel $K_0(x,y,\lambda)$:

\[
K_0(x,y,\lambda) =\left[
\begin{array}{cc}
-e^{i\lambda y}\cos(\lambda x) [a(y)+ib(y)] &
e^{i\lambda y}\cos(\lambda x) [b(y)-ia(y)]\\
-e^{i\lambda y}\sin(\lambda x) [a(y)+ib(y)] & e^{i\lambda
y}\sin(\lambda x) [b(y)-ia(y)]
\end{array}
\right]
\]
if $y>x>0$ and
\[
K_0(x,y,\lambda)= \left[
\begin{array}{cc}
-ie^{i\lambda x} [\cos(\lambda y)b(y)+\sin(\lambda y)a(y)] &
ie^{i\lambda x} [\sin(\lambda y) b(y)-\cos(\lambda y)a(y)]\\
-e^{i\lambda x} [\cos(\lambda y)b(y)+\sin(\lambda y)a(y)] &
e^{i\lambda x} [\sin(\lambda y)b(y)-a(y)\cos(\lambda y)]
\end{array}
\right]
\]
if $0<y<x$. Since $A(r)\in L^2(\mathbb{R}^+)$,  we have
$(\cal{D}_{0}-\lambda)^{-1}Q \in \cal{S}^2$ for any $\lambda\in
\mathbb{C}^+$ and the regularized determinant exists. Now, fix
$\lambda\in \mathbb{C}^+$. The function $\Pi_*(2R,\lambda)\to
\Pi_\alpha(\lambda)$ as $R\to\infty$ (see Theorem \ref{main-lsz}).
On the other hand, one can easily check that $\det {}_2  (I+
(\cal{D}_{0,R}^*-\lambda)^{-1}Q)\to \det {}_2  (I+
Q(\cal{D}_{0}-\lambda)^{-1}Q)$ as well.
\end{proof}


\smallskip
Now, let us study the wave operators for $\cal{D}$. The following
result establishes a connection between the stationary and
non-stationary scattering approaches.
\begin{theorem}
If $a(x), b(x)\in L^{2}(\mathbb{R}^+)$, then the wave operators
\[
\Omega _{\pm }{f}=\lim_{t\rightarrow \mp \infty
}e^{it\cal{D}}e^{-it\cal{D}_{0}}{f} \] exist. The limit is
understood in the strong sense, ${f}=(f_1,f_2)^t\in
[L^2(\mathbb{R}^+)]^2$. \label{wave-op}
\end{theorem}
\begin{proof}
The free evolution of $\cal{D}_0$ is given in Lemma
\ref{evolution} from Appendix. It is actually a shift after some
unitary transformations. Since each of the operators
$e^{it\cal{D}},e^{-it\cal{D}_{0}}$ is unitary, it suffices
to check the existence of strong limit for vectors ${f}%
=(f,0)^{t}$, where the scalar function \mbox{$f(x)\in
C_0^\infty(\mathbb{R}^+)$}. The existence of strong limit for
vectors with zero as the first coordinate can be proved in the
same way. Due to linearity, that is enough to conclude the
convergence for all $C_0^\infty(\mathbb{R}^+)$ vector-valued
functions that give
rise to subspace dense in $L^2(\mathbb{R}^+)\times L^2(\mathbb{R}^+)$. From Lemma \ref{evolution}, we have%
\begin{equation*}
e^{-it\cal{D}_{0}}{f}=\frac{1}{2}\left[
\begin{array}{c}
f(x+t)+f(x-t) \\
-i(f(x-t)-f(x+t))%
\end{array}%
\right] =\frac{1}{2}\left[
\begin{array}{c}
f(|x-t|) \\
-if(|x-t|)%
\end{array}%
\right] ,
\end{equation*}
where the last formula holds for $t$ large enough because the
support of $f
$ is finite (say, in the interval $(0,a)$). Consider the function%
\begin{equation*}
\vartheta (t,\lambda )=\frac{e^{it\lambda }}{2}\left[ \int\limits_{0}^{%
\infty }f(|x-t|)\varphi (x,\lambda )dx-i\int\limits_{0}^{\infty
}f(|x-t|)\psi (x,\lambda )dx\right] .
\end{equation*}
To prove Theorem \ref{wave-op}, it suffices to show that the limit
of $\vartheta (t,\lambda )$ in $L^{2,2\sigma }(R)$ exists as
$t\rightarrow \infty$. The following relations are true
\begin{eqnarray*}
\vartheta (t,\lambda ) =\frac{e^{it\lambda
}}{2}\int\limits_{0}^{\infty
}f(|x-t|)e^{i\lambda x}\overline{P(2x,\lambda )}dx=\frac{1}{2}%
\int\limits_{-a}^{a}f(|s|)e^{-i\lambda s}P_{\ast }(2t+2s,\lambda )ds \\
=\frac{1}{2}\int\limits_{-a}^{a}\left[ \int\limits_{-a}^{s}f(|\tau
|)e^{-i\lambda \tau }d\tau \right]'P_{\ast }(2t+2s,\lambda )ds
=P_{\ast }(2t+2a,\lambda )\int\limits_{0}^{a}f(\tau )\cos (\lambda
\tau
)d\tau \\
+\int\limits_{-a}^{a}A(2t+2s)P(2t+2s,\lambda
)\int\limits_{-a}^{s}f(|\tau |)e^{-i\lambda \tau }d\tau ds.
\end{eqnarray*}
Since $f(x)$ is smooth,
\begin{equation}
\left| \, \int\limits_{-a}^{s}f(|\tau |)e^{-i\tau \lambda }d\tau
\right| \leq \frac{C}{\sqrt{\lambda ^{2}+1}}.\label{ref1}
\end{equation}
Due to Lemma \ref{asymp-l2} and Theorem \ref{main-lsz},
\begin{equation}
P_{\ast }(2t+2a,\lambda )\int\limits_{0}^{a}f(\tau )\cos (\lambda
\tau )d\tau \rightarrow \Pi_\alpha (\lambda +i0)\chi
_{E^{c}}(\lambda) \int\limits_{0}^{a}f(\tau )\cos (\lambda \tau
)d\tau , \label{form}
\end{equation}
where the possible singular component of $d\sigma$ is supported on
the Borel set $E$, and $\chi _{E^{c}}$ is the characteristic
function of the complement to $E$. The convergence is understood
in $L^{2,2\sigma }(R)$ sense. Here we also used Remark after the
Lemma~\ref{asymp-l2}. Generalized Minkowski inequality and
(\ref{ref1}) yield
\begin{eqnarray*}
\left\| \int\limits_{-a}^{a}A(2t+2s)P(2t+2s,\lambda
)\int\limits_{-a}^{s}f(|\tau |)e^{-i\lambda \tau }d\tau ds\right\|
_{2,\sigma }
\\
\leq C\left[ \int\limits_{-a}^{a}|A(2t+2s)|ds\right] \left[
\sup_{x\geq 0}\int\limits_{-\infty }^{\infty }\frac{\left|
P(x,\lambda )\right| ^{2}}{\lambda ^{2}+1}d\sigma \right]^{1/2}.
\end{eqnarray*}%
The second factor is bounded due to Lemma \ref{kernel-diag}.
Function $A(x)\in L^{2}(\mathbb{R}^+)$, therefore the first factor
tends to $0$ as $t\rightarrow \infty$. \end{proof}

\textbf{Remark.} We not only proved the existence of the wave
operators, but also deduced the formula for them, the right-hand
side of (\ref{form}). Notice that this map is isometry. Part of
the arguments above are well-known in the theory of polynomials
orthogonal on the unit circle \cite{Geronimus, Simon}.

Very interesting effect can be observed in the case
$d\sigma\in$(Szeg\H{o}), and $A(r)\to 0$ at infinity in some sense
(say, $A(r)\in L^p(\mathbb{R}^+), p<\infty$). As was discussed
before (see the paragraph after Theorem \ref{sdmt}), the limit
$P_*(r_n,\lambda)$ is not necessarily uniquely defined and might
depend upon the choice of the subsequence $r_n\to\infty$. One can
easily show that the proof of the Theorem above can be adjusted to
this situation with the exception that the limit $\lim_{n\to
\infty} e^{it_n\cal{D}}e^{-it_n\cal{D}_0}$ depends upon the choice
of time sequence $t_n$ and the  limiting operators will actually
differ only by the unimodular factor. However, due to Lemma
\ref{real-valued-case}, this phenomena cannot be observed for
real-valued $A(r)$ (i.e. when $b(r)=0$).

At this point, we need to mention that in OPUC theory, the free
CMV matrix is unitarily equivalent to the shift in
$\ell^2(\mathbb{Z})$ space. Infact, the same is true about the
evolution for free Dirac operator. Indeed, consider the following
operator
\[
\tilde{\cal{D}}_0=Z^{-1}\cal{D}_0Z, Z=\frac{1}{\sqrt 2}\left[
\begin{array}{cc}
i & -1\\
1 & -i
\end{array}\right],\tilde{\cal{D}}_0=\left[
\begin{array}{cc}
\displaystyle -i\frac{d}{dr} & 0\\
 0 & \displaystyle i\frac{d}{dr}
\end{array}\right]
\]
The new domain of definition is $[H^1(\mathbb{R}^+)]^2$ with
additional condition $f_1(0)=if_2(0)$. Now, if one maps the
Hilbert space $f=(f_1,f_2)^t\in [L^2(\mathbb{R}^+)]^2$ to
$L^2(\mathbb{R})$ by
\[
f(r)\longrightarrow g(x)=\left\{
\begin{array}{lr}
f_1(x),& x>0 \\
if_2(-x),& x<0
\end{array}
\right.
\]
then the operator $\tilde{\cal{D}}$ happens to be unitarily
equivalent to the selfadjoint operator $\displaystyle
\cal{L}_0=-i\frac{d}{dx}$ on $L^2(\mathbb{R})$ with domain of
definition $H^1(\mathbb{R})$. Its free evolution is just the
shift: $\exp(it\cal{L}_0)g(x)=g(x+t), t\in \mathbb{R}$. Since we
performed only the unitary transformations, the evolution
$\exp(it\cal{D}_0)$ (see Lemma \ref{evolution} in Appendix) is
unitarily equivalent to the shift in $L^2(\mathbb{R})$. Notice
that potential is transformed to
\begin{equation}
Z^{-1}QZ=\left[
\begin{array}{cc}
0 & 2\overline{A(2r)}\\
2A(2r) & 0
\end{array}\right] \label{different-form}
\end{equation} The unitary transformation of $[L^2(\mathbb{R}^+)]^2$ to
$\mathbb{R}$ will result in nonlocal perturbation of $\cal{L}_0$:
\[
[\cal{L}g](x)=-i\frac{dg}{dx}+\widetilde{A}(x)[\cal{S}g](x)
\]
where Hermitian function $\widetilde{A}(x)$ and operator $\cal{S}$
are given by
\[
\widetilde{A}(x)=i\left\{
\begin{array}{lr}
-2\overline{A(2x)},& x>0 \\
{2A}(-2x),& x<0
\end{array}
\right., \cal{S}g(x)=g(-x)
\]
The analysis of $\cal{L}$ is non-trivial (rather than in the case
 $\cal{S}=I$)
 and is equivalent to
analysis of the original Dirac operator $\cal{D}$ or Krein system.
Notice that in the Fourier space, this operator can be formally
written as
\[
-\lambda f(\lambda)+\int\limits_{-\infty}^\infty
V(\lambda+t)f(t)dt
\]
with real-valued $V(\lambda)$ being the Fourier transform of
$\widetilde{A}(x)$. The analysis we have done before implies the
corresponding properties of this operator.

Let us consider the scattering problem for Dirac operator and
relate scattering parameters to the parameters of the
corresponding Krein system. Consider, for simplicity, operator
$\cal D$ with finitely supported coefficients $a$ and $b$. Then,
there is the so-called Jost solution $F(r,\lambda): {\cal
D}F=\lambda F$ defined by the asymptotics at infinity:
$F(r,\lambda)=(f_1(r,\lambda),f_2(r,\lambda)^t=\exp(i\lambda r)
(i,1)^t$ for $r$ large enough. Let us introduce the scattering
data for the Dirac operator:
\begin{eqnarray*}
A_d(\lambda)=(f_2(0,\lambda)-if_1(0,\lambda))/2,
B_d(\lambda)=(f_2(0,\lambda)+if_1(0,\lambda))/2,\\
T_d(\lambda)=A_d^{-1}(\lambda),
R_d(\lambda)=B_d(\lambda)/A_d(\lambda)
\end{eqnarray*}
Coefficient $T_d(\lambda)$ is called the transmission coefficient,
$R_d(\lambda)$ is the reflection coefficient, $f_2(0,\lambda)$ is
Jost function. These notations are quite natural. If one extends
$a$ and $b$ to the negative half-line as zero, then
\[
F(r,\lambda)=A_d(\lambda)\exp(i\lambda
r)(i,1)^t+B_d(\lambda)\exp(-i\lambda r)(-i,1)^t, r<0,
\]
or
\[
T_d(\lambda)F(r,\lambda)=\exp(i\lambda
r)(i,1)^t+R_d(\lambda)\exp(-i\lambda r)(-i,1)^t, r<0,
\]
\begin{lemma}\label{dirac-scat}
The following relations are true
\[
A_d(\lambda)=\A(\lambda), B_d(\lambda)=\B(\lambda),
R_d(\lambda)=\f(\lambda), f_2(0,\lambda)=\Pi_\alpha(\lambda),
\]
\begin{equation}
\sigma_d'(\lambda)=\frac{1}{\pi |f_2(0,\lambda)|^2}=\frac{1}{\pi
|\Pi(\lambda)|^2}=2\sigma'(\lambda) \label{krein-dirac}
\end{equation}
$\f(\lambda)$-- Schur function of the Krein system, functions
$\A(\lambda), \B(\lambda), \Pi_\alpha(\lambda)$ are taken from
consideration of $A(r)\in L^2(\mathbb{R}^+)$ or $A(r)\in
L^1(\mathbb{R}^+)$ cases.
\end{lemma}
\begin{proof} The proof is a straightforward calculation.
\end{proof}

The last formula in (\ref{krein-dirac}) is of great importance. It
gives a factorization of the spectral measure density via some
function analytic in the upper half-plain. Formulas of that sort
have analogs in the scattering problems for some PDE
\cite{Den-CPDE, Den-IMRN}. The equivalence of reflection
coefficient from the scattering theory of quantum mechanics and
Schur function is a remarkable fact, which, perhaps, was not
completely understood and used by both mathematical physicists and
analysts.

The general spectral theory allows to get new natural
interpretation for various quantities considered before. One
example is provided by the following Lemma.
\begin{lemma}
For any $\lambda_0\in \mathbb{C}^+$, the operator $\Im
(\cal{D}-\lambda_0)^{-1}$ has matrix-valued kernel
$G(x,y,\lambda_0)$ and
\[
2\Im \lambda_0 \int\limits_{-\infty}^\infty
\frac{|P_*(2r,\lambda)|^2}{|\lambda-\lambda_0|^2}\,
d\sigma(\lambda)={\rm Tr}\,  G(r,r,\lambda_0)
\]
\end{lemma}
\begin{proof}
The spectral representation for the resolvent yields (see
(\ref{delta-function}))
\begin{equation}
\int\limits_{-\infty}^\infty \frac{1}{\lambda-\lambda_0}
\left[\begin{array}{cc}
\varphi(x,\lambda)\varphi(y,\lambda) & \varphi(x,\lambda)\psi(y,\lambda)\\
\psi(x,\lambda)\varphi(y,\lambda) & \psi(x,\lambda)\psi(y,\lambda)
\end{array}\right]d\sigma_d(\lambda)=(\cal{D}-\lambda_0)^{-1}(x,y)
\end{equation}
Therefore,
\[
{\rm Tr}\, \Im G(x,y,\lambda_0)=\int\limits_{-\infty}^\infty
\frac{\Im \lambda_0}{|\lambda-\lambda_0|^2}
\left[\varphi(x,\lambda)\varphi(y,\lambda) +
\psi(x,\lambda)\psi(y,\lambda)\right] d\sigma_d(\lambda)
\]
Now, the Lemma is straightforward.
\end{proof}
This Lemma allows to control the integral
\[
\int\limits_{-\infty}^\infty
\frac{|P_*(r,\lambda)|^2}{\lambda^2+1}d\sigma(\lambda)
\]
by using the standard tools of, say, perturbation theory. In
particular, if $A\to 0$ in some sense, then this integral tends to
$1/2$, the value for unperturbed case.

 Recall that the CMV matrix corresponding to Verblunsky
coefficients $a_n$ is

$$
\cal{C}=\left[
\begin{array}{cccccc}
*&*&*&0&0&\ldots\\
*&*&*&0&0&\ldots\\
0&*&*&*&*&\ldots\\
0&*&*&*&*&\ldots\\
0&0&0&*&*&\ldots\\
\vdots&\vdots&\vdots&\vdots&\vdots&\ddots\\
\end{array}
\right] = \left[
\begin{array}{ccc}
A_0&0&\ldots\\
0&A_1&\ldots\\
\vdots&\vdots&\ddots
\end{array}
\right]
$$

\begin{eqnarray*}
A_k&=& \left[
\begin{array}{cccc}
\bar a_{2k}\rho_{2k-1}&-\bar a_{2k}a_{2k-1}&\bar
a_{2k+1}\rho_{2k}&\rho_{2k+1}\rho_{2k}
\\
\rho_{2k}\rho_{2k-1}&-\rho_{2k}a_{2k-1}&-\bar
a_{2k+1}a_{2k}&-\rho_{2k+1}a_{2k}
\end{array}
\right]\\
A_0&=& \left[
\begin{array}{ccc}
\bar a_0&\bar a_1\rho_0&\rho_1\rho_0\\
\rho_0&-\bar a_1a_0&-\rho_1a_0
\end{array}
\right]
\end{eqnarray*}
and $\rho_k=(1-|a_k|^2)^{1/2}$. To show how  CMV matrix
corresponds to Dirac operators, we prefer to write it in the
equivalent way (by introducing two Hilbert spaces
$\ell^2(\mathbb{Z}^+)$ corresponding to even and odd indices):

$$
\cal{C}= \left[
\begin{array}{cc}
\cal{C}_{11} & \cal{C}_{12}\\
\cal{C}_{21} & \cal{C}_{22}
\end{array}
\right]
$$
with
$$
\cal{C}_{11}=
\left[
\begin{array}{ccccc}
\bar a_0&\rho_0\rho_1&0&0&\ldots\\
0&-a_1\bar a_2&\rho_2\rho_3&0&\ldots\\
0&0&-a_3\bar a_4&\rho_4\rho_5&\ldots\\
0&0&0&*&\ldots\\
0&0&0&0&\ldots\\
\vdots&\vdots&\vdots&\vdots&\ddots\\
\end{array}
\right],\quad\, \cal{C}_{12}= \left[
\begin{array}{cccccc}
\rho_0 \bar a_1 &0&0&0&0&\ldots\\
\rho_1 \bar a_2 &\rho_2 \bar a_3&0&0&0&\ldots\\
0&\rho_3 \bar a_4 &\rho_4 \bar a_5&0&0&\ldots\\
0&0&\rho_5 \bar a_6&*&0&\ldots\\
0&0&0&*&*&\ldots\\
\vdots&\vdots&\vdots&\vdots&\vdots&\ddots\\
\end{array}
\right]
$$
$$
\cal{C}_{21}= \left[
\begin{array}{ccccc}
\rho_0  &-\rho_1 a_0&0&0&\ldots\\
0 &-\rho_2  a_1&-\rho_3 a_2&0&\ldots\\
0&0  &-\rho_4  a_3&-\rho_5 a_4&\ldots\\
0&0&0&*&\ldots\\
0&0&0&0&\ldots\\
\vdots&\vdots&\vdots&\vdots&\ddots\\
\end{array}
\right], \cal{C}_{22}= \left[
\begin{array}{ccccc}
-a_0 \bar a_1&0&0&0&\ldots\\
\rho_1\rho_2 &-a_2\bar a_3&0&0&\ldots\\
0&\rho_3\rho_4&-a_4\bar a_5&0&\ldots\\
0&0&\rho_5 a_6&*&\ldots\\
0&0&0&*&\ldots\\
\vdots&\vdots&\vdots&\vdots&\ddots\\
\end{array}
\right]
$$

It is well known that the formal discretization of the continuous
Schr\"odinger operator produces a discrete Schr\"odinger operator,
a particular case of the Jacobi matrix. For the Dirac operator,
the situation is a little bit different because it is self-adjoint
and the CMV matrix (an analog of Jacobi matrix in this case) is
unitary. Notice that $\cal{D}$ is unitarily equivalent to

\begin{equation}
\widehat{\cal{D}}=u\cal{D}u^{-1}=\left[
\begin{array}{cc}
-\displaystyle i\frac{d}{dr}  &  -2i\overline{A(2r)}\\
2i{A(2r)} & \displaystyle i\frac{d}{dr}
\end{array}\right],\quad
u=\frac{1}{\sqrt 2}\left[
\begin{array}{cc}
1 &  i\\
1 & -i
\end{array}
 \right]\label{dirac-df}
\end{equation}
with an appropriate boundary condition at zero. Conjugating
$\widehat{\cal{D}}$ by the matrix
\[
\tau=\left[
\begin{array}{cc}
i  &  0\\
0 &   -1
\end{array},\right]
\]
we get that $\tau^{-1}\widehat{\cal{D}}\tau$ has the same diagonal
but
 the off-diagonal elements have the form (\ref{different-form}). Consider the following formal
discretization of $\widehat{\cal{D}}$
\[
\widehat{\cal{D}}_h =\left[
\begin{array}{cc}
-ih^{-1} (L-I) &  -2i\overline{A(2r_n)}\\
2i{A(2r_n)} & ih^{-1} (I-R)
\end{array}
 \right]
\]
where $R$ is the right shift, $L$ -- the left shift. Then,
consider

\[
I+ih \widehat{\cal{D}}_h= \left[
\begin{array}{cc}
L & 2h\overline{A(2r_n)}\\
-2h{A(2r_n)} & R
\end{array}
 \right]
\]
Taking the Verblunsky coefficients $\{a^{(h)}_k\}: a_{2n}^{(h)}=0,
a_{2n-1}^{(h)}=2hA(2r_n)$, we get $I+ih \widehat{\cal{D}}_h={\rm
CMV}+\bar{o}(h)$. Thus, formally, CMV matrices and discretization
of Dirac operators are related via this very simple identity.
Since $L-I$ and $I-R$ are of order $h$ when acted on smooth
functions, we could also say that
\[
\exp\left( ih\widehat{\cal{D}}_h\right)= \left[
\begin{array}{cc}
L & 2h\overline{A(2r_n)}\\
-2h{A(2r_n)} & R
\end{array}
 \right]+\bar{o}(h)
\]
We do not pursue the goal of making any accurate statements
regarding these discretizations but that can be done.


{\bf Remarks and historical notes.} The one-to-one correspondence
between Krein systems and Dirac operators was discovered by M.G.
Krein in his seminal paper \cite{Krein2}. Unfortunately, no proofs
were given. The determinantal formulas obtained in this section
are new to the best of our knowledge. In the meantime, analogous
results for differential equations were obtained earlier (e.g.
\cite{Jost}). These ideas were also used quite recently
\cite{killip}.
 If
$a(x), b(x)\in L^1(\mathbb{R}^+)$, the existence of wave operators
follows from trace-class perturbation argument \cite{RS}. In the
case $a(x), b(x)\in L^p(\mathbb{R}^+),\ 1<p<2$, the  wave
operators were studied by Christ and Kiselev \cite{CK}. The
analysis was based on establishing the asymptotics of generalized
eigenfunctions (essentially, asymptotics of $P_*(r,\lambda)$) for
Lebesgue almost any value of $\lambda$. If $p>2$, one can use
results from \cite{kls} to construct examples with no absolutely
continuous spectrum. Thus in this case the wave operators might
not exist at all. The proof of Theorem \ref{wave-op} is taken from
\cite{Den-GAFA}. Independently, Barry Simon obtained analogous
results for CMV matrices in the Szeg\H{o} case.

\bigskip

\newpage

\section{Schr\"{o}dinger operators}

Let us consider Dirac operator (\ref{Dir}) with $b(r)=0$ and absolutely
continuous $%
a(r)$. For the corresponding Krein system, we have $A(r)\in
\mathbb{R}$ and from Lemma \ref{conjugation} we learn that $H(x)$
is real and continuous on $\mathbb{R}$, the measure $d\sigma$ is
even.
Operator $\cal{D}$ takes form%
\[
\cal{D}=\left[
\begin{array}{cc}
0 & d/dr-a \\
-d/dr-a & 0%
\end{array}%
\right]
\]%
It has the following domain of definition $\{f_{1}(r),f_{2}(r)\in
L^{2}(\mathbb{R}^{+})\times L^{2}(\mathbb{R}^{+})\}$, $f_{1(2)}$ are
absolutely continuous, $%
f_{2}'-af_{2},\ f_{1}'+af_{1}\in L^{2}(\mathbb{R}^{+})$, $f_{2}(0)=0$.
Consider operator
\[
\cal{D}^{2}=\left[
\begin{array}{cc}
\cal{H}_{1} & 0 \\
0 & \cal{H}_{2}%
\end{array}%
\right]
\]%
where
\[
\cal{H}_{1}=-\frac{d^{2}}{dr^{2}}+q_{1},\ f_{1}^{\prime }(0)+a(0)f_{1}(0)=0
\]

\begin{equation}
\cal{H}_{2}=-\frac{d^{2}}{dr^{2}}+q_{2},\ f_{2}(0)=0
\label{gamilt2}
\end{equation}
and potentials are%
\[
q_{1}=a^{2}-a^{\prime },\ q_{2}=a^{2}+a^{\prime }
\]%
Obviously, $\cal{H}_{1}$ and $\cal{H}_{2}$ are different only by
the order in factorization: $\cal{H}_{1}=\cal{O}^{\ast }\cal{O},\
\cal{H}_{2}=\cal{O}\cal{O}^{\ast }$, where $\cal{O}=-d/dr-a$ is
formal differential expression.

The following is true
\begin{eqnarray*}
-\frac{d^{2}\varphi}{dr^{2}}+q_{1}\varphi &=&\lambda ^{2}\varphi
,\ \varphi (0,\lambda )=1,\varphi
^{\prime }(0,\lambda )=-a(0)\varphi (0,\lambda ) \\
-\frac{d^{2}\psi}{dr^{2}}+q_{2}\psi  &=&\lambda ^{2}\psi ,\ \psi
(0,\lambda )=0,\psi ^{\prime }(0,\lambda )=\lambda
\end{eqnarray*}
That means $\varphi$ and $\psi$ are generalized eigenfunctions for
$\cal{H}_1$ and $\cal{H}_2$, respectively.

Since $\cal{D}(\varphi(r,0),\psi(r,0))^t=0$, we get (compare with
Lemma \ref{real-a})
 \begin{equation}\label{potential-a}
 \psi(r,0)=0, \varphi(r,0)=\exp\left(-\int\limits_0^r a(t)dt\right)
 \end{equation}

\bigskip {\bf Definition 1.5.} The non-decreasing function $\rho _{(h)}(E)$
is called a spectral measure for the general Schr\"odinger operator
$\cal{H}_1=$ $%
-d^{2}/dr^{2}+q\ $with mixed boundary condition $f'(0)=hf(0)$ if
the following is true. For any $f(r)\in L^{2}(\mathbb{R}^{+})$, we
have
\[
\int\limits_{0}^{\infty }|f(r)|^{2}dr=\int\limits_{-\infty }^{\infty
}\left\vert \int\limits_{0}^{\infty }f(r)\varphi _{(h)}(r,E)dr\right\vert
^{2}d\rho _{(h)}(E)
\]%
where $\varphi_{(h)} (r,E )$ is the generalized eigenfunction,
i.e. for any $E\in \mathbb{R}$
\begin{equation}
-\varphi_{(h)}^{\prime \prime }+q\varphi _{(h)}=E\varphi _{(h)},\
\varphi
_{(h)}(0,E)=1,\varphi _{(h)}^{\prime }(0,E)=h \label{problem}
\end{equation}
If $h=0$, we get the Neumann boundary condition.

\begin{definition}
The non-decreasing function $\rho_{(\infty)}(E)$  is
called a spectral measure for the general Schr\"odinger operator
$\cal{H}_{2}$
with the
Dirichlet boundary condition if the following is true. For any $f(r)\in
L^{2}(\mathbb{R}^{+})$, we have
\[
\int\limits_{0}^{\infty }|f(r)|^{2}dr=\int\limits_{-\infty }^{\infty
}\left\vert \int\limits_{0}^{\infty }f(r)\varphi
_{(\infty)}(r,E)dr\right\vert
^{2}d\rho _{(\infty)}(E)
\]%
where $\varphi_{(\infty)}(r,E )$ is the generalized eigenfunction,
i.e.
\[
-\varphi_{(\infty)}^{\prime \prime
}+q\varphi_{(\infty)}=E\varphi_{(\infty)},\
\varphi_{(\infty)}(0,E)=0,\varphi_{(\infty)}^{\prime }(0,E)=1
\]
\end{definition}
The spectral measure for the Schr\"odinger operator with locally
integrable potential always exists. But it is not necessarily
unique \cite{Naimark, Levitan}.

One can use Lemma \ref{number2} to prove the following Theorem. We
consider the usual normalization of the measure $d\sigma$ by
saying that the function $\sigma$ is odd and
$\sigma(\lambda)=(\sigma(\lambda-0)+\ \sigma(\lambda+0))/2$.

\begin{theorem}
If $\rho _{1(2)}$ are spectral measures for
the operators $\cal{H}_{1(2)}$,  then
\begin{equation}
\rho _{1}(\lambda )=\left\{
\begin{array}{cc}
4\sigma (\sqrt{\lambda }), & \lambda \geq 0 \\
0, & \lambda <0%
\end{array}%
\right.,\quad
\rho _{2}(\lambda )=\left\{
\begin{array}{cc}
\displaystyle 4\int\limits_{0}^{\sqrt{\lambda }}\xi ^{2}d\sigma (\xi ), &
\lambda \geq 0
\\
0, & \lambda <0%
\end{array}%
\right.\label{number3}
\end{equation}
\label{number4}
\end{theorem}

\begin{proof} Indeed, take any function $f(x)\in L^{2}(\mathbb{R}^{+})$. Let $%
f(-x)=f(x), x>0$ and consider $f(x)$ on the whole line. We have%
\[
{[\cal{W}f](\lambda )}=\int\limits_{-\infty }^{\infty }{\cal
E}(x,\lambda )f(x)dx=2\int\limits_{0}^{\infty }\varphi (x,\lambda
)f(x)dx
\]%
Therefore, the first formula of (\ref{number3}) is straightforward
due to Lemma \ref{number2}. To get an expression for $\rho
_{2}(\lambda )$, one should take the odd continuation of function
$f(x)$. \end{proof}

Notice that both $\rho _{1(2)}(\lambda )$ are constants for
$\lambda <0$. That means $\cal{H}_{1(2)}$ are both nonnegative
operators. That is not surprising since $\cal{D}^{2}\geq 0$ and
$\cal{D}^{2}$ is decoupled into the direct sum of $\cal{H}_{1}$
and $\cal{H}_{2}$. Formula for $\rho _{2}(\lambda )$ shows that
not only $\rho_{2}(\lambda)$ has no jump at $0$ (no zero
eigenvalue for $\cal{H}_2$), but also it decays at $0$ in a
certain way.

Consider the so-called ``free" case, i.e.  $a(r)=0$. Then,
$A(r)=0$, $d\sigma =d\lambda /(2\pi )$, $q_{1}=q_{2}=0,$ $\varphi
(x,\lambda )=\sin (\lambda x),$ $\psi (x,\lambda )=\cos (\lambda
x)$. Moreover, $\rho _{1}(\lambda )=2\lambda ^{1/2}/\pi $ on
$\mathbb{R}^{+}$ is the standard spectral measure for the free
Schr\"odinger operator on $\mathbb{R}^{+}$ with the
Neumann boundary condition and $%
\rho _{2}(\lambda )=2\lambda ^{3/2}/(3\pi )$ on $\mathbb{R}^{+}$ is the
spectral measure
for the free Schr\"odinger operator on $\mathbb{R}^{+}$ with the Dirichlet
boundary condition.

\bigskip Let us assume we are given nonnegative self-adjoint Schr\"odinger operator $%
\cal{H}_h=-d^{2}/dr^{2}+q$ with mixed boundary condition
$f'(0)=hf(0)$ at zero and $q\in L^1_{\rm{loc}}(\mathbb{R}^+)$.
Denote its spectral measure by $d\rho_h(E) $. Then, there is the
unique Dirac operator and Krein system that generate this
Schr\"odinger operator in a way described above. Indeed, define
the functions%
\begin{equation}\label{shred-krein}
a(r)=-\frac{\psi ^{\prime }(r)}{\psi (r)}, A(r)=a(r/2)/2
\end{equation}
where $\psi (r)$ is solution to the equation $-\psi ^{\prime
\prime }+q\psi =0,$ $\psi ^{\prime }(0)=h$, $\psi (0)=1$. We have
$\psi''(r)\in L^1_{\rm{loc}}(\mathbb{R}^+)$ and $\psi(r)>0$ by the
oscillation theory for Sturm-Liouville operators (\cite{Weidmann},
p. 218). Therefore, $a(r)$ is absolutely continuous on
$\mathbb{R}^+$ and $a(r)$ satisfies the Riccati equation
\begin{equation}
q=a^{2}-a^{\prime } \label{Riccati}
\end{equation}
 Let $A(r)=a(r/2)/2$. Notice that
$A(r)$ is absolutely continuous. Krein system with coefficient
$A(r)$ generates $\cal{H}_h$. Let us assume that there are two
different $A_{1}(r)$ and $A_{2}(r)$ that generate the same
$\cal{H}_h$. Then, due to Theorem \ref{number4}, measures
$\sigma_1$ and $\sigma_2$ of these Krein systems are the same.
Then, $A_{1}=A_{2}$. Notice that $A(0)=-h/2$. Thus, we proved
\begin{lemma}\label{mixed}
The operator $\cal{H}_h$ is generated by the Krein systems if and
only if $\cal{H}_h$ is nonnegative. The Krein system is unique and
$A(r)$ is given by (\ref{shred-krein}).
\end{lemma}

In general, Riccati equation (\ref{Riccati}) has many solutions.
For instance, if $q=0$, the general solution is given by
\begin{equation}
a_{h }(r)=-h(1+rh)^{-1}, h\geq 0 \label{example}
\end{equation}
But $a_{h}(0)=-h$ so they can be all distinguished by the value at
zero. If one considers $q=0$ and boundary condition $f^{\prime
}(0)=hf(0),$ then the corresponding operator is nonnegative for
$h\geq 0$ only. In this case, $\psi(r)=hr+1$ and the formula
(\ref{shred-krein}) gives exactly  $a_h (r)$.

The case of a Dirichlet boundary conditions is a bit subtle.
Consider positive $\cal{H}_\infty=-d^{2}/dr^{2}+q$ with Dirichlet
boundary condition at zero and locally summable potential $q$. The
problem here is that we don't know $a(0)$, the initial condition
for solving Riccati
equation%
\[
q=a^{2}+a^{\prime }
\]
Interestingly enough, we might have many Krein systems that
correspond to the same Schr\"odinger operator $\cal{H}_\infty$.
For instance, all $A_h(r)=-a_h(r/2)/2$ (formula (\ref{example}))
generate the same Schr\"odinger operator with Dirichlet boundary
condition and $q=0$. Different $A_h$ have different $\sigma_h$.
Due to (\ref{number3}), these measures are different by the jump
at zero only. Let us consider $\sigma (\lambda )=\sigma
_{0}(\lambda )+h\theta (\lambda )/2$, where $\theta (\lambda )$ is
Heaviside function, $h/2\geq 0$ is the jump at zero. Then, the
corresponding $H(x)=$ $h/2$ for all $x$.
Given any $%
r>0$, solution of equation (\ref{basic2}) is $\Gamma
_{r}(s,t)=h(2+hr)^{-1}$. Then, $%
A(r)=h(2+hr)^{-1}=A_h(r)$.

Thus, the natural questions are when is $\cal{H}_\infty$ generated
by Krein system and how to describe all Krein systems that give
rise to $\cal{H}_\infty$?
\begin{lemma}\label{dirichlet}
The operator $\cal{H}_\infty$ is generated by some Krein system if
and only if there is some $\psi(r)>0$ for $r\geq 0$ such that
$-\psi ^{\prime \prime }+q\psi =0$. Moreover, if $A(r)$ is
coefficient of Krein systems generating $\cal{H}_\infty$, then
\begin{equation}\label{krein-shredinger}
\quad A(r)=a(r/2)/2,\,{\it where}\quad
a(r)=\frac{\psi'(r)}{\psi(r)}
\end{equation}
with some $\psi(r)$ satisfying the properties given above.
\end{lemma}
\begin{proof}
If there is some positive $\psi(r)$ satisfying equation, then the
Krein system can be easily constructed by letting
\[
a(r)=\frac{\psi'(r)}{\psi(r)},\quad A(r)=a(r/2)/2
\]
Conversely, assume that there is at least one $A(r)$ generating
$\cal{H}_\infty$. Then, the dual system with coefficient $-A(r)$
will generate Dirac operator $\cal{D}_-$. Let
$\varphi_-(r,\lambda),\psi_-(r,\lambda)$ be the corresponding
generalized eigenfunctions. Then, (\ref{potential-a}) yields
\[
a(r)=\varphi_-'(r,0)/\varphi_-(r,0), A(r)=a(r/2)/2
\]
and $\varphi(r,0)$ satisfies $-\varphi''(r,0)+q(r)\varphi(r,0)=0$.
\end{proof}
There are several different ways to reformulate this criteria. For
instance, solution $\psi(r,0)$ mentioned above exists if and only
if the corresponding operator with the boundary condition
$f'(0)=h_0f(0)$ is non-negative for some $h_0$. Notice that if
that is true for $h_0$, then it must be true for any $h>h_0$.
One can also easily state this criteria in terms of the spectral measure $%
\rho_2 $. If $\sigma $, obtained from the formula (\ref{number3}),
generates a Krein systems, then this Krein system generates
$\cal{H}_\infty$.

In view of these two Lemmas, one can suggest the following
reduction of Schr\"odinger operator to the Krein system. Let
$\cal{H}_h$ be any Schr\"odinger operator {\bf bounded from
below}. Add large positive number $\gamma$ to $\cal{H}_h$ so that
it becomes strictly positive. This transformation simply moves the
spectrum to the right and does not change the spectral types. For
$\cal{H}_h+\gamma$, the Lemma \ref{mixed} is applicable. For
$\cal{H}_\infty$, the algorithm is the same but one has to apply
Lemma \ref{dirichlet}.

Now, let us briefly discuss the solution to the inverse problem
for Shr\"odinger operators. For $\cal{H}_h$ bounded from below,
the problem can be reduced to the inverse problem for Krein system
which we know how to solve (just follow the construction in the
first Sections). Assume that we are given the spectral measure
$\rho $ of Schr\"odinger operator $\cal{H}$ bounded from below.
Assume also that the potential $q$ we want to find is continuous.
Then, the asymptotics of $\rho $ at infinity \cite{Levitan} is
\begin{eqnarray*}
\rho (\lambda ) &=&2\lambda ^{3/2}/(3\pi )+\bar{o}(\lambda ^{1/2}),{\rm {%
\, for \, the \, Dirichlet \, b.c.}} \\ \rho (\lambda )
&=&2\lambda ^{1/2}/\pi -h +\bar{o}(1),{\rm \, for\, mixed\,
b.c.\,} f^{\prime }(0)=h f(0)
\end{eqnarray*}%
From this asymptotics, we can find the corresponding boundary
condition and apply one of the algorithms discussed above to find
measure $\sigma $ for one of the Krein systems, generating
$\cal{H}$. Once $\sigma$ is known, we
can find $%
A, $ then $a,$ and, finally, $q$. Notice also that this method
gives a one-to-one correspondence between all Schr\"odinger
operators, bounded from below, and spectral measures that yield
accelerant: $H\in C^{m+1}(\mathbb{R}^+)$ iff the potential $q\in
C^{m}(\mathbb{R}^{+})$, $m$ is an integer. An accelerant $H$ is
a.c. on $\mathbb{R}^{+}$ iff $q\in
L_{\rm{loc}}^{1}(\mathbb{R}^{+})$.

There is a direct way of solving the inverse spectral problem for
Schr\"odinger operators. This method is due to Gelfand and Levitan
\cite{GelfLev}. Let us discuss this method and compare it to
Krein's approach. Consider the operator (\ref{problem}) with a
continuous potential. Assume that its spectral measure
$\rho_{(h)}$ is given. Define
\[
\beta(\lambda)=\left\{
\begin{array}{cc}
\rho_{(h)}(\lambda)-2\lambda^{1/2}/\pi,& \lambda\geq 0 \\
\rho_{(h)}(\lambda),& \lambda<0
\end{array}
\right.
\]
and
\[
F(x)=\lim_{n\to\infty}
\int\limits_{-\infty}^n \cos(\lambda^{1/2}x)
d\beta(\lambda), \, x\geq 0
\]
It turns out that the limit exists and $F(x)$ is continuously
differentiable in $x$. Consider
\[
F(x,y)=\frac{F(x+y)+F(x-y)}{2}
\]
and an integral equation
\begin{equation}
K(x,y)+F(x,y)+\int\limits_0^x K(x,t)F(t,y)dt=0,\quad 0\leq y\leq
x<\infty \label{urav}
\end{equation}
One can prove that the solution $K(x,y)$ exists and is unique.
Then, the following relations solve the inverse problem.
\begin{equation}
h=K(0,0)=-F(0,0),\quad
q(x)=2\frac{d}{dx}K(x,x)
\end{equation}

Now, let us make an assumption that the operator $\cal{H}$ is
nonnegative. Then, we can find the unique Krein system that
generates $\cal{H}$. Take $\sigma(\lambda)=\rho(\lambda^2)/4$. For
an accelerant, we have the formal representation
\[
H(x)=\int\limits_{-\infty}^\infty \cos(\lambda x)d\left(
\sigma(\lambda)-\lambda/(2\pi)\right)=\frac{1}{2}
\int\limits_0^\infty \cos(\sqrt{\mu} x)d\left(
\rho(\mu)-2\mu^{1/2}/\pi\right)=F(x)/2
\]
One can check that
\begin{equation}
-K(x,y)=\Gamma_{2x} (x+y,0)+\Gamma_{2x} (x-y,0)
\label{eeer}
\end{equation}
Indeed, from (\ref{basic2}), we have the following identities
\[
\Gamma_{2x} (x+y,0)+\int\limits_0^x
H(y-u)
\Gamma_{2x}(x+u,0)du
+\int\limits_0^x H(y+u) \Gamma_{2x}
(x-u,0) du=H(x+y)
\]
\[
\Gamma_{2x} (x-y,0)+\int\limits_0^x H(-y-u) \Gamma_{2x}(x+u,0)du
+\int\limits_0^x H(u-y) \Gamma_{2x}
(x-u,0)du =H(x-y)
\]
Since $A(r)$ is real-valued, $H(x)$ is a real-valued, even function.
So, by adding the last two formulas, we get
\[
[\Gamma_{2x} (x+y,0)+\Gamma_{2x} (x-y,0)] + \int\limits_0^x \left[
H(y-u)+H(y+u)\right]
\left[ \Gamma_{2x}
(x-u,0)+\Gamma_{2x} (x+u,0) \right]du=
\]
\[
=H(x+y)+H(x-y)
\]
Therefore, we have (\ref{eeer}), and  $K(x,x)=-[\Gamma_{2x}
(2x,0)+\Gamma_{2x} (0,0)]$.
Taking the derivative, we obtain
\[
2\frac{d}{dx}K(x,x)=2(-A'(2x)+2A^2(2x))=a^2-a'=q
\]
and
\[
f'(0)/f(0)=-a(0)=-2A(0)=K(0,0)
\]
Thus, Krein's approach gives the same answer and these two methods
are essentially identical. The difference is that imposing
condition on (\ref{urav}) to have the unique solution is weaker
than saying that $F(x)/2$ is an accelerant. That allowed authors
of \cite{GelfLev} to deal with a more general situation. Notice
that the inverse problem for the Krein system is in fact the
problem of the factorization for integral operators. Indeed, given
measure $\sigma $, we construct the accelerant. For any $r>0$,
operator $1+{\cal H}_{r}>0$. Therefore, $\Gamma _{r}(x,y)$ exists.
To find it, we need to factorize $(1+{\cal H}_{r})^{-1}$. Once we
do that, $A(r)=\Gamma _{r}(0,r)=-V_{+}(0,r)$ by (\ref{kernels}).

Let us consider the scattering theory for Schr\"odinger operator
$\cal{H}_2$ defined by (\ref{gamilt2}). We assume that $a$ has a
compact support that belongs to, say, $[0,R]$. That means
potential $q_2$ has a compact support too. Consider the Jost
solution $F(r,\lambda): \cal{H}_2 F=\lambda^2 F$ defined by its
asymptotics at infinity: $F(r,\lambda)=\exp(i\lambda r), r>R$. Let
us introduce the scattering data
\begin{eqnarray*}
A_s(\lambda)=(F'(0,\lambda)+i\lambda F(0,\lambda))/(2i\lambda),
B_s(\lambda)=(i\lambda F(0,\lambda)-F'(0,\lambda))/(2i\lambda),\\
T_s(\lambda)=A_s(\lambda)^{-1},
R_s(\lambda)=B_s(\lambda)/A_s(\lambda)
\end{eqnarray*}
Function $F(0,\lambda)$ is called the Jost function for
Schr\"odinger operator. The next Lemma relates scattering and
spectral data for Schr\"odinger operator with Dirichlet boundary
conditions to the corresponding parameters of Krein system. One
should remember that we deal not with arbitrary Schr\"odinger but
with the one generated by Krein system with compactly supported
absolutely continuous coefficient.

\begin{lemma}
The following relations hold true
\begin{eqnarray*}
A_s(\lambda)=\A(\lambda)+\frac{a(0)(\A(\lambda)+\B(\lambda))}{2i\lambda},
B_s(\lambda)=\B(\lambda)-\frac{a(0)(\A(\lambda)+\B(\lambda))}{2i\lambda},\\
R_s(\lambda)=\f(\lambda)-\frac{a(0)(1+\f(\lambda))^2}{2i\lambda
+a(0)(1+\f(\lambda))},
F(0,\lambda)=A_s(\lambda)+B_s(\lambda)=\A(\lambda)+\B(\lambda)=\Pi(\lambda),\\
\rho_2'(\lambda^2)/\lambda=\frac{1}{\pi
|F(0,\lambda)|^2}=\frac{1}{\pi |\Pi(\lambda)|^2}=2\sigma'(\lambda)
\end{eqnarray*}
\end{lemma}
\begin{proof} The proof is a direct corollary from Lemma \ref{dirac-scat}.
Notice that if $a(0)=0$, the data coincide with main parameters in
Krein system.
\end{proof}
Simple calculations show that
$|A_s(\lambda)|^2=1+|B_s(\lambda)|^2$ if $\lambda\in
\mathbb{R}\backslash\{0\}$. That follows from the identity
$|\A(\lambda)|^2=1+|\B(\lambda)|^2$ which holds for any
$\lambda\in \mathbb{R}$. These formulas also show that the
analytical properties of $A_s(\lambda)$ and $R_s(\lambda)$ are
worse than of the analogous functions for Dirac operator or Krein
system. For instance,  $A_s(\lambda)$ has pole at zero iff
$a(0)\neq 0$. Moreover, \[A_s(\lambda)=
\frac{\Pi(\lambda)}{2}\left[
\frac{a(0)}{i\lambda}+1+F(\lambda)\right]
\]
where $F(\lambda)=\widehat{\Pi}(\lambda)\Pi^{-1}(\lambda)$ is
Weyl-Titchmarsh function for Krein system. Thus, we see that
$A_s(\lambda)$ might also have zeroes in $\mathbb{C}^+$. These
zeroes must be purely imaginary. Indeed, if $\lambda_0$ is such a
zero, then $\lambda_0^2$ is eigenvalue for the Schr\"odinger
operator considered on the whole line, a selfadjoint operator.
Therefore, $\lambda_0^2<0$ and $\lambda_0$ is purely imaginary.

\bigskip
{\bf Remarks and historical notes.}

For the first time, the inverse spectral problem for the
Schr\"odinger operator was solved by Gelfand and Levitan
\cite{GelfLev}. Their approach is applicable to any operator, not
necessarily bounded from below. In the recent paper
\cite{Simon-4}, Simon essentially introduced an ``accelerant"
directly for the Schr\"odinger operator. Different factorizations
of Schr\"odinger operators and applications were discussed in many
papers (see, for instance, \cite{Deift}). These methods allow one
to insert eigenvalues below the essential spectrum, for instance.

\newpage
\section{Scattering theory for Krein systems}
In this section, we will consider the scattering theory for Krein
systems and Dirac operators from slightly different perspective.
The strategy is close to what is best known as approach by
Marchenko and Agranovich to solution of inverse scattering problem
for Sturm-Liouville operators \cite{Marchenko}. We will try to
emphasize the algebraic aspect of this argument, i.e. why Hankel
operators appear and how their inversion is related to scattering
data. Let us start with Krein systems. For simplicity, assume that
the coefficient $A(r)$ is finitely supported within, say, interval
$[0,R]$ and is continuous on this interval. Then, clearly, there
is a unique solution $X_{\rm sc}(r,\lambda)$ such that $X_{\rm
sc}(r,\lambda)=X_0(r,\lambda)$, if $r>R$ where
\[
X_0(r,\lambda)=\left[
\begin{array}{cc}
e^{i\lambda r} & 0\\
0 & 1
\end{array}
\right]
\]
is the fundamental solution for $A(r)=0$. This solution $X_{\rm
sc}(r,\lambda)$ is normalized at infinity by its asymptotical
behavior. We will study this solution and the scattering data it
defines. Then, we will find its relation to spectral data and show
how to solve an inverse scattering problem. This construction will
be valid for more general case $A(r)\in L^1(\mathbb{R}^+)$. In the
meantime, let us assume first that $A(r)$ is compactly supported
and do some preliminary calculations. Clearly, if $X_{\rm sc}$ is
given by its value at $r=R$ (i.e., $X_{\rm sc}(R)=X_0(R)$), we
might try to study it by solving the Krein system backwards, from
$R$ to $0$. That is equivalent to dealing with ``mirrored"
coefficient $A^{(R)}(r)=A(R-r)\cdot \chi_{[0,R]}(r)$. Therefore,
we can study the following problem: consider the Krein system with
coefficient $A_1(r)$, it will be later taken equal to
$-A^{(R)}(r)$ but so far it is an arbitrary function with support
within $[0,R]$.  For this $A_1(r)$, consider the following
solution
\[
Y(r)=X(r)\left[
\begin{array}{cc}
e^{-i\lambda R} & 0\\
0 & 1
\end{array}
\right]
\]
where $X$ is  the fundamental solution for $A_1(r)$ normalized by
$X(0)=I$. Notice that
\begin{equation}
X_1(r,\lambda)=Y(R-r,-\lambda) \label{mirror-eq}
\end{equation}
 satisfies the following
properties: $X_1'(r,\lambda)=V_{\{\lambda,
-A_1^{(R)}\}}X_1(r,\lambda)$ and $X_1(R,\lambda)=X_0(R,\lambda)$.
Thus, $X_1(r,\lambda)$ is the scattering solution for $-A_1^{(R)}$
and so if one wants to study the scattering problem for $A(r)$, we
just need to study $Y(r)$ associated to $A_1(r)=-A^{(R)}(r)$ and
then use the formula (\ref{mirror-eq}).

Now, let us start with obtaining the formulas for elements of the
matrix $Y(r,\lambda)$. To simplify the calculations, take
\[
Y_1(r,\lambda)=Y(r,\lambda)\left[
\begin{array}{cc}
e^{i\lambda (R-r)} & 0\\
0 & 1
\end{array}\right]
=X(r)\left[
\begin{array}{cc}
e^{-i\lambda r} & 0\\
0 & 1
\end{array}
 \right]
\]
 We have
\footnote{In the calculations below we assume that $\lambda$ is
real.}
\begin{equation}
Y_1(r,\lambda)=\left[
\begin{array}{cc}
\overline{{\A}(r,\lambda)} & e^{i\lambda r}\overline{{\B}(r,\lambda)} \\
e^{-i\lambda r}{\B}(r,\lambda) & {\A}(r,\lambda)
\end{array}
\right]\label{formula-y1}
\end{equation}
Let us write the following identity which follows from
(\ref{simplect})
\[
 Y_1(r,\lambda)\cdot  \left[
\begin{array}{cc}
{\A(r,\lambda)} & -e^{i\lambda r}\overline{\B(r,\lambda)} \\
-e^{-i\lambda r}\B(r,\lambda) & \overline{\A(r,\lambda)}
\end{array}
\right]=I
\]
or
\begin{equation}
Y_1(r,\lambda)\cdot \left[
\begin{array}{cc}
1 & \overline{s_r(\lambda)} \\
{s_r(\lambda)} & 1
\end{array}
\right]=\left[
\begin{array}{cc}
\A^{-1}(r,\lambda) & 0 \\
0 & \overline{\A(r,\lambda)}^{-1}
\end{array}
\right]\label{identity-hankel}
\end{equation}
where $s_r(\lambda)=-e^{-i\lambda r}
{\B(r,\lambda)}\A^{-1}(r,\lambda)$.  Let
\[
Y_1(r,\lambda)=I+\int\limits_0^r K^{(r)}(s,0) \left[
\begin{array}{cc}
 e^{-i\lambda s} & 0 \\
0 & e^{i\lambda s}
\end{array}
\right]  ds, K^{(r)}(s,0)=\left[
\begin{array}{cc}
\overline{K_1^{(r)}(s,0)} & \overline{K_2^{(r)}(s,0)} \\
{K_2^{(r)}(s,0)} & {K_1^{(r)}(s,0)}
\end{array}
\right]
\]
\[
\frac{\B(r,\lambda)}{\A(r,\lambda)}=\int\limits_0^\infty
C_r(x)e^{i\lambda x}dx, s_r(\lambda)=-\int\limits_{-r}^\infty
C_r(x+r)e^{i\lambda x}dx, J_r(x)=-C_r(r-x)\cdot \chi_{[0,r]}(x)
\]
Subtract the identity matrix from both sides of
(\ref{identity-hankel}), take adjoint of the matrices on both
sides, and act by the operator
\[
\left[
\begin{array}{cc}
 \cal{P}_+ & 0 \\
0 & \cal{P}_-
\end{array}
\right]
\]
from the left. We recall that $\cal{P}_{\pm}$ are projections from
$L^2(\mathbb{R})$ onto $H^2(\mathbb{R})$ and
$\overline{H^2(\mathbb{R})}$. Then, we have the following integral
equation for matrix $K^{(r)}(x,0)$, $x\in [0,\infty)$
\begin{equation} \left[
\begin{array}{cc}
 I & \cal{G}_r^* \\
\cal{G}_r & I
\end{array}
\right] K^{(r)*}(\cdot,0)+\left[
\begin{array}{cc}
 0 & \overline{J_r(x)} \\
{J_r(x)} & 0
\end{array}
\right]=0 \label{kernel-hankel}
\end{equation}
and an operator
\[
[\cal{G}_r f](x)=\int\limits_0^\infty J_r(x+u) f(u)du
\]
is acting in $L^2(\mathbb{R}^+)$. Notice that $K^{(r)}(x,0)=0$ if
$x>r$. Consider an operator $\cal{K}^{(r)}$ in
$[L^2(\mathbb{R}^+)]^2$ such that
\[
\left[I+\left[
\begin{array}{cc}
 0 & \cal{G}_r^* \\
\cal{G}_r & 0
\end{array}
\right]\right](I+\cal{K}^{(r)*})=I
\]
Operator $\cal{G}_r$ is contractive Hankel operator since
$\B(r,\lambda)/\A(r,\lambda)$ is analytic contraction. Therefore,
$\cal{K}_r$ exists, is self-adjoint,  and has matrix-valued kernel
$K^{(r)}(x,y)$ such that $K^{(r)}(x,0)$ is exactly the solution of
(\ref{kernel-hankel}). Now, let us compare these equations for
different $r\in [0,R]$. Notice that $\B(R,\lambda)/\A(R,\lambda)$
is the Schur function for $A_1(r)$. Denote
\begin{equation}
 C(x)=C_R(x), J(x)=-C(R-x)\cdot
\chi_{[0,R]}(x)\label{pereschet1}
\end{equation}

 As it follows from (\ref{param-sh}) and
(\ref{sasp}), we have $C_r(x)=C(x)$ for $x\in [0,r]$. Therefore,

\begin{equation}
J_r(x)=J(x+(R-r)) \label{pereschet2}
\end{equation}
 and $K^{(r)}(x,0)$ can be obtained by inverting
the matrix-valued Hankel operator. On the other hand, it can be
expressed through kernels $\Gamma, \widehat\Gamma$ by formula
(\ref{formula-y1}). We then have
\[
K^{(r)}(x,0)=-\frac 12 \left[
\begin{array}{cc}
 {\Gamma_r(x,0)}+{\widehat{\Gamma}_r(x,0)} & {\Gamma_r(r,x)}-{\widehat\Gamma_r(r,x)} \\
 {\Gamma_r(x,r)}-{\widehat\Gamma_r(x,r)}
&\Gamma_r(0,x)+{\widehat{\Gamma}_r(0,x)}
\end{array}
\right]
\]
Theis formula allows to find $A_1(r)$ as off-diagonal elements in
the matrix
\[
K^{(r)}(0,0)=-\frac 12 \left[
\begin{array}{cc}
 {\Gamma_r(0,0)}+{\widehat{\Gamma}_r(0,0)} & 2\overline{A_1(r)} \\
2A_1(r)&\Gamma_r(0,0)+{\widehat{\Gamma}_r(0,0)}
\end{array}
\right]
\]
Now, we are ready to translate these calculations to the original
setting for coefficient $A(r)$. Take $A_1(r)=-A^{(R)}(r)$ and then
use the formula (\ref{mirror-eq}). We then have
\[
X_{\rm sc}(r,\lambda)= Y_1(R-r,-\lambda) \left[
\begin{array}{cc}
 e^{i\lambda r/2} & 0 \\
0 & e^{-i\lambda r/2}
\end{array}
\right] e^{i\lambda r/2}=e^{i\lambda r/2}Z(r,\lambda)
\]
For $Z(r,\lambda)$\footnote{The motivation to introduce
$Z(r,\lambda)$ comes from the fact that it is the scattering
solution for Dirac operator, that will be made clear later.}
\[
Z(r,\lambda)=\left[
\begin{array}{cc}
 e^{i\lambda r/2} & 0 \\
0 & e^{-i\lambda r/2}
\end{array}
\right]+\int\limits_{r/2}^{R-r/2} K^{(R-r)}(s-r/2,0)\left[
\begin{array}{cc}
 e^{i\lambda s} & 0 \\
0 & e^{-i\lambda s}
\end{array}
\right]ds
\]
\[
=\left[
\begin{array}{cc}
 e^{i\lambda r/2} & 0 \\
0 & e^{-i\lambda r/2}
\end{array}
\right]+\int\limits_{r/2}^{\infty} L^{(r/2)}(s,r/2)\left[
\begin{array}{cc}
 e^{i\lambda s} & 0 \\
0 & e^{-i\lambda s}
\end{array}
\right]ds
\]
where $L^{(r/2)}(x,y)=K^{(R-r)}(x-r/2,y-r/2); \quad x,y>r/2$
satisfies
\[
L^{(r/2)}(x,y)+\int\limits_{r/2}^\infty \left[
\begin{array}{cc}
 0 & \overline{J_{R-r}}(x+u-r) \\
J_{R-r}(x+u-r) & 0
\end{array}
\right]L^{(r/2)}(u,y)du
\]
\[
+\left[
\begin{array}{cc}
 0 & \overline{J_{R-r}}(x+y-r) \\
J_{R-r}(x+y-r) & 0
\end{array}
\right]=0;\quad x,y>r/2
\]
Now, let us obtain the formula for $J_{R-r}(x+y-r)$. Notice that
the Schur function for the coefficient $A_1(r)=-A^{(R)}(r)$ is
equal to $-{\B(R,-\lambda )}/{\A^{\ast }(R,-\lambda )}$ (Lemma
\ref{mirror}). If
\[
\frac{\B(R,\lambda)}{\overline{A(R,\lambda)}}=\int_{\mathbb{R}}D(x)e^{i\lambda
x}dx
\]
then $J_{R-r}(x+y-r)=D(x+y)$ by (\ref{pereschet1}) and
(\ref{pereschet2})\footnote{Notice also that $D(x)=0$ for $x>R$.}.
Therefore, $L^{(r/2)}(x,y)$ is an integral kernel of the operator
$\cal{L}^{(r/2)}$ in $\left[L^2[r/2,+\infty)\right]^2$ given by
\[
\left[I+\left[
\begin{array}{cc}
 0 & \cal{D}_{r/2}^* \\
\cal{D}_{r/2} & 0
\end{array}
\right]\right](I+\cal{L}^{(r/2)})=I
\]
and
\[
[\cal{D}_{r/2} f](x)=\int\limits_{r/2}^\infty D(x+u) f(u)du
\]
Since $|{\B(R,\lambda)}/{\overline{A(R,\lambda)}}|<1$,
$\|\cal{D}_\rho\|<1$ for any $\rho>0$. Therefore,
\[
\left[I+\left[
\begin{array}{cc}
 0 & \cal{D}_{r/2}^* \\
\cal{D}_{r/2} & 0
\end{array}
\right]\right]
\]is invertible.
Also, $\cal{D}_{r/2}=0$ if $r>R$ so $\cal{L}^{(r/2)}=0$ for $r>R$.

Now, we want to take $R\to\infty$. For this, we need some
regularity at infinity, just like we needed regularity of, say
accelerant near zero in the previous constructions. The simplest
and quite natural class of functions to consider is
$L^1(\mathbb{R}^+)\cap C_0(\mathbb{R}^+)$.

\begin{theorem}
Let $A(r)\in L^1(\mathbb{R}^+)\cap C_0(\mathbb{R}^+)$ be the
coefficient in Krein's system. Then, there exists the scattering
solution $X_{\rm sc}(r,\lambda)$ such that
\begin{equation}\label{normirovka}
X_{\rm sc}(r,\lambda)=X_0(r,\lambda)+\bar{o}(1)
\end{equation}
 as $\lambda\in
\mathbb{R}, r\to\infty$. This solution can be obtained as follows

\[
X_{\rm sc}(r,\lambda)= e^{i\lambda r/2}Z(r,\lambda),
Z(r,\lambda)=\left[
\begin{array}{cc}
 e^{i\lambda r/2} & 0 \\
0 & e^{-i\lambda r/2}
\end{array}
\right]+\int\limits_{r/2}^{\infty} L^{(r/2)}(s,r/2)\left[
\begin{array}{cc}
 e^{i\lambda s} & 0 \\
0 & e^{-i\lambda s}
\end{array}
\right]ds
\]
where $L^{(r/2)}(x,y); x,y>r/2$ is the kernel of the operator
$\cal{L}^{(r/2)}:$
\[
\left[I+\mathfrak{D}_{r/2}\right](I+\cal{L}^{(r/2)})=I,\quad
\mathfrak{D}_{r/2}= \left[
\begin{array}{cc}
 0 & \cal{D}_{r/2}^* \\
\cal{D}_{r/2} & 0
\end{array}
\right], [\cal{D}_{r/2} f](x)=\int\limits_{r/2}^\infty D(x+u)
f(u)du
\]

and $D(x)$ is given by
\[
g(\lambda)=\f(\lambda)\frac{\A(\lambda)}{\overline{\A(\lambda)}}=\int\limits_{-\infty}^{\infty}
D(x)e^{i\lambda x}dx, D(x)\in L^1(\mathbb{R})\cap C_0(\mathbb{R})
\]
The coefficient $A(r)$ can be found from the following identity
\begin{equation}
A(r)=L^{(r/2)}_{21}(r/2,r/2) \label{inverse-scat}
\end{equation}
\end{theorem}
\begin{proof}
As it follows from the section on the Baxter Theorem,  Levy-Wiener
Theorem, and Lemma \ref{auxil-scat} $D(x)\in L^1(\mathbb{R}^+)\cap
C_0(\mathbb{R}^+)$. Also, $\cal{D}_\rho$ is contraction in
$L^2[\rho,\infty)$. By the general theory of Hankel operators,
$\cal{D}_\rho$ is compact in any of $L^p[\rho,\infty),
\infty>p\geq 1$ and any eigenvalue in $L^1[\rho,\infty)$ is also
an eigenvalue in $L^2[\rho,\infty)$. This is because $D(x)\in
L^\infty(\mathbb{R})$. In particular, that means
$I+\mathfrak{D}_\rho$ is invertible in $L^1[\rho,\infty)$ and,
therefore, $L^{(\rho)}(x,\rho)\in L^1[\rho,\infty)\cap
C_0(\mathbb{R}^+)$. Moreover,
$\|L^{(\rho)}(x,\rho)\|_{L^1[\rho,\infty)}\to 0$ as
$\rho\to\infty$ which can be checked by simple iterations. Thus,
$X_{\rm sc}(r,\lambda)$ is well-defined and satisfies
(\ref{normirovka}). Now, let us show that it is actually a
solution.

Consider $R$-- any positive number and let $A_R(r)=A(r)\cdot
\chi_{[0,R]}(r)$. It has finite support and therefore the
statement of the Theorem follows from the calculations given
above. If $D_R(x)$ is the corresponding function, then
$\|D_R(x)-D(x)\|_1\to 0$ as $R\to \infty$. That follows from the
proof of Baxter's Theorem. Consequently, $L^{(r)}_R(x,r)\to
L^{(r)}(x,r)$ and $X_{\rm sc}^R(r,\lambda)\to X_{\rm
sc}(r,\lambda)$ as $R\to\infty$ and $r$ is fixed.  On the other
hand,
\[
X_{\rm
sc}^{R}(r,\lambda)=X(r,\lambda)X^{-1}(R,\lambda)X_0(R,\lambda),
r\in [0,R]
\]
Then,
\[
X_{\rm sc}^R(r,\lambda)=X(r,\lambda)\left[
\begin{array}{cc}
\A(R,\lambda) & -\overline{\B(R,\lambda)}\\
-\B(R,\lambda) & \overline{\A(R,\lambda)}
\end{array}
\right], r\in [0,R]
\]
For fixed $r$,
\[
X_{\rm sc}^R(r,\lambda)\to X(r,\lambda)\left[
\begin{array}{cc}
\A(\lambda) & -\overline{\B(\lambda)}\\
-\B(\lambda) & \overline{\A(\lambda)}
\end{array}
\right]
\]
as $R\to\infty$. Therefore,
\[
X_{\rm sc}(r,\lambda)= X(r,\lambda)\left[
\begin{array}{cc}
\A(\lambda) & -\overline{\B(\lambda)}\\
-\B(\lambda) & \overline{\A(\lambda)}
\end{array}
\right]
\]
Consequently, $X_{\rm sc}(r,\lambda)$ is a solution. Relation
(\ref{inverse-scat}) is true for truncated $A(r)$ and therefore
holds after taking $R\to\infty$.
\end{proof}

Notice that
\[
X_{\rm sc}(0,\lambda)=\left[
\begin{array}{cc}
\A(\lambda) & -\overline{\B(\lambda)}\\
-\B(\lambda) & \overline{\A(\lambda)}
\end{array}
\right]
\]
This matrix can be regarded as scattering data for Krein's system.
On the other hand, function $D(x)$ can also be regarded as
scattering data. What is the relation between these functions? If
$\f(\lambda)$ or $\B(\lambda)$ are given, then $|\A(\lambda)|$ can
be found from $|\A(\lambda)|^{-2}=1-|\f(\lambda)|^2$ or
$|\A(\lambda)|^2=1+|\B(\lambda)|^2$, respectively. Then,
$\A(\lambda)$ can be found from $|\A(\lambda)|$ because it is
outer. Function $\B(\lambda)$ can be obtained from
$\B(\lambda)=\f(\lambda)\A(\lambda)$ and $D(x)$ may be recovered
by taking inverse Fourier transform of
$\B(\lambda)/\overline{\A(\lambda)}$.  The converse result can be
obtained from the following calculations which is in the core of
the method. Assume that $A(r)\in L^1(\mathbb{R}^+)$. Then,
\begin{equation}
\left[
\begin{array}{cc}
\A(\lambda) & -\overline{\B(\lambda)}\\
-\B(\lambda) & \overline{\A(\lambda)}
\end{array}
\right] \left[
\begin{array}{cc}
1 & \overline{g(\lambda)}\\
g(\lambda) & 1
\end{array}
\right]=\left[
\begin{array}{cc}
\overline{\A^{-1}}(\lambda) & 0\\
0 & \A^{-1}(\lambda)
\end{array}
\right]\label{main-algebra}
\end{equation}
All elements in the matrices are from the Wiener algebra. Let
\[
\A(\lambda)=1+\int\limits_0^\infty \alpha(x)e^{i\lambda x }dx,
\B(\lambda)=\int\limits_0^\infty \beta(x)e^{i\lambda x}dx
\]
Writing down the integral equation for coefficients and taking
suitable projections, we get
\begin{equation}
\left[
\begin{array}{cc}
1 & -\cal{D}^*\\
-\cal{D} & 1
\end{array}
\right]\left[
\begin{array}{cc}
\overline{\alpha(\cdot)} & \overline{\beta(\cdot)}\\
\beta(\cdot) & \alpha(\cdot)
\end{array}
\right]=\left[
\begin{array}{cc}
0 & \overline{D(x)}\\
D(x) & 0
\end{array}
\right]\label{twist}
\end{equation}
where $\cal{D}$ is the corresponding Hankel operator. Of course,
this equation can be used to find $\B(\lambda)$ and $\A(\lambda)$.
Consequently,  if $D(x)$ is generated by Krein system, then it
defines this system uniquely. The converse is also true.
\begin{theorem}
Assume that $D(x)$ is a given function from $L^1(\mathbb{R}^+)\cap
C_0(\mathbb{R}^+)$. If the Hankel operator $\cal{D}$ is
contraction in $L^2(\mathbb{R}^+)$, then there is a unique Krein
system that generates this $D(x)$. Moreover, the coefficient
$A(r)\in L^1(\mathbb{R}^+)\cap C_0(\mathbb{R}^+)$. Conversely, any
$A(r)\in L^1(\mathbb{R}^+)\cap C_0(\mathbb{R}^+)$ generates $D(x)$
with these properties.
\end{theorem}
\begin{proof}
The second part follows from the arguments given above. Now, let
us start with $D(x)$. Consider (\ref{main-algebra}). Let
$\A(\lambda)=1+\phi(\lambda),
g(\lambda)=g_1(\lambda)+\overline{g}_2(\lambda),
\A^{-1}(\lambda)=1+\psi(\lambda)$, where $\phi(\lambda),
g_{1(2)}(\lambda), \psi(\lambda)\in W_+(\mathbb{R})$. Then,
(\ref{main-algebra}) is equivalent to the algebraic system
\[
\phi=\overline{\psi}+\overline{\B}(g_1+\overline{g}_2),\B=g_1+\overline{g}_2+\overline{\phi}(g_1+\overline{g}_2)
\]
That can be rewritten as a system for $\phi$ and $\B$:
\begin{equation}
\phi=\cal{P}_+(\overline{\B} g_1),
\B=g_1+\cal{P}_+(\overline{\phi} g_1) \label{main-algebra1}
\end{equation}
with two more equations
\begin{equation}
g_2=-\frac{\cal{P}_+(\phi \overline{g_1})}{1+\phi},\psi=-\B
g_2-\cal{P}_+(\B \overline{g_1}) \label{main-algebra2}
\end{equation}
that determine $g_2$ and $\psi$ consecutively. The system
(\ref{main-algebra1}) is equivalent to (\ref{twist}). Since
$\cal{D}$ is contraction, (\ref{twist}) has the unique solution.
Now, that $\phi$ and $\B$ are found, $g_2$ and $\psi$ can be found
by (\ref{main-algebra2}). Notice carefully that we do not know yet
that $\psi(\lambda)+1$  is indeed an inverse to $1+\phi(\lambda)$.
It will be clear in a second.  We have
\begin{lemma}
Matrix
\[
\left[
\begin{array}{cc}
\A(\lambda) & \overline{\B(\lambda)}\\
\B(\lambda) & \overline{\A(\lambda)}
\end{array}
\right]
\]
obtained in this way is $J$-unitary.
\end{lemma}
\begin{proof}
We only need to show that $|\A|^2=1+|\B|^2$. The following is true
\[
|\A|^2=1+|\B|^2 \Leftrightarrow
\phi+\overline\phi=|\B|^2-|\phi|^2\Leftrightarrow
\phi=\cal{P}_+(|\B|^2-|\phi|^2) \Leftrightarrow
\phi=\cal{P}_+(\B\overline{\B}-\phi\overline{\phi})
\]
Plug in the expressions for $\B$ and $\phi$ from
(\ref{main-algebra1}) into the last formula. Thus we just need to
check that
\[
\phi=\cal{P}_+(\overline{\B}g_1)+\cal{P}_+(\overline{\B} \cal{P}_+
( g_1 \overline{\phi}))-\cal{P}_+(\overline{\phi}\cal{P}_+(g_1
\overline{\B}))
\]
Due to the first identity in (\ref{main-algebra1}), the last
equality can be written as
\[
\cal{P}_+(\overline{\B} \cal{P}_+ ( g_1
\overline{\phi}))=\cal{P}_+(\overline{\phi}\cal{P}_+(g_1
\overline{\B}))
\]
but this is always true since the left-hand side is
\[
\cal{P}_+(\overline{\B} \cal{P}_+ ( g_1
\overline{\phi}))=\cal{P}_+(\overline{\B} g_1 \overline\phi -
\overline \B \cal{P}_-( g_1 \overline{\phi}))=\cal{P}_+(g_1
\overline{\B}\overline{\phi})
\]
and the right-hand side is
\[
\cal{P}_+(\overline{\phi} \cal{P}_+ ( g_1
\overline{\B}))=\cal{P}_+(\overline{\phi} g_1 \overline\B -
\overline \phi \cal{P}_-( g_1 \overline{\B}))=\cal{P}_+(g_1
\overline{\B}\overline{\phi})
\]
\end{proof}
As a corollary, we get $(1+\phi)(1+\psi)=1$. Indeed, we have $g=\B
\overline{\A}^{-1}$ from (\ref{main-algebra1}) and
(\ref{main-algebra2}). Then,
$1+\psi=\overline{\A}-\B\overline{g}=(|\A|^2-|\B|^2)\A^{-1}$=$(1+\phi)^{-1}$
by Lemma above.

Now, we can say that the function $\f=\B\A^{-1}$ is analytic
contraction for which the conditions of the Baxter Theorem hold
true. Therefore, it generates the Krein system with $A(r)\in
L^1(\mathbb{R}^+)$. Moreover, $\alpha(x),\beta(x) \in
C_0(\mathbb{R}^+)$. Therefore, $C(x)\in C_0(\mathbb{R}^+)$. By
Lemma \ref{auxil-scat}, we have $A(r)$ is also from
$C_0(\mathbb{R}^+)$. Clearly, the function $D(x)$ corresponds to
this Krein system.

\end{proof}
The class of $D(x)$ we considered was the simplest one. In
principle, the Hankel operator $\mathfrak{D}$ is bounded under
much weaker conditions (e.g. BMOA space for $g_1$, see
\cite{Peller}).

Let us make a remark regarding the Dirac operator. Notice that the
matrix-function $Z(r,\lambda)$ introduced above satisfies the
following properties
\[
\left[-iJ\frac{d}{dr}+\left[
\begin{array}{cc}
0 & -i\overline{A(r)}\\
iA(r) & 0
\end{array}
\right]\right]Z=\frac{\lambda}{2}\,Z,\quad Z(r,\lambda)=\left[
\begin{array}{cc}
e^{i\lambda r/2} & 0\\
0 & e^{-i\lambda r/2}
\end{array}
\right]+\bar{o}(1)
\]
Thus, $Z(r,\lambda)$ is the scattering solution for the Dirac
operator (written in a slightly different way, see
(\ref{dirac-df})). All results from the section can be easily
translated to this case.\vspace{2cm}

\bigskip
{\bf Remarks and historical notes.}

The main results of this section are not new. They are essentially
contained in the papers by Krein and Melik-Adamyan
\cite{krein-madamyan, melik-adamyan1, melik-adamyan2}. The authors
were motivated by certain problem of continuation related to
Hankel operators (and, essentially, coming from the complete
solution of Nehari problem, see \cite{AAK, Peller}). Also,
Melik-Adamyan considers slightly different canonical system. Using
the approach of this section, one can provide another proof for
continuous analog of Baxter's Theorem. We tried to make the
argument almost purely algebraic.

Also, one can use upper(lower)-triangular factorization of
operator $\mathfrak{D}_r$ to represent its determinant via certain
integral of $A(r)$. Analogous calculations will be done in the
next section for truncated Wiener-Hopf operators. This is a way to
obtain formula similar to the so-called Borodin-Okounkov identity
\cite{basor}.

\newpage
\section{Truncated Wiener-Hopf operators. The Strong Szeg\H{o}
Theorem}\label{sect-ss}

\bigskip
Consider the continuous accelerant $H$. In this section, we obtain
an important formula for the Fredholm determinant of $I+\cal{H}_r$
in terms of the coefficient $A(r)$ of the associated Krein system
and prove continuous analog of the so-called Strong Szeg\H{o}
Theorem \cite{Simon}. We start with the following well-known
result \cite{Akhiezer}. We omit the proof which can be obtained,
e.g., by upper(lower)-triangular factorization of $I+\cal{H}_r$.

\begin{theorem}
If $H(x)$ is a continuous accelerant on $\mathbb{R}$, then
\[
\det(1+{\cal H}_{r})=\exp \left[ \int\limits_{0}^{r}\Gamma _{u}(0,0)du%
\right]
\]
\end{theorem}

Now, assume that we are given a continuous accelerant $H(x).$ We
have a relation
\[
\frac{d}{du}\Gamma _{u}(0,0)=-|A(u)|^{2}
\]%
and
\[
\Gamma _{r}(0,0)=H(0)-\int\limits_{0}^{r}|A(s)|^{2}ds
\]%
\begin{equation}
\det(1+{\cal H}_{r})=\exp [H(0)r]\exp \left[ -\int%
\limits_{0}^{r}(r-s)|A(s)|^{2}ds\right] \label{kac}
\end{equation}
and
\begin{equation}\label{det-2}
\det{}_2(1+{\cal H}_{r})=\exp \left[ -\int%
\limits_{0}^{r}(r-s)|A(s)|^{2}ds\right]
\end{equation}
Notice, that
\[
\frac{d}{dr}\ln \det(1+{\cal
H}_{r})=H(0)-\int\limits_{0}^{r}|A(s)|^{2}ds,\quad \frac{d}{dr}\ln
\det{}_2(1+{\cal H}_{r})=-\int\limits_{0}^{r}|A(s)|^{2}ds
\]%
The fact that the both sides have limits for $A(r)\in
L^2(\mathbb{R}^+)$ can be regarded as the weak Szeg\H{o} theorem.
Notice that an approximation argument yields (\ref{det-2}) without
continuity assumption on $H(x)$. The regularity $H(x)\in L^2_{\rm
loc}(\mathbb{R})$ is enough and the weak Szeg\H{o} theorem reads:
$\ln \det{}_2(1+{\cal H}_{r})$ is nonincreasing and
\[
\lim_{r\to\infty} \frac{d}{dr}\ln \det{}_2(1+{\cal
H}_{r})=-\int\limits_{0}^{\infty}|A(s)|^{2}ds
\]
The strong Szeg\H{o} asymptotics is as follows
\[
\frac{\det (I+\cal{H}_r)}{\exp \left[r\Bigl(H(0)-\displaystyle
\int\limits_0^r |A(s)|^2 ds\Bigr)\right]}= \frac{\det{}_2
(I+\cal{H}_r)}{\exp \left[-r\displaystyle \int\limits_0^r |A(s)|^2
ds\right]}
 =\exp \left[\int\limits_0^r
s|A(s)|^2ds\right]
\]
and the limit of left-hand side exists iff
\begin{equation}\label{sst1}
G=\int\limits_0^\infty r|A(r)|^2dr< \infty
\end{equation}
Notice that under this condition one has
\[
r\int\limits_0^r |A(s)|^2 ds=r\int\limits_0^\infty |A(s)|^2ds
+\bar{o}(1)
\]
and assuming that $A(r)\in L^2(\mathbb{R}^+)$, we get
\begin{equation}
T_r=\frac{\det (I+\cal{H}_r)}{\exp \left[r\Bigl(H(0)-\displaystyle
\int\limits_0^\infty |A(s)|^2 ds\Bigr)\right]}= \frac{\det{}_2
(I+\cal{H}_r)}{\exp \left[-r\displaystyle \int\limits_0^\infty
|A(s)|^2 ds\right]}
 =\exp \int\limits_0^\infty
\nu_r(s) s|A(s)|^2ds \label{G=T}
\end{equation}
where $\nu_r(x)=1$ for $x\in [0,r]$ and $\nu_r(x)=rx^{-1}$ for
$x>r$. Since $\nu_r(x)$ is non-negative, monotone in $r$, and
$\nu_r(x)\leq 1$, the left-hand side is also increasing and has a
finite limit iff (\ref{sst1}) holds. The following result also
gives a characterization of this case in terms of the spectral
measure. That is also a continuous analog of Ibragimov's and
Golinskii-Ibragimov's Theorems \cite{Gol-Ibr,Ibr, Simon}. Some
parts of the arguments below are borrowed from the discrete case
\cite{Simon}.

\begin{theorem}\label{strong-szego-theorem}
Assume $A(r)\in L^2(\mathbb{R}^+)$. Then, the following statements
are equivalent
\begin{itemize}
\item[(i)] $G<\infty$

\item[(ii)] The measure $d\sigma$ is purely a.c. and
\begin{equation}
\ln (2\pi\sigma'(\lambda))=\int\limits_{-\infty}^\infty
l(x)\exp(i\lambda x)dx \in H^{1/2}(\mathbb{R})
\end{equation}

\item[(iii)] $T_r$ is bounded
\end{itemize}
Moreover,
\begin{equation}
G=T_\infty=L=I \label{GLI}
\end{equation}
where\footnote{ Recall that if $A(r)\in L^2(\mathbb{R})$, then
$\ln \left[2\pi\sigma'(\lambda)\right]\in
L^1(\mathbb{R})+L^2(\mathbb{R})$ (Corollary \ref{cor-sum2}).
 Also, notice that condition $A(r)\in
L^2(\mathbb{R})$ can be expressed purely in spectral terms, i.e.
through $d\sigma$. In particular, the Theorem \ref{l2-test2} and
Lemma \ref{ss-l2} from Appendix imply that conditions
$d\sigma_s=0$ and $\ln (2\pi\sigma'(\lambda))\in
H^{1/2}(\mathbb{R})$ guarantee $A(r)\in L^2(\mathbb{R}^+)$.}
\[
L=\int\limits_0^\infty x |l(x)|^2dx,\quad I=\frac{1}{\pi}
\int\limits_{\mathbb{C}^+} \left|\frac{\partial_\lambda
\Pi_\alpha(\lambda)}{\Pi_\alpha(\lambda)}\right|^2d^2\lambda
\]
where $\Pi_\alpha(\lambda)$ is given by {\rm (\ref{pisub})}.
\end{theorem}
\begin{proof}
The equivalence of (i) and (iii) as well as $G=T_\infty$ follow
from the argument above (see (\ref{G=T})).  The rest of the proof
is divided into several Lemmas.
\begin{lemma}\label{ss-lemma1}
If $G<\infty$, then $d\sigma_s=0$.
\end{lemma}
\begin{proof}
For real $\lambda$,
\[
|P_*(r,\lambda)|\geq  \exp \left[-\int\limits_0^r
|A(s)|ds\right]\geq  \exp [ -C \sqrt{\ln r}]
\]
That follows from Lemma \ref{gronwall-p*}. From
(\ref{decay-improved}), we get
\[
\int\limits_{-\infty}^\infty
\frac{|P_*(r,\lambda)|^2}{\lambda^2+1}d\sigma_s(\lambda)\leq
C\int\limits_r^\infty |A(s)|^2ds\leq Cr^{-1}
\]
Therefore,
\[
\int\limits_{-\infty}^\infty
\frac{d\sigma_s(\lambda)}{\lambda^2+1}\leq \frac{C\exp(C\sqrt{\ln
r})}{r}\to 0,\, {\rm as}\quad r\to \infty
\]
Thus $d\sigma_s=0$.
\end{proof}
\begin{lemma}\label{ss-lemma2}
Assume $g(\lambda)$ is outer in $N(\mathbb{C}^+)$, $g(i\infty)=1$,
and $\ln |g(\lambda)|\in L^2(\mathbb{R})+L^1(\mathbb{R})$. Then,
{\rm (\ref{cauchy})} below holds.
\end{lemma}
\begin{proof}
Using multiplicative representation for outer functions and
normalization at $i\infty$, we have
\[
\ln g(\lambda)=2\int\limits_0^\infty l(x)\exp(i\lambda x)dx
\]
with $l(x)\in L^2(\mathbb{R}^+)+W(\mathbb{R})\cdot
\chi_{\mathbb{R}^+}$. Simple calculations show that
\[
\int\limits_0^\infty xe^{-2\epsilon x}
|l(x)|^2dx=(4\pi)^{-1}\int\limits_{\Im \lambda>\epsilon}
|\partial_\lambda \ln
g(\lambda)|^2d^2\lambda=(4\pi)^{-1}\int\limits_{\Im
\lambda>\epsilon} \left|\frac{\partial_\lambda
g(\lambda)}{g(\lambda)}\right|^2d^2\lambda
\]
Therefore, we always have
\begin{equation}\label{cauchy}
\int\limits_0^\infty x
|l(x)|^2dx=(4\pi)^{-1}\int\limits_{\mathbb{C}^+}
\left|\frac{\partial_\lambda
g(\lambda)}{g(\lambda)}\right|^2d^2\lambda
\end{equation}
 even though the both quantities can be infinite.
\end{proof}
For $A(r)\in L^2(\mathbb{R}^+)$, we have $\ln
(2\pi\sigma'(\lambda))\in L^2(\mathbb{R})+L^1(\mathbb{R})$,
$\Pi_\alpha(i\infty)=1$ and the Lemma can be applied to
$g(\lambda)=\Pi_\alpha^{-2}(\lambda)$ and thus we get $L=I$ in
(\ref{GLI}).

We need the following auxiliary
\begin{lemma}\label{ss-lemma3}
For any $R$, we have the following inequality
\begin{equation}\label{auxil-ss1}
\int\limits_0^R r|A(r)|^2dr\geq {\pi}^{-1}
\int\limits_{\mathbb{C}^+}\left| \frac{\partial_\lambda
P_*(R,\lambda)}{P_*(R,\lambda)}\right|^2d^2\lambda
\end{equation}
\end{lemma}
\begin{proof}
For any $T>0$, consider $\Omega_T=\{\lambda\in
\overline{\mathbb{C}^+}, |\lambda|\leq T\}$. Let us show that

\begin{equation}
\int\limits_0^R r|A(r)|^2dr\geq \pi^{-1}\int\limits_{\Omega_T}
\left|\frac{\partial_\lambda
P_*(R,\lambda)}{P_*(R,\lambda)}\right|^2d^2\lambda
\label{T-domain}
\end{equation}
for any $T>0$. Then the general statement follows upon taking
$T\to\infty$. Fix any $T$. Then, it is sufficient to prove
(\ref{T-domain}) assuming that $A(r)\in C[0,R]$. Indeed, any
$A(r)\in L^2[0,R]$ can be approximated by continuous functions in
$L^2[0,R]$ norm and the both sides of (\ref{T-domain}) are
continuous in $A(r)$ with respect to $L^2[0,R]$ metric.

Then, we just need to use the suitable formula from the discrete
case and an approximation result given by Corollary
\ref{approx-poly}. Consider large $n$, the discretization step
$h=R/n$, Verblunsky parameters given by (\ref{discrete-A}), and
the corresponding monic orthogonal polynomials $P_k(z)$ and
$P_k^*(z)$. Then, we have the following formula (\cite{Simon},
Theorem 2.1.4)
\begin{equation}
\ln \left[\prod_{j=0}^n (1-|a_j|^2)^{-j-1}\right]=
\frac{1}{\pi}\int\limits_{\mathbb{D}} \left|\frac{\partial_z
P_n^*(z)}{P_n^*(z)}\right|^2d^2z \label{sst-d}
\end{equation} With our choice of
Verblunsky parameters,
\[
\ln \left[\prod_{j=0}^n (1-|a_j|^2)^{-j-1}\right]\to
\int\limits_0^R r|A(r)|^2dr, n\to\infty
\]
by the Riemann sum approximation.
 Over $\Omega_T$,
we have $P_n^*(e^{i\lambda h})\to P_*(R,\lambda)$ and
\[
ihe^{i\lambda h} \partial_z P_n^* (e^{i\lambda h})\to
\partial_\lambda P_*(R,\lambda)
\]
Since $e^{i\lambda h}\to 1$ uniformly over $\Omega_T$, we also
have
\begin{equation}
ih \partial_z P_n^* (e^{i\lambda h})\to
\partial_\lambda P_*(R,\lambda) \label{der-ha}
\end{equation}
Therefore, making the change of variables and using (\ref{sst-d})
 and (\ref{der-ha}), we get (\ref{T-domain}).

\end{proof}
For $A(r)\in L^2(\mathbb{R}^+)$, we have $P_*(R,\lambda)\to
\Pi_\alpha(\lambda)$ uniformly in $\{\Im \lambda>\delta\}$ for any
$\delta>0$. Therefore, we always have an estimate $G\geq I$ and
consequently (i) implies (ii) due to Lemmas \ref{ss-lemma1},
\ref{ss-lemma2}.

Now, we are left with proving that $G\leq I$. Let us assume we
have purely a.c. measure with density satisfying $\ln
(2\pi\sigma'(\lambda))\in H^{1/2}(\mathbb{R})$ so that $L<\infty$
for the corresponding function $l(x)$.  We need to show $G\leq L$
for the associated Krein system.

We know that $A(r)\in L^2(\mathbb{R}^+)$. Then,
$2\pi\sigma'(\lambda)-1=\psi(x)$ with $\psi(x)\in L^2(\mathbb{R})$
by Lemma \ref{ss-l2}. Formulas (\ref{derivative-l2}) and
(\ref{sst-refl2}) imply
\begin{equation}
H(x)=2\pi\hat\psi(x)\label{slowant}
\end{equation}
where $\hat{\psi}$ is Fourier transform of $\psi$.

 Let us take Hermitian $l_R(x)$ such that
$l_R(x)$ is continuous with compact support within $[-R,R]$ and
\[
\int\limits_{-\infty}^\infty (x^2+1)^{1/2} |l_R(x)-l(x)|^2dx\to 0
\]
as $R\to\infty$. Formula (\ref{slowant}),  Lemma \ref{ss-l2}, and
Lemma \ref{l2regular} show that the corresponding $A_R(r)\to A(r)$
in $L^2[0,T]$ for any $T>0$. For each $R$, we can apply the
Theorem \ref{approx-3}. For the corresponding sequence of
Verblunsky parameters, we have \cite{Simon}, Chapter 6:
\begin{equation}\label{discrete-sst}
\prod_{j=0}^\infty
(1-|a_j^{(n)}|^2)^{-j-1}=\exp\left[\sum\limits_{k=1}^\infty k
|\hat{L}_k^{(n)}|^2\right]
\end{equation}
where
\[
\ln \mu'_n(\theta)=\sum\limits_{k=-\infty}^\infty \hat{L}_k^{(n)}
e^{ik\theta}
\]
Clearly,
\[
\sum\limits_{k=1}^\infty k |\hat{L}_k^{(n)}|^2\to \int\limits_0^R
x|l_R(x)|^2dx
\]
as $n\to\infty$. On the other hand, for any $\delta>0$, we have
\[
\prod_{\delta<jh<n} (1-|a_j^{(n)}|^2)^{-j-1} \to \exp\left[
\int\limits_\delta^R r|A_R(r)|^2dr\right]
\]
as it follows from the Theorem \ref{approx-3}. Therefore,
\[
\int\limits_0^R r|A_R(r)|^2\leq \int\limits_0^R x|l_R(x)|^2dx
\]
and then $G\leq I$ since $A_R(r)\to A(r)$ in $L^2_{\rm
loc}(\mathbb{R}^+)$. So, $G=I$ and the proof is finished.
\end{proof}

The proof we used essentially utilized the strong Szeg\H{o}
Theorem for Toeplitz matrices and the approximation of continuous
orthogonal system by the sequence of discrete ones. In the
meantime, one could have adjusted the various proofs directly to
continuous case (see, e.g. \cite{basor} for continuous analog of
Borodin-Okounkov identity which can probably be used for this
purpose). Notice that if $d\sigma_{k}, k=1,2$ satisfy conditions
of Theorem \ref{strong-szego-theorem} then
$d\sigma=\left[\sigma_1'\right]^{\gamma}\left[\sigma_2'\right]^{1-\gamma}d\lambda,
\gamma\in [0,1]$ also satisfies these conditions.

 Formula (\ref{kac}) is very important for many applications.
In the theory of random matrices, one needs to calculate
asymptotics of Fredholm determinants for some specific
accelerants. Assume that coefficient $A$, corresponding to a given
accelerant $H$, tends to zero fast enough. Then, solving the
inverse scattering problem by methods of the last Section, one can
obtain an asymptotics of $A(r)$ at infinity.  Assume that $A(r)$
decays at infinity fast enough such that
\[
A(r)=\sum\limits_{n=2}^{N}\frac{\gamma
_{n}}{r^{n}}+\frac{C_{N+1}(r)}{r^{N+1}%
},
\]%
holds with constants $\gamma _{n}$ and $C_{N+1}(r)\in L^{\infty
}(R^{+}),$ $N $ is arbitrary. Formula (\ref{kac}) and
\[
\int\limits_{0}^{r}(r-s)|A(s)|^{2}ds=r\int\limits_{0}^{\infty
}|A(s)|^{2}ds-r\int\limits_{r}^{\infty
}|A(s)|^{2}ds-\int\limits_{0}^{\infty
}s|A(s)|^{2}ds+\int\limits_{r}^{\infty }s|A(s)|^{2}ds
\]%
shows that as long as all $\gamma _{n}$, $\displaystyle
\int\limits_{0}^{\infty }|A(s)|^{2}ds$, and $\displaystyle
\int\limits_{0}^{\infty }s|A(s)|^{2}ds$ are known, we can compute
complete asymptotic of $\det(1+{\cal H}_{r})$ as $r\rightarrow
\infty $. But these two integrals can be explicitly expressed via
the spectral data. The idea of using the inverse scattering theory
to compute asymptotics of Fredholm determinants was pioneered by
Dyson \cite{dyson}.\vspace{1cm}

{\bf Remarks and historical notes.}

 Various generalizations of the strong Szeg\H{o} formula to
continuous case were obtained in \cite{Akhiezer, Kac, Solev}. Our
version is optimal and new to our knowledge. On application of
inverse scattering to random matrices see \cite{dyson}. In
\cite{widom}, the theory of Krein systems was used to study an
asymptotics of the certain Toeplitz determinants.

\newpage

\section{Appendix }

In this Appendix, we collected the general results that we used in
the main text.

\subsection{For section \ref{sect-fact}}

The proofs of the following two Lemmas are given in
\cite{Smithies}, p.71 and p.99.
\begin{lemma} \label{fredholm}
If $\Gamma_r(x,y)$ is the resolvent kernel for  $K(x,y)$, which is
continuous on $[0,r]^2$ (see Section 2), then
\[
\Gamma_r(x,y)=\frac{\delta_r(x,y)}{\delta_r}
\]
where
\[
\delta_r(x,y)=K \left(
\begin{array}{c}
x\\
y\end{array}\right)+\frac{1}{1!}\int\limits_0^r K\left(
\begin{array}{cc}
x&\xi_1\\
y&\xi_1\end{array} \right)d\xi_1+\ldots
\]
\[
+\frac{1}{n!}\int\limits_0^r\ldots \int\limits_0^r K\left(
\begin{array}{cccc}
x&\xi_1&\ldots &\xi_n \\
y&\xi_1&\ldots &\xi_n \end{array} \right)d\xi_1 \ldots
d\xi_n+\ldots
\]

\[
\delta_r=1+\int\limits_0^r K \left(
\begin{array}{c}
\xi_1\\
\xi_1\end{array}\right)d\xi_1+\frac{1}{2!}\int\limits_0^r
\int\limits_0^r K\left(
\begin{array}{cc}
\xi_1&\xi_2\\
\xi_1&\xi_2\end{array} \right)d\xi_1 d\xi_2+\ldots
\]
\[
+\frac{1}{n!}\int\limits_0^r\ldots \int\limits_0^r K\left(
\begin{array}{cccc}
\xi_1&\xi_2&\ldots &\xi_n \\
\xi_1&\xi_2&\ldots &\xi_n \end{array} \right)d\xi_1 \ldots
d\xi_n+\ldots
\]
and
\[
K\left(
\begin{array}{cccc}
\xi_1&\xi_2&\ldots &\xi_n \\
\eta_1&\eta_2&\ldots &\eta_n \end{array}
\right)=\det\{K(\xi_k,\eta_l)\}, 1\leq k,l\leq n
\]
The series converges absolutely.
\end{lemma}

The next Lemma gives Carleman-Hilbert determinantal representation
for resolvent. We recommend an excellent book \cite{G-G-K}
(Theorem 2.2, p.207) for the modern presentation of that subject.
It also contains the discussion of integral operators with
discontinuity on the diagonal.
\begin{lemma} \label{fredholm-mod}
If $\Gamma_r(x,y)$ is the resolvent kernel for  $K(x,y)\in
\hat{C}([0,r]^2)$  (see Section 2), then
\[
\Gamma_r(x,y)=\frac{\hat\delta_r(x,y)}{\hat\delta_r}\in
\hat{C}([0,r]^2)
\]
where $\hat\delta_r(x,y)$ and $\hat\delta_r$
 are defined as $\delta_r(x,y)$ and $\delta_r$ given above, but
 relative to the modified kernel  $\widehat{K}(x,y)=K(x,y)$ if $x\neq y$ and $\widehat{K}(x,y)=0$ on
the diagonal.
\end{lemma}

\subsection{For section \ref{four}}

We begin with some simple facts about the linear spaces with
indefinite metric. Let $[x,y]=(Jx,y)$ be indefinite inner product
(in $\mathbb{C}^2$). For any matrix $M$, introduce $M^c=JM^*J$.
Then, $[Mx,y]=[x,M^cy]$. Clearly, $(AB)^c=B^cA^c,
(A^c)^{-1}=(A^{-1})^c$ if $A$ is invertible. The following result
is well-known \cite{Iohvidov, Azizov}
\begin{lemma}
If $M$ is $J$--unitary, then $|\det M|=1$, and $M^{-1}, M^*$ are
$J$--unitary too. If $M$ is $J$-- contraction then $M^*$ is $J$--
contraction also. \label{jproperty1}
\end{lemma}

\begin{proof} Taking determinant of $M^* J M=J$, we get $|\det M|=1$.
It is straightforward that $M$ is $J$-- unitary if and only if
$M^{-1}$ is $J$--unitary.

Then, clearly, $M$ is $J$--unitary if and only if
\[
M^*JMJ=I
\]
but that means
\[
M^*J=(MJ)^{-1}
\]
so
\[
MJM^*J=I
\]
which is the same as saying that $M^*$ is $J$-- unitary.

Assume that $M$ is $J$--contraction and $1$ is not an eigenvalue.
Then, we have a general algebraic formula
\[
(M-I)^{-1}(MM^c-I)(M^c-I)^{-1}=(M^c-I)^{-1}(M^cM-I)(M-I)^{-1}=I+(M-I)^{-1}+(M^c-I)^{-1}
\]
which can be rewritten as
\[
MJM^*-J=JQ^*(M^*JM-J)QJ, Q=(M-I)^{-1}(M^c-I)
\]
Since $Q$ is invertible, we have $M^*JM\leq J$ iff $MJM^*\leq J$.
 In other words, $M$ is $J$--contractive
iff $M^*$ is $J$-contractive. If $1$ is an eigenvalue, multiply
$M$ by unimodular scalar factor such that $1$ is not an eigenvalue
of the resulting operator and apply the argument above.
\end{proof}

\begin{lemma} \label{uniqueness1}
Assume that
\[
p(\lambda)=1+\int\limits_0^r f(x)\exp(i\lambda x)dx
\]
where $f(x)\in L^2[0,r]$ and $p(\lambda)\neq 0$ for $\lambda\in
\mathbb{R}$. Then, $p(\lambda)$ is uniquely determined by
\[
\cal{P}_{[-r,r]} \left[ \frac{1}{|p(\lambda)|^{2}}-1\right]
\]
\end{lemma}
\begin{proof}
By Levy-Wiener theorem,
\begin{equation}
\frac{1}{|p(\lambda)|^2}-1=\int\limits_{-\infty}^\infty
\overline{h(x)}\exp (i\lambda x)dx \label{appendix1}
\end{equation}
where $h(x)\in L^1(\mathbb{R})$ and $h(-x)=\overline{h(x)}$. We
know $h(x)$ on $[-r,r]$ and need to find $f(x)$ on $[0,r]$. As in
Lemma \ref{lemmaimp},
\[
\cal{P}_{+}\left[ p(\lambda)\left( 1+\int\limits_{-\infty }^{\infty }%
\overline{h(s)}\exp (i\lambda s)ds\right)-1\right] =0
\]%
which gives
\[
f(x)+\overline{h(x)}+\int\limits_0^r \overline{h(x-t)}f(t)dt=0,
0<x<r
\]
This equation determines $f(x)$ on $[0,r]$ uniquely from $h(x)$ on
$[-r,r]$ since $I+\cal{H}>0$ where
\[
\cal{H}g(x)=\int\limits_0^r \overline{h(x-t)}g(t)dt
\]
The positivity of $I+\cal{H}$ follows from (\ref{appendix1}).
\end{proof}
\subsection{For section \ref{sect-szego}}

\begin{theorem}\label{nnn1}
Let $d\mu(\lambda)$ be finite nonnegative measure defined on the
whole line. Consider the linear manifold $L$ consisting of the
finite linear combinations of exponents $\exp(i\lambda r)$ with
$r\geq 0$. Then, $L$ is not dense in $L^2(d\mu)$ iff
\begin{equation}
\int\limits_{-\infty}^\infty \frac{\ln
\mu'(\lambda)}{1+\lambda^2}d\lambda>-\infty \label{metka1}
\end{equation}
\end{theorem}
\begin{proof} Denote the closure of $L$ in $L^2(d\mu)$ by
$L_\mu$. Consider the function
\[
f(\lambda)=\int\limits_0^\infty \exp(-x+i\lambda
x)dx=\frac{i}{i+\lambda}
\]
We can find sequence $\{f_n\}$ in $L$ which is uniformly bounded
on $\mathbb{R}$ in $L^\infty$ norm and $f_n\to f$ in the uniform
norm on any fixed compact in $\mathbb{R}$. For instance,
\[
f_n (\lambda)=\sum_{j=0}^{n^2-1} \exp(i\lambda j/n)
\int\limits_{jn^{-1}}^{(j+1)n^{-1}} \exp(-x) dx
\]
Therefore, $f\in L_\mu$ and $g=(\lambda-i)(\lambda+i)^{-1}=1-2f
\in L_\mu $. Let us show that $L_\mu$ is invariant under the
multiplication by $g$. Because $g$ is an elementary Blaschke
factor, $|g|=1$ on the real line. Therefore, we only need to show
that $g\psi \in L_\mu$ for any $\psi\in L$. But that clearly
follows from
\begin{itemize}
\item[(i)]{ $\psi$ is finite linear combination of exponents
$\exp(i\lambda r)$ for different $r\geq 0$,} \item[(ii)]{ space
$L_\mu$ is invariant under the multiplication by $\exp(i\lambda
r)$ for any $r\geq 0$,} \item[(iii)]{ $g$ itself belongs to
$L_\mu$.}
\end{itemize} Now, once we know that $L_\mu$ is invariant under
the multiplication by $g$, we know that $g^n\in L_\mu$ for all
$n\in \mathbb{Z}^+$. Let us consider the standard conformal map of
$\lambda\in \mathbb{C}^+$ onto the unit disc: $w\in \mathbb{D},
w=(\lambda-i)(\lambda+i)^{-1}$. Under this map, measure $d\mu$
goes into a new finite measure $d\tau$ on the unit circle, which
generates a new Hilbert space $L^2(d\tau)$. The subspace $L_\mu$
goes into the subspace $L_\tau$. Moreover, since all $g^n\in
L_\mu$, $w^n\in L_\tau $ for any $n\in \mathbb{Z}$, $w\in
\mathbb{T}$. Let us consider the subspace $Z_\tau$ obtained by
closure of $Z=$ Span$\{1,w,\ldots, w^n,\ldots\}$ in the
$L^2(d\tau)$. Clearly $Z_\tau\subseteq L_\tau$. Let us show that
actually $Z_\tau=L_\tau$. To do that, it is enough to prove that
any function
\[
h(w)=\exp\left[ \frac{w+1}{w-1} ~r \right], r\geq 0, w\in
\mathbb{T}
\]
(which is an image of $\exp(i\lambda r)$ under the conformal map)
can be approximated by analytic polynomials in $w$ in the
$L^2(d\tau)$ metric. Consider functions
\[
h_\rho (w)=\exp\left[ \frac{\rho w+1}{\rho w-1} ~r \right], w\in
\mathbb{T}
\]
for $\rho<1$. Since $d\mu$ is finite on $\mathbb{R}$, we have
$\tau(-\varepsilon, \varepsilon)\to 0$ as $\varepsilon \to 0$. In
other words, $\tau$ has no mass point at $w=1$. Function $h(w)$ is
bounded in $\overline{\mathbb{D}}$ and is continuous there except
for the point $w=1$. Therefore, $h_\rho\to h$ in $L^2(d\tau)$ as
$\rho\to 1$. At the same time, each $h_\rho$ is analytic in small
neighborhood of $\mathbb{D}$ and therefore can be approximated by
polynomials uniformly on the unit circle. Thus, $L_\tau=Z_\tau$.
The Szeg\H{o} theorem (\cite{Simon}, Chapter 2) says that $Z_\tau$
is not dense in $L^2(d\tau)$ iff
\[
\int\limits_{-\pi}^\pi \ln \tau'(\theta)d\theta>-\infty
\]
Clearly the last condition is equivalent to $(\ref{metka1})$.
\end{proof}

Notice that due to
\[
\int\limits_{-\infty}^\infty d\mu(\lambda) <\infty
\]
we have that (\ref{metka1}) is equivalent to
\begin{equation}
\int\limits_{-\infty}^\infty \frac{\ln^-
\mu'(\lambda)}{1+\lambda^2}d\lambda>-\infty \label{metka2}
\end{equation}

This Theorem has interpretation in the theory of Gaussian
stationary processes with continuous time and that is very useful
point of view on the whole theory of Krein systems. In the
meantime, there are the so-called Krein strings \cite{DymMcKean},
differentiable operators more suitable to deal with stationary
processes.

\begin{theorem}\label{appendix-szego}
Assume that $d\sigma$ is a measure on the real-line such that
\[
\int\limits_{-\infty}^\infty
\frac{d\sigma(\lambda)}{1+\lambda^2}<\infty
\]
Consider the linear manifold $X$ of functions
\[
\hat{f}(\lambda)=\int\limits_0^\infty \exp(i\lambda x)
{f}(x)dx,\quad 0\leq r_1<r_2
\]
where ${f}(x)\in C^1[r_1,r_2]$ and is zero outside
$[r_1,r_2]\subseteq [0,\infty)$. Then, $X$ is not dense in
$L^2(d\sigma)$ iff
\begin{equation}
\int\limits_{-\infty}^\infty \frac{\ln \sigma'
(\lambda)}{1+\lambda^2}d\lambda >-\infty
\end{equation}
Moreover, let $\lambda_0\in \mathbb{C}^+$. Then
\begin{equation}
{\rm Dist} \left( \frac{1}{\lambda-\lambda_0}, \bar{X}
\right)_{L^2(d\sigma)} = \frac{1}{\sqrt{2\Im\lambda_0}} \exp
\left[ \frac{\Im\lambda_0}{2\pi}\int\limits_{-\infty}^{\infty}
\frac{\ln
(2\pi\sigma'(\lambda))}{|\lambda-\lambda_0|^2}d\lambda\right]
\end{equation}
\end{theorem}
\begin{proof} Consider a new measure
$d\mu=d\sigma/(1+\lambda^2)$ which is finite on the real line.
Denote by $Y$ the linear manifold of functions of the following
form $(\lambda+i)\hat f(\lambda), \hat f\in X$. Let $Y_\mu$ be the
closure of $Y$ in $L^2(d\mu)$. We only need to show that $Y_\mu
\neq L^2(d\mu)$ iff
\begin{equation}
\int\limits_{-\infty}^\infty \frac{\ln \mu'
(\lambda)}{1+\lambda^2}d\lambda >-\infty
\end{equation}
Let $L_\mu$ be the space of functions from the proof of the
Theorem \ref{nnn1}, i.e. the closure in $L^2(d\mu)$ of finite
linear combinations of exponents $\exp(i\lambda r), r\geq 0$. It
is not difficult to show that $\exp(i\lambda r)\in Y_\mu$ for any
$r\geq 0$. That follows from the representation
\[
\exp(i\lambda r)=-i\exp(r)(\lambda+i)\int\limits_r^\infty
\exp(-x)\exp(i\lambda x)dx
\]

So, $L_\mu\subseteq Y_\mu$. At the same time, each function
\begin{eqnarray*}
(\lambda+i)\int\limits_{r_1}^{r_2} \exp(i\lambda x) {f}(x)dx= i
\Bigl( f(r_1)\exp(i\lambda r_1)-f(r_2)\exp(i\lambda r_2)\Bigl.\\
\Bigl.+ \int\limits_{r_1}^{r_2} \exp(i\lambda x) [{f}'(x)
+{f}(x)]dx\Bigr)
\end{eqnarray*}
can be approximated in $L^2(d\mu)$ by the finite linear
combinations of exponents $\exp(i\lambda r)$. One should replace
the integral by the Riemann sum and use continuity of the
functions ${f}, {f}'$ to estimate the error. Thus $L_\mu = Y_\mu$
and one can use Theorem \ref{nnn1} to finish the proof of the
first statement of the Theorem.

Now, let us obtain the formula for the distance.
 For
simplicity, consider $\lambda_0=i$. The general case can be
treated in the same way. We have
\[
\inf\limits_{\hat f\in \bar{X}} \int\limits_{-\infty}^{\infty}
\left| \frac{1}{\lambda-i} -\hat f(\lambda)\right|^2 d\sigma=
\inf\limits_{\hat f\in \bar{X}} \int\limits_{-\infty}^{\infty}
\left| 1-\frac{\lambda-i}{\lambda+i}(\lambda+i)\hat
f(\lambda)\right|^2 \frac{d\sigma}{1+\lambda^2}
\]
\[
=\inf\limits_{y\in Y_\mu} \int\limits_{-\infty}^{\infty} \left|
1-\frac{\lambda-i}{\lambda+i}\,y(\lambda)\right|^2
d\mu(\lambda)=\inf\limits_{y\in L_\mu}
\int\limits_{-\infty}^{\infty} \left|
1-\frac{\lambda-i}{\lambda+i}\,y(\lambda)\right|^2 d\mu(\lambda)
\]
\[
=\inf\limits_{v\in Z_\tau} \int\limits_{\mathbb{T}} \left|
1-wv(w)\right|^2 d\tau(w)
\]
where the measure $d\tau(w)$ was obtained from $d\mu(\lambda)$ by
mapping $\mathbb{C}^+$ onto $\mathbb{D}$ via
$w=(\lambda-i)(\lambda+i)^{-1}$. Here we also used an
approximation result from the proof of Theorem \ref{nnn1}. For the
last $\inf$, we can use the Szeg\H{o} formula \cite{Simon}, i.e.

\[
{\rm Dist} \left(1,wZ_\tau \right)_{L^2(d\tau)}=\exp \left[
\frac{1}{4\pi} \int\limits_0^{2\pi} \ln (2\pi
\tau'(\theta))d\theta \right]= \frac{1}{\sqrt 2}  \exp \left[
\frac{1}{2\pi}\int\limits_{-\infty}^{\infty} \frac{\ln (2\pi
\sigma'(\lambda))}{1+\lambda^2}d\lambda\right]
\]
and the proof is finished. \end{proof}

\subsection{For section \ref{sect-appro}}

\begin{lemma}\label{regular-inf}
If $C(x)\in L^1[0,R]$, $C(x)$ is continuous at zero and
\[
f(\lambda)=\int\limits_0^R C(x) \exp(i\lambda x)dx
\]
then
\[
\lim_{y\to\infty} \frac 1y \int\limits_0^y sf(is)ds=C(0)
\]

\end{lemma}
\begin{proof}
The proof follows from the standard estimates:
\[
\lim\limits_{y\to\infty} \frac 1y \int\limits_0^y sf(is)ds=C(0)+
\lim_{y\to\infty} \frac 1y \int\limits_0^y s \left[
\int\limits_0^R [C(x)-C(0)] \exp(-sx)dx\right] ds=C(0)
\]
because the second term before the $\lim$ can be bounded by
\[
C\left[\omega_\delta(C)+\frac{\|C\|_1+|C(0)|}{y}\int\limits_0^y
s\exp(-s\delta)ds\right]
\]
where $\omega_\delta(C)=\sup_{x\in [0,\delta]}|C(x)-C(0)|\to 0$ as
$\delta\to 0$.
\end{proof}

\subsection{For section \ref{sect-zeros}}

The following Lemma controls the zeroes of the continuous
orthogonal polynomial
\begin{lemma}\label{zeroes}
Let
\[
p(\lambda)=1-\int\limits_0^r \gamma(x)e^{-i\lambda x}dx
\]
where $\gamma(x)\in C^2[0,r]$, $\gamma(r)\neq 0$. If $\lambda_n$
are zeroes of $p(\lambda)$ and $|\lambda_1|\leq |\lambda_2|\leq
\ldots$, then $\lambda_n=\lambda_n^0+\bar{o}(1), n\to\infty,$
where $\lambda_n^0= x_n+iy_n$ and
\begin{equation}\label{details-roots}
 x_n^2+y_n^2= |\gamma(r)|^2\exp(2ry_n),\quad
x_n=r^{-1}\left[\pi/2+\pi n+\varphi\right], n\in \mathbb{Z},
\end{equation}
Here, $\gamma(r)=|\gamma(r)|e^{i\varphi}$.
\end{lemma}
\begin{proof}
From Lebesgue-Riemann Lemma, $\lambda_n\in \mathbb{C}^+$ for large
$n$. Integrating by parts, we get
\[
p(\lambda)=1+\frac{1}{i\lambda}\left[\gamma(r)e^{-i\lambda
r}-\gamma(0)\right]-\frac{1}{\lambda^2}\left[
\gamma'(r)e^{-i\lambda r}-\gamma'(0)\right]+\frac{1}{\lambda^2}
\int\limits_0^r \gamma''(s)e^{-i\lambda s}ds
\]
Therefore, the equation $p(\lambda)=0$ can be rewritten as
\[
\frac{e^{-i\lambda r}}{i\lambda}=
\left[-1+\frac{\gamma(0)}{i\lambda}-\frac{\gamma'(0)}{\lambda^2}\right]
\cdot \left[
\gamma(r)+\frac{\gamma'(r)}{i\lambda}-\frac{1}{i\lambda}\int\limits_0^r
\gamma''(s)e^{i\lambda(r-s)}ds\right]^{-1}
\]
\[=-\frac{1}{\gamma(r)}+O(|\lambda|^{-1})
\]
By Rouche's Theorem, $\lambda_n$ will be approaching the roots
$\lambda_n^0$ of equation
\[
\frac{e^{-i\lambda r}}{i\lambda}=-\frac{1}{\gamma(r)}
\]
Then, the first equation in (\ref{details-roots}) easily follows
upon taking the absolute value squared. The second one can be
obtained by taking the real part of identity
$\gamma(r)e^{-ir(x+iy)}=-i(x+iy)$, which yields
$\cos(rx-\varphi)=ye^{-ry}|\gamma(r)|^{-1}\to 0$, as $y\to\infty$.
The last equation yields the needed quantization for $x_n$.
\end{proof}

The following result is due to Widom (see, e.g. \cite{Simon},
Lemma 8.1.9)
\begin{lemma}{\rm (Widom's lemma).}\label{widom}
Let $F,D$ be disjoint compact sets in $\mathbb{C}$ and
$\mathbb{C}\backslash F$-- connected. Then there is $m$ such that
for any $z_1, z_2, \ldots, z_m\in D$ there is a monic polynomial
$\tilde{Q}_m(z)$ of degree $m$, such that
\[
\sup_{z\in F}
\left|\frac{\tilde{Q}_m(z)}{\prod_{j=1}^m(z-z_j)}\right|\leq \frac
12
\]
\end{lemma}

The following result is the mean-values formula for analytic
functions of a special type.

\subsection{For section \ref{sect-l2}}

\begin{lemma}\label{mean-value}
Assume that $g(\lambda)\in B(\mathbb{C}^+)$, $g(\lambda)\in
L^1(\mathbb{R})$, $g(\lambda)=\bar{o}(1)$ as $|\lambda|\to\infty$
and $g(iy)=\bar{o}(y^{-1})$ as $y\to +\infty$. Then,
\[
\int\limits_{-\infty}^\infty \ln |1+g(\lambda)|d\lambda=0
\]
\end{lemma}
\begin{proof}
Since $\Re (1+g)>0$ in $\mathbb{C}^+$ and $g\in B(\mathbb{C}^+)$,
the function $1+g$ is outer from $N(\mathbb{C}^+)$. Therefore,

\[
\ln |1+g(iy)|=\frac y\pi \int\limits_{-\infty}^\infty \frac{\ln
|1+g(\lambda)|}{\lambda^2+y^2}d\lambda
\]
Multiply the last identity by $y$ and take $y\to+\infty$.
\end{proof}

\begin{lemma}\label{auxil-2}
If $f(\lambda)$ is such that $(\lambda+i)^{-1}f(\lambda)\in
H^2(\mathbb{C}^+)$ and $f(\lambda)\in L^2(\mathbb{R})$, then
$f(\lambda)\in H^2(\mathbb{C}^+)$.
\end{lemma}
\begin{proof}
Since $f(\lambda)=(\lambda+i)g(\lambda)$ with $g(\lambda)\in
H^2(\mathbb{C}^+)$, we have $f(\lambda)\in N(\mathbb{C}^+)$. Then,
the statement of the Lemma follows, for example, from the
multiplicative representation of $N(\mathbb{C}^+)$.
\end{proof}

\subsection{For section \ref{sect-baxter}}
The following considerations are used in the discussion regarding
the case  $A(r)\in L^1(\mathbb{R}^+)$. We borrow the notations,
statements, and proofs from \cite{Simon}, Chapter 5. For the
reader's convenience, we decided to include this material.

Let $X$ be a Banach space, $\cal{C}$-- linear bounded operator,
and $P_+$-- projection (i.e. linear bounded operator such that
$P_+^2=P_+$). Notice that $P_-=I-P_+$ is also a projection.

\begin{definition}
 The Toeplitz operator is an operator acting
in $P_+(X)$ by the formula $\cal{T}=P_+\cal{C}P_+$.
\end{definition}

\begin{definition}
A linear bounded operator $\cal{U}$ is called upper triangular if
$P_-\cal{U}P_+=0$ and $\cal{L}$ is lower triangular is
$P_+\cal{L}P_-=0$.
\end{definition}

\begin{definition}
A linear bounded operator $\cal{C}$ is a Wiener-Hopf operator if
$\cal{C}=\cal{L}\cal{U}$ where $\cal{L},\cal{U}$-- invertible,
$\cal{L}, \cal{L}^{-1}$-- lower triangular,
$\cal{U},\cal{U}^{-1}$-- upper triangular.
\end{definition}

\begin{theorem}{\rm (Wiener-Hopf Theorem).}
Let $\cal{C}$ be a Wiener-Hopf operator. Then, the corresponding
Toeplitz operator $\cal{T}=P_+\cal{C}P_+$ is invertible and
\[
\cal{T}^{-1}=(P_+\cal{U}^{-1}P_+)(P_+\cal{L}^{-1}P_+)
\]
\end{theorem}
Assume that $Q, R$ are projections and
\begin{equation}\label{baxter-algebra}
QR=RQ=0, Q+R=P_+
\end{equation}
\begin{theorem}{\rm(Baxter's Lemma).} Let $\cal{C}$ be a Wiener-Hopf
operator so that $\cal{C}=\cal{L}\cal{U}=\cal{U}\cal{L}$. Consider
$Q,R$ obeying (\ref{baxter-algebra}). Assume that
\[
R\cal{L}Q=R\cal{L}^{-1}Q=Q\cal{U}R=Q\cal{U}^{-1}R=0
\]
and
\[
\|P_-\cal{L}^{-1}R\cal{U}\|<1/2, \|R\cal{U}^{-1}P_-\cal{L}\|<1/2
\]
Then, $\cal{T}_Q=Q\cal{C}Q$ is invertible and
\[
\|\cal{T}_Q^{-1}\|<\|\cal{L}^{-1}\cal{U}^{-1}\|+2\max
(\|\cal{U}^{-1}\|, \|\cal{L}^{-1}\|)(\|P_-
\cal{L}^{-1}\|+\|R\cal{U}^{-1}\|)
\]
\end{theorem}
\begin{corollary}\label{baxter-cor1}
Let $\cal{C}$ be a Wiener-Hopf operator and $\{Q_n, R_n\}$--
sequence of projections obeying (\ref{baxter-algebra}) with $Q_n
x\to x$ for any $x\in P_+(X)$. If they also satisfy conditions of
Baxter's Lemma, then
\[
(Q_n \cal{C} Q_n)^{-1}Q_n x- Q_n\cal{T}^{-1}x\to 0, x\in P_+(X)
\]
\end{corollary}

Now, let us apply these mainly algebraic results to the concrete
situation. Let $H(x)\in L^1(\mathbb{R})$-- Hermitian function and
$1+\nu(\lambda)>0$, where $\nu(\lambda)$-- Fourier transform of
$H$. Let $X$ be $L^1(\mathbb{R})$, $\cal{C}f=f+H\ast f$,
$\left[P_+f\right](x)=\chi_{\mathbb{R}^+}(x)f(x)$. The function
$\nu(\lambda)\in W(\mathbb{R})$ and $1+\nu(\lambda)>0$. Therefore,
by general result from the Wiener algebra theory, we have
$\hat{g}(\lambda)=\ln(1+\nu(\lambda))\in W(\mathbb{R})$ so
\[
\hat{g}=\int\limits_{-\infty}^\infty g(x)\exp(i\lambda
x)dx=\int\limits_{-\infty}^0 g(x)\exp(i\lambda
x)dx+\int\limits_{0}^\infty g(x)\exp(i\lambda x)dx=\hat{g}_-
+\hat{g}_+
\]
Notice that $\hat{u}=\exp(\hat{g}_+)-1\in W_+$ and
$\hat{l}=\exp(\hat{g}_-)-1\in W_-$. Therefore,
$\cal{C}=\cal{L}\cal{U}$, where $\cal{L}f=f+l\ast f,
\cal{U}f=f+u\ast f$. Both operators $\cal{U}$ and $\cal{L}$ are
invertible and one can easily check that $\cal{U}, \cal{U}^{-1}$
are upper triangular, $\cal{L}, \cal{L}^{-1}$ -- lower triangular
(notice that at the moment the definition of upper(lower)
triangular operator is different from what we used in the section
on factorization of integral operators). Therefore, the
Wiener-Hopf theorem is applicable to the operator
$I+\cal{H}_\infty=P_+\cal{C} P_+$. Let
$\Gamma(x)=(I+\cal{H}_\infty)^{-1}H(x)$.

Then, consider the following projections
$\left[Q_nf\right](x)=\chi_{[0,n]}(x)f(x)$ and
$\left[R_nf\right](x)=\chi_{[n,\infty]}(x)f(x)$. The result below
is what we use in the proof of continuous analog of Baxter's
Theorem for OPUC. Recall that $I+\cal{H}_r$ is given by
(\ref{intop}) and can be regarded as an operator from $L^1[0,r]$
to $L^1[0,r]$ due to Young's inequality.
\begin{corollary}\label{baxter-cor2}
If $n>n_0$, then $\|(I+\cal{H}_n)^{-1}\|_{L^1[0,n], L^1[0,n]}<C$.
Moreover, $\|\Gamma_n(0,x)-\Gamma(x)\|_{L^1[0,n]}\to 0$, where
$\Gamma_n(0,x)=(I+\cal{H}_n)^{-1}\left[ \chi_{[0,n]}(x)\cdot
H(x)\right]$.
\end{corollary}
\begin{proof}
Indeed, for $n$ large enough, all conditions of Baxter's lemma are
satisfied which yields the necessary estimates on the norms. Then,
in the Corollary \ref{baxter-cor1}, take $x=H(t)$. All conditions
of Corollary \ref{baxter-cor1} are satisfied and we get the second
statement on convergence.
\end{proof}

\subsection{For section \ref{sect-dirac}}

The next Lemma shows that the spectral measure for Dirac operator
is uniquely defined.
\begin{lemma}\label{uniqueness}
The spectral measure $d\sigma_d$ for Dirac operator $\cal{D}$ is
unique.
\end{lemma}
\begin{proof}
Indeed, if $\tau(\lambda)$ is another spectral measure, then
 Lemma \ref{transformation-dirac} yields that
 \begin{equation}
\hat{f}(\lambda)= \int\limits_{-r}^r f(x)\exp(i\lambda x)dx\in
 L^2(\mathbb{R},d\tau)\label{scale-new}
 \end{equation}
for any $f(x)\in L^2[-r,r]$, $r>0$. By taking
$f(x)=\chi_{[0,R]}(x)\cdot \exp(-x)$ ($R$ is large), we have
\begin{equation}\label{apriory-33}
\int_{\mathbb{R}} \frac{d\tau(\lambda)}{\lambda^2+1}<\infty
\end{equation}
Moreover, from the definition of the spectral measure,
\[
\int_{\mathbb{R}} |\hat{f}(\lambda)|^2
d(\sigma_d(\lambda)-\tau(\lambda))=0
\]
Using (\ref{apriory-33}), we can approximate $\chi_{[a,b)}$ by
functions (\ref{scale-new}) in both
$L^2(\mathbb{R},d\sigma_d(\lambda))$ and
$L^2(\mathbb{R},d\tau(\lambda))$. We have
\[
\int\limits_a^b d(\sigma_d(\lambda)-\tau(\lambda))=0
\] for any $a$ and $b$. That implies $d\tau=d\sigma_d$.
\end{proof}

The following result is quite elementary. It is used in the proof
of existence of wave operators for Dirac operator with square
summable potential.

\begin{lemma}For the unperturbed operator
$\cal{D}_{0}$, the action of the group $e^{-it\cal{D}_{0}}$ is
given by the formulas
\begin{equation*}
\mathit{e}^{-it\cal{D}_{0}}\left[
\begin{array}{c}
f \\
0%
\end{array}%
\right] \mathit{=}\frac{1}{2}\left[
\begin{array}{c}
f(x+t)+f(x-t) \\
-i(f(x-t)-f(x+t))%
\end{array}%
\right] \mathit{,}
\end{equation*}
where $f(x)\in L^{2}(\mathbb{R}^+)$ is extended to the whole line
as an even function.

\begin{equation*}
\mathit{e}^{-it\cal{D}_{0}}\left[
\begin{array}{c}
0 \\
f%
\end{array}%
\right] \mathit{=}\frac{1}{2}\left[
\begin{array}{c}
-i(f(x+t)-f(x-t)) \\
f(x-t)+f(x+t)%
\end{array}%
\right] \mathit{,}
\end{equation*}
$f(x)$ is extended to $\mathbb{R}$ as an odd
function.\label{evolution}
\end{lemma}

\begin{proof} One can use the definition of $\exp(-it\cal{D}_0)$
to verify this statement directly. Another way to see that is to
use the spectral resolution for $\cal{D}_{0}$.

\begin{equation}
\left[
\begin{array}{c}
f_{1} \\
f_{2}%
\end{array}%
\right] \rightarrow F (\lambda )=\int\limits_{0}^{\infty
}f_{1}(x)\cos (\lambda x)dx+\int\limits_{0}^{\infty }f_{2}(x)\sin
(\lambda x)dx,
\end{equation}

\begin{eqnarray*}
f_{1}(x) =\frac{1}{\pi }\int\limits_{-\infty }^{\infty } F
(\lambda )\cos (\lambda x)d\lambda , \quad f_{2}(x) =\frac{1}{\pi
}\int\limits_{-\infty }^{\infty } F (\lambda )\sin (\lambda
x)d\lambda .
\end{eqnarray*}
Therefore,
\begin{eqnarray*}
\left( e^{-it\cal{D}_{0}}{f}\right) _{1}(x) =\frac{1}{\pi
}\int\limits_{-\infty }^{\infty }e^{-it\lambda } F
(\lambda )\cos (\lambda x)d\lambda , \\
 \left(
e^{-it\cal{D}_{0}}{f}\right) _{2}(x) =\frac{1}{\pi
}\int\limits_{-\infty }^{\infty }e^{-it\lambda } F (\lambda )\sin
(\lambda x)d\lambda .
\end{eqnarray*}
It now suffices to apply the Fourier inversion formula.\end{proof}

\subsection{For section \ref{sect-ss}}
The following Lemma proves one simple property of the exponential
map on the $H^{1/2}(\mathbb{R})$ functions.
\begin{lemma} \label{ss-l2}
If $\gamma(x)\in H^{1/2}(\mathbb{R})$, then $e^{\gamma(x)}-1\in
L^2(\mathbb{R})$ and this map is continuous.
\end{lemma}
\begin{proof}
We have
\[
e^{\gamma(x)}=1+\sum\limits_{n=1}^\infty \frac{\gamma^n(x)}{n!},
\|\gamma^n(x)\|_2=\|\hat\gamma *\ldots*\hat\gamma\|_2
\]
\[
\|\hat\gamma(\omega)\|_p\leq
C\left(\frac{2-p}{2p-2}\right)^{(2-p)/(2p)}
\|\gamma\|_{H^{1/2}(\mathbb{R})}, 1<p<2
\]
by Holder's inequality. By Young's inequality we now have
\[
\|\hat\gamma *\ldots*\hat\gamma\|_2\leq \|\hat\gamma\|_{p_n}^n,
p_n=2n(2n-1)^{-1}
\]
So, $\|\gamma^n(x)\|_2\leq
Cn^{n/2}\|\gamma\|^n_{H^{1/2}(\mathbb{R})}$. The Stirling formula
for factorial yields convergence of the series and continuity of the
exponential map.
\end{proof}

\end{document}